\documentclass[11pt]{amsbook}

\usepackage{color}
\usepackage{pgfpages}
\usepackage{amsfonts,amssymb}
\usepackage[T1]{fontenc}
\usepackage[utf8]{inputenc}
\usepackage{comment}
\usepackage{amsmath}
\usepackage{amsthm}
\usepackage{amsrefs}
\usepackage{qsymbols}
\usepackage{latexsym}
\usepackage{chngcntr}
\usepackage[noadjust]{cite}
\usepackage{paralist}
\usepackage{esint}

\newtheorem{theorem}{\bf Theorem}[chapter]

\newtheorem{lemma}[theorem]{\bf Lemma}
\newtheorem{proposition}[theorem]{\bf Proposition}
\newtheorem{corollary}[theorem]{\bf Corollary}
\theoremstyle{definition}
\newtheorem{definition}[theorem]{\bf Definition}
\newtheorem{remark}[theorem]{\bf Remark}
\newtheorem{convention}[theorem]{\bf Convention}

\newtheorem{assumption}[theorem]{\bf Assumption}
\numberwithin{subsection}{chapter}

\newcommand{\IR}{\mathbb{R}}
\newcommand{\IC}{\mathbb{C}}
\newcommand{\IN}{\mathbb{N}}
\newcommand{\IZ}{\mathbb{Z}}

\newcommand{\IP}{\mathbb{P}}

\newcommand{\IF}{\mathbb{F}}
\newcommand{\IE}{\mathbb{E}}

\newcommand{\IS}{\mathbb{S}}

\newcommand{\Lop}{{\mathcal{L}}}
\newcommand{\cS}{{\mathcal{S}}}
\newcommand{\IT}{{\mathbb{T}}}
\newcommand{\Sec}{\mathrm{S}}
\newcommand{\loc}{\mathrm{loc}}

\newcommand{\cL}{\mathcal{L}}
\newcommand{\cM}{\mathcal{M}}

\newcommand{\with}{\quad\hbox{with}\quad}
\newcommand{\andf}{\quad\hbox{and}\quad}

\newcommand{\cA}{\mathcal{A}}
\newcommand{\cB}{\mathcal{B}}

\newcommand{\cH}{\mathcal{H}}
\newcommand{\cC}{\mathcal{C}}
\newcommand{\cX}{\mathcal{X}}
\newcommand{\cT}{\mathcal{T}}
\newcommand{\cR}{\mathcal{R}}

\newcommand{\LL}{\mathrm{L}}
\renewcommand{\H}{\mathrm{H}}

\newcommand{\B}{\mathrm{B}}
\newcommand{\C}{\mathrm{C}}
\newcommand{\W}{\mathrm{W}}
\newcommand{\X}{\mathrm{X}}
\newcommand{\Y}{\mathrm{Y}}
\newcommand{\cF}{\mathcal{F}}
\newcommand{\cY}{\mathcal{Y}}

\renewcommand{\d}{\mathrm{d}}
\newcommand{\ii}{\mathrm{i}}
\newcommand{\e}{\mathrm{e}}
\newcommand{\eps}{\varepsilon}

\DeclareMathOperator{\supp}{supp}

\DeclareMathOperator{\sgn}{sgn}

\DeclareMathOperator{\divergence}{div}

\renewcommand{\Re}{\operatorname{Re}}
\renewcommand{\Im}{\operatorname{Im}}
\DeclareMathOperator{\Id}{Id}
\newcommand{\Rg}{\mathcal{R}}

\newcommand{\dom}{\mathcal{D}}

\def\ov{\overline}

\def\wh{\widehat}
\def\wt{\widetilde}
\def\ddj{\dot\Delta_j}
\def\ddk{\dot\Delta_k}

\def\du{\delta\!u}
\def\df{\delta\!f}
\def\dP{\delta\!P}

\def\dg{\delta\!g}
\def\dh{\delta\!h}
\DeclareMathOperator{\adj}{adj}
\newcommand{\E}{\mathrm{E}}
\def\dv{\delta\!v}

\newcommand{\bA}{{\bf A}}

\newcommand{\cb}{\color{black}}

\numberwithin{equation}{chapter} 

\title{Free Boundary Problems by \\ Da Prato -- Grisvard Theory}

\author{Rapha\"el Danchin \\ Matthias Hieber \\ Piotr Bogus\l aw Mucha \\
 Patrick Tolksdorf \\ ~ \\}

\begin{document}

\begin{abstract} An $\LL_1$-maximal regularity theory for parabolic evolution equations inspired by the pioneering work of Da Prato and Grisvard~\cite{DaPrato_Grisvard} is developed.  
Besides of its own interest, the approach yields a framework allowing  {\em global-in-time} control of the change of  Eulerian to Lagrangian coordinates in various problems related to fluid mechanics. 
This  property is of course decisive for free boundary problems. This concept is illustrated by the analysis of the free boundary value problem describing  the  motion of  viscous, incompressible 
Newtonian fluids without surface tension and, secondly, the motion of compressible pressureless gases. 
 
For this purpose, an endpoint maximal $\LL_1$-regularity approach to the Stokes and Lam\'e systems is developed. It is applied then to establish 
global, strong well-posedness results for the free boundary problems described above in the case where the initial domain coincides with the half-space, and the initial velocity is  
small with respect to  a suitable scaling invariant norm.
\end{abstract}

\maketitle

\cleardoublepage
\setcounter{page}{1}
\pagenumbering{roman}
\tableofcontents

\chapter{Introduction}


This memoir aims at providing the reader with an abstract machinery that allows to solve global-in-time regularity issues for nonlinear parabolic problems in an adapted functional framework.
To illustrate the efficiency of this machinery, we prove the global-in-time existence of strong solutions to free boundary problems with infinite depth and small initial data, a question that has been left open for about 30 years.

Our approach is part of the so-called maximal regularity techniques, which consist in recovering the full regularity information of the solution provided by the data. 
Striving for {\em maximal elliptic regularity} goes back to 
J.~Schauder's works   on the Laplace operator (supplemented with simple boundary conditions)  
 in the  30s~\cite{Scha34}.
The {\em systematic} investigation of linear elliptic problems subject to general boundary conditions began in the late 1950's  with the works of  Y. B. Lopatinskii~\cite{Lop53} and 
Z. Shapiro~\cite{Sha53}. They gave for the first time a characterization of those  classes of problems for which solvability results and 
a priori estimates can be obtained. They showed  that, 
whenever  $\cA(x,D)$ is a differential  operator of order $2m$ and $(\cB_j(x,D))_{1\leq j\leq m},$ 
an $m$-tuple of  boundary operators of order $m_j<2m,$
the solvability of the system
\begin{align*}
\begin{aligned}
\cA(x,D)u &= f && \text{in}\; \Omega \\
\cB_j(x,D)u &= g_j && \text{on}\; \partial \Omega ,\quad j=1 , \dots , m,
\end{aligned}
\end{align*}
in some smooth domain $\Omega$  of $\IR^n$ is essentially equivalent to 
 some  algebraic condition involving the symbols of operators $\cA(x,D)$  and  $\cB_j(x,D).$  

Nowadays, this condition is called the `complementing condition' or the `Lopatinskii--Shapiro condition'. This result initiated the study of such  systems 
in various function spaces, such as H\"older and Lebesgue spaces and culminated
in the celebrated paper of Agmon, Douglis and Nirenberg~\cite{ADN59}. The approach of~\cite{ADN59}
was based on representations of solutions in the case of constant coefficient operators, 
which were given in terms of so-called Poisson kernels.  A decisive advantage of maximal regularity results 
is that they are stable by perturbation. Consequently, for a smooth enough domain 
$\Omega$ and  under suitable regularity assumptions 
for the  coefficients  of $\cA$ and $\cB_j,$ one can go from 
the aforementioned results in the half-space (and whole space) 
for constant coefficients operators, to the general case, getting  estimates  of the type
$$\begin{aligned}
\| u\| _{\C^{2m,\alpha }(\overline{\Omega})} &\leq C \Big( \| f\| _{\C^{0,\alpha }(\overline{\Omega})}+ \| u\| _{\C^{0,\alpha }(\overline {\Omega})}  + \sum_j\| g_j\| _{\C^{2m-m_j,\alpha }(\partial \Omega)}\Big) \quad \mbox{ and }\\
\| u\| _{\H^{2m}_p(\Omega)} &\leq C \Big( \| f\| _{\LL_p(\Omega)} + \| u\| _{\LL_p(\Omega)} + \sum_j\| g_j\| _{\H_p^{2m-m_j-\frac{1}{p}}(\partial \Omega)}\Big)\cdotp
\end{aligned}$$
These estimates can be used to derive existence theorems in various function spaces and allow also to 
prove further results like, e.g., the Fredholm property. One can refer here to the work of Browder~\cite{Bro59}, the book of Lions and Magenes~\cite{LM68} 
as well as to Ladyzenskaja-Uralceva~\cite{LadU64} and Solonnikov~\cite{Sol65} and the references given therein. 

The investigations to generalize these a priori estimates to the case of systems gave rise to many notions of elliptic systems, 
such as Legendre-Hadamard elliptic, determinant-elliptic, normally elliptic and elliptic in the sense of Agmon, Douglis, Nirenberg. Note that these are all finite-dimensional concepts.
A very general approach was developed in~\cite{DHP03}, where  parameter-elliptic operators with coefficients taking values in the space of all bounded operators acting on an arbitrary 
Banach space $X$ were introduced, hereby unifying several notions of ellipticity for $N\times N$-systems.    

The development of the corresponding evolutionary  parabolic theory can be described as follows: introducing a parameter-dependent Lopatinskii-Shapiro condition for $\lambda $ belonging to a 
suitable sector of the complex plane, one sees that for all $f \in \LL_p(\Omega;\IC),$ there exists a unique $u \in  \H_p^{2m}(\Omega;\IC)$ satisfying
\begin{align*}
\begin{aligned}
(\lambda + \cA(x,D))u &= f && \text{in}\; \Omega\\
\cB_j(x,D) u &= 0 && \text{on}\; \partial \Omega,\quad j=1 , \dots , m,
\end{aligned}
\end{align*}
as well as a priori estimates similar to the ones described above. Defining the 
$\LL_p(\Omega;\IC)$-realization 
of the above boundary value problem to be
\begin{equation*}
A_B u:=\cA(x,D)u, \qquad \dom(A_B):=\{u \in \H_p^{2m}(\Omega;\IC):\cB_j(x,D) u =0\; \text{for all}\; j=1 , \dots , m \}, 
\end{equation*}
it follows that $-A_B$ generates a strongly continuous analytic semigroup on $\LL_p(\Omega;\IC)$ for all $p\in(1,\infty)$;
see, e.g., Agmon~\cite{Agm62}, Friedman~\cite{Fri64}, Amann~\cite{Ama90} and  Tanabe~\cite{Tan97}. 
\medbreak
Let us consider an arbitrary linear and closed operator
$A:{\mathcal D}(A)\subset X\to X$ in a 
Banach space $X,$ and the evolution equation
\begin{equation}\label{eq:parabolic}
\begin{aligned}
 u'(t) + Au(t) &= f(t), && t > 0\\
 u(0)&=0.\end{aligned}\end{equation}  
Getting maximal parabolic regularity results  for~\eqref{eq:parabolic} within the $\LL_q$-scale for $1<q< \infty$, i.e.,  
 estimates of the form
\begin{equation}\label{eq:mrp}\| u'\|_{\LL_q(\IR_+ ; X)} + \| A u \|_{\LL_q(\IR_+ ; X)} \leq C\| f\|_{\LL_q(\IR_+ ; X)}\end{equation}
is a delicate matter.
Combined with interpolation theory, such results provide a very powerful tool for solving
large classes of nonlinear parabolic systems~\cite{Ama90, Lunardi, PS16}, including free boundary value problems.  Classical results in this  direction can be found  in 
\cite{LSU64} in  the $\LL_p$-$\LL_q$-setting for second-order equations with continuous coefficients, subject to  Dirichlet boundary conditions. 
Using heat-kernel estimates, off-diagonal estimates or weak reverse H\"older inequalities, these results can be extended~\cite{HP97, Egert, Tolksdorf_elliptic} to the case of rough coefficients and  boundaries.

Putting together  a theorem  due to Dore and Venni~\cite{DV87}, which says that an estimate of the above form holds true provided the operator $A_B$ admits bounded imaginary powers of an angle 
less than $\pi/2$  and  a  result of  Seeley~\cite{See67} on imaginary powers of the operator $A_B$ where all the coefficients are  $\C^\infty$-functions, one obtains maximal 
$\LL_p$-$\LL_q$-regularity results for this class of scalar-valued boundary value problems for all $1<p,q<\infty$.  

Consider  a Banach space $X$ of class $\cH\cT$,
that is to say such that the Hilbert transform 
is bounded on $\LL_q(\IR;X)$ for some (all) $q \in (1,\infty)$ (see~\cite{Ama95,DHP03,KW04,HNVW17} for more information).
For a sectorial operator $A$ on $X,$ of angle $\phi_A<\pi/2,$  it was shown by  Weis~\cite{Wei01} that the maximal $\LL_q$-regularity estimate~\eqref{eq:mrp} is equivalent to the 
{$\cR$-sectoriality} of $A$. Here an operator $A$ is called
{\em $\cR$-sectorial of angle $\phi \in (0,\pi)$} provided $\sigma(A) \subset \overline{\Sigma_\phi}$
and for all $\phi'>\phi$ the set 
$$
\{\lambda R(\lambda,A): \phi' \leq |\arg(\lambda)| \leq \pi\} 
$$ 
is $\cR$-bounded. The infimum over all $\phi$ for which the above set is $\cR$-bounded is called   
the {\em $\cR$-angle} \index{$\cR$-angle} of $A$ and is denoted by $\phi_A^\cR$. 

In fact, let $A$ be the generator of a bounded analytic semigroup on a 
Banach space $X$ of class $\cH\cT$. Then, the characterization theorem due to Weis~\cite{Wei01} says that  
$A$ has maximal $\LL_q$-regularity for some (all) $q \in (1,\infty)$ on $\IR_+$ if and only if 
$-A$ is $\cR$-sectorial of angle $\phi_A^\cR < \pi/2$.

Striving, however, for estimates which are useful in nonlinear problems, one is forced to 
look for minimal smoothness assumptions on the coefficients. In~\cite{DHP03}, the concept of $\cR$-boundedness is employed to obtain the following result  
for general systems subject to variable, $\cL(E)$-valued coefficients whenever $E$ is of class $\cH\cT.$ Assume that 
\begin{enumerate}
 \item[(i)] the top-order coefficients of $\cA(x,D)$ belong to $\C(\Omega ; \cL(E))$,
 \item[(ii)] the coefficients of the boundary operators $\cB_j(x,D)$ belong to $\C^{2m-m_j}(\partial \Omega;\cL(E))$,
 \item[(iii)] the principal symbol of $\cA$ is parameter-elliptic of angle $\phi_\cA$,
 \item[(iv)] the Lopatinskii--Shapiro condition holds.
\end{enumerate}
Then, if the angle of ellipticity satisfies $\phi_{\cA}<\frac \pi 2$, the parabolic initial boundary value problem
\begin{align}\label{eq:pivp}
 \left\{\begin{aligned}
\partial_t u + A_B u  &= f, \quad t>0, \\
 u(0) &= 0 
 \end{aligned} \right.
\end{align}
has the property of maximal regularity in $\LL_q(\IR_+; \LL_p(\Omega;E))$ for all  $p,q\in (1,\infty)$.

One may wonder whether the normal ellipticity condition on $\cA$ and the Lopatinskii--Shapiro condition on $(\cA, \cB_1,\ldots, \cB_m)$ are necessary for getting the 
above  $\LL_p$-$\LL_q$-estimates.  It was proved in~\cite{DHP} that this is indeed the case. 
At this point,  one might also ask whether {\em all} generators of analytic semigroups on $\LL_p(\Omega,\mu)$, where $1<p<\infty$ and $(\Omega,\mu)$ denotes a measure space possess the property of maximal 
$\LL_p$-$\LL_q$-regularity. Kalton and Lancien~\cite{KL99} gave a complete answer to this difficult question by showing that if $E$ has an unconditional basis and all generators of analytic semigroups on $E$  have the 
maximal $\LL_q$-regularity property, then $E$ has to be isomorphic to a Hilbert space.  
\smallbreak
When dealing with nonlinear problems, it is helpful to  formulate the maximal regularity property of the solution of an evolution equation in terms of an {\em isomorphism between data and solution spaces}.
More precisely, consider again  the equation~\eqref{eq:parabolic}
but with nonzero data, and define $\IF:=\LL_q(\IR_+;E)$ and $\IE:=\H^1_q(\IR_+;E) \cap \LL_q(\IR_+;\dom(A))$. Then, given $q \in (1,\infty)$,
and assuming that $-A$ has maximal $\LL_q$-regularity and $0 \in \rho(A)$, it follows that 
$$  \Big( \frac{d}{dt} + A, {\rm tr} \Big) \in \mbox{Isom}(\IE,\IF \times X_\gamma), $$
where $X_\gamma$ stands for the real interpolation space $(X , \dom(A))_{1-1/q,q}$ and ${\rm tr}(u):=u(0)$ denotes the trace operator.

In~\cite{DHP}, a characterization of optimal $\LL_p$-$\LL_q$-regularity for solutions of general  boundary value problems of the form   
\begin{align}\label{1-1}
 \left\{ \begin{aligned}
\partial_t u + \cA(x,D)u & = f(t,x), && t\in J,\; x\in \Omega,\\
\cB_j(x,D) u & = g_j(t,x), && t\in J,\; x\in\partial \Omega,\;j=1,\ldots,m,\\
u(0,x) & = u_0(x), && x\in \Omega
\end{aligned} \right.
\end{align}
in terms of the data was given for the case where the principal symbol of $\cA$  is parameter 
elliptic and  the Lopatinskii--Shapiro condition holds. 
Moreover, it was shown in~\cite{DDHPV04} that $A_B$ admits  a bounded $\H^\infty$-calculus on $\LL_p(\Omega;E)$  for $1<p<\infty$ provided the top-order coefficients are H\"older continuous. This allows to 
characterize the domains of the fractional powers $A_B^\alpha$ of $A_B$ as the complex interpolation spaces $[X,\dom(A_B)]_\alpha$, which is important for the nonlinear theory.
\smallbreak
This article focuses on  the limit case of maximal $\LL_1$-regularity, that
is when $q=1.$ In order to motivate our study, let us resume, for the time being,   
to the basics of the  description of solutions to PDEs that model the evolution of particles or elements/agents in a physical situation. 
Being interested in describing the long time behavior of solutions, any information about possible stabilization effects of solutions is very much desirable. 

A key example in this respect can be  represented within the framework of classical mechanics by Lagrangian coordinates:
 the solution  is described  by variables that follow 
each particle trajectory  knowing its initial  position $\xi.$
By contrast,  within the Eulerian coordinates system, the  position of a  particle {\cb at time  $t$} 
 is given by $(t,x)$, where $x$ denotes the space variable. 
 According to basic kinematic laws, this alternative viewpoint requires the knowledge of  the velocity of each particle. 
 More precisely, 
let $X_u (t , \xi) \in \IR^n$ denote the position of a fluid particle at time $t$ which was  located in  $\xi \in \IR^n$ at initial time $t_0.$ If $u (t , \xi)$ denotes the velocity 
field of this particle, then  the particle trajectory is given by
\begin{align*}
 X_u(t , \xi) = \xi + \int_{t_0}^t u(s , \xi) \; \d s.
\end{align*}
The Lagrangian variables are well suited to study free boundary  problems, as in these variables, the underlying, originally moving, domain is the initial domain and thus does \textit{not} depend on the 
time variable.  In fact, for a number of free boundary problems, combining classical maximal regularity tools and Lagrangian coordinates already allows to get short time existence of strong 
solutions for smooth enough data and long-time or global existence for small data.

In this direction,  one may mention the {\em Stefan problem} modelling the melting of ice in water. Another prominent problem is the \textit{one-} or 
 {\em two-phase Navier-Stokes} problem describing, e.g., the motion of a droplet of oil in water. 

To have a well-defined coordinate transform between Eulerian and Lagrangian variables at hand, it is fundamental that, for each $t \geq t_0,$ the mapping
\begin{align*}
 \xi \mapsto X_u(t , \xi)
\end{align*}
is invertible and that, when measured in suitable norms, does not develop singularities. Employing the inverse function theorem shows that the invertibility is ensured if the Jacobian of 
$X_u(t , \cdot)$ is invertible. This Jacobian is given by
\begin{align}
\label{Eq: Jacobian of X}
 D_{\xi} X_u (t , \xi) = \Id + \int_{t_0}^t D_{\xi} u (s , \xi) \; \d s.
\end{align}
Clearly, a way to obtain a quantitative control on all norms of derivatives of the transformation $\xi \mapsto X_u(t , \xi)$ is to use a Neumann-series argument to invert the 
Jacobian of $X_u(t , \cdot)$. Moreover, if one desires to prove global existence for small initial data, it is highly desirable to have a \textit{global} coordinate transform at hand. This is indeed the case if one can find  a constant $c \in (0 , 1)$ such that for all $t > t_0,$ one has
\begin{align}
\label{Eq: The L1 integral}
 \Big\| \int_{t_0}^t D_{\xi} u (s , \xi) \; \d s \Big\|_{\LL_{\infty}} \leq c.
\end{align}
One directly sees that whether one has a global Lagrangian coordinate transform at disposal is intimately connected to a \textit{bound in $\LL_1(\IR_+;\LL_\infty(\Omega))$} of the Jacobian of 
the Lagrangian fluid velocity.  Now, if~\eqref{Eq: The L1 integral} follows from  a
 $\LL_q$-in-time estimate for $u$ for some $q > 1$,  then one  obtains a $t$-dependent bound, which can only be balanced if further 
sufficiently strong decay properties (exponential decay, typically) of solutions  is known. Unfortunately, some very simple geometric configurations like the half-space do not allow for 
having more than algebraic decay estimates. 
\smallbreak
Summarizing, to establish an efficient method for proving global-in-time
results for free boundary problems, we see that it would be desirable to have an {\em $\LL_1$-maximal regularity theory} at hand. It is, however, known that estimate 
\eqref{eq:mrp} fails to be true for $q=1$, whenever $X$ is reflexive. On the other hand,   {\em Da Prato -- Grisvard theorem}~\cite{DaPrato_Grisvard} tells us 
that~\eqref{eq:mrp} holds for all $1\leq q \leq \infty$ provided
the Banach space $X$ is replaced with a suitable interpolation space. In fact, let $A$ be the generator of a bounded analytic semigroup on a Banach space $X$ with domain $\dom(A)$, $0 \in \varrho(A)$, $\theta \in (0 , 1)$ and $1 \leq q < \infty$. Then, 
there exists a constant $C > 0$ such that for all $f \in \LL_q(\IR_+;\dom_{A}(\theta,q))$, the function $u$ given for $t>0$ by 
$$u(t):= \int_0^t \e^{(t - s)A}f(s)\, \d s$$ satisfies
\begin{align*}
 \| Au \|_{\LL_q(\IR_+; \dom_{A}(\theta , q))} \leq C \| f \|_{\LL_q(\IR_+; \dom_{A}(\theta , q))}.
\end{align*} 
Here  $\dom_{A}(\theta,q)$ is defined by 
\begin{align*}
 \dom_{A}(\theta , q) := \bigg\{ x \in X : [x]_{\theta,q}:= \Big(\int_0^\infty\|t^{1 - \theta} A \e ^{tA} x\|_X^q \; \frac{\d t}{t}\Big)^{1/q} < \infty \bigg\}\cdotp
\end{align*} 
When equipped with the norm $\|x\|_{\theta,q}:=\|x\|_X + [x]_{\theta,q}$, the space $\dom_{A}(\theta,q)$ becomes a Banach space. It is known  that 
$\dom_{A} (\theta , q)$ coincides with the real interpolation space $(X , \dom(A))_{\theta , q}$ and that the respective norms are equivalent. 
Furthermore, if $0 \in \varrho(A)$, then the real interpolation space norm is 
equivalent to the homogeneous norm $[\cdot]_{\theta,q}$. 

In this context, we refer to a  result due to Dore~\cite{Dor99} saying that a sectorial operator $A$ satisfying $0 \in \varrho(A)$ admits a 
bounded $\H^\infty$-calculus on $\dom_{A}(\theta,q)$ for $\theta \in (0,1)$ and $1\leq q \leq \infty$. 
It was also shown by Dore~\cite{Dor01} that if $0 \in \sigma(A)$, then $A$ admits a bounded $\H^\infty$-calculus on the real interpolation space between $X$ and $\dom(A) \cap \cR(A)$.

The problems we have in mind to apply the  Da Prato -- Grisvard approach to obtain global Lagrangian coordinate transforms  are the limit cases,  where the corresponding 
semigroups do \textit{not} have exponential decay.  Since the classical Da Prato -- Grisvard theorem gives only a global-in-time maximal $\LL_q$-regularity estimate for invertible generators, we  
present first an alternative \emph{homogeneous} formulation of this result. More precisely, we will show that if one endows the space $\dom_A (\theta , q)$ with the \textit{homogeneous} 
norm $[\cdot]_{\theta , q}$, then a global maximal $\LL_q$-regularity result can be obtained whenever $1 \leq q \leq \infty$ 
\emph{without assuming invertibility} of the generator of the associated semigroup.

In contrast to the inhomogeneous case where $\dom_A (\theta , q)$ could be identified as the real interpolation space $(X , \dom(A))_{\theta , q}$, in the homogeneous case one can recover the 
norm $[\cdot]_{\theta , q}$ as a real interpolation space norm between $X$ and the domain of a \textit{homogeneous} version of the operator $A$; see for example the book~\cite[Sec.~6.3-6.4]{Haase} by Haase 
or the survey by Kunstmann~\cite{kunstmann21}
for a presentation of this homogeneous interpolation theory. In comparison to~\cite{Haase}, we will provide here another approach to these spaces, which will give us more 
flexibility in the choice of the underlying function spaces. 

If $A$ denotes a differential operator on $\LL_p$ (or a subspace of $\LL_p$) of order $2 m$ without lower-order terms  then, formally, the domain of the homogeneous version of $A$ should be a 
homogeneous Bessel potential space of order $2 m$ that additionally incorporates boundary values and/or other (algebraic) conditions like solenoidality. A computation of the  interpolation space 
required by the homogeneous version of the Da Prato -- Grisvard theorem then \textit{naturally} leads to the \textit{homogeneous Besov spaces} $\dot \B^{2 m \theta}_{p , q}$ as a ground space for 
maximal $\LL_q$-regularity. Notice that the last index $q$ is the same as the integrability index in the maximal regularity result and that also boundary conditions and other side conditions have to 
be taken into account in rigorous results. 

It is well-known that, in the whole space $\IR^n,$  the \textit{only} homogeneous Besov space that embeds into $\LL_{\infty}$ is the space $\dot \B^{n / p}_{p , 1}$. Thus, recalling the discussion on the Lagrangian coordinate 
transform, we are aiming for a functional setting that guarantees $\nabla u \in \dot \B^{n / p}_{p , 1}$. As the gradient operator allows to change the differentiability scale on 
homogeneous Besov spaces, we thereby find that $\nabla^2 u \in \dot \B^{n / p - 1}_{p , 1}$ and thus realize that $\dot \B^{n / p - 1}_{p , 1}$ is the right  choice of a ground space for maximal $\LL_1$-regularity 
if for example the Stokes operator is considered. 

It is astonishing, that the last index being $1$ exactly 
ensures both $\nabla u$ to be in $\LL_{\infty}$ in space \textit{and} $\LL_1$ in time. For related results on endpoint maximal $\LL_1$-regularity or associated Besov spaces we refer to~\cite{Cannone, chemin, kunstmann21, ogawa-shimizu-jde}. This shows that the Da Prato -- Grisvard theorem on maximal regularity 
combined with the study of homogeneous Besov spaces is the right tool to ensure the existence of a global-in-time Lagrangian coordinate transform.

Another  difficulty when investigating free boundary  or moving domain problems arising for incompressible fluids is the condition $\divergence v =0$, which needs to be incorporated 
in the underlying function spaces. This procedure is rather well understood within the $\LL_p$-setting for $1<p<\infty$ and various domains $\Omega$ including the half-space $\IR_+^n$, 
the bent half-space, bounded or exterior domains with smooth boundaries subject to various boundary conditions. Let us refer to the above domains as \emph{standard domains},  and let $\IP$ be the 
Helmholtz projection from $\LL_p(\Omega) \to \LL_{p , \sigma} (\Omega)$, where $$\LL_{p , \sigma} (\Omega):= \overline{\{u \in \C_c^\infty(\Omega ; \IC^n): \divergence\; u=0\;\mbox{ in } \Omega\}}^{\|\cdot\|_{\LL_p}}.$$ 
Solonnikov's~\cite{Sol77} classical result  for {\em Dirichlet boundary conditions}, (see Abels~\cite{Abe05} for $n=2$) reads as follows: let $n\ge 2$, $J=(0,T)$ for some $T>0$, $1<p<\infty$ and 
assume that $\Omega\subset\IR^n$ is a bounded domain of class $\C^3$. Then, the Stokes operator defined by 
$$
A_{p}u:=- \IP\Delta u, \quad \dom(A_{p}):= \W^2_p(\Omega) \cap \W^1_{p , 0}(\Omega)\cap \LL_{p , \sigma} (\Omega)
$$
admits maximal $\LL_q$-regularity on $\LL_{p , \sigma} (\Omega)$
for all $q\in(1,\infty).$ In particular, the solution $u$ to the Cauchy problem 
$$ 
u'(t) + A_p u(t) = f(t), \,\,  t>0, \quad u(0) = u_0, 
$$
satisfies the estimate
$$
\|u'\|_{\LL_q(J;\LL_p(\Omega))} + \|A_p u\|_{\LL_q(J;\LL_p(\Omega))} \leq C( \|f\|_{\LL_q(J;\LL_p(\Omega))} + \|u_0\|_{X_\gamma}),
$$
for some $C>0$ independent of $f \in \LL_q(J;\LL_{p , \sigma} (\Omega))$ and $u_0 \in X_\gamma:= (\LL_{p , \sigma} (\Omega),\dom(A_p))_{1-1/q,q}$. For global-in-time maximal regularity estimates for exterior 
domains, we refer to  Giga and Sohr~\cite{GS91}. For a modern and short proof of Solonnikov's result we refer to~\cite{GHHSS10}. For a survey of results in this direction, see~\cite{HS18b}. Notice that the endpoint cases $q = 1$ and $q = \infty$ are not included in these results.

Multiplying the Stokes equation with  cut-off functions leads to a localized perturbed version of this equation. Since the Stokes equation is invariant under 
rotations and translations, we may assume that the localized version is either an equation on $\IR^n$ or on a bent half-space $H:=\{x\in\IR^n:\ x_n>h(x')\}$
with a certain bending function $h:\IR^{n-1}\to\IR$. We  further transform the localized system by $v(x',x_n):=(u\circ\phi)(x',x_n):= u(x',x_n + h(x'))$ for $(x',x_n)\in\IR^n_+$ to 
an equation on $\IR^n_+$. By this procedure, the Stokes equation  on a domain is reduced to at most countable equations on $\IR^n_+$ or $\IR^n$.

A fundamental problem arising here is the fact that $\divergence u=0$ is not preserved, neither under multiplication with a cut-off function nor by the above transformation.  This might be the reason 
why proofs of the analyticity of the Stokes semigroup as well as maximal regularity estimates  rely on different methods and why different assumptions are  imposed, depending on 
the approaches being used.

Boundary conditions different from the Dirichlet condition arise in many applications. For example, when considering free boundary problems, then there is no stress
at the surface of the fluid and it is natural to impose the condition 
$$
\IT(u,P)\bar n = 0 \ \hbox{ on the  boundary.}
$$    
 Here $\IT(u,P):=2\mu D(u) - P \Id$ denotes the stress tensor, $D(u)=\tfrac{1}{2} (\nabla u + [\nabla u]^{\top})$ the deformation tensor and $\bar n$ the outer normal. This condition, called 
free or Neumann boundary condition, belongs to the class of  energy preserving boundary conditions. This is due to the fact that the kinetic energy balance related to the 
homogeneous Stokes equation subject to boundary conditions $B(u,P)=0$ is given by 
$$
\frac12 \frac{d}{dt} \int_\Omega |u|^2 dx + 2 \mu \int_\Omega |D(u)|^2 dx = \int_{\partial\Omega} [u \cdot \IT(u,P)\bar n] \,\d \sigma. 
$$  
Whenever $B(u,P)$ is chosen in such a way that the above right-hand side vanishes, $B(u,P)=0$ is called an energy preserving boundary condition.

A first approach related to the free boundary condition $B(u,P)=\IT(u,P)\bar n$ was given by Solonnikov in~\cite{Sol84} and~\cite{Sol91}. Nowadays, one knows that the inhomogeneous 
Stokes problem on a half-space or on bounded domains with smooth boundaries subject to Neumann boundary conditions admits a unique solution $(u,P)$ within the maximal $\LL_q$-regularity class  
if and only if the data belong to (precisely defined) data space. For details, see, e.g., the work of Bothe, K\"ohne and Pr\"uss~\cite{BKP13}.

It seems that Saal~\cite{Saa06} was the first to consider the Stokes equation subject to Navier or Neumann conditions on a half-space $\IR^n_+$ in the framework of maximal $\LL_p$-$\LL_q$-regularity for 
$1<p,q<\infty$. Maximal regularity results for a wide class of boundary conditions including the free boundary condition,  based on pseudo-differential methods, were derived by 
Grubb and Solonnikov~\cite{GrSo91}. In a series of articles~\cite{SS07, Shibata1, Shibata2, Shibata_Shimizu} Shibata and Shimizu also proved maximal $\LL_p$-$\LL_q$-estimates for free boundary or Neumann conditions on a half-space,  
bounded domains or exterior domains. Their results  imply in particular that the Stokes operator subject to this boundary condition defines a sectorial operator in $\LL_{p , \sigma}(\IR^n_+)$ or generates a {\em bounded analytic} semigroup  on  $\LL_{p , \sigma}(\IR^n_+)$ and not only an analytic semigroup.

The Navier-Stokes equations subject to {\em Neumann or free boundary conditions} were investigated first by  Beale~\cite{Beale} and Solonnikov~\cite{Sol84},~\cite{Sol91} in a series of 
articles, see~\cite{SD18} for a survey. 

In the case where the initial domain $\Omega_0$ is bounded, the existence of a unique local as well as a unique global solution was proved by Solonnikov  
in~\cite{Solonnikov1,Sol84} within the $\LL_p$-$\LL_p$-setting and also by Mucha and Zajaczkowski in~\cite{MuZa1,Mucha_Zajaczkowski}. Shibata and Shimizu~\cite{SS07} considered  the corresponding $\LL_p$-$\LL_q$-setting. 
At the same time, in the case of an initial boundary with high curvature, blow-up may happen (see~\cite{CCDFG,Coutand_Shkoller}).

We mention here also the semigroup approach due to Schweizer~\cite{Sch97} for the case of non-vanishing surface tension and the result on exponential stability of global strong solutions 
by K\"ohne, Pr\"uss and Wilke~\cite{KPW13} for two-phase flows based on Eulerian coordinates and the Hanzawa transform. A related approach to global well-posedness  for  the Navier-Stokes equations 
subject to free boundary conditions is  described by Shibata in~\cite{Shi20} for initial domains close to a ball.

The case where the initial domain $\Omega$ is an infinite layer of {\em finite} depth 
with gravity was investigated by many authors after the pioneering work by  Beale~\cite{Beale} 
where almost global  existence was established. We refer here, e.g., to the works of  Tani and Tanaka~\cite{TT}, 
Abels~\cite{Abels} and Saito~\cite{Saito}.   In this setting, it looks that the first global existence 
result for small and smooth data  has been obtained by D.L. Sylvester in~\cite{Sylvester} 
(see  also the much more recent results 
 by  Y.\@ Guo and I.\@ Tice  in~\cite{Guo1,Guo2} where global-in-time results
 are obtained in various frameworks thanks to  high order energy type techniques).

In order to prove global existence of strong solutions one needs, following the Lagrangian approach, the integrability of $\nabla u(t,x)$ with respect to 
$t \in (0,\infty)$. This difficulty was overcome by Saito~\cite{Saito}, who proved maximal $\LL_q$-regularity for $1<q<\infty$ as well as {\em exponential stability} of the linearization in this setting.        

Considering  general  unbounded domains as initial domains, exponential stability of the associated linearization is not expected since in many cases, as for example for $\Omega=\IR^n$ or $\IR^n_+$,  
$0$ belongs to its spectrum. A global existence result for the case of the bottomless ocean was nevertheless obtained by Saito and Shibata~\cite{SS19} in the case with surface tension  by combining the Hanzawa transform approach with 
$\LL_r$-$\LL_s$-decay estimates for the semigroup associated with the linearization. Similar results are true also for exterior domains, see~\cite{Shi20}. 

In this article, we develop  an $\LL_1(\IR_+;\dot\B^s_{p,1}(\IR_+^n))$-maximal regularity approach to the Stokes and Lam\'e systems. It is then applied to establish two 
global strong well-posedness results within the theory of free boundary problems. 

The first one concerns the motion of  viscous, incompressible Newtonian fluids  
without surface tension, the second one pressureless gases. In the case of Newtonian fluids we consider as initial domain the half-space $\IR^3_+$ which 
corresponds to the 
bottomless ocean. Recall that  $\IT (u , P):=  D(u) - P \Id$ denotes the stress tensor with deformation tensor
$D(u) := \bigl(\nabla u + [\nabla u]^{\top}\bigr)\cdotp$

Our maximal $\LL_1$-regularity result for the Stokes system reads as follows. 

\begin{theorem}
Let  $n\geq2,$ $p\in (1,\infty)$ and $s \in (0 , 1)$ satisfy $s \leq n / p$. For any solenoidal $f \in \LL_1(\IR_+  ; \dot\B^s_{p,1}(\IR_+^n ; \IC^n))$ and solenoidal $u_0$
in $ \dot\B^s_{p,1}(\IR_+^n ; \IC^n)$, there exists a unique solution $(u , P)$ with $u \in {\mathcal C}_b(\IR_+  ; \dot \B^{s}_{p ,1} (\IR^n_+))$
  and $\partial_tu,\nabla^2u,\nabla P\in \LL_1(\IR_+; \dot \B^{s}_{p ,1} (\IR^n_+))$ to 
$$ 
\left\{
  \begin{aligned}
   \partial_t u  - \divergence \IT (u , P) &= f, && t \in\IR_+  , \:x \in \IR_+^n \\
   \divergence  u &= 0, &&  t \in\IR_+  , \:x \in \IR_+^n \\
   \IT (u , P)\cdot \e_n &= 0, && t \in \IR_+  ,\: x \in \partial \IR_+^n \\
u|_{t = 0} &= u_0, && x \in  \IR_+^n.
  \end{aligned} \right.
$$
Furthermore, there exists a constant $C>0$ such that 
$$
\displaylines{\quad
  \|u\|_{\LL_\infty(\IR_+;\dot \B^{s}_{p,1}(\IR^n_+))}+
  \|u_t\|_{\LL_1(\IR_+ ; \dot \B^{s}_{p,1}(\IR^n_+))} + \| \nabla^2 u\|_{\LL_1(\IR_+ ; \dot \B^{s}_{p,1}(\IR^n_+ ))} + \| \nabla P \|_{\LL_1(\IR_+ ; \dot \B^{s}_{p,1}(\IR^n_+))} \hfill\cr\hfill
\leq C\Big( \| f\|_{\LL_1(\IR_+ ; \dot \B^{s}_{p,1}(\IR^n_+))} + \|u_0\|_{ \dot \B^{s}_{p,1}(\IR^n_+)}\Big)\cdotp}
$$
\end{theorem}
Let us  mention here that a related result has been 
established independently by Ogawa and Shimizu in~\cite{OS},
that covers the case of non-positive values of $s$ only, while the theorem above deals with positive values of $s$. 
\smallbreak

We then consider the {\em free boundary problem for viscous, incompressible Newtonian flows  without surface tension}. Denoting by  $v$ the velocity of the fluid and 
by $Q$  its pressure, the motion of the fluid is governed by the following set of equations in the (unknown) moving domain $\Omega_t\subset\IR^n$: 
\begin{align}
\label{Eq: Free boundary problem}
 \left\{
  \begin{aligned}
   \partial_t v + (v \cdot \nabla) v - \divergence \IT (v , Q) &= 0, \qquad && t \in\IR_+  , \: x \in \Omega_t \\
   \divergence v &= 0, \qquad &&  t \in\IR_+  ,\: x \in \Omega_t \\
   \IT (v , Q) \overline{n} &= 0, \qquad && t \in \IR_+  ,\: x \in \partial \Omega_t \\
  v \cdot \overline{n} &= - (\partial_t \eta) / \lvert \nabla_x \eta \rvert ,\qquad && t \in\IR_+  ,\: x \in \partial \Omega_t \\
v|_{t = 0} &= v_0, \qquad && x \in \Omega_0 \\
   \Omega_t |_{t = 0} &= \IR_+^n.
  \end{aligned}
 \right.
\end{align}
We  focus on  small amplitude motions so that the boundary $\partial \Omega_t$  may be described  by some (unknown) function $\eta=\eta(t,x)$ through
$$
\partial \Omega_t=\bigl\{x\in\IR^n\,:\, \eta(t,x)=0\bigr\}\cdotp
$$ 
Consequently,  the outward unit normal vector $\overline{n}$  to $\partial \Omega_t$ is just $\overline{n} = \nabla_x \eta / \lvert \nabla_x \eta \rvert$. 
\medbreak
Taking advantage of our maximal regularity result applied to the particular case of the Stokes equation in the half-space
will enable us to prove a global existence result for small data. 
Before giving the statement, let us specify what a solution to~\eqref{Eq: Free boundary problem} is.

\begin{definition}\label{def:NS}
We call the triplet $(v , Q , \Omega_t)$ a global solution to~\eqref{Eq: Free boundary problem} if  $(v , Q , \Omega_t)$ satisfies~\eqref{Eq: Free boundary problem} almost everywhere
and if, for all $t > 0,$ the velocity field satisfies $v(t , \cdot) \in \dot \B^{n / p - 1}_{p , 1} (\Omega_t)$ and  
$$\sup_{t\in\IR_+} \|v(t)\|_{\dot \B^{n/p-1}_{p,1}(\Omega_t)}+
\int_{\IR_+} \bigl(\|\partial_t v , \nabla^2 v , \nabla Q(t)\|_{\dot \B^{n/p-1}_{p,1}(\Omega_t)}+ \|\nabla v(t)\|_{\dot \B^{n/p}_{p,1}(\Omega_t)}\bigr)\, \d t<\infty$$
with $\Omega_t$ given, for each $t > 0$, by
\begin{align*}
 \Omega_t = X_v(t , \IR^n_+).
\end{align*}
Above,  the flow $X_v$ of $v$ is defined by
\begin{align*}
 X_v (t , x) = x + \int_0^t v(\tau , X_v(\tau , x)) \, \d \tau,
\end{align*}
and is required  to satisfy:
\begin{align*}
 \nabla X_v - \Id \in \cC(\IR_+ ; \dot \B^{{n}/{p}}_{p , 1} (\IR_+^n)).
\end{align*}
\end{definition}

\begin{theorem} \label{Thm:NS3D, intro} 
Assume that $n\geq3$. Let $v_0$ be a solenoidal vector-field on $\IR^n_+,$ with coefficients in the homogeneous Besov space $\dot\B^{n/p-1}_{p,1}(\IR^n_+)$ for some $p\in(n-1,n).$  
Then, there exists $c>0$ such that if 
 \begin{equation}
  \|v_0\|_{\dot \B^{n/p-1}_{p,1}(\IR^n_+)}\leq c,
 \end{equation}
then  system~\eqref{Eq: Free boundary problem} 
admits  a unique, global solution  $(v,Q,\Omega_t)$.
\end{theorem}

As a second  application, we consider the  free boundary problem for pressureless gases  in the half-plane $\IR_+^n$. This system is a simplified model for the dynamics of galaxies. Formally, in $\IR^n$ 
solutions to the pressureless gas system can be obtained as limits of solutions to the Euler equations. For more information, we refer, e.g., to the book by Dafermos~\cite{Daf05}.

The governing equations read as 
\begin{align}
\label{eq:presslessint}
 \left\{
 \begin{aligned}
   \rho(\partial_t v + v \cdot \nabla v) - \divergence \IS (v) &= 0, \qquad && t \in\IR_+,\  x \in \Omega_t, \\
 \partial_t\rho + \divergence (\rho v) &= 0, \qquad &&  t \in\IR_+ ,\  x \in \Omega_t, \\
   \IS (v) \overline{n} &= 0, \qquad && t \in \IR_+ ,\  x \in \partial \Omega_t, \\
      v \cdot \overline{n} &= - (\partial_t \eta) / \lvert \nabla_x \eta \rvert ,\qquad && t \in\IR_+,\  x \in \partial \Omega_t,\\
   v|_{t = 0} &= v_0, \qquad && x \in \IR_+^n, \\
   \Omega_t |_{t = 0} &= \IR_+^n.\end{aligned} \right.
\end{align}
Similarly as above, the unknowns are the velocity field $v$, the density $\rho$ and the time-dependent fluid domain $\Omega_t$. Here   $\IS(v)$ stands for the viscous stress tensor given by 
$$
\IS (v):= \mu D(v) + \lambda\divergence v\,\Id,\qquad \mu>0,\quad 2 \mu + n \lambda>0.$$

Our maximal $\LL_1$-regularity result for the Lam\'e system
\begin{equation}\label{eq:lameint}
\left\{\begin{aligned} 
\partial_t u - \divergence\IS(u) &= f \quad &\hbox{in }\ & \IR_+\times\IR^n_+, \\
\IS(u) \e_n|_{\partial\IR^n_+}&= g \quad &\hbox{on }\ & \IR_+\times\partial\IR^n_+,\\
u|_{t = 0} &= u_0  \quad &\hbox{in }\ & \IR^n_+\end{aligned} \right.
\end{equation}
reads as follows.

\begin{theorem}\label{Thm:lameint} Let $1 < p < \infty$ and $0<s<1/p$ with $s\leq n/p-1$.
Assume that $\mu>0$ and $2 \mu + n \lambda >0$.   
There exists a constant $C > 0$ such that for every 
$u_0\in  \dot \B^s_{p , 1} (\IR^{n}_+)$, $f\in \LL_1(\IR_+;\dot\B^s_{p , 1} (\IR^{n}_+))$ and
\begin{align*}
 g \in \dot\Y^s_p:=\dot \B^{\frac{1}{2}(s+1-\frac1p)}_{1 , 1} (\IR_+ ; \dot \B^0_{p , 1} (\partial\IR^{n}_+)) 
 \cap \LL_1 (\IR_+ ; \dot \B^{s + 1 - \frac{1}{p}}_{p,1} (\partial\IR^{n}_+)),
\end{align*}
there exists a unique solution $u$ to~\eqref{eq:lameint} that
satisfies
\begin{multline}\label{est:lame0} \|u\|_{\E^s_p}:=\| \partial_tu\|_{\LL_1 (\IR_+ ; \dot \B^s_{p , 1} (\IR^n_+))} 
+\|\nabla u \|_{\LL_1 (\IR_+ ; \dot \B^{s+1}_{p , 1} (\IR^n_+))} 
+\|u \|_{\cC_b (\IR_+ ; \dot \B^s_{p , 1} (\IR^n_+))} +
\|\nabla u|_{\partial\IR^n_+}\|_{\dot\Y^s_p}\\[5pt]
\leq C\bigl(\|u_0\|_{\dot \B^s_{p , 1} (\IR^{n}_+)}
+\|f \|_{\LL_1 (\IR_+ ; \dot \B^s_{p , 1} (\IR^n_+))} +
\| g \|_{\dot\Y^s_p}\bigr)\cdotp
\end{multline}
\end{theorem}

Note that this result relies partly on resolvent estimates for the Lam\'e operator given in  Section~\ref{Sec: The Lame operator in the upper half-plane}. 

\vspace{.3cm}
The result for the free boundary value problem~\eqref{eq:presslessint} reads as follows. Similarly to Definition \ref{def:NS} we call a triplet $(v,\rho,\Omega_t)$ a global solution to \eqref{eq:presslessint} if $(v,\rho,\Omega_t)$ satisfies \eqref{eq:presslessint} almost everywhere and $ v(t,\cdot) \in \dot \B^{n/p-1}_{p,1}(\Omega_t)$ and 
\begin{align}
\label{Eq: Solution estimate pressureless}
\sup_{t\geq0} \|v(t)\|_{\dot \B^{n/p-1}_{p,1}(\Omega_t)}<\infty\andf
\int_{\IR_+}\ \bigl( \|(\partial_tv , \nabla^2 v)(t)\|_{\dot \B^{n/p - 1}_{p,1}(\Omega_t)} + \|\nabla v(t)\|_{\dot \B^{n/p}_{p,1}(\Omega_t)}\bigr)\, \d t<\infty, 
\end{align}
as well as 
$$
\sup_{t\geq 0} \|\rho(t)-1\|_{\LL_\infty(\Omega_t)} < \infty. 
$$
Here  $\Omega_t$ and $\partial \Omega_t$ are given for each $t>0$ by 
$$
\Omega_t=X_v(t,\IR^n_+) \mbox{ and } \partial\Omega_t = X_v(t,\partial\IR^n_+)
$$
and $X_v$ defined by 
$$
X_v(t,x) = x + \int_0^t v(\tau,X_v(\tau,x))d\tau
$$
is  required to satisfy 
$$
\nabla X_v - \Id \in \cC(\IR_+ ; \dot \B^{n/p}_{p , 1} (\IR_+^n)).
$$

\begin{theorem}\label{thm:lpresslint} Let $p\in(n-1,n)$, 
 $v_0 \in \dot \B^{n/p-1}_{p,1}(\IR^n_+)$ and $\rho_0 \in \LL_\infty(\IR^n_+) \cap  
 \cM(\dot \B^{n/p-1}_{p,1}(\IR^n_+))$, where $\cM(X)$ denotes the multiplier  space of $X$. 
 Then there exists a constant $c=c(p,\lambda,\mu)$ such that, if  
$$
\|\rho_0-1\|_{\LL_\infty(\IR^n_+)  \cap \cM(\dot \B^{n/p-1}_{p,1}(\IR^n_+))}  
 +  \|v_0\|_{\dot \B^{n/p-1}_{p,1}(\IR^n_+)}\leq c,
 $$
then system~\eqref{eq:presslessint} admits a unique, global  solution  $(v,\rho,\Omega_t)$ and there exists a constant $C>0$ such that 
$$ 
\sup_{t\geq 0} \|\rho-1\|_{\LL_\infty(\Omega_t)}\leq Cc.
$$

 \end{theorem}

\begin{remark}
In Chapter 7 we state and prove  a more general result than the one stated in Theorem \ref{thm:lpresslint} above by replacing the half-space $\IR^n_+$ by a bent half-space $\Omega$.
\end{remark}

 Proving  the above assertions  is based  on the following strategy: instead of solving the original free boundary problems, we introduce 
 Lagrangian coordinates.  The corresponding system is  then given  in a fixed domain, namely the half-plane or the half-space,  and 
 can be solved by means of a contracting mapping argument in the solution space of maximal regularity.  
 
The fixed point map is built upon the linearization of this system,  namely  the Stokes or the Lam\'e system. 
In order to meet the requirements of the fixed point theorem, one has to resort on the one side to rather standard nonlinear estimates
 and on the other side and,  more importantly,  to the $\LL_1$-in-time analysis developed  above.  

An additional difficulty is that going to Lagrangian coordinates does not preserve the solenoidal condition  and the homogeneity of the data at the boundary. 
Therefore, we will have to consider a version of the $\LL_1$-in-time regularity theorem allowing for nonzero boundary data and, 
in the Stokes case, with nonzero divergence. 
We chose to tackle the Stokes and Lam\'e systems differently. 
In the Stokes case, we avoided using  explicit formulae
for the solution (in contrast with what Ogawa and Shimizu did
in~\cite{OS}) by working out suitable extensions for 
the boundary data and then, eventually, used (our homogeneous version of) the Da Prato and Grisvard theorem to handle the initial data and source
term. 

In the Lam\'e case,  we first computed explicitly the solution $v$ corresponding
to nonzero boundary data, then used results on multipliers
(in the parabolic scale) so as to prove 
that $v$ satisfies Inequality~\eqref{est:lame0}. 
Then, again, Da Prato and Grisvard's theorem
enabled us to handle the initial data and  source
term. 
\medbreak
This article  is structured as follows. In Chapter~\ref{Sec: Da Prato--Grisvard theorem}, we revisit the Da Prato and Grisvard theory in the homogeneous space setting. Then, 
in Chapter~\ref{Sec: The functional setting and basic interpolation results}, we present the homogeneous functional framework that will be appropriate 
for our further investigations, and prove various extension results from the half-space to the whole space. 
Next, we use the previous chapters to prove  an  $\LL_q$-in-time maximal regularity theorem for the Stokes system
supplemented with a Neumann-type  boundary condition \emph{including the case $q=1$}. We then prove  two milestone results in this direction: the 
 $\LL_1$-in-time maximal regularity theorems for the Stokes and the Lam\'e systems. The remaining text is dedicated to solving the associated free boundary problem 
by the strategy  presented above. Some technical results are summarized in the Appendix. 

\vskip1cm

{\bf Acknowledgments.} The authors wish to express their gratitude to Yoshihiro Shibata for fruitful discussions. Special thanks the authors pass to Winfried Sickel for exhibiting a 
counterexample in Chapter~\ref{Sec: The free boundary problem for incompressible fluids with infinite depth} to some borderline product law for Besov spaces.  

The first, second and fourth author  have been partially supported by   ANR-15-CE40-0011.
PBM has been partly supported  by the Polish National Science Centre’s grant No2018/29/B/ST1/00339 (OPUS).

\setcounter{page}{1}
\pagenumbering{arabic}
\chapter{The Da Prato -- Grisvard theorem}
\label{Sec: Da Prato--Grisvard theorem}

In this chapter we extend the classical abstract maximal regularity result of Da Prato -- Grisvard 
\cite[Thm.~4.7]{DaPrato_Grisvard} from 1975 to the homogeneous setting.

\smallskip 

Let $X$ be a Banach space and $\cA : \dom(A) \subset X \to X$ be a linear and closed operator. For $\vartheta \in [0 , \pi)$ define the sector $\Sec_{\vartheta}$ in the complex plane as
\begin{align*}
 \Sec_{\vartheta} := (0 , \infty) \quad \text{if} \quad \vartheta = 0
\end{align*}
and
\begin{align*}
 \Sec_{\vartheta} := \{ z \in \IC \setminus \{ 0 \} : \lvert \arg(z) \rvert < \vartheta \} \quad \text{if} \quad \vartheta \in (0 , \pi).
\end{align*}
With $\sigma (\cA)$ denoting the spectrum of $\cA$ and $\rho(\cA)$ its resolvent set, we say that $\cA$ is \textit{sectorial of angle $\omega \in [0 , \pi)$} if
\begin{align*}
 \sigma (\cA) \subset \overline{\Sec_{\omega}}
\end{align*}
and if for every $\vartheta \in (\omega , \pi)$ there exists $C > 0$ such that
\begin{align}
\label{Eq: McIntosh Resolvent estimate}
 \| \lambda (\lambda - \cA)^{-1} \|_{\Lop(X)} \leq C \qquad (\lambda \in \IC \setminus \overline{\Sec_{\vartheta}}).
\end{align}
It is well-known~\cite[Thm.~II.4.6]{Engel_Nagel} that $- \cA$ generates a strongly continuous bounded analytic semigroup if and only if $\cA$ is densely defined and sectorial of some angle $\omega \in [0 , \tfrac{\pi}{2})$. In the following, we will always assume that $- \cA$ is the generator of a strongly continuous bounded analytic semigroup on $X$. This semigroup will be denoted by $(\e^{- t \cA})_{t \geq 0}$. The bounded analyticity of the semigroup further implies the validity of certain smoothing estimates. Indeed, by~\cite[Thm.~II.4.6]{Engel_Nagel} one further knows that the ranges of the semigroup operators for positive times $t > 0$, i.e., $\Rg(\e^{- t \cA})$, lie in $\dom(\cA)$ and that there exists $M > 0$ such that
\begin{align}
\label{Eq: Smoothing estimate}
 \| t \cA \e^{- t \cA} \|_{\Lop(X)} \leq M \qquad (t > 0).
\end{align}

Given $0<T\leq \infty$, we are interested in this chapter in maximal regularity estimates to the abstract Cauchy problem
\begin{align}
\tag{ACP}\label{Eq: ACP}
\left\{ \begin{aligned}
 \dot u (t) + \cA u (t) &= f(t) \qquad (0 < t < T) \\
 u(0) &= x.
\end{aligned} \right.
\end{align}
Remember that, if $x \in X$ and $f \in \LL_1 (0 , T ; X)$, then the unique mild solution to~\eqref{Eq: ACP} is given by the  following variation of 
constants formula (see~\cite[Prop.~3.1.16]{Arendt_Batty_Hieber_Neubrander}):
\begin{align}
\label{Eq: Variation of constants}
 u (t) = \e^{- t \cA} x + \int_0^t \e^{- (t - s) \cA} f (s) \; \d s.
\end{align}
The basic maximal regularity result that builds the basis for the further investigation is the theorem of Da Prato and Grisvard~\cite[Thm.~4.7]{DaPrato_Grisvard}, which can be seen as the first abstract 
result on maximal regularity in the mathematical literature. As it was already described in the introduction, this theorem delivers maximal regularity estimates for generators of bounded analytic semigroups on finite time intervals or on $\IR_+$  if, additionally, $\cA$ is boundedly invertible. The aim of this section is to establish a homogeneous version of this theorem which guarantees global-in-time estimates also if $0 \in \sigma (\cA)$. 

\section{Homogeneous operators and  spaces}\label{Sec: The homogeneous operator and homogeneous spaces} 

First, we have to  introduce a homogeneous version of the operator $\cA$
in the case where $- \cA$  is  the generator of a strongly continuous bounded analytic semigroup on $X$. To define the homogeneous version of $\cA$ we make the following assumptions.

\begin{assumption}
\label{Ass: Homogeneous operator}
The operator $\cA$ is injective and there exists a normed vector space $Y$ (not necessarily complete) such that $\dom(\cA) \subset Y$ and such that there exist two constants $C_1 , C_2 > 0$ such that
\begin{align}
\label{Eq: Equivalence of norms in abstract setting}
 C_1 \| \cA x \|_X \leq \| x \|_Y \leq C_2 \| \cA x \|_X \qquad (x \in \dom(\cA)).
\end{align}
\end{assumption}
\noindent
In the case where $\cA$ stands for the Laplace operator on $\IR^n_+,$ a prominent  example of a couple $(X,Y)$  is, of course,  
$X = \LL_{p}(\IR^n_+)$  and $Y = \dot \H^2_{p} (\IR^n_+).$
\begin{definition}
\label{Def: Definition homogeneous operator}
If $- \cA$ satisfies Assumption~\ref{Ass: Homogeneous operator}, then define the domain of the homogeneous version $\dot \cA$ of $\cA$ by
\begin{align*}
 \dom(\dot \cA) := \{ y \in Y : \exists (x_k)_{k \in \IN} \subset \dom(\cA) \text{ with } x_k \to y \text{ in } Y \text{ as } k \to \infty \}.
\end{align*}
With this definition, set $\dot \cA y$  to be the following  limit in $X$:
\begin{align*}
 \dot \cA y := \lim_{k \to \infty} \cA x_k \qquad (y \in \dom(\dot \cA)).
\end{align*}
\end{definition}

\begin{remark}
Note that~\eqref{Eq: Equivalence of norms in abstract setting} implies that $\dot \cA y$ is independent of the approximating sequence and thus well-defined. Moreover, since $X$ is complete, ~\eqref{Eq: Equivalence of norms in abstract setting} implies that $\dot \cA y \in X$ for all $y \in \dom(\dot \cA)$.
\end{remark}

The vector space $\dom(\dot \cA)$ will be considered as endowed with the graph norm $\| y \|_{\dom(\dot \cA)} := \| \dot \cA y \|_X$. That this is indeed a norm is implied by the following lemma.

\begin{lemma}
\label{Lem: Injectivity of homogeneous abstract operator}
The operator $\dot \cA$ is injective.
\end{lemma}

\begin{proof}
Assume that $\dot \cA y = 0$ for some $y \in \dom(\dot \cA)$. Let $(x_k)_{k \in \IN} \subset \dom(\cA)$ be such that $x_k \to y$ in $Y$ as $k \to \infty$. Then we find by~\eqref{Eq: Equivalence of norms in abstract setting} that
\begin{align*}
 0 = \| \dot \cA y \|_X = \lim_{k \to \infty} \| \cA x_k \|_X \geq C_2^{-1} \lim_{k \to \infty} \| x_k \|_Y = \| y \|_Y.
\end{align*}
It follows that $y = 0$.
\end{proof}

From now on, $Y$ will denote the normed vector space whose existence is postulated in Assumption~\ref{Ass: Homogeneous operator}. Following the terminology in~\cite[Sec.~2.3]{Bergh_Lofstrom} we say that $X$ and $Y$ are \textit{compatible} if there exists a Hausdorff topological vector space $Z$ such that $X$ and $Y$ are subspaces of $Z$. In this case the sum space $X + Y$ and intersection space $X \cap Y$ with canonical norms are well defined and again normed vector spaces. \par
Finally, we will impose a condition on the operator $\cA$ and its homogeneous counterpart $\dot \cA$ which ensures that $\dot \cA$ and $Y$ carry enough information to recover the operator $\cA$ again.

\begin{assumption}
\label{Ass: Intersection property} The operator  $\cA$ and the normed vector space $Y$ are such that
\begin{align}
\label{Eq: Intersection property}
 \dom(\dot \cA) \cap X = \dom(\cA).
\end{align}
\end{assumption}

Recall the following elementary result from semigroup theory, see~\cite[Lem.~II.1.3]{Engel_Nagel}. For general $x \in X$ one has for $t > 0$
\begin{align}
\label{Eq: Allerweltsformel 1}
 \int_0^t \e^{- s \cA} x \, \d s \in \dom(\cA) \quad \text{and} \quad \e^{- t \cA} x - x = - \cA \int_0^t \e^{- s \cA} x \, \d s.
\end{align}
Moreover, if $x \in \dom(\cA)$, then it even holds
\begin{align}
\label{Eq: Allerweltsformel 2}
 \e^{- t \cA} x - x = - \int_0^t \e^{- s \cA} \cA x \, \d s.
\end{align}
This motivates, with $\cA$, $X$, $\dot \cA$, $Y$ and $Z$ subject to the assumptions above, the following definition of an extension of the semigroup operators $\e^{- t \cA}$ for $t > 0$ to the sum space $X + \dom(\dot \cA)$:
\begin{align}
\label{Eq: Extension semigroup operators}
 \e^{- t \cA} : X + \dom(\dot \cA) \to X + \dom(\dot \cA), \quad \e^{- t \cA} (x + y) := \e^{- t \cA} x + y - \int_0^t \e^{- s \cA} \dot \cA y \, \d s,
\end{align}
where $x \in X$ and $y \in \dom(\dot \cA)$.

\begin{proposition}
\label{Prop: Representation abstract semigroup}
The operator in~\eqref{Eq: Extension semigroup operators} is well-defined and for all $t > 0$, $x \in X$ and $y \in \dom(\dot \cA)$ it holds that $\e^{- t \cA} (x + y) \in \dom(\dot \cA)$ and that
\begin{align}
\label{Eq: Canonical representation on sum space}
 \dot \cA \e^{- t \cA} (x + y) = \cA \e^{- t \cA} x + \e^{- t \cA} \dot \cA y.
\end{align}
\end{proposition}

\begin{proof}
First of all, notice that all expressions in~\eqref{Eq: Extension semigroup operators} make sense, so that $\e^{- t \cA} (x + y)$ exists as an element in $X + \dom(\dot \cA)$. To show that the definition is independent of the representation of $x + y$ let $x_1 , x_2 \in X$ and $y_1 , y_2 \in \dom(\dot \cA)$ with $x_1 + y_1 = x_2 + y_2.$
Hence
\begin{align}
\label{Eq: Relation of differences}
 x_1 - x_2 = y_2 - y_1 \in X \cap \dom(\dot \cA),
\end{align}
and thus, $y_2 - y_1 \in \dom(\cA)$ by Assumption~\ref{Ass: Intersection property}. It follows that
\begin{align}
\label{Eq: Well-posedness semigroup}
\begin{aligned}
 \e^{- t \cA} (x_1 + y_1) &= \e^{- t \cA} x_1 + y_1 - \int_0^t \e^{- s \cA} \dot \cA y_1 \, \d s \\
 &= \e^{- t \cA} x_1 + y_1 - \int_0^t \e^{- s \cA} \cA (y_1 - y_2) \, \d s - \int_0^t \e^{- s \cA} \dot \cA y_2 \, \d s.
 \end{aligned}
\end{align}
Using~\eqref{Eq: Allerweltsformel 2} together with~\eqref{Eq: Relation of differences} in~\eqref{Eq: Well-posedness semigroup} delivers
\begin{align*}
 \e^{- t \cA} (x_1 + y_1) = \e^{- t \cA} x_1 + y_1 + \e^{- t \cA} (y_1 - y_2) - (y_1 - y_2)  - \int_0^t \e^{- s \cA} \dot \cA y_2 \, \d s = \e^{- t \cA} (x_2 + y_2).
\end{align*}

\noindent To establish~\eqref{Eq: Canonical representation on sum space} notice first that~\eqref{Eq: Allerweltsformel 1} together with the definition in~\eqref{Eq: Extension semigroup operators} implies that $\e^{- t \cA} (x + y) \in \dom(\dot \cA)$ for all $x \in X$ and $y \in \dom(\dot \cA)$. Now, calculate by virtue of~\eqref{Eq: Allerweltsformel 1}
\begin{align*}
 \dot \cA \e^{- t \cA} (x + y) = \cA \e^{- t \cA} x + \dot \cA y - \cA \int_0^t \e^{- s \cA} \dot \cA y \, \d s = \cA \e^{- t \cA} x + \e^{- t \cA} \dot \cA y. &\qedhere
\end{align*}
\end{proof}

\begin{remark}
\label{Rem: Semigroup property of extended semigroup}
The family $(\e^{-t\cA})_{t\geq0}$ satisfies the semigroup property on  $X + \dom (\dot \cA),$ that is to say,
$$\e^{-(s+t)\cA}z= \e^{-s\cA} \e^{-t\cA} z\qquad z\in X + \dom (\dot \cA) \text{ \ and \ } s , t \geq 0. $$
Indeed, let $x \in X$, $y \in \dom(\dot \cA)$ and $s , t > 0$ (the cases $s = 0$ or $t = 0$ being trivial). By Proposition~\ref{Prop: Representation abstract semigroup} we find for $s , t > 0$ that
\begin{align}
\label{Eq: Calculation for semigroup property of extended semigroup}
\begin{aligned}
 \dot \cA \e^{- (s + t) \cA} (x + y) &= \cA \e^{- (s + t) \cA} x + \e^{- (s + t) \cA} \dot \cA y \\ &= \e^{- s \cA} \big( \cA \e^{- t \cA} x + \e^{- t \cA} \dot \cA y \big) = \e^{- s \cA} \dot \cA \e^{- t \cA} (x + y).
\end{aligned}
\end{align}
Thus, the statement follows from the injectivity of $\dot \cA$, see Lemma~\ref{Lem: Injectivity of homogeneous abstract operator}, once we prove that for $z \in \dom(\dot \cA)$ one has $\e^{- s \cA} \dot \cA z = \dot \cA \e^{- s \cA} z$. But this directly follows by~\eqref{Eq: Allerweltsformel 1} combined with the definition~\eqref{Eq: Extension semigroup operators} by the following calculation
\begin{align*}
 \e^{- s \cA} \dot \cA z = \dot \cA z - \cA \int_0^s \e^{- \tau \cA} \dot \cA z \, \d \tau = \dot \cA \e^{- s \cA} z.
\end{align*}
\end{remark}

The following extends the injectivity of $\dot \cA$ to its action on the semigroup operators.

\begin{lemma}
\label{Lem: Injectivity semigroup family}
Let $z \in X + \dom(\dot \cA)$. If $\dot \cA \e^{- t \cA} z = 0$ for all $t > 0$ then $z = 0$.
\end{lemma}

\begin{proof}
By Lemma~\ref{Lem: Injectivity of homogeneous abstract operator} we immediately find that $\e^{- t \cA} z = 0$ for all $t > 0$. Write $z = x + y$ for some $x \in X$ and $y \in \dom (\dot \cA)$. By~\eqref{Eq: Extension semigroup operators} and~\eqref{Eq: Allerweltsformel 1} we thus find that
\begin{align*}
 y = \int_0^t \e^{- s \cA} \dot \cA y \, \d s - \e^{- t \cA} x \in \dom(\cA).
\end{align*}
Employing~\eqref{Eq: Allerweltsformel 2} then it delivers
\begin{align*}
 y = \cA \int_0^t \e^{- s \cA} y \, \d s - \e^{- t \cA} x = y - \e^{- t \cA} y - \e^{- t \cA} x.
\end{align*}
In the limit $t \to 0$ the strong continuity of the semigroup now implies that $x = - y$ and thus that $z = 0$.
\end{proof}

We continue by introducing the crucial function spaces that are required for the formulation of the Da Prato -- Grisvard theorem. The inhomogeneous function spaces that were already investigated by Da Prato and Grisvard are given as follows. For $\theta \in (0 , 1)$ and $1 \leq q \leq \infty$, denote by $\dom_{\cA} (\theta , q)$ the space
\begin{equation}\label{eq:DAtheta}
 \dom_{\cA} (\theta , q) = \{ x \in X : t \mapsto t^{1 - \theta} \cA \e ^{- t \cA} x \in \LL_{q , *} (\IR_+  ; X) \}
\end{equation}
endowed with the norm
\begin{align}
\label{Eq: Da Prato-Grisvard norm}
 \| x \|_{\dom_{\cA} (\theta , q)} := \| x \|_X + \| t \mapsto t^{1 - \theta} \cA \e ^{- t \cA} x \|_{\LL_{q , *} (\IR_+ ; X)}.
\end{align}
Here, we denote by $\LL_{q , *} (\IR_+  ; X)$ the Bochner $\LL_q$-space endowed with the Haar measure $\d t / t$. In order to underline the integration with respect to the first variable in 
spaces of type $\LL_{q , *} (\IR_+ ; X)$ we use notation $t \mapsto f(t)$.

\par
We introduce their homogeneous counterparts as 
\begin{equation}\label{eq:DAthetadot}
 \dot \dom_{\cA} (\theta , q) := \{ x \in X + \dom(\dot \cA) : \| x \|_{\dot \dom_{\cA} (\theta , q)} := \| t \mapsto t^{1 - \theta} \dot \cA \e ^{- t \cA} x \|_{\LL_{q , *} (\IR_+  ; X)} < \infty \}.
\end{equation}
By virtue of Lemma~\ref{Lem: Injectivity semigroup family} the mapping $\| \cdot \|_{\dot \dom_{\cA} (\theta , q)}$ defines indeed a norm on $\dot \dom_{\cA} (\theta , q)$.

\begin{remark}
Notice that the semigroup $(\e^{- t \cA})_{t \geq 0}$ satisfies a uniform boundedness estimate with respect to the norms of $\dom_{\cA} (\theta , q)$ and $\dot \dom_{\cA} (\theta , q)$. This stems from the fact that for $s , t > 0$ and $z \in X + \dom(\dot \cA)$ one has by virtue of Remark~\ref{Rem: Semigroup property of extended semigroup} and~\eqref{Eq: Calculation for semigroup property of extended semigroup} the commutator property
\begin{align*}
 \dot \cA \e^{- t \cA} \e^{- s \cA} z = \dot \cA \e^{- (s + t) \cA} z = \e^{- s \cA} \dot \cA \e^{- t \cA} z
\end{align*}
and the uniform boundedness of $(\e^{- s \cA})_{s \geq 0}$ on $X$.
\end{remark}

It was already known due to Da Prato and Grisvard that the inhomogeneous spaces $\dom_{\cA} (\theta , q)$ can be characterized as the real interpolation spaces $(X , \dom (\cA))_{\theta , q}$. Moreover, Haase~\cite[Sec.~6.4]{Haase} proved the same for homogeneous spaces that are defined similarly as above. That  is still valid with the definitions presented above is established in the following section.

\section{Homogeneous spaces and real interpolation}\label{Sec: Homogeneous spaces and real interpolation} To proceed we introduce basic notions from the theory of real interpolation, see, e.g.,~\cite{Bergh_Lofstrom, Triebel}. A couple $(\cX , \cY)$ of normed vector spaces $\cX$ and $\cY$ is called an interpolation couple if they are compatible, i.e., if they are continuously included into a common Hausdorff topological vector space. As it was already mentioned above, the intersection space $\cX \cap \cY$ as well as the sum space $\cX + \cY$ endowed with the canonical norms are well-defined. We define the $K$-functional of an element $z \in \cX + \cY$ and $t > 0$ as
\begin{align}
\label{Eq: K-functional}
 K(t , z ; \cX , \cY) := \inf \{ \| x \|_{\cX} + t \| y \|_{\cY} : x \in \cX, \, y \in \cY \text{ with } x + y = z \}.
\end{align}
For $\theta \in (0 , 1)$ and $q \in [1 , \infty]$, the real interpolation space between $\cX$ and $\cY$ with parameters $\theta$ and $q$ is then defined as
\begin{align}
\label{Eq: Real interpolation space}
\begin{aligned}
 \big( \cX , \cY \big)_{\theta , q} \;\; &:= \{ z \in \cX + \cY : t \mapsto t^{- \theta} K(t , z ; \cX , \cY) \in \LL_{q , *} (0 , \infty) \}, \\
 \| z \|_{(\cX , \cY)_{\theta , q}} &:= \| t \mapsto t^{- \theta} K(t , z ; \cX , \cY) \|_{\LL_{q , *} (0 , \infty)}.
\end{aligned}
\end{align}
An important information on the density of function spaces related to the real interpolation is the following: In~\cite[Thm.~3.4.2]{Bergh_Lofstrom} it is proved that the intersection space $\cX \cap \cY$ is always dense in the real interpolations space $(\cX , \cY)_{\theta , q}$ whenever $1 \leq q < \infty$. An application of this density property is given in the following lemma.

\begin{lemma}
\label{Lem: Density in abstract spaces}
Let $1 \leq q < \infty$ and $\theta \in (0 , 1)$. Then $\dom(\cA)$ is dense in  the space $\dom_{\cA} (\theta , q)$ 
defined in~\eqref{eq:DAtheta}, 
and in $(X , \dom(\dot \cA))_{\theta , q}$.
\end{lemma}

\begin{proof}
By Assumption~\ref{Ass: Intersection property} we have $\dom(\cA) = X \cap \dom(\dot \cA)$ so that the density in $(X , \dom(\dot \cA))_{\theta , q}$ directly follows from~\cite[Thm.~3.4.2]{Bergh_Lofstrom}. %
In the case of $\dom_{\cA} (\theta , q),$ the density follows from the strong continuity of the semigroup $(\e^{- t \cA})_{t \geq 0}$ on $\dom_{\cA} (\theta , q)$, which in turn follows from the strong continuity of $(\e^{- t \cA})_{t \geq 0}$ on $X$, its uniform boundedness with respect to $t$, and the dominated convergence theorem.
\end{proof}

In the following, we are going to identify $(X , \dom(\dot \cA))_{\theta , q}$ as $\dot \dom_{\cA} (\theta , q)$ (defined 
in~\eqref{eq:DAthetadot}). As a preparation, we record the following well-known lemma, see, e.g.,~\cite[Sec.~6]{Haase}, but for the convenience of the reader present the proof.

\begin{lemma}
\label{Lem: Hardy}
For each $z \in \dot \dom_{\cA}(\theta , q)$, $1 \leq q \leq \infty$ and each $\theta \in (0 , 1),$ 
we have \begin{align*}
 \Big\| t^{- \theta} \int_0^t \| s \dot \cA \e^{- s \cA} z \|_X \, \frac{\d s}{s} + t^{1 - \theta} \| \dot \cA \e^{- t \cA} z \|_X \Big\|_{\LL_{q , *} (0 , \infty)} \leq \frac{1 + \theta}{\theta} \| z \|_{\dot \dom_{\cA} (\theta , q)}.
\end{align*}
\end{lemma}

\begin{proof}
The result follows from the triangle inequality and an application of Hardy's inequality, see~\cite[Lem.~6.2.6]{Haase} to estimate for $\eps > 0$
\begin{align*}
 \Big\| t^{- \theta} \int_0^t \chi_{(\eps , \infty)} (s) \| s \dot \cA \e^{- s \cA} z \|_X \, \frac{\d s}{s} \Big\|_{\LL_{q , *} (0 , \infty)} \leq \frac{1}{\theta} \Big\| \chi_{(\eps , \infty)} (t) t^{1 - \theta} \| \dot \cA \e^{- t \cA} z \|_X \Big\|_{\LL_{q , *} (0 , \infty)}.
\end{align*}
Here, $\chi_{(\eps , \infty)} (t)$ denotes the characteristic function of $(\eps , \infty)$. 
To above inequality is obtained via Riesz-Thorin interpolation. For $q=\infty$ and $q=1$ the proofs are immediate. 
Since $z \in \dot \dom_{\cA}(\theta , q)$ the right-hand side is estimated by $\| z \|_{\dot \dom_{\cA} (\theta , q)}$. This allows to take the limit $\eps \to 0$ on the left-hand side.
\end{proof}

\begin{proposition}
\label{Prop: Equivalent norms}
Let $\theta \in (0 , 1)$ and $1 \leq q \leq \infty$. Then we have 
\begin{align*}
 \big(X , \dom(\dot \cA)\big)_{\theta , q} = \dot \dom_{\cA} (\theta , q).
\end{align*} 
In particular, there exist two constants $C_1 , C_2 > 0$ such that for all $z \in \dot \dom_{\cA} (\theta , q)$ we have
\begin{align}
\label{Eq: Equivalent norms}
 C_1 \| z \|_{\dot \dom_{\cA} (\theta , q)} \leq \| z \|_{(X , \dom(\dot \cA))_{\theta , q}} \leq C_2 \| z \|_{\dot \dom_{\cA} (\theta , q)}.
\end{align}
\end{proposition}

\begin{proof}
Let $z \in (X , \dom(\dot \cA))_{\theta , q}$ and let $z = x + y$ for some $x \in X$ and $y \in \dom(\dot \cA)$. Then, by virtue of~\eqref{Eq: Canonical representation on sum space},~\eqref{Eq: Smoothing estimate} as well as the boundedness of the semigroup family we find that
\begin{align*}
 t \| \dot \cA \e^{- t \cA} z \|_X \leq  \| (t \cA) \e^{- t \cA} x \|_X + t \| \e^{- t \cA} \dot \cA y \|_X \leq C \Big( \| x \|_X + t \| \dot \cA y \|_X \Big).
\end{align*}
Since $x$ and $y$ are arbitrary, one  takes the infimum over all $x \in X$ and $y \in \dom(\dot \cA)$ with $z = x + y$ and obtain
\begin{align*}
 t \| \cA \e^{- t \cA} z \|_X \leq C K (t , z ; X , \dom(\dot \cA)).
\end{align*}
Consequently, by definition of the real interpolation space norm in~\eqref{Eq: Real interpolation space} it follows that
\begin{align*}
 \| x \|_{\dot \dom_{\cA} (\theta , q)} \leq C \| t \mapsto t^{- \theta} K(t , x ; X , \dom(\dot \cA)) \|_{\LL_{q , *} (0 , \infty)} = C\|x\|_{(X,\dom(\dot\cA))_{\theta,q}}.
\end{align*}

\indent For the other direction, let $z \in \dot \dom_{\cA} (\theta , q)$ and write $z = x + y$ for appropriate $x \in X$ and $y \in \dom(\dot \cA)$. By virtue of~\eqref{Eq: Extension semigroup operators} and~\eqref{Eq: Allerweltsformel 1} we find that
\begin{align*}
 \e^{- t \cA} z &= z - \cA \int_0^t \e^{- s \cA} x \, \d s - \int_0^t \e^{- s \cA} \dot \cA y \, \d s, \end{align*}
 whence
 $$ z =  a_t+b_t \with 
 a_t:=  \cA \int_0^t \e^{- s \cA} x \, \d s + \int_0^t \e^{- s \cA} \dot \cA y \, \d s\andf
 b_t := \e^{- t \cA} z.$$
  By Proposition~\ref{Prop: Representation abstract semigroup}, we know that 
$b_t \in \dom(\dot \cA)$ and it is clear that $a_t \in X$. For $t > 0$ this gives the following control on the $K$-functional
\begin{align*}
 K (t , x ; X , \dom(\dot \cA)) \leq \| a_t \|_X + t \| \dot \cA b_t \|_X.
\end{align*}
The second term on the right-hand side can  be estimated by virtue of Lemma~\ref{Lem: Hardy} by
\begin{align*}
 \| t^{1 - \theta} \dot \cA b_t \|_{\LL_{q , *} (0 , \infty ; X)} \leq C \| z \|_{\dot \dom_{\cA} (\theta , q)}.
\end{align*}
The term involving $a_t$ requires some analysis. Notice that Lemma~\ref{Lem: Hardy} implies that for $t > 0$ we have
\begin{align*}
 s \mapsto \dot \cA \e^{- s \cA} z \in \LL_1 (0 , t ; X).
\end{align*}
Moreover, the properties of the semigroup imply that for  $t > 0$ we have
\begin{align*}
 s \mapsto \e^{- s \cA} \dot \cA y \in \LL_1 (0 , t ; X).
\end{align*}
Since $\dot \cA \e^{- s \cA} z = \cA \e^{- s \cA} x + \e^{- s \cA} \dot \cA y$ we thus find that for  $t > 0$ we have
\begin{align*}
 s \mapsto \cA \e^{- s \cA} x \in \LL_1 (0 , t ; X).
\end{align*}
Since $\cA$ is closed, we conclude that for such $t$ one has
\begin{align*}
 \int_0^t \e^{- s \cA} x \, \d s \in \dom(\cA) \quad \text{and} \quad \cA \int_0^t \e^{- s \cA} x \, \d s = \int_0^t \cA \e^{- s \cA} x \, \d s.
\end{align*}
For these $t$ we thus conclude that
\begin{align*}
 a_t = \cA \int_0^t \e^{- s \cA} x \, \d s + \int_0^t \e^{- s \cA} \dot \cA y \, \d s = \int_0^t \dot \cA \e^{- s \cA} z \, \d s.
\end{align*}
Again by Lemma~\ref{Lem: Hardy} we conclude that
\begin{align*}
 \| t^{- \theta} a_t \|_{\LL_{q , *} (0 , \infty ; X)} \leq \Big\| t^{- \theta} \int_0^t \| s \dot \cA \e^{- s \cA} z \|_X \, \frac{\d s}{s} \Big\|_{\LL_{q , *} (0 , \infty)} \leq C \| z \|_{\dot \dom_{\cA} (\theta , q)}. &\qedhere
\end{align*}
\end{proof}

\begin{remark}
Notice that Haase gives a similar result as Prop.~\ref{Prop: Equivalent norms} in~\cite[Thm.~6.4.5]{Haase}. Also, in~\cite[Sec.~6.3]{Haase} a homogeneous version of an operator in a very general and abstract setting is defined. In his context, however, the domain of the homogeneous operator is complete, see~\cite[bottom of p.~6]{Haak_Haase_Kunstmann}, and thus in general different from the definition given at the beginning of this section.
 In~\cite[p.~233/234]{Haase} the domain of the homogeneous Laplacian on the whole space is calculated and it turns out that the homogeneous operator defined by Haase acts on homogeneous Bessel potential spaces  which are defined as subspaces of $\cS^{\prime} (\IR^n) / \mathcal{P} (\IR^n)$, where $\mathcal{P} (\IR^n)$ denotes the space of all polynomials on $\IR^n$. \par
It goes without saying that an investigation of PDE problems that incorporates also boundary conditions or nonlinearities is quite unpleasant in such spaces. Hence, the formalism above shall provide a tool that is better suited to tackle problems in 
nontrivial domains, or nonlinear PDEs.
\end{remark}

\section{Da Prato -- Grisvard's theorem with homogeneous estimates}\label{Sec: From the Da Prato--Grisvard semi-norm to a real interpolation space norm} 

Next, let $0 < T \leq \infty$, $f \in \LL_q(0 , T ; \dom_{\cA} (\theta , q))$, and let $u$ be given by~\eqref{Eq: Variation of constants} with $x = 0$. Then the following proposition is a classical result of Da Prato and Grisvard~\cite[Thm.~4.7]{DaPrato_Grisvard}. In the literature~\cite{DaPrato_Grisvard, Haase}, however, the result is  proved only on finite time intervals or on the interval $\IR_+ $ but under the condition that $0 \in \rho(\cA)$. To avoid that condition
that will not be satisfied in the applications we have in mind, 
we restrict ourselves to homogeneous estimates (cf.~\cite[Thm.~1.4]{Iwabuchi} for the case of the Laplacian defined on homogeneous Besov-type spaces or~\cite[Sec.~4]{Lunardi} in H\"older spaces). We remark that since the space $Y$ is not necessarily complete, the homogeneous space $\dot \dom_{\cA} (\theta , q)$ does not have to be complete either. As vector valued integration in spaces that are not complete is a bit delicate, we will 
concentrate on the proof of \emph{homogeneous} estimates 
for  solutions to~\eqref{Eq: ACP} corresponding 
to data in \emph{nonhomogeneous} spaces. 
\begin{proposition}
\label{Prop: Da Prato-Grisvard}
Let $\theta \in (0 , 1)$, $1 \leq q \leq \infty$ and $0 < T \leq \infty$. Then, there exists a constant $C > 0$ such that for all $f \in \LL_q(0 , T ; \dom_{\cA} (\theta , q)),$ 
the solution $u$ to $(ACP)$ defined by~\eqref{Eq: Variation of constants} with $x=0$
satisfies  $u(t) \in \dom(\cA)$ for almost every $0 < t < T$ and the homogeneous estimate:
\begin{align*}
 \| \cA u \|_{\LL_q(0 , T ; \dot\dom_{\cA} (\theta , q))} \leq C \| f \|_{\LL_q(0 , T ; \dot \dom_{\cA} (\theta , q))}.
\end{align*}
\end{proposition}

\begin{proof} For the reader's convenience, we present a proof here. 
It will  be divided into five consecutive steps. 
\subsection*{Step~1} First we show that $u$  satisfies $u(t) \in \dom(\cA)$ for almost every $t\in(0,T)$.
For that,  notice that the integral that defines $u$ converges absolutely in $X$ since the semigroup is bounded on the interval $[0 , t]$ and since $f$ lies  in $\LL_q(0 , T ; X)$. To show that $u(t) \in \dom(\cA)$ for almost every $t \in (0 , T)$, we have to establish that  $s \mapsto \cA \e^{- (t - s) \cA} f(s)$ is, for almost every $t \in (0 , T),$ integrable on $(0 , t)$ with respect to the norm of $X$. To do so in 
the case $q<\infty,$ perform first the substitution $\tau = t - s$ and then estimate by H\"older's inequality to obtain
\begin{align*}
 \int_0^t \| \cA \e^{- (t - s) \cA} f(s) \|_X \; \d s &= \int_0^t \tau^{\theta} \tau^{1 - \theta} \| \cA \e^{- \tau \cA} f(t - \tau) \|_X \; \frac{\d \tau}{\tau} \\
 &\leq C t^{\theta} \bigg( \int_0^t \| \tau^{1 - \theta} \cA \e^{- \tau \cA} f(t - \tau) \|_X^q \; \frac{\d \tau}{\tau} \bigg)^{1 / q} \cdotp
\end{align*}
Now, let $0 < T^{\prime} < T$ be arbitrary. According to Fubini's theorem and the substitution $s = t - \tau$, we have
\begin{align*}
 \int_0^{T^{\prime}} \int_0^{t} \| \tau^{1 - \theta} \cA \e^{- \tau \cA} f(t-\tau) \|_X^q \frac{\d \tau}{\tau} \; \d t &= \int_0^{T^{\prime}} \int_{\tau}^{T^{\prime}}   \| \tau^{1 - \theta} \cA \e^{- \tau \cA} f(t-\tau) \|_X^q \,\d t\,\frac{\d \tau}{\tau}\\
  &= \int_0^{T^{\prime}} \int_0^{T^{\prime}-\tau} \| \tau^{1 - \theta} \cA \e^{- \tau \cA} f(s) \|_X^q \; \d s\, \frac{\d \tau}{\tau}\\
   &\leq \int_0^{T^{\prime}} \int_0^{T} \| \tau^{1 - \theta} \cA \e^{- \tau \cA} f(s) \|_X^q \; \frac{\d \tau}{\tau} \, \d s.
\end{align*}
The right-hand side is finite since $f$ is in $\LL_q(0 , T ; \dom_{\cA} (\theta , q))$. 
Since $0 < T^{\prime} < T$ was arbitrary, this implies the finiteness of 
$\displaystyle \int_0^t \| \cA \e^{- (t - s) \cA} f(s) \|_X \,\d s$ for almost every $t \in (0 , T),$
and thus the desired conclusion for the first step
(the easier case $q=\infty$ being left to the reader). 

\subsection*{Step~2} The aim of Steps~2 and~3 is to bound $\cA e^{-\tau\cA}\cA u$ in $\LL_q(0,T;X)$ for each fixed $\tau>0$. We assume  $q$ is finite and consider the case $q = \infty$ in the final fifth step.

We start with
\begin{align*}
 \| \cA \e^{- \tau \cA} \cA u (t) \|_X \leq \int_0^t \| \cA^2 \e^{- (t + \tau - s) \cA} f(s) \|_X \; \d s
 = \int_0^t \| \cA^2 \e^{- (\tau + s) \cA} f(t - s) \|_X \; \d s.
\end{align*}
Next, let $q^{\prime}$ denote the H\"older conjugate exponent to $q$, with the usual convention $1 / \infty = 0$. Let further $\gamma_1 , \gamma_2 \in (0 , 1)$ with $\gamma_1 + \gamma_2 = 1$ and $\gamma_2 > 1 / q^{\prime}$. Then, the analyticity of the semigroup, cf.~\eqref{Eq: Smoothing estimate}, together with H\"older's inequality deliver
\begin{align*}
 \int_0^t \| \cA^2 \e^{- (\tau + s) \cA} f(t - s) \|_X \; \d s &\leq C \int_0^t \frac{1}{\tau + s} \| \cA \e^{- \frac12(\tau + s) \cA} f(t - s) \|_X \; \d s \\
 &\leq C \bigg( \int_0^t \frac{1}{(\tau + s)^{\gamma_1 q}} \| \cA \e^{- \frac12(\tau + s) \cA} f(t - s) \|_X^q \; \d s \bigg)^{\frac{1}{q}} \\
 &\hspace{6.5cm} \cdot \| (\tau + \cdot)^{- \gamma_2} \|_{\LL_{q^{\prime}} (0 , t)} \\
 &\leq C \tau^{1 / q^{\prime} - \gamma_2} \bigg( \int_0^t \frac{1}{(\tau + s)^{\gamma_1 q}} \| \cA \e^{- \frac12(\tau + s) \cA} f(t - s) \|_X^q \; \d s \bigg)^{\frac{1}{q}}.
\end{align*}
 We used the splitting $e^{-t\cA}=e^{-\frac{1}{2} t \cA}e^{-\frac12 t \cA}$
in order to keep  fine properties of considered terms
(note that $C$ depends only on $\sup_{0 < s < \infty} \| s \cA \e^{- s \cA} \|_{\Lop(X)},$ and on $\gamma_2$ and $q$).
Summarizing, we find that for all $0 < t < T$ and $\tau > 0$
\begin{align}
\label{Eq: Estimate of Step 2}
 \| \cA \e^{- \tau \cA} \cA u (t) \|_X \leq C \tau^{1 / q^{\prime} - \gamma_2} \bigg( \int_0^t \frac{1}{(\tau + s)^{\gamma_1 q}} \| \cA \e^{- \frac12(\tau + s)\cA} f(t - s) \|_X^q \; \d s \bigg)^{\frac{1}{q}}.
\end{align}
\subsection*{Step~3}
By virtue of~\eqref{Eq: Estimate of Step 2} and Fubini's theorem, we compute for $\tau > 0$
\begin{align*}
 \| \cA \e^{- \tau \cA} \cA u \|_{\LL_q (0 , T ; X)}^q &= \int_0^T \| \cA \e^{- \tau \cA} \cA u (t) \|_X^q \; \d t \\
 &\leq C \tau^{q / q^{\prime} - \gamma_2 q} \int_0^T \int_0^t \frac{1}{(\tau + s)^{\gamma_1 q}} \| \cA \e^{- \frac12(\tau + s) \cA} f(t - s) \|_X^q \; \d s \; \d t \\
 &= C \tau^{q / q^{\prime} - \gamma_2 q} \int_0^T \biggl(\int_s^T \| \cA \e^{- \frac12(\tau + s) \cA} f(t - s) \|_X^q \; \d t \biggr)\frac{\d s}{(\tau + s)^{\gamma_1 q}}\cdotp
\end{align*}
Finally, the substitution $t^{\prime} = t - s$ delivers, since
$\gamma_1=1-\gamma_2,$ 
\begin{align}
\label{Eq: Estimate of Step 3}
 \| \cA \e^{- \tau \cA} \cA u \|_{\LL_q (0 , T ; X)}^q \leq C \tau^{q / q^{\prime} +\gamma_1 q-q} \int_0^T \biggl(\int_0^{T - s} \| \cA \e^{- \frac12(\tau + s)\cA} f(t) \|_X^q \; \d t\biggr) \frac{\d s}{(\tau + s)^{\gamma_1 q}} \cdotp
\end{align}
\subsection*{Step~4} Still assuming that $q$ is finite, we estimate the full norm of $Au$ in $\LL_q(0 , T ; \dot \dom_{\cA} (\theta , q)).$
Let  $\gamma := (\gamma_1 + 1 / q^{\prime} - \theta) q$.   Use Fubini's theorem as well as~\eqref{Eq: Estimate of Step 3} as follows:
\begin{align*}
 \| \cA u \|_{\LL_q(0 , T ; \dot \dom_{\cA} (\theta , q))}^q &= \int_0^T \int_0^{\infty} \| \tau^{1 - \theta} \cA \e^{- \tau \cA} \cA u(t) \|_X^q \; \frac{\d \tau}{\tau} \; \d t \\
 &= \int_0^{\infty} \tau^{(1 - \theta) q - 1} \int_0^T \| \cA \e^{- \tau \cA} \cA u(t) \|_X^q \; \d t \; \d \tau \\
 &\leq C \int_0^{\infty} \tau^{\gamma - 1} \int_0^T \int_0^{T - s} \| \cA \e^{- \frac12(\tau + s) \cA} f(t) \|_X^q \; \d t \frac{\d s}{(\tau + s)^{\gamma_1 q}}  \; \d \tau.
\end{align*}
Continue to manipulate this expression in the following calculation, by using Fubini's theorem to integrate over $t$ last, proceeding then with the substitution $s^{\prime} = \tau + s$, and using Fubini's theorem again to switch the order of integration over $s$ and $\tau$. Choose $\gamma_1$ in such a way that $\gamma_1 > \theta - 1 / q^{\prime}$, so that $\gamma > 0.$  This results in
\begin{align*}
 \| \cA u \|_{\LL_q(0 , T ; \dot \dom_{\cA} (\theta , q))}^q &\leq C \int_0^T \int_0^{\infty} \tau^{\gamma - 1} \int_0^{T - t} \frac{1}{(\tau + s)^{\gamma_1 q}} \| \cA \e^{- \frac12(\tau + s) \cA} f(t) \|_X^q \; \d s \; \d \tau \; \d t \\
 &= C \int_0^T \int_0^{\infty} \tau^{\gamma - 1} \int_{\tau}^{T + \tau - t} \frac{1}{s^{\gamma_1 q}} \| \cA \e^{-\frac s2 \cA} f(t) \|_X^q \; \d s \; \d \tau \; \d t \\
 &= C \int_0^T \int_0^{\infty}  \frac{1}{s^{\gamma_1 q}} \| \cA \e^{- \frac s2 \cA} f(t) \|_X^q \int_{\max\{ 0 , s + t - T \}}^s \tau^{\gamma - 1} \; \d \tau \; \d s \; \d t \\
 &\leq \frac{C}{\gamma} \int_0^T \int_0^{\infty} s^{\gamma - \gamma_1 q} \| \cA \e^{- \frac s2 \cA} f(t) \|_X^q \; \d s \; \d t \\
 &= \frac{C}{\gamma} \int_0^T \int_0^{\infty} \| s^{1 - \theta} \cA \e^{- \frac s2 \cA} f(t) \|_X^q \; \frac{\d s}{s} \; \d t \\
 &= \frac{C}{\gamma} \| f \|_{\LL_q(0 , T ; \dot \dom_{\cA}(\theta , q))}^q.
\end{align*}

\subsection*{Step 5} Here we consider the case $q = \infty$. For $0 < t < T$ and $0 < \tau < \infty$ we directly find by~\eqref{Eq: Smoothing estimate} that
\begin{align*}
 \Big\| \tau^{1 - \theta} \cA \e^{- \tau \cA} \cA \int_0^t \e^{- (t - s) \cA} f(s) \, \d s \Big\|_X &\leq C \tau^{1 - \theta} \int_0^t \frac{1}{\tau + t - s} \| \cA \e^{- \frac{1}{2} (\tau + t - s) \cA} f(s) \|_X \, \d s \\
 &\leq C \tau^{1 - \theta} \int_0^t \frac{1}{(\tau + t - s)^{2 - \theta}} \, \d s \| f \|_{\LL_{\infty} (0 , T ; \dot \dom_{\cA} (\theta , \infty))}.
\end{align*}
Finally, notice that by the substitution rule, we have
\begin{align*}
 \tau^{1 - \theta} \int_0^t \frac{1}{(\tau + t - s)^{2 - \theta}} \, \d s \leq \int_0^{\infty} \frac{1}{(1 + r)^{2 - \theta}} \, \d r. &\qedhere
\end{align*}
\end{proof}

In order to establish that the solution is bounded in time when spatially measured in its trace space, we prove a Fubini-type property of the Da Prato -- Grisvard norm with respect to $\theta$. We start with the following Schur type estimate, which is a quantified generalization of~\cite[Thm.~5.2.2]{Haase}.

\begin{lemma}
\label{Lem: Schur type estimate}
Let $\cA$ be sectorial of angle $\omega \in [0 , \pi)$ and let $\omega < \phi < \pi$. Let $\Phi , \Psi : \Sec_{\phi} \to \IC$ be holomorphic such that there exists $C > 0$ and $0 < \alpha < \beta < \infty$ with
\begin{align}
\label{Eq: Quantified H_0^infty}
 \lvert \Phi(z) \rvert , \lvert \Psi (z) \rvert \leq \frac{C \lvert z \rvert^{\alpha}}{1 + \lvert z \rvert^{\beta}} \qquad (z \in \Sec_{\phi}).
\end{align}
Define for $s , t > 0$ the operators $\Phi(t \cA)$ and $\Psi(s \cA)$ as usual via the Cauchy integral, i.e., if $\gamma$ describes $\partial \Sec_{\vartheta}$ in a counterclockwise manner for some $\vartheta \in (\omega , \phi)$, define
\begin{align*}
 \Phi(t \cA) := \frac{1}{2 \pi \ii} \int_{\gamma} \Phi(t z) (z - \cA)^{-1} \; \d z \quad \text{and} \quad \Psi(s \cA) := \frac{1}{2 \pi \ii} \int_{\gamma} \Psi(s z) (z - \cA)^{-1} \; \d z.
\end{align*}
Then, for each $\eta \in \IR$ with $\lvert \eta \rvert < \min\{ \alpha , \beta - \alpha \}$ and each $1 \leq q \leq \infty$ it holds
\begin{align*}
 \sup_{s > 0} \| t \mapsto (t s^{-1})^{\eta} \Phi(t \cA) \Psi (s \cA) \|_{\LL_{q , *} (0 , \infty ; \Lop(X))} < \infty.
\end{align*}
\end{lemma}

\begin{proof}
Notice that, 
\begin{align*}
 \Phi(t \cA) \Psi(s \cA) = \frac{1}{2 \pi \ii} \int_{\gamma} \Phi(t z) \Psi(s z) (z - \cA)^{-1} \; \d z.
\end{align*}
Estimating by virtue of~\eqref{Eq: Quantified H_0^infty} and~\eqref{Eq: McIntosh Resolvent estimate} together with the substitution $z^{\prime} = s z$ yields
\begin{align*}
 \| \Phi(t \cA) \Psi(s \cA) \|_{\Lop(X)} \leq C \int_{\gamma} \frac{(ts^{-1})^{\alpha} \lvert z \rvert^{2 \alpha}}{(1 + \lvert t s^{-1} z \vert^{\beta}) (1 + \lvert z \vert^{\beta})} \; \frac{\lvert \d z \rvert}{\lvert z \rvert}\cdotp
\end{align*}
Now, multiply by $(t s^{-1})^{\eta}$ and take the $\LL_{q , *}(0 , \infty)$-norm of this inequality. 
Together with  the triangle inequality for integrals, this gives
\begin{align*}
 \| t \mapsto (t s^{-1})^{\eta} \Phi(t \cA) \Psi (s \cA) \|_{\LL_{q , *} (0 , \infty ; \Lop(X))} \leq C \int_{\gamma} \Big\| t \mapsto \frac{(ts^{-1} \lvert z \rvert)^{\alpha + \eta}}{1 + (t s^{-1} \lvert z \rvert)^{\beta}} \Big\|_{\LL_{q , *} (0 , \infty)} \frac{\lvert z \rvert^{\alpha - \eta}}{1 + \lvert z \rvert^{\beta}} \; \frac{\lvert \d z \rvert}{\lvert z \rvert}\cdotp
\end{align*}
Finally, perform in the $\LL_{q , *} (0 , \infty)$-norm the substitution $t^{\prime} = t s^{-1} \lvert z \rvert$ yielding
\begin{align*}
 \| t \mapsto (t s^{-1})^{\eta} \Phi(t \cA) \Psi (s \cA) \|_{\LL_{q , *} (0 , \infty ; \Lop(X))} \leq C \int_{\gamma} \Big\| t \mapsto \frac{t^{\alpha + \eta}}{1 + t^{\beta}} \Big\|_{\LL_{q , *} (0 , \infty)} \frac{\lvert z \rvert^{\alpha - \eta}}{1 + \lvert z \rvert^{\beta}} \; \frac{\lvert \d z \rvert}{\lvert z \rvert}\cdotp
\end{align*}
Observe that the right-hand side is independent of $s$ and is finite if $\lvert \eta \rvert < \alpha$ and $\lvert \eta \rvert < \beta - \alpha$. Conclude by taking the supremum over $s > 0$.
\end{proof}

The following lemma permits to represent the norms $\dot \dom_{\cA} (\theta , q)$ in a different manner. Notice that this is a special case of the far more general statement in~\cite[Thm.~6.4.2]{Haase}.

\begin{lemma}
\label{Lem: Other representation of norm}
For $0 < \theta < 1$, $1 \leq q \leq \infty$ and $x \in \dot \dom_{\cA} (\theta , q)$ it holds
\begin{align*}
 \| x \|_{\dot \dom_{\cA} (\theta , q)} \simeq \| t \mapsto t^{2 - \theta} \cA \dot \cA \e^{- t \cA} x \|_{\LL_{q , *} (0 , \infty ; X)}.
\end{align*}
\end{lemma}

\begin{proof}
The bound of the right-hand side by the left-hand side directly follows from the following estimate, which in turn is a consequence of Remark~\ref{Rem: Semigroup property of extended semigroup},~\eqref{Eq: Calculation for semigroup property of extended semigroup} and~\eqref{Eq: Smoothing estimate}:
\begin{align*}
 t^{2 - \theta} \| \cA \dot \cA \e^{- t \cA} x \|_X \leq C t^{1 - \theta} \| \dot \cA \e^{- \frac{t}{2} \cA} x \|_X.
\end{align*}
\indent To bound the left-hand side by the right-hand side, notice that since $- \cA$ generates a bounded analytic semigroup, $\cA$ is sectorial of some angle $\omega \in [0 , \pi / 2)$. Let $\phi \in (\omega , \pi / 2)$ and define the holomorphic functions $f(z) := z^3 \e^{- 2 z}$ and $\Psi (z) := z \e^{- z}$. Note that $f(z) = z^2 \e^{- z} \Psi(z)$ and that  $\Psi$ satisfies the assumptions of Lemma~\ref{Lem: Schur type estimate} for $\alpha = 1$ and arbitrary large $\beta$ on $\Sec_{\phi}$. Define the constant
\begin{align*}
 \frac{1}{c} := \int_0^{\infty} f(t) \; \frac{\d t}{t} \in (0 , \infty).
\end{align*}
Then, by virtue of McIntosh's approximation theorem~\cite[Thm.~5.2.6]{Haase} it holds in the sense of an improper Riemann integral for all $y \in \overline{\dom(\cA) \cap \Rg(A)}$
\begin{align}
\label{Eq: McIntosh approximation Phi}
 c \int_0^{\infty} f(t \cA) y \; \frac{\d t}{t} = y.
\end{align}
For $s > 0$ apply~\eqref{Eq: McIntosh approximation Phi} with $y := s^{1 - \theta}  \dot\cA \e^{- s \cA} x \in \overline{\dom (\cA) \cap \Rg(\cA)}$ followed by H\"older's inequality to obtain
\begin{align*}
 s^{1 - \theta} \dot \cA \e^{- s \cA} x &= c \int_0^{\infty} \Big(\frac{t}{s} \Big)^{\theta} \Psi(t \cA) \Psi(s \cA) t^{2 - \theta} \cA \dot \cA \e^{- t \cA} x \; \frac{\d t}{t}\cdotp
\end{align*}
In the case $q = \infty$, we already get
\begin{align*}
 \sup_{s > 0} \| s^{1 - \theta} \dot \cA \e^{- s \cA} x \|_X \leq c \sup_{s > 0} \int_0^{\infty} \Big(\frac{t}{s} \Big)^{\theta} \| \Psi(t \cA) \Psi(s \cA) \|_{\Lop(X)} \; \frac{\d t}{t} \; \sup_{t > 0} \| t^{2 - \theta} \cA \dot \cA \e^{- t \cA} x \|_X.
\end{align*}
Notice that the constant on the right-hand side is finite by Lemma~\ref{Lem: Schur type estimate}. \par
If $1 \leq q < \infty$, then H\"older's inequality together with Lemma~\ref{Lem: Schur type estimate} imply
\begin{align*}
 \| s^{1 - \theta} \dot \cA \e^{- s \cA} x \|_X \leq C \bigg( \int_0^{\infty} \Big( \frac{t}{s} \Big)^{\theta} \| \Psi(t \cA) \Psi(s \cA) \|_{\Lop(X)} \| t^{2 - \theta} \cA \dot \cA \e^{- t \cA} x \|_X^q \; \frac{\d t}{t} \bigg)^{\frac{1}{q}}\cdotp
\end{align*}
By Fubini's theorem and again by Lemma~\ref{Lem: Schur type estimate}, we finally find
\begin{align*}
 &\bigg( \int_0^{\infty} \| s^{1 - \theta} \dot \cA \e^{- s \cA} x \|_X^q \; \frac{\d s}{s} \bigg)^{\frac{1}{q}} \\
 &\qquad \leq C \bigg( \int_0^{\infty} \int_0^{\infty} \Big( \frac{t}{s} \Big)^{\theta} \| \Psi(t \cA) \Psi(s \cA) \|_{\Lop(X)} \; \frac{\d s}{s} \; \| t^{2 - \theta} \cA \dot \cA \e^{- t \cA} x \|_X^q \; \frac{\d t}{t} \bigg)^{\frac{1}{q}} \\
 &\qquad \leq C \bigg( \sup_{t > 0} \int_0^{\infty} \Big( \frac{t}{s} \Big)^{\theta} \| \Psi(t \cA) \Psi(s \cA) \|_{\Lop(X)} \; \frac{\d s}{s} \bigg)^{\frac{1}{q}} \bigg( \int_0^{\infty} \| t^{2 - \theta} \cA \dot \cA \e^{- t \cA} x \|_X^q \; \frac{\d t}{t} \bigg)^{\frac{1}{q}}\cdotp \qedhere
\end{align*}
\end{proof}

The following proposition states a Fubini type property of the Da Prato -- Grisvard semi-norm with respect to the parameter $\theta$. 

\begin{proposition}
\label{Prop: Da Prato-Grisvard Reiteration}
For $\theta , \eta > 0$ with $\theta + \eta < 1$, $1 \leq q \leq \infty$, and $x \in \dot \dom_{\cA} (\theta + \eta , q)$ it holds
\begin{align*}
 \| t \mapsto t^{1 - \eta} \dot \cA \e ^{- t \cA} x \|_{\LL_{q , *} (\IR_+  ; \dot \dom_{\cA} (\theta , q))} \simeq \| x \|_{\dot \dom_{\cA} (\theta + \eta , q)}.
\end{align*}
\end{proposition}

\begin{proof}
We start by bounding the left-hand side by the right-hand side. Here it holds in the case $q < \infty$ by using the definition of the semi-norms, the smoothing estimate~\eqref{Eq: Smoothing estimate} of the semigroup as well as the substitution $s^{\prime} = s + t$ and Fubini's theorem that
\begin{align*}
 \bigg( \int_0^{\infty} \| t^{1 - \eta} \dot \cA \e ^{- t \cA} x \|_{\dot \dom_{\cA} (\theta , q)}^q \; \frac{\d t}{t} \bigg)^{\frac{1}{q}} &= \bigg( \int_0^{\infty} \int_0^{\infty} \| s^{1 - \theta} t^{1 - \eta} \cA \dot \cA \e ^{- (s + t) \cA} x \|_X^q \; \frac{\d s}{s} \; \frac{\d t}{t} \bigg)^{\frac{1}{q}} \\
 &\leq C \bigg( \int_0^{\infty} \int_t^{\infty} \Big\| \frac{(s - t)^{1 - \theta} t^{1 - \eta}}{s} \dot \cA \e ^{- \frac{s}{2} \cA} x \Big\|_X^q \; \frac{\d s}{s - t} \; \frac{\d t}{t} \bigg)^{\frac{1}{q}} \\
 &= C \bigg( \int_0^{\infty} \int_0^s \Big\| \frac{(s - t)^{1 - \theta} t^{1 - \eta}}{s} \dot \cA \e ^{- \frac{s}{2} \cA} x \Big\|_X^q \; \frac{\d t}{t(s - t)} \; \d s \bigg)^{\frac{1}{q}}\cdotp
\end{align*}
Next, perform the substitution $t^{\prime} = t / s$ what delivers
\begin{align*}
 \bigg( \int_0^{\infty} \int_0^s \Big\| \frac{(s - t)^{1 - \theta} t^{1 - \eta}}{s} \dot \cA \e ^{- \frac{s}{2} \cA} x \Big\|_X^q \; \frac{\d t}{t(s - t)} \; \d s \bigg)^{\frac{1}{q}} \leq C \bigg( \int_0^{\infty} \| s^{1 - (\theta + \eta)} \dot \cA \e^{- \frac{s}{2} \cA} x \|_X^q \; \frac{\d s}{s} \bigg)^{\frac{1}{q}}\cdotp
\end{align*}
The constant $C > 0$ in the latter inequality is given by
\begin{align*}
 C = \bigg( \int_0^1 \big( (1 - t)^{1 - \theta} t^{1 - \eta} \big)^q \; \frac{\d t}{t(1 - t)} \bigg)^{\frac{1}{q}},
\end{align*}
which is finite whenever $\theta , \eta < 1$. This proves the inequality in the case $q < \infty$. In the case $q = \infty$ calculate similarly for $t , s > 0$
\begin{align*}
 \| s^{1 - \theta} t^{1 - \eta} \cA \dot \cA \e^{- (s + t) \cA} x \|_X \leq C \frac{s^{1 - \theta} t^{1 - \eta} }{s + t} \| \dot \cA \e^{- \frac{s + t}{2} A} x \|_X \leq C \frac{s^{1 - \theta} t^{1 - \eta} }{(s + t)^{2 - (\theta + \eta)}} \| x \|_{\dot \dom_{\cA} (\theta + \eta , \infty)}.
\end{align*}
Notice that
\begin{align*}
 \frac{s^{1 - \theta} t^{1 - \eta} }{(s + t)^{2 - (\theta + \eta)}} = \frac{(t / s)^{1 - \eta} }{(1 + (t / s))^{2 - (\theta + \eta)}}
\end{align*}
and the right-hand side is bounded for $t/s$ in the range  $(0 , \infty).$
 Thus, taking the supremum over $s$ and then over $t$ in the inequality above delivers 
 the desired estimate in the case $q = \infty$. \par
To bound the right-hand side by the left-hand side,  it suffices to 
follow the end of the proof of  Lemma~\ref{Lem: Other representation of norm}. 
Keeping the same notation, we take 
$$
 s > 0,   \qquad y := (2 s)^{2 - (\theta + \eta)} \cA \dot \cA \e^{- 2 s A} x \in \dom(\cA) \cap \Rg(\cA).
 $$
 Then,  apply~\eqref{Eq: McIntosh approximation Phi} but with $f (z) := z^2 \e^{- z} = z \e^{- z} \Psi(z)$ followed by H\"older's inequality to obtain
\begin{align*}
& \| (2 s)^{2 - (\theta + \eta)} \cA \dot \cA \e^{- 2 s \cA} x \|_X \\
 &\qquad \leq C \int_0^{\infty} \Big( \frac{t}{s} \Big)^{\eta} \| \Psi (t \cA) \Psi(s \cA) s^{1 - \theta} \dot \cA \e^{- s \cA} t^{1 - \eta} \dot \cA \e^{- t \cA} x \|_X \; \frac{\d t}{t} \\
 &\qquad \leq C \| t \mapsto (t s^{-1})^{\eta} \Psi(t \cA) \Psi(s \cA) \|_{\LL_{q^{\prime} , *} (0 , \infty ; \Lop(X))} \| t \mapsto s^{1 - \theta} \dot \cA \e^{- s \cA} t^{1 - \eta} \dot \cA \e^{- t \cA} x \|_{\LL_{q , *} (0 , \infty ; X)}.
\end{align*}
Taking the $\LL_{q , *} (0 , \infty)$-norm of this inequality with respect to $s > 0$, we conclude the proof by virtue of Lemmas~\ref{Lem: Schur type estimate} and~\ref{Lem: Other representation of norm}.
\end{proof}

The result of Da Prato -- Grisvard, Proposition~\ref{Prop: Da Prato-Grisvard}, gives only regularity estimates to solutions to~\eqref{Eq: ACP} with $x = 0$. The following lemma provides homogeneous estimates in the case $x \neq 0$ and $f = 0$.

\begin{lemma}
\label{Lem: Abstract estimate of initial values}
Let $1 \leq q < \infty$ and $\theta \in (0 , 1 / q).$  There exists a constant $C > 0$ such that for all $x \in \dot \dom_{\cA} (1 + \theta - 1 / q , q)$ and   $0 < T \leq \infty,$
we have 
\begin{align*}
 \| t \mapsto \dot \cA \e^{- t \cA} x \|_{\LL_q (0 , T ; \dot \dom_{\cA} (\theta , q))} \leq C \| x \|_{\dot \dom_{\cA} (1 + \theta - 1 / q, q)}. 
\end{align*}
If $q = \infty$, then there exists $C > 0$ such that for each $\theta \in (0 , 1)$ and all $x \in \dom(\cA^2)$ it holds
\begin{align*}
 \| t \mapsto \dot \cA \e^{- t \cA} x \|_{\LL_{\infty} (0 , T ; \dot \dom_{\cA} (\theta , \infty))} \leq \| \cA x \|_{\dot \dom_{\cA}(\theta , \infty)}.
\end{align*}
\end{lemma}

\begin{proof}
The statement for $q = \infty$ readily follows from the fact that $(\e^{- t \cA})_{t \geq 0}$ is uniformly bounded with respect to the $\dot \dom_{\cA} (\theta , \infty)$-norm. \par
Let $1 \leq q < \infty$. By virtue of~\eqref{Eq: Smoothing estimate}, the following estimate is valid
\begin{align*}
 \| t \mapsto \dot \cA \e^{- t \cA} x \|_{\LL_q (0 , T ; \dot \dom_{\cA} (\theta , q))}^q &= \int_0^T \int_0^{\infty} \| \tau^{1 - \theta} \cA \dot \cA \e^{- (t + \tau) \cA} x \|_X^q \; \frac{\d \tau}{\tau} \; \d t \\
 &\leq C \int_0^T \int_0^{\infty} \frac{\tau^{q (1 - \theta)}}{(t + \tau)^q} \| \dot \cA \e^{- \frac12(t + \tau)  \cA} x \|_X^q \; \frac{\d \tau}{\tau} \; \d t.
\end{align*}
Continue with Fubini's theorem followed by the substitution $s = t + \tau$ to obtain
\begin{align*}
 \int_0^T \int_0^{\infty} \frac{\tau^{q (1 - \theta)}}{(t + \tau)^q} \| \dot \cA \e^{- \frac12(t + \tau)  \cA} x \|_X^q \; \frac{\d \tau}{\tau} \; \d t = \int_0^{\infty} \int_{\tau}^{T + \tau} \frac{\tau^{q (1 - \theta)}}{s^q} \| \dot \cA \e^{- \frac s2 \cA} x \|_X^q \; \d s \; \frac{\d \tau}{\tau}\cdotp
\end{align*}
Another application of Fubini's theorem finally yields
\begin{align*}
 \int_0^{\infty} \int_{\tau}^{T + \tau} \frac{\tau^{q (1 - \theta)}}{s^q} \| \dot \cA \e^{- \frac s2 \cA} x \|_X^q \; \d s \; \frac{\d \tau}{\tau} &= \int_0^{\infty} \biggl(\int_{\max\{ 0 , s-T \}}^s \tau^{q(1 - \theta) - 1} \; \d \tau \biggr) s^{- q} \| \dot \cA \e^{- \frac{s}{2} \cA} x \|_X^q \; \d s \\
 &\leq C \int_0^{\infty} \| s^{1 / q - \theta} \dot \cA \e^{- \frac{s}{2} \cA} x \|^q_X \; \frac{\d s}{s} \\
 &= C\| x \|_{\dot \dom_{\cA} (1 + \theta - 1 / q , q)}^q. \qedhere
\end{align*}
\end{proof}

If $u$ is given by~\eqref{Eq: Variation of constants}, then it can be shown that $u$ is bounded with respect to $t$ with values in $\dot \dom_{\cA} (1 + \theta - 1 / q , q)$. While in the general maximal regularity theory this follows by the trace method, we present here an elementary proof that follows from Da Prato -- Grisvard theorem. In the following lemma, we are required to assume that $f$ takes values in the inhomogeneous space $\dom_{\cA} (\theta , q)$ 
so that  the convolution integral in~\eqref{Eq: Variation of constants} is well defined.
\begin{lemma}
\label{Lem: Linfty estimate}
Let $1 \leq q < \infty$, $\theta \in (0 , 1 / q)$ and $0 < T \leq \infty$. Then, there exists a constant $C > 0$ such that for all $x \in \dot \dom_{\cA} (1 + \theta - 1 / q , q)$ and $f \in \LL_q (0 , T ; \dom_{\cA} (\theta , q)),$ it holds, for $u$ given by~\eqref{Eq: Variation of constants},
\begin{align*}
 \| u \|_{\LL_{\infty} (0 , T ; \dot \dom_{\cA} (1 + \theta - 1 / q , q))} \leq C \Big( \| x \|_{\dot \dom_{\cA} (1 + \theta - 1 / q , q)} + \| f \|_{\LL_q (0 , T ; \dot \dom_{\cA} (\theta , q))} \Big)\cdotp
\end{align*}
Furthermore, the function $t \mapsto u (t)$ belongs to $\cC(\IR_+;  \dot \dom_{\cA} (1 + \theta - 1 / q , q)).$
\end{lemma}

\begin{proof} First consider the case  $q > 1$. Then, the fundamental theorem of calculus yields
\begin{align*}
 u(t) = x + \int_0^t u^{\prime} (s) \; \d s,\quad  0 < t < T. 
\end{align*}
Since $x \in \dot \dom_{\cA} (1 + \theta - 1 / q , q)$, the first  term is estimated trivially. For the other term, we use Proposition~\ref{Prop: Da Prato-Grisvard Reiteration} and~\eqref{Eq: Smoothing estimate}:
\begin{align*}
 \Big\| \int_0^t u^{\prime} (s) \; \d s \Big\|_{\dot \dom_{\cA} (1 + \theta - 1 / q, q)} &\leq C  \Big\| \tau \mapsto \tau^{1 / q} \cA \e^{- \tau \cA} \int_0^t u^{\prime} (s) \; \d s \Big\|_{\LL_{q , *} (0 , \infty ; \dot \dom_{\cA} (\theta, q))} \\
 &\leq C \Big\| \tau \mapsto \tau^{1 / q - 1} \int_0^t u^{\prime} (s) \; \d s \Big\|_{\LL_{q , *} (0 , \infty ; \dot \dom_{\cA} (\theta, q))}.
\end{align*}
Now, Hardy's inequality~\cite[Lem.~6.2.6]{Haase} yields
\begin{align*}
 \Big\| \tau \mapsto \tau^{1 / q - 1} \int_0^t u^{\prime} (s) \; \d s \Big\|_{\LL_{q , *} (0 , \infty ; \dot \dom_{\cA} (\theta, q))} &\leq \frac{q}{q - 1} \| s \mapsto s^{1 + 1 / q - 1} u^{\prime} (s) \|_{\LL_{q , *} (0 , \infty ; \dot \dom_{\cA} (\theta , q))} \\
 &= \frac{q}{q - 1} \| u^{\prime} \|_{\LL_q (0 , \infty ; \dot \dom_{\cA} (\theta , q))}.
\end{align*}
The latter term is controlled by virtue of Proposition~\ref{Prop: Da Prato-Grisvard}. \par
The case $q = 1$ follows directly by the variation of constants formula. Indeed, due to the boundedness of the semigroup on $\dot \dom_{\cA} (\theta , 1)$ the term $\e^{- t \cA} x$ is harmless and
the second one is estimated as
\begin{align*}
 \Big\| \int_0^t \e^{- (t - s) \cA} f(s) \; \d s \Big\|_{\dot \dom_{\cA} (\theta , 1)} \leq C \int_0^t \| f(s) \|_{\dot \dom_{\cA} (\theta , 1)} \; \d s.
\end{align*}
\indent We leave the proof of the time continuity to the reader.
\end{proof}


Altogether, this gives the following version of  Da Prato and Grisvard theorem.

\begin{theorem}
\label{Thm: Full Da Prato - Grisvard}
Let $X$ be a Banach space and $- \cA$ be the generator of a bounded analytic semigroup on $X$ subject to Assumptions~\ref{Ass: Homogeneous operator} and~\ref{Ass: Intersection property}. Let $1 \leq q < \infty$, $\theta \in (0 , 1 / q)$ and let $0 < T \leq \infty$. For $f \in \LL_q (0 , T ; \dom_{\cA} (\theta , q))$ and $x \in \dot \dom_{\cA} (1 + \theta - 1 / q , q),$ the mild solution  
$$u \in \cC ([0 , T) ; \dot \dom_{\cA} (1 + \theta - 1 / q , q))$$
to~\eqref{Eq: ACP} is given by
\begin{align*}
 u(t) = \e^{- t \cA} x + \int_0^t \e^{- (t - s) \cA} f(s) \; \d s
\end{align*}
and satisfies $u(t) \in \dom(\dot \cA)$ for almost every $t \in (0 , T).$
\smallbreak
 Furthermore, there exists $C > 0$ such that
\begin{align*}
 \| u \|_{\LL_{\infty} (0 , T ; \dot \dom_{\cA} (1 + \theta - 1 / q , q))} + \| u^{\prime} ,  \cA u \|_{\LL_q (0 , T ; \dot \dom_{\cA} (\theta , q))} \leq C \Big( \| f \|_{\LL_q (0 , T ; \dot \dom_{\cA} (\theta , q))} + \| x \|_{\dot \dom_{\cA} (1 + \theta - 1 / q , q)} \Big)\cdotp
\end{align*}
In case $q=\infty$ we assume in addition that $x \in \dom(\cA^2)$ and then 
for each $\theta \in (0 , 1)$ 
\begin{align*}
 \| u^{\prime} , \cA u \|_{\LL_{\infty} (0 , T ; \dot \dom_{\cA} (\theta , \infty))}
 \leq C\Big(\| f \|_{\LL_\infty(0,T;{\dot \dom_{\cA}(\theta , \infty))}}+\| \cA x \|_{\dot \dom_{\cA}(\theta , \infty)}\Big)\cdotp
\end{align*}
\end{theorem}

\bigskip

The above result is a homogeneous version of the classical one. This abstract theorem will be a subject to nontrivial application for systems of PDEs. The definition of $\dot \dom_{\cA} (\theta , q)$ related to the real interpolation space $(X , \dom(\dot \cA))_{\theta , q}$ yields a natural framework of homogeneous Besov spaces. As these spaces might not be complete for particular choices of parameters, the space $\dot \dom_{\cA} (\theta , q)$ might not be complete as well. For this reason, Theorem~\ref{Thm: Full Da Prato - Grisvard} is formulated for $f$ being merely in the inhomogeneous space $\LL_q (0 , T ; \dom_{\cA} (\theta , q))$, but with estimates in homogeneous spaces. Notice that the density result proven in Lemma~\ref{Lem: Density in abstract spaces} allows to formulate Theorem~\ref{Thm: Full Da Prato - Grisvard} for all $f \in \LL_q (0 , T ; \dot \dom_{\cA} (\theta , q))$ whenever $\dot \dom_{\cA} (\theta , q)$ is complete. \smallbreak
Systems in unbounded domains 
require homogeneous settings. 
For that reason, having 
Theorem~\ref{Thm: Full Da Prato - Grisvard} at hand will be crucial
in the  two chapters dedicated to fluid mechanics systems
in the half-space (or perturbation of it).

\chapter{The functional setting and basic interpolation results}
\label{Sec: The functional setting and basic interpolation results}

In this chapter, we introduce the functional framework which builds the basis for our further  investigations. 
 We start with the definition of  homogeneous Besov, Bessel potential and Sobolev spaces and continue with their solenoidal counterparts, 
 having in mind to derive interpolation identities
 first in the $\IR^n$ case, then in the $\IR^n_+$ case. 

Let us first fix some notation. Throughout the first part of the paper, 
if not otherwise stated then the space dimension $n$ is any positive integer. We denote the Schwartz space by $\cS (\IR^n)$, the space of tempered distributions by $\cS^{\prime} (\IR^n)$ and the Fourier transform on $\cS^{\prime} (\IR^n)$ by $\cF$. Throughout, $p^{\prime} \in [1 , \infty]$ denotes the H\"older conjugate exponent to $p \in [1 , \infty]$, i.e. $1/p+1/p'=1$. The natural numbers are denoted  by $\IN = \{1 , 2 , 3 , \dots \}$ and $\IN_0 := \IN \cup \{ 0 \}$. 
Finally, for any wo linear normed spaces $X$ and $Y$, the notation $X\hookrightarrow Y$ means continuous embedding.

\section{Besov spaces and Littlewood-Paley decomposition}
In order to define Besov spaces and Littlewood-Paley decomposition, we fix a smooth function $\chi$ supported in, say, the ball $B(0,4/3)$ of $\IR^n,$
and with value $1$ on $B(0,3/4),$ then we set 
$$\varphi:=\chi(\cdot/2)-\chi(\cdot).$$
We thus have
\begin{align}
\label{Eq: Defining Fourier cutoff}
 \supp \varphi \subset \{ \xi \in \IR^n : 3/4 \leq \lvert \xi \rvert \leq 8/3 \}
 \andf\sum_{k = - \infty}^{\infty} \varphi (2^{- k} \xi) = 1 \text{ for all } \xi \in \IR^n \setminus \{ 0 \}.
\end{align}
The definitions of the dyadic blocks
\begin{align*}
 \dot \Delta_k f := \cF^{-1} \varphi (2^{- k} \cdot) \cF f \quad \text{for} \quad k \in \IZ ,\  f \in \cS^{\prime} (\IR^n)
\end{align*}
and of the low  frequency cut-off 
\begin{align*}
 \dot S_k f := \cF^{-1} \chi(2^{-k}\cdot) \cF f,\quad k\in\IZ
\end{align*}
lead for $s \in \IR$, $p , q \in [1 , \infty]$ and $f \in \cS^{\prime} (\IR^n)$ to the homogeneous Besov space norm and the inhomogeneous Besov spaces norms, respectively, which are defined as
\begin{align*}
 \| f \|_{\dot \B^s_{p , q} (\IR^n)} := \big\| (2^{s k} \| \dot \Delta_k f \|_{\LL_p (\IR^n)})_{k \in \IZ} \big\|_{\ell^q (\IZ)}
\end{align*}
and
\begin{align*}
 \| f \|_{\B^s_{p , q} (\IR^n)} := \| \dot S_0 f \|_{\LL_p (\IR^n)} + \big\| (2^{s k} \| \dot \Delta_k f \|_{\LL_p (\IR^n)})_{k \in \IN_0} \big\|_{\ell^q (\IN_0)}.
\end{align*}

As $\|f\|_{\dot \B^s_{p , q} (\IR^n)} = 0$ for $f$ in $\cS'(\IR^n)$ does not
imply that $f\equiv0,$ functions contained in a homogeneous Besov space will be assumed to have the following control on the low Fourier frequencies (after the definition given 
in~\cite[Chap.\@ 2]{Bahouri_Chemin_Danchin}).
\begin{definition}
\label{Def: Low frequency control}
Let $\cS_h^{\prime} (\IR^n)$ denote the space of all tempered distributions $f$ that satisfy
\begin{align*}
 \lim_{\lambda \to \infty} \| \theta (\lambda D) f \|_{\LL_{\infty} (\IR^n)} = 0 \quad 
 \hbox{for any } \ 
 \theta \in \C_c^{\infty} (\IR^n).
\end{align*}
Here, $\theta (\lambda D) f$ stands for $\cF^{-1} \theta (\lambda \cdot) \cF f$.
\end{definition}

The above condition of convergence is equivalent to the fact 
that 
\begin{equation}\label{eq:BF}
\|\dot S_jf\|_{\LL_\infty(\IR^n)}\to 0\quad\hbox{when}\quad j\to-\infty.\end{equation}
Owing to Bernstein's inequality, if   $f \in \cS^{\prime} (\IR^n)$ satisfies  $\dot S_j f \in \LL_p (\IR^n)$ for some $1 \leq p < \infty$ and $j\in\IZ,$  then $f$ is contained in $\cS_h^{\prime} (\IR^n)$.

\begin{definition}
For $s \in \IR$ and $p , q \in [1 , \infty]$ the space $\dot \B^s_{p , q} (\IR^n)$ is defined to be a set of all distributions $f \in \cS_h^{\prime} (\IR^n)$ that satisfy
\begin{align*}
 \| f \|_{\dot \B^s_{p , q} (\IR^n)} < \infty.
\end{align*}
The space $\B^s_{p , q} (\IR^n)$ is the  set of all tempered distributions $f \in \cS^{\prime} (\IR^n)$ that satisfy
\begin{align*}
 \| f \|_{\B^s_{p , q} (\IR^n)} < \infty.
\end{align*}
\end{definition}

Considering only tempered distributions in $\cS_h^{\prime} (\IR^n)$ ensures that $\dot \B^s_{p , q} (\IR^n)$ is a normed space. Moreover, it is a Banach space whenever 
(see~\cite{Bahouri_Chemin_Danchin})
\begin{align}
\label{Eq: Banach space conditions}
 s < \frac{n}{p} \quad \text{ or }\quad q = 1\!\andf\! s \leq \frac{n}{p}\cdotp
\end{align}
The properties of $\varphi$ and the definition of the homogeneous dyadic blocks $\dot \Delta_k$ suggest that  the following \emph{Littlewood-Paley decomposition} holds true:
\begin{align}
\label{Eq: Representation of f}
 f = \sum_{k \in \IZ} \dot \Delta_k f.
\end{align}
However, for a general tempered distribution, this equality only holds \emph{up to a polynomial}. 
This is actually another motivation for considering the space  $\cS_h^{\prime} (\IR^n)$
since, as proved in~\cite[p.~61/62]{Bahouri_Chemin_Danchin}, we have 
\eqref{Eq: Representation of f}  for all $f \in \cS_h^{\prime} (\IR^n).$ 
\medbreak
In the same spirit, we have the  following lemma. 
\begin{lemma}
\label{Lem: Representation of f}
Let $m>0$ and  $\psi : \IR^n \setminus \{ 0 \}\to\IR$ be smooth and homogeneous of degree $m$.
Then, for all $f \in \cS_h^{\prime} (\IR^n),$ the series
\begin{align*}
 \sum_{k \in \IZ} \cF^{-1} \psi \cF \dot \Delta_k f
\end{align*}
converges in $\cS_h^{\prime} (\IR^n)$. 
\end{lemma}

\begin{proof}
 The convergence of
\begin{align*}
 \bigg(\sum_{k = 0}^M \cF^{-1} \psi \cF \dot \Delta_k f \bigg)_{M \in \IN} \quad \text{to} \quad \cF^{-1} \psi (1-\chi) \cF f \quad \text{in} \quad \cS^{\prime}(\IR^n)
\end{align*}
follows by a direct calculation. To show convergence for the low frequencies let $\eps > 0$ and let $M \geq L \geq L_0$ where $L_0 \in \IN$ is chosen such that for all $k \in \IZ$ that satisfy $k \leq - L_0$
\begin{align*}
   \| \Phi (2^{- k} D) f \|_{\LL_{\infty} (\IR^n)} < \eps, \quad \text{ \ where \ \ } \quad \Phi(\xi) := \varphi (\xi) \psi(\xi).
\end{align*}
The existence of $L_0$ follows from Definition~\ref{Def: Low frequency control}. Now, let $g \in \cS (\IR^n)$ and employ the homogeneity of $\psi$ as well as H\"older's inequality to get
 \begin{align*}
 \Big\lvert \Big\langle \sum_{k = - M}^0 \cF^{-1}\psi \cF \dot \Delta_k f - \sum_{k = - L + 1}^0 \cF^{-1} \psi \cF \dot \Delta_k f , g \Big\rangle \Big\rvert &\leq \sum_{k = - M}^{- L} 2^{k m} \| \Phi (2^{- k} D) f \|_{\LL_{\infty} (\IR^n)} \| g \|_{\LL_1 (\IR^n)} \\
 &\leq C\eps \| g \|_{\LL_1 (\IR^n)}. 
\end{align*}
Hence the low frequency part of the series converges in $\LL_\infty(\IR^n).$
A tiny modification of the above argument ensures that~\eqref{eq:BF} is satisfied by the series.
This completes the proof. 
\end{proof}

This directly leads to the definition of homogeneous Bessel potential spaces.

\begin{definition}
\label{Def: Homogeneous Bessel potential spaces}
For $m \in \IZ$ and $1 \leq p \leq \infty,$ we define   $\dot \H^m_p (\IR^n)$ to be the space of all tempered distributions $f \in \cS_h^{\prime} (\IR^n)$ such that
\begin{align*}
 \sum_{k \in \IZ} \cF^{-1} \lvert \xi \rvert^m \cF \dot \Delta_k f \in \LL_p (\IR^n)
\end{align*}
endowed with the norm
\begin{align*}
 \| f \|_{\dot \H^m_p (\IR^n)} := \Big\| \sum_{k \in \IZ} \cF^{-1} \lvert \xi \rvert^m \cF \dot \Delta_k f \Big\|_{\LL_p (\IR^n)}.
\end{align*}
\end{definition}

Along with the above definition, we introduce the homogeneous Sobolev spaces of non-negative integer order as follows. 

\begin{definition}
For $m \in \IN_0$ and $1 \leq p \leq \infty,$ we define $\dot \cH^m_p (\IR^n)$ to be  the space of all tempered distributions $f \in \cS_h^{\prime} (\IR^n)$ that satisfy
\begin{align*}
 \| f \|_{\dot \cH^m_p (\IR^n)} := \sum_{j = 1}^n \| \partial_j^m f \|_{\LL_p (\IR^n)} < \infty.
\end{align*}
\end{definition}

In order to show that for $m \in \IN_0$ the homogeneous Bessel potential and Sobolev spaces coincide, the following lemma is crucial.

\begin{lemma}
\label{Lem: Density of functions with no low frequencies}
For each $1 < p < \infty$, $1 \leq q < \infty$ and $s \in \IR$ the class of functions
\begin{align*}
 \cS_c (\IR^n) := \{ \psi \in \cS (\IR^n) : 0 \notin \supp(\cF^{-1} \psi) \}
\end{align*}
is dense in $\LL_p (\IR^n)$ and in $\dot \B^s_{p , q} (\IR^n)$.
\end{lemma}

\begin{proof} 
The case of the homogeneous Besov space is proven in~\cite[Prop.~2.27]{Bahouri_Chemin_Danchin}. Thus, let us concentrate on the $\LL_p$-space. It suffices to prove that the annihilator $\cS_c (\IR^n)_{\perp}$ with respect to the $\LL_p$-duality product contains only the function that is identically zero. \par
Let $f \in \cS_c (\IR^n)_{\perp} \subset \LL_{p^{\prime}} (\IR^n)$. Then, for all $\psi$ in $\cS_c(\IR^n),$ 
we have
\begin{align*}
 \langle \cF f , \cF^{-1} \psi \rangle = \langle f , \psi \rangle = 0.
\end{align*}
Since $\cF \C_c^{\infty} (\IR^n \setminus \{ 0 \}) \subset \cS_c (\IR^n)$, this implies that the support of $\cF f$ is contained in $\{ 0 \}$ so that $f$ is a polynomial. But the only polynomial that is contained in $\LL_{p^{\prime}} (\IR^n)$ is the zero polynomial, so that $f = 0$.
\end{proof}

Now, we prove that the homogeneous Bessel potential space as well as the homogeneous Sobolev space of order $m$ are equal. The proof substantially uses Lemma~\ref{Lem: Density of functions with no low frequencies} and the fact that each $f \in \cS_h^{\prime} (\IR^n)$ has the Littlewood-Paley decomposition~\eqref{Eq: Representation of f}.

\begin{proposition}
\label{Prop: Bessel equals Sobolev}
For all $m \in \IN_0$ and  $1 < p < \infty,$ the spaces $\dot \H^m_p (\IR^n)$ and $\dot \cH^m_p (\IR^n)$ coincide with equivalent norms.
\end{proposition}

\begin{proof}
In the case $m = 0$ there is nothing to do so let $m > 0$. For $1 \leq j \leq n$ and $f \in \cS_h^{\prime} (\IR^n)$ write
\begin{align*}
 \partial_j^m f = (2 \pi \ii)^m \cF^{-1} \xi_j^m \cF f.
\end{align*}
Then, using~\eqref{Eq: Representation of f} and Lemma~\ref{Lem: Representation of f} with  $\psi (\xi) := \xi_j^m$ together with the fact that the Fourier transform is an isomorphism on $\cS^{\prime}(\IR^n)$ imply
$$
 \partial_j^m f = (2 \pi \ii)^m \sum_{k \in \IZ} \cF^{-1} \xi_j^m \cF \dot \Delta_k f = (2 \pi \ii)^m \sum_{k \in \IZ} \cF^{-1} \xi_j^m \lvert \xi \rvert^{- m} \cF \cF^{-1} \lvert \xi \rvert^m \cF \dot \Delta_k f
 \quad\hbox{in }\ \cS^{\prime}(\IR^n).
$$
 Now,  the above identity  delivers for $g \in \cS_c (\IR^n)$: 
\begin{align*}
(2 \pi \ii)^{-m} \langle \partial_j^m f , g \rangle = \Big\langle \sum_{k \in \IZ} \cF^{-1} \lvert \xi \rvert^m \cF \dot \Delta_k f , \cF \xi_j^m \lvert \xi \rvert^{- m} \cF^{-1} g \Big\rangle\cdotp
\end{align*}
Since by Marcinkiewicz--Mikhlin's multiplier theorem,
\begin{align*}
 \cF \xi_j^m \lvert \xi \rvert^{- m} \cF^{-1} : \LL_{p^{\prime}} (\IR^n) \to \LL_{p^{\prime}} (\IR^n)
\end{align*}
is a bounded operator, it follows that the functional $\cS_c (\IR^n) \ni g \mapsto \langle \partial_j^m f , g \rangle$ satisfies the estimate
\begin{align*}
 \lvert \langle \partial_j^m f , g \rangle \rvert \leq (2 \pi)^m \| \cF^{-1} \xi_j^m \lvert \xi \rvert^{- m} \cF \|_{\Lop(\LL_{p^{\prime}} (\IR^n))} \| f \|_{\dot \H^m_p (\IR^n)} \| g \|_{\LL_{p^{\prime}} (\IR^n)}.
\end{align*}
Since $\cS_c (\IR^n)$ is dense in $\LL_{p^{\prime}} (\IR^n)$ by Lemma~\ref{Lem: Density of functions with no low frequencies}, this shows that $\dot \H^m_p (\IR^n) \subset \dot \cH^m_p (\IR^n)$. \par
For the other direction, assume that $f \in \dot \cH^m_p (\IR^n)$. Notice that the proof of Lemma~\ref{Lem: Representation of f} with $\psi (\xi) := \lvert \xi \rvert^m$ gives with $\vartheta:=1-\chi,$
\begin{align*}
 \Big\lvert \Big\langle \sum_{k \in \IZ} \cF^{-1} \lvert \xi \rvert^m \cF \dot \Delta_k f , g \Big\rangle \Big\rvert \leq \liminf_{M \to \infty} \| \cF^{-1} \vartheta(2^M \xi) \lvert \xi \rvert^m \cF f \|_{\LL_p (\IR^n)} \| g \|_{\LL_{p^{\prime}} (\IR^n)}.
\end{align*}
Let $\Upsilon : \IR \to \IR$  be a smooth function supported away from the origin
and with value $1$ on $|t|\geq 3 / (8 \sqrt{n}).$  Since
\begin{align*}
 \xi \mapsto \vartheta (\xi) \lvert \xi \rvert^m \Big( \sum_{j = 1}^n \Upsilon(\xi_j) \lvert \xi_j \rvert^m \Big)^{-1} \quad \text{and} \quad \xi \mapsto \Upsilon(\xi_j) \lvert \xi_j \rvert^m \xi_j^{- m}
\end{align*}
are $0$ order multipliers, Marcinkiewicz--Mikhlin's theorem  implies that for fixed $M \in \IN$
\begin{align*}
 \| \cF^{-1} \vartheta(2^M \xi) \lvert \xi \rvert^m \cF f \|_{\LL_p (\IR^n)} &\leq C 2^{- m M} \Big\| \cF^{-1} \sum_{j = 1}^n \Upsilon(2^M \xi_j) \lvert 2^M \xi_j \rvert^m \cF f \Big\|_{\LL_p (\IR^n)} \\
 &\leq C \sum_{j = 1}^n \| \cF^{-1} \xi_j^m \cF f \|_{\LL_p (\IR^n)} \\
 &\leq C \| f \|_{\dot \cH^m_p (\IR^n)}
\end{align*}
with $C > 0$ being independent of $M$. One may thus conclude that
$$
 \Big\lvert \Big\langle \sum_{k \in \IZ} \cF^{-1} \lvert \xi \rvert^m \cF \dot \Delta_k f , g \Big\rangle \Big\rvert \leq 
 C \| f \|_{\dot \cH^m_p (\IR^n)} \| g \|_{\LL_{p^{\prime}} (\IR^n)},
$$
whence the embedding $\dot \cH^m_p (\IR^n)\hookrightarrow \dot \H^m_p (\IR^n).$
\end{proof}

\begin{remark}
Proposition~\ref{Prop: Bessel equals Sobolev} is known from the book of Bergh and L\"ofstr\"om~\cite[Thm.~6.3.1]{Bergh_Lofstrom} but for elements $f \in \cS^{\prime} (\IR^n)$ that vanish in a neighborhood of the origin. The assumptions here are weakened in the sense that only the control on the low Fourier frequencies given in Definition~\ref{Def: Low frequency control} is required.
\end{remark}

\begin{convention}
Since Proposition~\ref{Prop: Bessel equals Sobolev} shows that the definitions of homogeneous Bessel potential spaces and homogeneous Sobolev spaces of non-negative order $m$ are equivalent, we will only use the notation $\dot \H^m_p (\IR^n)$ in the following. Furthermore, it is clear that $\LL_p (\IR^n) \cap \dot \cH^m_p (\IR^n)$ coincides with the usual Sobolev space of $m$ times weakly differentiable functions with derivatives up to order $m$ in $\LL_p (\IR^n)$. Thus, let us from now on write $\H^m_p (\IR^n)$ for this Sobolev space endowed with the usual norm.
\end{convention}

\begin{remark}
Let $s > 0$ and $p , q \in [1 , \infty]$. Then the equality $\dot \B^s_{p , q} (\IR^n) \cap \LL_p (\IR^n) = \B^s_{p , q} (\IR^n)$ is also valid.
\end{remark}

$$
$$

\begin{remark}
Let $0 < s < n / p$, $1 \leq p \leq \infty$ and $1 \leq q \leq \infty$ or let $s = n / p$ and $q = 1$. In this case, the homogeneous Besov space continuously embeds into $\LL_{p , \loc} (\IR^n)$. To see this, let $K \subset \IR^n$ be a compact set and let $f \in \dot \B^s_{p , q} (\IR^n)$. Then the Littlewood--Paley decomposition of $f$ and Bernstein's inequality yield
\begin{align*}
 \| f \|_{\LL_p (K)} &\leq \sum_{\ell \leq 0} \| \dot \Delta_{\ell} f \|_{\LL_p (K)} + \sum_{\ell \geq 1} \| \dot \Delta_{\ell} f \|_{\LL_p (K)} \\
 &\leq \lvert K \rvert^{\frac{1}{p}} \sum_{\ell \leq 0} \| \dot \Delta_{\ell} f \|_{\LL_{\infty} (K)} + C \| f \|_{\dot \B^s_{p , q} (\IR^n)} \\
 &\leq \lvert K \rvert^{\frac{1}{p}} \sum_{\ell \leq 0} 2^{(\frac{n}{p} - s) \ell} 2^{s \ell} \| \dot \Delta_{\ell} f \|_{\LL_p (\IR^n)} + C \| f \|_{\dot \B^s_{p , q} (\IR^n)} \\
 &\leq C (\lvert K \rvert^{\frac{1}{p}} + 1) \| f \|_{\dot \B^s_{p , q}(\IR^n)}.
\end{align*}
\end{remark}

\section{Solenoidal function spaces and extension operators}

To investigate the Stokes problem, we introduce the spaces of vector fields that are solenoidal.
 In these spaces, the boundary condition may have to be taken into account
even in some cases where the classical trace operator is not defined. For example, for Dirichlet boundary conditions, we have to consider  (with $\e_n$ denoting the $n$-th standard unit vector)
\begin{align*}
 \Lop_{p , \sigma} (\IR^n_+) := \{ f \in \LL_p (\IR^n_+ ; \IC^n) : \divergence f = 0 \text{ and } \e_n \cdot f = 0 \text{ on } \partial \IR^n_+\}
\end{align*}
while  Neumann boundary conditions correspond to the (larger) space
\begin{align}
\label{Eq: Lp-sigma}
 \LL_{p , \sigma} (\IR^n_+) := \{ f \in \LL_p (\IR^n_+ ; \IC^n) : \divergence f = 0\}.
\end{align}
More generally, for an open subset $\Omega \subset \IR^n$, with some abuse of language we  define
\begin{align*}
 \LL_{p , \sigma} (\Omega) := \{ f \in \LL_p (\Omega ; \IC^n) : \divergence f = 0\}.
\end{align*}
Since our aim here is  to investigate the Stokes problem with Neumann boundary conditions, 
we adopt the second definition from now on. 
\medbreak
Next, we define the solenoidal counterparts of the Bessel potential and Besov spaces on~$\IR^n$.
\begin{definition}
\label{Def: Solenoidal spaces}
Let $1 \leq p , q \leq \infty$, $s \in \IR$ and $m \in \IZ$. Define
\begin{align*}
 \dot \B^s_{p , q , \sigma} (\IR^n) := \{ f \in \dot \B^s_{p , q} (\IR^n ; \IC^n) : \divergence f = 0 \}
\end{align*}
and
\begin{align*}
 \dot \H^m_{p , \sigma} (\IR^n) := \{ f \in \dot \H^m_p (\IR^n ; \IC^n) : \divergence f = 0 \}.
\end{align*}
\end{definition}

To handle the half-space case, we proceed as usual by defining  spaces by restriction.

\begin{definition}
\label{Def: Spaces on the half-space}
Let $1 \leq p , q \leq \infty$, $s \in \IR$, $m \in \IZ$ and $\Omega \subset \IR^n$ be an open set. Let $X$ denote one of the symbols $\dot \B^s_{p , q}$, $\dot \B^s_{p , q , \sigma}$, $\dot \H^m_p$ and $\dot \H^m_{p , \sigma}$
(or their inhomogeneous counterpart), then, define $X(\Omega)$ by
\begin{align*}
 X (\Omega) := \{ f|_{\Omega} : f \in X (\IR^n) \}
\end{align*}
endowed with the quotient norm.
\end{definition}

To proceed recall the basic notions on real interpolation described at the beginning of Section~\ref{Sec: Homogeneous spaces and real interpolation}. For the non-solenoidal spaces on the whole space the following real interpolation result is classical\footnote{Adapting
it to our definition of homogeneous spaces is straightforward as, by construction, 
$\big( \dot \H^0_p (\IR^n) , \dot \H^m_p (\IR^n) \big)_{\theta , q}$ is included in $\cS^{\prime}_h (\IR^n).$} 
 cf.~\cite[Thm.~6.3.1]{Bergh_Lofstrom}.

\begin{proposition}
\label{Prop: Interpolation on whole space}
Let $\theta \in (0 , 1)$, $1 \leq p , q \leq \infty$ and $m \in \IN$. Then
\begin{align*}
  \big( \dot \H^0_p (\IR^n) , \dot \H^m_p (\IR^n) \big)_{\theta , q} = \dot \B^{m \theta}_{p , q} (\IR^n)
\end{align*}
with equivalent norms.
\end{proposition}

To prove this equality for the corresponding spaces on $\IR^n_+,$  a suitable extension operator is needed. 
In  the homogeneous setting, that operator is required to satisfy \emph{homogeneous} estimates. For later use, we also define an extension operator that respects the solenoidality of functions.

\begin{lemma}
\label{Lem: Extension operators}
Let $m \in \IN_0$. There exist extension operators $E$ and $E_{\sigma}$ (depending on $m$), i.e., $[E f]|_{\IR^n_+} = f$ and $[E_{\sigma} f]|_{\IR^n_+} = f$, that map measurable functions on $\IR^n_+$ to measurable functions on $\IR^n$, and such that for all $1 < p < \infty$, $0 \leq k \leq m + 1$ 
and $0 \leq \ell \leq m,$ it holds
\begin{align*}
 E \in \Lop(\H_p^k (\IR^n_+) , \H^k_p (\IR^n)) \quad \text{and} \quad E_{\sigma} \in \Lop(\H^\ell_p (\IR^n_+ ; \IC^n) , \H^\ell_p (\IR^n ; \IC^n)).
\end{align*}
The operator $E_{\sigma}$ satisfies $\divergence(E_{\sigma} f) = 0$ whenever $\divergence(f) = 0$ in $\IR^n_+$. Moreover, there exists a constant $C > 0$ such that
\begin{align*}
\begin{aligned}
 \| E f \|_{\dot \H^k_p (\IR^n)} &\leq C \| f \|_{\dot \H^k_p (\IR^n_+)} && (f \in \H^k_p (\IR^n_+)), \\
 \| E_{\sigma} f \|_{\dot \H^\ell_p (\IR^n ; \IC^n)} &\leq C \| f \|_{\dot \H^\ell_p (\IR^n_+ ; \IC^n)} && (f \in \H^\ell_p (\IR^n_+ ; \IC^n)).
\end{aligned}
\end{align*}
\end{lemma}

\begin{proof} Constructing   $E$ relies on  the higher-order reflection principle, as described in e.g.,~\cite[Thm.~7.58]{Renardy_Rogers}. Let us shortly recall how to proceed. Let $f : \IR^n_+ \to \IC$ be measurable. Then, we define $E f$ by $Ef := f$ on $\IR^n_+$ and  
\begin{align}
\label{Eq: Homogeneous extension}
 [E f] (x) := \sum_{j = 0}^m \alpha_j f \Big( x^{\prime} , - \frac{x_n}{j + 1} \Big) \qquad (x \in \IR^n \setminus \IR^n_+),
\end{align}
where $x^{\prime} := (x_1 , \cdots , x_{n - 1})$ and the real numbers $\alpha_j $ are chosen so that $E$ maps $\C^m$-functions on $\IR^n_+$ to $\C^m$-functions on $\IR^n$. This holds true if the numbers $\alpha_j$ satisfy for each $\iota = 0 , \dots , m$
\begin{align*}
 \sum_{j = 0}^m \Big( - \frac{1}{j + 1} \Big)^{\iota} \alpha_j = 1.
\end{align*}
This amounts to solving an $(m+1)\times(m+1)$ linear system associated to a 
 Vandermonde matrix, which is known to be invertible.
 
  Since $E$ maps $\C^m$-functions on $\IR^n_+$ to $\C^m$-functions on $\IR^n,$ it also maps $\H^{m + 1}_p$-functions on $\IR^n_+$ to $\H^{m + 1}_p$-functions on $\IR^n$. Moreover, for $0 \leq k \leq m + 1$, Proposition~\ref{Prop: Bessel equals Sobolev} and~\eqref{Eq: Homogeneous extension} deliver
\begin{align*}
 \| E f \|_{\dot \H^k_p (\IR^n)} \leq C \sum_{\gamma = 1}^n \| \partial_{\gamma}^k E f \|_{\LL_p (\IR^n)} \leq C \sum_{\gamma = 1}^n  \| \partial_{\gamma}^k f \|_{\LL_p (\IR^n_+)} \qquad (f \in \H^k_p (\IR^n_+)).
\end{align*}
Finally, if $F \in \dot \H^k_p (\IR^n)$ is any extension of $f$ to $\IR^n$, then
\begin{align*}
 \sum_{\gamma = 1}^n \| \partial_{\gamma}^k f \|_{\LL_p (\IR^n_+)} \leq \sum_{\gamma = 1}^n \| \partial_{\gamma}^k F \|_{\LL_p (\IR^n)} \leq C \| F \|_{\dot \H^k_p (\IR^n)}.
\end{align*}
Taking the infimum over all these extensions then delivers
\begin{align*}
 \| E f \|_{\dot \H^k_p (\IR^n)} \leq C \| f \|_{\dot \H^k_p (\IR^n_+)}.
\end{align*}
The construction of $E_{\sigma}$ is similar. For a measurable vector field $f: \IR^n_+ \to \IC^n,$ we adopt the notation $f = (f^{\prime} , f_n)$ with $f^{\prime} := (f_1 , \cdots , f_{n - 1}).$
 Define $E_{\sigma} f := f$ on $\IR^n_+$ and define $E_{\sigma} f$ on $\IR^n \setminus \IR^n_+$ by
\begin{align}
\label{Eq: Definition of solenoidal extension}
 [E_{\sigma} f] (x) := \bigg( \sum_{j = 0}^m \alpha^{\prime}_j f^{\prime} \Big( x^{\prime} , - \frac{x_n}{j + 1} \Big) , \sum_{j = 0}^m \beta^{\prime}_j f_n \Big(x^{\prime} , - \frac{x_n}{j + 1} \Big) \bigg) \qquad (x \in \IR^n \setminus \IR^n_+),
\end{align}
for appropriately chosen numbers $\alpha_j^{\prime} , \beta_j^{\prime} \in \IR$, $j = 0 , \dots , m$. We impose the conditions that $E_{\sigma} f$ shall map $\C^{m - 1}$-functions on $\IR^n_+$ to $\C^{m - 1}$-functions on $\IR^n$, what results, for  $\iota = 0 , \dots , m - 1,$ in 
\begin{align}
 \label{Eq: First regularity condition} \sum_{j = 0}^m \Big( - \frac{1}{j + 1} \Big)^{\iota} \alpha_j^{\prime} = 1, \\
 \label{Eq: Second regularity condition} \sum_{j = 0}^m \Big( - \frac{1}{j + 1} \Big)^{\iota} \beta_j^{\prime} = 1.
\end{align}
 Furthermore, if $f$ is locally integrable, then the distributional divergence of $E_{\sigma} f$ on $\IR^n \setminus \overline{\IR^n_+}$ is given for $\varphi \in \C_c^{\infty} (\IR^n \setminus \overline{\IR^n_+})$ by
\begin{align*}
 \langle \divergence(E_{\sigma} f) , \varphi \rangle = - \sum_{j = 0}^m \int_{\IR^n_-}& \Big( (j + 1) \alpha_j^{\prime} \nabla_{x^{\prime}} (\varphi(x^{\prime} , - (j + 1) x_n)) \cdot f^{\prime} (x) \\
 &\qquad - \beta_j^{\prime} \partial_{x_n} (\varphi (x^{\prime} , - (j + 1) x_n)) f_n (x) \Big) \; \d x.
\end{align*}
Thus, if $\divergence f = 0$ and if the condition
\begin{align}
\label{Eq: Divergence free condition}
 (j + 1) \alpha_j^{\prime} = - \beta_j^{\prime}
\end{align}
holds for all $j = 0 , \dots , m$, then $\divergence(E_{\sigma} f) = 0$ in $\IR^n \setminus \overline{\IR^n_+}$. Plugging~\eqref{Eq: Divergence free condition} into~\eqref{Eq: First regularity condition} yields
\begin{align*}
 \sum_{j = 0}^m \Big( - \frac{1}{j + 1} \Big)^{\iota} \beta_j^{\prime} = 1
 \quad\hbox{for }\  \iota = 1 , \dots , m. 
\end{align*}
Thus, in view of~\eqref{Eq: Second regularity condition} it follows that the coefficients with the desired properties exist if~\eqref{Eq: Second regularity condition} holds, but for $\iota = 0 , \dots , m$. This is exactly the condition from the first part of the proof including the Vandermonde matrix, so that the numbers $\beta_j^{\prime}$, $j = 0 , \dots , m$ indeed exist. \par
To show that $E_{\sigma} f$ maps locally integrable solenoidal vector fields into locally integrable solenoidal vector fields on $\IR^n$ (the solenoidality across the half-space boundary has to be verified) let $\varphi \in \C_c^{\infty} (\IR^n)$. The distributional divergence of $E_{\sigma} f$ is calculated as
\begin{align*}
 \langle \divergence(E_{\sigma} f) , \varphi \rangle = - \int_{\IR^n_+} \nabla \Big( \varphi(x) - \sum_{j = 0}^m \beta_j^{\prime} \varphi(x^{\prime} , - (j + 1) x_n) \Big) \cdot f (x) \; \d x.
\end{align*}
If $f$ if solenoidal and if $ \IR^n_+ \ni x \mapsto \varphi(x) - \sum_{j = 0}^m \beta_j^{\prime} \varphi(x^{\prime} , - (j + 1) x_n) =: \psi(x)$ has trace zero at $\{ x_n = 0 \}$ the right-hand side vanishes (this follows by approximating $\psi$ by $\C_c^{\infty} (\IR^n_+)$ functions). Notice that in view of~\eqref{Eq: Second regularity condition} with $\iota = 0,$ the trace of $\psi$  vanishes on the half-space boundary. It follows that the distributional divergence of $E_{\sigma} f$ is indeed zero. \par
From this point, proving  the boundedness estimates is similar as  for the operator $E$.
\end{proof}

\begin{remark}
\begin{enumerate}
\item If $m = 0$, then the operator $E$ constructed in Lemma~\ref{Lem: Extension operators} is the even reflection at the half-space boundary.
\item Extension operators for inhomogeneous solenoidal spaces on more general domains where constructed in~\cite{Kato_Mitrea_Ponce_Taylor}.
\end{enumerate}
\end{remark}

The following are direct corollaries of Lemma~\ref{Lem: Extension operators}.

\begin{corollary}
\label{Cor: Density}
Let $m \in \IN_0$ and $1 < p < \infty$. Then
\begin{align*}
 \C_{c , \sigma}^{\infty} (\overline{\IR^n_+}) := \{ \psi|_{\IR^n_+} : \psi \in \C^{\infty}_{c , \sigma} (\IR^n) \}   \mbox{ \ \ \ is dense in $\H^m_{p , \sigma} (\IR^n_+)$.}
\end{align*}
\end{corollary}

\begin{proof}
Approximate $E_{\sigma} f$ with functions in $\C^{\infty}_{c , \sigma} (\IR^n)$ for $f \in \H^m_{p , \sigma} (\IR^n_+)$.
\end{proof}

\begin{corollary}
\label{Cor: Lp-sigma in right scale}
For $1 < p < \infty$ the space $\dot \H^0_{p , \sigma} (\IR^n_+)$ coincides with $\LL_{p , \sigma} (\IR^n_+)$.
\end{corollary}

\begin{proof}
This follows from  $\dot \H^0_{p , \sigma} (\IR^n) = \LL_{p , \sigma} (\IR^n),$ 
since $\LL_{p , \sigma} (\IR^n_+)$-functions can be generalized to functions in $\LL_{p , \sigma} (\IR^n)$ by using the same extension operator $E_{\sigma}$ from Lemma~\ref{Lem: Extension operators}.
\end{proof}

The boundedness of the extension operators in Lemma~\ref{Lem: Extension operators} can be extended to homogeneous Besov spaces as follows.

\begin{proposition}
\label{Prop: Proper boundedness extension operators}
Let $m \in \IN_0$ and let $E$ and $E_{\sigma}$ denote the extension operators defined in~\eqref{Eq: Homogeneous extension} and~\eqref{Eq: Definition of solenoidal extension}, respectively. Furthermore, let $1 < p < \infty$, $1 \leq q \leq \infty$ and $-1 +1/p < s < m + 1 + 1 / p$ satisfy~\eqref{Eq: Banach space conditions}. Then, $E$ defines a bounded operator
\begin{align*}
 E : \dot \B^s_{p , q} (\IR^n_+) \to \dot \B^s_{p , q} (\IR^n).
\end{align*}
Moreover, if $1 < p < \infty$, $1 \leq q \leq \infty$ and $-1 +1/p < s < m + 1 / p$ satisfy~\eqref{Eq: Banach space conditions}, then $E_{\sigma}$ defines a bounded operator
\begin{align*}
 E_{\sigma} : \dot \B^s_{p , q} (\IR^n_+ ; \IC^n) \to \dot \B^s_{p , q} (\IR^n ; \IC^n).
\end{align*}
\end{proposition}

\begin{proof}
Fix $m \in \IN_0$ and let $E$ denote the corresponding extension operator from Lemma~\ref{Lem: Extension operators}. In the following, we concentrate on the operator $E$ as the proof for $E_{\sigma}$ is literally the same. Using the boundedness properties of $E$ stated in Lemma~\ref{Lem: Extension operators} together with real interpolation, yields that
\begin{align*}
 E : \B^s_{p , q} (\IR^n_+) \to \B^s_{p , q} (\IR^n) \quad \text{is bounded}
\end{align*}
for all $0 < s < m + 1$, all $1 < p < \infty$ and all $1 \leq q \leq \infty$. Next, notice that by definition of $E$ in~\eqref{Eq: Homogeneous extension}, $E$ commutes with the dilation operator $\delta_{\lambda} f := f(\lambda \cdot)$ (for $\lambda > 0$). Thus, for $f \in \B^s_{p , q} (\IR^n_+)$ and any $\lambda > 0$ one concludes by scaling properties of the homogeneous Besov space norm that
\begin{align*}
 \| E f \|_{\dot \B^s_{p , q} (\IR^n)} = \| \delta_{\lambda} E \delta_{\lambda^{-1}} f \|_{\dot \B^s_{p , q} (\IR^n)} &\simeq \lambda^{s - \frac{n}{p}} \| E \delta_{\lambda^{-1}} f \|_{\dot \B^s_{p , q} (\IR^n)} \\
 &\leq C \lambda^{s - \frac{n}{p}} \big( \| \delta_{\lambda^{-1}} f \|_{\LL_p (\IR^n_+)} + \| \delta_{\lambda^{-1}} f \|_{\dot \B^s_{p , q} (\IR^n)} \big) \\
 &\lesssim \lambda^s \| f \|_{\LL_p (\IR^n_+)} + \| f \|_{\dot \B^s_{p , q} (\IR^n_+)}.
\end{align*}
Letting $\lambda \to 0$ shows that
\begin{align}
\label{Eq: Extension operator homogeneous scaling estimate}
 \| E f \|_{\dot \B^s_{p , q} (\IR^n)} \leq C \| f \|_{\dot \B^s_{p , q} (\IR^n_+)} \qquad (f \in \B^s_{p , q} (\IR^n_+)).
\end{align}
If $s$, $p$ and $1 \leq q < \infty$ satisfy~\eqref{Eq: Banach space conditions} the completeness of $\dot \B^s_{p , q}$ allows us to conclude that~\eqref{Eq: Extension operator homogeneous scaling estimate} holds for all $f \in \dot \B^s_{p , q} (\IR^n_+)$ by density. \par
In the case $q = \infty$ this density argument does not work out anymore. Instead, we use a weak approximation scheme. Let $f \in \dot \B^s_{p , \infty} (\IR^n_+)$ and let $F \in \dot \B^s_{p , \infty} (\IR^n)$ with $F|_{\IR^n_+} = f$. Define
\begin{align*}
 F_k := \sum_{m = - k}^k \dot \Delta_m F \qquad (k \in \IN).
\end{align*}
Since the Littlewood--Paley decomposition holds true for $F$, it holds $F_k \to F$ in $\cS^{\prime} (\IR^n)$ as $k \to \infty$. Moreover, notice that for some absolute constant $C > 0$ one has that
\begin{align}
\label{Eq: Weak approximation scheme}
 \| F_k \|_{\dot \B^s_{p , \infty} (\IR^n)} \leq C \| F \|_{\dot \B^s_{p , \infty} (\IR^n)}.
\end{align}
Define $f_k := F_k|_{\IR^n_+}$. The estimates~\eqref{Eq: Extension operator homogeneous scaling estimate} and~\eqref{Eq: Weak approximation scheme} together with the fact that the norm on the half-space is given by the quotient norm then gives
\begin{align}
\label{Eq: Weak approximation extension}
 \| E f_k \|_{\dot \B^s_{p , \infty} (\IR^n)} \leq C \| F \|_{\dot \B^s_{p , \infty} (\IR^n)}.
\end{align}
By the Fatou property of homogeneous Besov spaces~\cite[Thm.~2.25]{Bahouri_Chemin_Danchin} there exists a subsequence $(f_{k_j})_{j \in \IN}$ and a function $h \in \dot \B^s_{p , \infty} (\IR^n)$ such that
\begin{align}
\label{Eq: Application Fatou}
 E f_{k_j} \to h \quad \text{in} \quad \cS^{\prime} (\IR^n) \qquad \text{and} \qquad \| h \|_{\dot \B^s_{p , \infty} (\IR^n)} \leq C \liminf_{j \to \infty} \| E f_{k_j} \|_{\dot \B^s_{p , \infty} (\IR^n)}.
\end{align}
If $\varphi \in \C_c^{\infty} (\IR^n_+)$, then the convergence in $\cS^{\prime} (\IR^n)$ yields
\begin{align*}
 \int_{\IR^n} h(x) \varphi(x) \; \d x = \lim_{j \to \infty} \int_{\IR^n} E f_{k_j} (x) \varphi(x) \; \d x = \int_{\IR^n} f(x) \varphi(x) \; \d x.
\end{align*}
This shows that $h$ coincides with $f$ in $\IR^n_+$. Similarly, one derives that $h = E f$ in $\IR^n_-$. Since $F$ was an arbitrary extension of $f$ to $\IR^n$ we thus derive by combining~\eqref{Eq: Weak approximation extension} and~\eqref{Eq: Application Fatou} that
\begin{align*}
 \| E f \|_{\dot \B^s_{p , \infty} (\IR^n)} \leq C \| f \|_{\dot \B^s_{p , \infty} (\IR^n_+)}.
\end{align*}
We continue now with the remaining cases. \par
If $-1 +1/p<s<1/p$ and $1 \leq q \leq \infty$, then the result stems from the explicit formula of $E$ given in~\eqref{Eq: Homogeneous extension}, 
and the fact that the extension by $0$ operator maps  $\dot \B^s_{p , q} (\IR^n_+) \to \dot \B^s_{p , q} (\IR^n)$
(see~\cite[Cor.~2.2.1]{Danchin_Mucha_Memoirs}). \par 
 If $m + 1 \leq s < m + 1 + 1 / p$ and $1 \leq q \leq \infty$ satisfy~\eqref{Eq: Banach space conditions}, let $f \in \B^s_{p , q} (\IR^n_+)$.
Observe that~\eqref{Eq: Homogeneous extension} ensures (with the notation of Lemma~\ref{Lem: Extension operators}) for $1 \leq \ell \leq n$ and $k := m + 1$
\begin{align*}
 \partial_\ell^k E f = \chi_{\IR^n_+} \partial_\ell^k f + (1 - \delta_{\ell , n}) \chi_{\IR^n \setminus \IR^n_+} E \partial_\ell^k f + \delta_{\ell , n} \sum_{j = 0}^{m} \alpha_j \Big( - \frac{1}{j + 1} \Big)^k \chi_{\IR^n \setminus \IR^n_+} (\partial_n^k f) \Big( x^{\prime} , - \frac{x_n}{j + 1} \Big),
\end{align*}
where $\delta_{\ell , n}$ denotes Kronecker's delta and $\chi_{V}$ the characteristic function of a subset $V \subset \IR^n$. Since $1 / p - 1 < s - k < 1 / p$, this together with~\cite[Cor.~2.2.1]{Danchin_Mucha_Memoirs} followed by boundedness properties of the gradient operator gives an estimate of the form
\begin{align}
\label{Eq: Boundedness of extension operators}
 \sum_{\ell = 1}^n \| \partial_\ell^k E f \|_{\dot \B^{s - k}_{p , q} (\IR^n)} \leq C \sum_{\ell = 1}^n \| \partial_\ell^k f \|_{\dot \B^{s - k}_{p , q} (\IR^n_+)} \leq C \| f \|_{\dot \B^{s}_{p , q} (\IR^n_+)}.
\end{align}
Additionally, on the whole space it is evident that
\begin{align*}
 \| E f \|_{\dot \B^{s}_{p , q} (\IR^n)} \leq C \sum_{j = 1}^n \| \partial_j^k E f \|_{\dot \B^{s - k}_{p , q} (\IR^n)},
\end{align*}
see~\cite[Cor.~2.32]{Bahouri_Chemin_Danchin}, which together with~\eqref{Eq: Boundedness of extension operators} implies the boundedness estimate
\begin{align*}
 \| E f \|_{\dot \B^{s}_{p , q} (\IR^n)} \leq C \| f \|_{\dot \B^{s}_{p , q} (\IR^n_+)} \qquad (f \in \dot \B^s_{p , q} (\IR^n_+)). &\qedhere
\end{align*}
\end{proof}

In the proof of Proposition~\ref{Prop: Proper boundedness extension operators}, we used that the operator
\begin{align*}
 \nabla : \dot \B^{s + 1}_{p , q} (\IR^n_+;\IC) \to \dot \B^s_{p , q} (\IR^n_+ ; \IC^n)\quad\hbox{is bounded.}
\end{align*}
This holds true for all $s \in \IR$ and all $1 \leq p , q \leq \infty$ and follows immediately by the definition of the space by restriction of whole space elements and by the corresponding boundedness property on the whole space. On $\IR^n$, however, there is also an estimate of the gradient from below, 
a consequence of the reverse Bernstein inequality. The following corollary gives the corresponding result on the half-space.

\begin{corollary}
\label{Cor: Comparison of gradients}
Let $1 < p < \infty$, $1 \leq q \leq \infty$ and $s>-1+1/p$ satisfy
\begin{align*}
 s < \frac{n}{p} - 1 \quad \text{or, \ \  if }\  q = 1, \ \  s \leq \frac{n}{p} - 1.
\end{align*}
Then there exists a constant $C > 0$ such that for all $f \in \dot \B^{s + 1}_{p , q} (\IR^n_+)$ it holds
\begin{align*}
 \| f \|_{\dot \B^{s + 1}_{p , q} (\IR^n_+)} \leq C \| \nabla f \|_{\dot \B^s_{p , q} (\IR^n_+ ; \IC^n)}.
\end{align*}
\end{corollary}

\begin{proof}
Let $E$ denote an extension operator from Proposition~\ref{Prop: Proper boundedness extension operators} that is bounded on the space $\dot \B^{s + 1}_{p , q} (\IR^n_+)$ (i.e., take $m \in \IN$ big enough therein).  
 Then, the corresponding whole space estimate,~\cite[Cor.~2.32]{Bahouri_Chemin_Danchin}, gives, if $s+1$ satisfies~\eqref{Eq: Banach space conditions},
\begin{align*}
 \| f \|_{\dot \B^{s + 1}_{p , q} (\IR^n_+)} \leq \| E f \|_{\dot \B^{s + 1}_{p , q} (\IR^n)} 
 \leq C \| \nabla E f \|_{\dot \B^s_{p , q} (\IR^n ; \IC^n)}.
\end{align*}
Next, proceed as in the first inequality of~\eqref{Eq: Boundedness of extension operators} with $k = 1$ to  get 
\begin{align*}
 \| \nabla E f \|_{\dot \B^s_{p , q} (\IR^n ; (\IC^n)^k)}\leq C
  \| \nabla f \|_{\dot \B^s_{p , q} (\IR^n_+ ; \IC^n)}. &\qedhere
\end{align*}
\end{proof}

To formulate another corollary in this section, we say that $v \in \dot \B^s_{p , q} (\IR^n_+ ; \IC^n)$ is curl free, if $\partial_j v_k = \partial_k v_j$ holds in $\IR^n_+$ in the sense of distributions.

\begin{corollary}
\label{Cor: Curl free vector fields}
Let $n \geq 2$, $1 < p < \infty$, $1 \leq q < \infty$ and $-1 +1/p < s < 1 / p$. Assume additionally that 
$s + 1$ fulfills~\eqref{Eq: Banach space conditions}. 
There exist  constants $C_1 , C_2 > 0$ such that for all curl free  $v \in \dot \B^s_{p , q} (\IR^n_+ ; \IC^n),$ 
one can find  $g \in \dot \B^{s + 1}_{p , q} (\IR^n_+)$ such that $v = \nabla g$ and 
\begin{align*}
 C_1 \| v \|_{\dot \B^s_{p , q} (\IR^n_+ ; \IC^n)} \leq \| g \|_{\dot \B^{s + 1}_{p , q} (\IR^n_+)} \leq C_2 \| v \|_{\dot \B^s_{p , q} (\IR^n_+ ; \IC^n)}.
\end{align*}
\end{corollary}

\begin{proof}
Let $E_{\sigma}$ denote the solenoidal extension operator defined in~\eqref{Eq: Definition of solenoidal extension} with $m = 0$. Notice that $E_{\sigma} : \dot \B^s_{p , q} (\IR^n_+ ; \IC^n) \to \dot \B^s_{p , q} (\IR^n ; \IC^n)$ is bounded by Proposition~\ref{Prop: Proper boundedness extension operators}. Next, notice that $E_{\sigma}$ maps curl free functions in $\IR^n_+$ to curl free functions in $\IR^n$. Now,~\cite[Cor.~2.32]{Bahouri_Chemin_Danchin} implies the existence of $a \in \dot \B^{s + 1}_{p , q} (\IR^n)$ with $E_{\sigma} v = \nabla a$ and that satisfies
\begin{align}
\label{Eq: Estimate curl free vector fields}
 C_1 \| E_{\sigma} v \|_{\dot \B^s_{p , q}(\IR^n ; \IC^n)} \leq \| a \|_{\dot \B^{s + 1}_{p , q} (\IR^n)} \leq C_2 \| E_{\sigma} v \|_{\dot \B^s_{p , q}(\IR^n ; \IC^n)}.
\end{align}
Define $g := a|_{\IR^n_+} \in \dot \B^{s + 1}_{p , q} (\IR^n_+)$ so that $v = \nabla g$. The estimate follows now from the boundedness of the gradient operator,~\eqref{Eq: Estimate curl free vector fields} and the boundedness of $E_\sigma$ through
\begin{align*}
\| v \|_{\dot \B^s_{p , q} (\IR^n_+ ; \IC^n)} \leq C \| g \|_{\dot \B^{s + 1}_{p , q} (\IR^n_+)} \leq C \| a\|_{\dot \B^{s + 1}_{p , q} (\IR^n)} 
 \leq C \| E_{\sigma} v \|_{\dot \B^{s}_{p , q} (\IR^n ; \IC^n)}\leq C \|v\|_{\dot \B^{s}_{p , q} (\IR^n_+ ; \IC^n)}. &\qedhere
\end{align*}
\end{proof}

Once the existence of an extension operator that acts boundedly on the two endpoint spaces of interpolation is guaranteed, the interpolation on domains follows usually from the retraction-coretraction principle, see, e.g.,~\cite[Sec.~1.2.4]{Triebel} or~\cite[Sec. I.2.3]{Ama95}. However, if the homogeneous Besov and Bessel potential spaces defined above are not complete, then  one is prohibited to perform a density argument in Lemma~\ref{Lem: Extension operators} to conclude that the extension operators act boundedly on homogeneous Bessel potential spaces, whence  the following  Condition~\eqref{Eq: Banach space condition for interpolation spaces} for the interpolation on half-spaces.

\begin{proposition}
\label{Prop: Interpolation on half-space}
Let $\theta \in (0 , 1)$, $1 \leq p , q \leq \infty$ and $m \in \IN$. Then
\begin{align*}
 \dot \B^{m \theta}_{p , q} (\IR^n_+) \hookrightarrow
  \big( \dot \H^0_p (\IR^n_+) , \dot \H^m_p (\IR^n_+) \big)_{\theta , q}.
\end{align*}
Besides, if
\begin{align}
\label{Eq: Banach space condition for interpolation spaces}
 m \theta < \frac{n}{p} \text{ \ and  \ } 1 < q < \infty \qquad \text{or} \qquad m \theta \leq \frac{n}{p} \text{  \ and  \ } q = 1,
\end{align}
then the equality
\begin{align*}
 \big( \dot \H^0_p (\IR^n_+) , \dot \H^m_p (\IR^n_+) \big)_{\theta , q} = \dot \B^{m \theta}_{p , q} (\IR^n_+)
\end{align*}
holds with equivalent norms.
\end{proposition}

\begin{proof}
Let $f \in \dot \B^{m \theta}_{p , q} (\IR^n_+)$. By virtue of Definition~\ref{Def: Spaces on the half-space} there exists $F \in \dot \B^{m \theta}_{p , q} (\IR^n)$ with
\begin{align*}
\mbox{ $F|_{\IR^n_+} = f$ \ \  and \ \ }
 \| F \|_{\dot \B^{m \theta}_{p , q} (\IR^n)} \leq 2 \| f \|_{\dot \B^{m \theta}_{p , q} (\IR^n_+)}.
\end{align*}
Since $\dot \B^{m \theta}_{p , q} (\IR^n) = ( \dot \H^0_p (\IR^n) , \dot \H^m_p (\IR^n))_{\theta , q}$ by Proposition~\ref{Prop: Interpolation on whole space}, there exist for each $t > 0$ functions $A_t \in \dot \H^0_p (\IR^n)$ and $B_t \in \dot \H^m_p (\IR^n)$ with $A_t + B_t = F$ and
\begin{align*}
 \| A_t \|_{\dot \H^0_p (\IR^n)} + t \| B_t \|_{\dot \H^m_p (\IR^n)} \leq 2 K(t , F ; \dot \H^0_p (\IR^n) , \dot \H^m_p (\IR^n)).
\end{align*}
Define $a_t := A_t|_{\IR^n_+}$ and $b_t := B_t|_{\IR^n_+}$. By all choices, it holds $a_t + b_t = f$ and
\begin{align*}
 K (t , f ; \dot \H^0_p (\IR^n_+) , \dot \H^m_p (\IR^n_+)) \leq \| a_t \|_{\dot \H^0_p (\IR^n_+)} + t \| b_t \|_{\dot \H^m_p (\IR^n_+)} \leq 2 K(t , F ; \dot \H^0_p (\IR^n) , \dot \H^m_p (\IR^n)).
\end{align*}
It follows that
\begin{align*}
 \| f \|_{( \dot \H^0_p (\IR^n_+) , \dot \H^m_p (\IR^n_+) )_{\theta , q}} \leq 2 \| F \|_{( \dot \H^0_p (\IR^n) , \dot \H^m_p (\IR^n) )_{\theta , q}} \leq C \| F \|_{\dot \B^{m \theta}_{p , q} (\IR^n)} \leq 2 C \| f \|_{\dot \B^{m \theta}_{p , q} (\IR^n_+)}.
\end{align*}

For the other inclusion notice that~\eqref{Eq: Banach space condition for interpolation spaces} implies that $p$ and $q$ are finite. Assume first that $f$ belongs to the \emph{inhomogeneous} Besov space $\B^{m \theta}_{p , q} (\IR^n_+)$. From the first part of the proof, we  know that $f \in (\dot \H^0_p (\IR^n_+) , \dot \H^m_p (\IR^n_+))_{\theta , q}$. Thus, let $a \in \dot \H^0_p (\IR^n_+)$ and $b \in \dot \H^m_p (\IR^n_+)$ be such that $f = a + b$. Since $b = f - a \in \dot \H^0_p (\IR^n_+)$ it particularly follows that $b \in \LL_p (\IR^n_+) \cap \dot \H^m_p (\IR^n_+) = \H^m_p (\IR^n_+)$. Let $E$ be an extension operator from Lemma~\ref{Lem: Extension operators} that satisfies boundedness estimates from $\dot \H^0_p (\IR^n_+)$ to $\dot \H^0_p (\IR^n)$ and from $\dot \H^m_p (\IR^n_+)$ to $\dot \H^m_p (\IR^n)$. Define $A := E a$, $B := E b$ and $F := E f$ and notice that $F = A + B$. Consequently,
\begin{align*}
 K (t , F ; \dot \H^0_p (\IR^n) , \dot \H^m_p (\IR^n)) \leq \| A \|_{\dot \H^0_p (\IR^n)} + t \| B \|_{\dot \H^m_p (\IR^n)} \leq C \Big(  \| a \|_{\dot \H^0_p (\IR^n_+)} + t \| b \|_{\dot \H^m_p (\IR^n_+)} \Big).
\end{align*}
Since $a$ and $b$ were arbitrary, it follows that
\begin{align*}
 K (t , F ; \dot \H^0_p (\IR^n) , \dot \H^m_p (\IR^n)) \leq C K (t , f ; \dot \H^0_p (\IR^n_+) , \dot \H^m_p (\IR^n_+))
\end{align*}
and thus, by the definition of the real interpolation space norm in~\eqref{Eq: Real interpolation space}, Proposition~\ref{Prop: Interpolation on whole space} and Definition~\ref{Def: Spaces on the half-space} that
\begin{align}
\label{Eq: Boundedness estimate for extension operator on Besov spaces}
 \| f \|_{\dot \B^{m \theta}_{p , q} (\IR^n_+)} \leq \| F \|_{\dot \B^{m \theta}_{p , q} (\IR^n)} \leq C \| f \|_{( \dot \H^0_p (\IR^n_+) , \dot \H^m_p (\IR^n_+) )_{\theta , q}}.
\end{align}
To obtain the embedding for all $f \in (\dot \H^0_p (\IR^n_+) , \dot \H^m_p (\IR^n_+))_{\theta , q}$ notice that, if $q < \infty$, the intersection space $\dot \H^0_p (\IR^n_+) \cap \dot \H^m_p (\IR^n_+) = \H^m_p (\IR^n_+)$, which is a subspace of $\B^{m \theta}_{p , q} (\IR^n_+)$, is always dense in the interpolation space $(\dot \H^0_p (\IR^n_+) , \dot \H^m_p (\IR^n_+))_{\theta , q}$ (this also holds if the interpolation couple does not consist of Banach spaces, see~\cite[Thm.~3.4.2]{Bergh_Lofstrom}). Thus, if the condition~\eqref{Eq: Banach space condition for interpolation spaces} is fulfilled, the Besov space $\dot \B^{m \theta}_{p , q} (\IR^n_+)$ is a Banach space and thus the embedding
\begin{align*}
 \big( \dot \H^0_p (\IR^n_+) , \dot \H^m_p (\IR^n_+) \big)_{\theta , q} \subset \dot \B^{m \theta}_{p , q} (\IR^n_+)
\end{align*}
holds by density.
\end{proof}


A key  point of the argument given in this article will be to understand real interpolation spaces between several kinds of solenoidal vector spaces. A helpful tool for their investigation is the Helmholtz projection $\IP$ whose mapping properties are stated below. Define
\begin{align*}
 \IP : \cS_c (\IR^n) \to \cS_c (\IR^n), \quad f \mapsto \cF^{-1} \Big( 1 - \frac{\xi \otimes \xi}{\lvert \xi \rvert^2} \Big) \cF f
\end{align*}
and notice that $\IP$ satisfies $\divergence(\IP f) = 0$.

\begin{proposition}
\label{Prop: Helmholtz projection}
Let $1 < p < \infty$, $1 \leq q < \infty$ and $s \in\IR$ satisfy~\eqref{Eq: Banach space conditions}. Then, the operator $\IP$ extends to a bounded operator on $\dot \B^s_{p , q} (\IR^n ; \IC^n)$ and to a bounded operator on $\LL_p (\IR^n ; \IC^n)$. Moreover, for each $m \in \IN_0,$ it maps 
\begin{align*}
 \Sigma := [\LL_p (\IR^n ; \IC^n) + \dot \B^s_{p , q} (\IR^n ; \IC^n)] \cap \dot \H^m_p (\IR^n ; \IC^n)
\end{align*}
into $\dot \H^m_{p,\sigma} (\IR^n ; \IC^n)$ and satisfies the estimate
\begin{align}
\label{Eq: Helmholtz on Bessel potential spaces}
 \| \IP f \|_{\dot \H^m_{p , \sigma} (\IR^n)} \leq C \| f \|_{\dot \H^m_p (\IR^n ; \IC^n)} \qquad (f \in \Sigma)
\end{align}
for some constant $C > 0$.
\end{proposition}

\begin{proof}
The statements concerning the boundedness on the Lebesgue and Besov spaces are well-known; the case of Lebesgue spaces directly follows from Marcinkiewicz--Mikhlin's theorem and the case of Besov spaces is proven in~\cite[Lem.~3.1.2]{Danchin_Mucha_Memoirs}. \par
To get~\eqref{Eq: Helmholtz on Bessel potential spaces}, it suffices to prove that for each $1 \leq j \leq n$, $m \in \IN_0$ and $f \in \Sigma,$ the operator $\IP$ satisfies
\begin{align*}
 \IP f \in \cS_h^{\prime} (\IR^n ; \IC^n) \quad \text{and} \quad \partial_j^m \IP f = \IP \partial_j^m f.
\end{align*}
The estimate~\eqref{Eq: Helmholtz on Bessel potential spaces} then follows from Proposition~\ref{Prop: Bessel equals Sobolev} and the boundedness estimate on $\LL_p$. \par
Since $\IP$ maps $\dot \B^s_{p , q} (\IR^n ; \IC^n)$ into itself, which is by definition a subspace of $\cS_h^{\prime} (\IR^n ; \IC^n)$ and since $\IP f$ lies in $\LL_p (\IR^n ; \IC^n)$ for $f \in \LL_p (\IR^n ; \IC^n)$ which is a subspace of $\cS_h^{\prime} (\IR^n ; \IC^n)$ by~\cite[p.~22]{Bahouri_Chemin_Danchin}, the property $\IP f \in \cS_h^{\prime} (\IR^n ; \IC^n)$ for $f \in \Sigma$ directly follows. \par
That the Helmholtz projection commutes with derivatives is well-known for $f \!\in\! \cS_c (\IR^n ; \IC^n)$ and may be proven for $f \in \Sigma$ by density of $\cS_c (\IR^n ; \IC^n)$ in $\Sigma$, see Lemma~\ref{Lem: Density of functions with no low frequencies}.
\end{proof}

We  are now in the  position to formulate our first result on real interpolation between solenoidal spaces. For the corresponding result in the inhomogeneous setting, we refer to~\cite[Thm. 3.4]{Ama00}.   

\begin{proposition}
Let $1 < p < \infty$, $1 \leq q < \infty$, $\theta \in (0 , 1)$ and $m \in \IN$. Then, we have
\begin{align*}
 \big( \LL_{p , \sigma} (\IR^n) , \dot \H^m_{p , \sigma} (\IR^n) \big)_{\theta , q} \hookrightarrow\dot \B^{m \theta}_{p , q , \sigma} (\IR^n).
\end{align*}
 Moreover, if
\begin{align*}
 0 < m \theta < \frac{n}{p} \ \text{ and }\ q > 1, \quad \text{or} \qquad 0 < m \theta \leq \frac{n}{p} \ \text{ and }\ q = 1,
\end{align*}
then
\begin{align*}
 \big( \LL_{p , \sigma} (\IR^n) , \dot \H^m_{p , \sigma} (\IR^n) \big)_{\theta , q} = \dot \B^{m \theta}_{p , q , \sigma} (\IR^n)
\end{align*}
with equivalent norms.
\end{proposition}

\begin{proof}
The inclusion
\begin{align}
\label{Eq: First interpolation inclusion without boundary conditions}
  \big( \LL_{p , \sigma} (\IR^n) , \dot \H^m_{p , \sigma} (\IR^n) \big)_{\theta , q} \hookrightarrow \dot \B^{m \theta}_{p , q , \sigma} (\IR^n)
\end{align}
follows  from the following argument: by Corollary~\ref{Cor: Lp-sigma in right scale} and Definition~\ref{Def: Solenoidal spaces}, we have
\begin{align*}
 \LL_{p , \sigma} (\IR^n) \hookrightarrow \dot \H^0_p (\IR^n ; \IC^n) \quad \text{and} \quad \dot \H^m_{p , \sigma} (\IR^n) \hookrightarrow\dot \H^m_p (\IR^n ; \IC^n). 
\end{align*}
Consequently, Proposition~\ref{Prop: Interpolation on whole space} implies
\begin{align*}
 \big( \LL_{p , \sigma} (\IR^n) , \dot \H^m_{p , \sigma} (\IR^n) \big)_{\theta , q} \hookrightarrow \big( \dot \H^0_p (\IR^n ; \IC^n) , \dot \H^m_p (\IR^n ; \IC^n) \big)_{\theta , q} = \dot \B^{m \theta}_{p , q} (\IR^n ; \IC^n).
\end{align*}
Finally, since the interpolation space is always contained in the sum-space, we have
\begin{align*}
 \big( \LL_{p , \sigma} (\IR^n) , \dot \H^m_{p , \sigma} (\IR^n) \big)_{\theta , q} \subset \LL_{p , \sigma} (\IR^n) + \dot \H^m_{p , \sigma} (\IR^n),
\end{align*}
what implies~\eqref{Eq: First interpolation inclusion without boundary conditions} since each $f \in \LL_{p , \sigma} (\IR^n) + \dot \H^m_{p , \sigma} (\IR^n)$ satisfies $\divergence f = 0$. \smallbreak
For proving the converse inclusion of~\eqref{Eq: First interpolation inclusion without boundary conditions}, 
consider $f \in \dot \B^{m \theta}_{p , q , \sigma} (\IR^n)$. Since $\dot \B^{m \theta}_{p , q , \sigma} (\IR^n) \subset \dot \B^{m\theta}_{p , q} (\IR^n ; \IC^n)$, Proposition~\ref{Prop: Interpolation on whole space} and the definition of the $K$-functional~\eqref{Eq: K-functional} imply that, for each $t > 0,$ there exist functions $A_t \in \dot \H^0_p (\IR^n ; \IC^n)$ and $B_t \in \dot \H^m_p (\IR^n ; \IC^n)$ with $A_t + B_t = f$ and
\begin{align}
\label{Eq: First embedding without boundary conditions}
 \| A_t \|_{\dot \H^0_p (\IR^n ; \IC^n)} + t \| B_t \|_{\dot \H^m_p (\IR^n ; \IC^n)} \leq 2 K (t , f ; \dot \H^0_p (\IR^n ; \IC^n) , \dot \H^m_p (\IR^n ; \IC^n)).
\end{align}
In the following, we are going to manipulate $A_t$ and $B_t$ to obtain functions $a_t \in \dot \H^0_{p , \sigma} (\IR^n)$ and $b_t \in \dot \H^m_{p , \sigma} (\IR^n)$ with $f = a_t + b_t$ which firstly will show that $f$ is contained in the sum space of $\dot \H^0_{p , \sigma} (\IR^n)$ and $\dot \H^m_{p , \sigma} (\IR^n)$ and, secondly,  will provide an estimate for the $K$-functional of $f$ with respect to the solenoidal spaces. \par
This is done by the following observation. Since $B_t = - A_t + f$, we have 
\begin{align*}
 B_t \in [\LL_p (\IR^n ; \IC^n) + \dot \B^{m \theta}_{p , q} (\IR^n ; \IC^n)] \cap \dot \H^m_p (\IR^n ; \IC^n). 
\end{align*}
Thus, Proposition~\ref{Prop: Helmholtz projection} yields for some constant $C > 0$ that
\begin{align}
\label{Eq: Second embedding without boundary conditions}
 \IP B_t \in \dot \H^m_{p , \sigma} (\IR^n) \quad \text{with} \quad \| \IP B_t \|_{\dot \H^m_{p , \sigma} (\IR^n)} \leq C \| B_t \|_{\dot \H^m_p (\IR^n ; \IC^n)}.
\end{align}
 For $A_t \in \dot \H^0_p (\IR^n ; \IC^n) = \LL_p (\IR^n ; \IC^n)$, Proposition~\ref{Prop: Helmholtz projection} yields that for some constant $C > 0,$
\begin{align}
\label{Eq: Third embedding without boundary conditions}
 \IP A_t \in \LL_{p , \sigma} (\IR^n) \quad \text{with} \quad \| \IP A_t \|_{\LL_{p , \sigma} (\IR^n)} \leq C \| A_t \|_{\dot \H^0_p (\IR^n ; \IC^n)}.
\end{align}
 Now, define
\begin{align*}
 a_t := \IP A_t \quad \text{and} \quad b_t := \IP B_t
\end{align*}
and notice that the divergence-free condition of $f$ implies $\IP f = f$ so that
\begin{align*}
 a_t + b_t = \IP A_t + \IP B_t = \IP f = f.
\end{align*}
This decomposition delivers the following estimate by combining~\eqref{Eq: Second embedding without boundary conditions},~\eqref{Eq: Third embedding without boundary conditions} and~\eqref{Eq: First embedding without boundary conditions}:
\begin{align*}
 K (t , f ; \LL_{p , \sigma} (\IR^n) , \dot \H^m_{p , \sigma} (\IR^n)) &\leq \| a_t \|_{\LL_{p , \sigma} (\IR^n)} + t \| b_t \|_{\dot \H^m_{p , \sigma} (\IR^n)} \\
 &\leq C \big\{ \| A_t \|_{\dot \H^0_p (\IR^n ; \IC^n)} + t \| B_t \|_{\dot \H^m_p (\IR^n ; \IC^n)} \big\} \\
 &\leq 2 C K (t , f ; \dot \H^0_p (\IR^n ; \IC^n) , \dot \H^m_p (\IR^n_+ ; \IC^n)).
\end{align*}
Now, Proposition~\ref{Prop: Interpolation on whole space} together with the definition of real interpolation spaces in~\eqref{Eq: Real interpolation space} and the bound on the $K$-functional implies that $f$ is contained in $( \LL_{p , \sigma} (\IR^n) , \dot \H^m_{p , \sigma} (\IR^n) )_{\theta , q}$ and
\begin{align*}
 \| f \|_{( \LL_{p , \sigma} (\IR^n) , \dot \H^m_{p , \sigma} (\IR^n) )_{\theta , q}} \leq C \| f \|_{\dot \B^{m \theta}_{p , q , \sigma} (\IR^n)},
\end{align*}
which completes the proof. 
\end{proof}

The existence of the extension operator $E_{\sigma}$ that respects the solenoidality and homogeneity, see Lemma~\ref{Lem: Extension operators}, leads to the following proposition.

\begin{proposition}
\label{Prop: Interpolation of solenoidal spaces without boundary conditions}
Let $1 < p < \infty$, $1 \leq q < \infty$, $\theta \in (0 , 1)$ and $m \in \IN$ satisfy
\begin{align*}
 0 < m \theta < \frac{n}{p} \ \text{ and }\ q > 1, \quad \text{or} \quad 0 < m \theta \leq \frac{n}{p} \ \text{ and }\ q = 1.
\end{align*}
Then
\begin{align*}
 \big( \LL_{p , \sigma} (\IR^n_+) , \dot \H^m_{p , \sigma} (\IR^n_+) \big)_{\theta , q} = \dot \B^{m \theta}_{p , q , \sigma} (\IR^n_+)
\end{align*}
with equivalent norms.
\end{proposition}

\begin{proof}
The proof of this result is literally the same as that of Proposition~\ref{Prop: Interpolation on half-space} but uses the extension operator $E_{\sigma}$ instead of $E.$
\end{proof}

\begin{corollary}
\label{Cor: Density of inhomogeneous spaces in homogeneous spaces}
Let $1 < p < \infty$, $1 \leq q < \infty$ and $s>0$ satisfying~\eqref{Eq: Banach space conditions}, then
$$
 \mbox{$\C_{c , \sigma}^{\infty} (\IR^n)$ is dense in $\dot \B^s_{p , q , \sigma} (\IR^n)$ and $\C_{c , \sigma}^{\infty} (\overline{\IR^n_+})$ is dense in $\dot \B^s_{p , q , \sigma} (\IR^n_+)$.}
$$
\end{corollary}

\begin{proof}
First of all, $\C_{c , \sigma}^{\infty} (\IR^n)$ is dense in $\H^m_{p , \sigma} (\IR^n)$ and $\C_{c , \sigma}^{\infty} (\overline{\IR^n_+})$ is dense in $\H^m_{p , \sigma} (\IR^n_+)$, for the latter statement, see Corollary~\ref{Cor: Density}. Moreover, the spaces $\H^m_{p , \sigma} (\IR^n)$ and $\H^m_{p , \sigma} (\IR^n_+)$ are dense in $\dot \B^s_{p , q , \sigma} (\IR^n)$ and $\dot \B^s_{p , q , \sigma} (\IR^n_+)$, respectively, since the intersection space is always dense in the real interpolation space if $q < \infty$, see~\cite[Thm.~3.4.2]{Bergh_Lofstrom}.
\end{proof}


\chapter{The Stokes operator with a Neumann-type boundary condition}
\label{Sec: The Stokes operator with Neumann boundary conditions}

This section is devoted to the study of the Stokes operator supplemented with  a homogeneous Neumann-type boundary condition in the half-space $\IR^n_+$ with $n\geq2.$ First of all, we recall some useful facts and  deal with mapping properties of the Stokes resolvent on $\H^1_p$-spaces. We focus on optimal estimates of the resolvent with respect to homogeneous norms, so as to be able to define a homogeneous version of this operator  on a subspace of $\dot \H^2_{p , \sigma} (\IR^n_+)$.
 This will enable us  to prove interpolation identities between $\LL^p_{\sigma} (\IR^n_+)$ and the domain of the homogeneous Stokes operator subject to this Neumann-type boundary condition. \par

\section{Generation properties of the Stokes operator}
\label{Subsec: Generation properties of the Stokes operator}

Recall that the stress tensor $\IT$ of a Newtonian fluid is given by
\begin{align*}
 \IT (u , P) = D(u) - P \Id \with D(u) = \nabla u + [\nabla u]^{\top}. 
\end{align*}
Then, the domain of the Stokes operator on $\LL_{p , \sigma} (\IR^n_+)$ is given, for $1<p<\infty$,  by
\begin{align}\label{dom:Ap}
 \dom(A_p) := \bigl\{ u \in \H^2_p  (\IR^n_+ ; \IC^n) :& \divergence u = 0 ,\ \exists P \in \widehat\W^1_p (\IR^n_+) \text{ with } \divergence \IT(u , P) \in \LL_{p , \sigma} (\IR^n_+) \\
& \qquad\qquad\text{ and } \IT(u , P) \e_n = 0 \text{ on } \partial \IR^n_+ \bigr\}\cdotp \nonumber
\end{align}
Here $\widehat\W^1_p (\IR^n_+)$ designates those $P \in \LL_{p , \loc} (\overline{\IR^n_+})$ such that $\nabla P \in \LL_p (\IR^n_+ ; \IC^n)$.
Note that the definition  implies $P$ to be harmonic
since  $\divergence \IT(u , P)$ is divergence free by 
assumption. With $u \in \dom(A_p)$ and $P$ being the associated pressure,  define the Stokes operator on $\LL_{p , \sigma} (\Omega)$ by
\begin{align*}
 A_p u := - \divergence \IT (u , P).
\end{align*}
In order to formulate results concerning the resolvent problem associated  to $A_p$, denote as above for $\vartheta \in (0 , \pi)$ the sector
\begin{align*}
 \Sec_{\vartheta} = \{ z \in \IC \setminus \{ 0 \} : \lvert \arg(z) \rvert < \vartheta \}.
\end{align*}
With these definitions, Shibata and Shimizu proved in~\cite[Thm.~1.1]{Shibata_Shimizu} the following theorem.

\begin{theorem}
\label{Thm: Shibata - Shimizu}
Let $1 < p < \infty$. Then, for all $\vartheta \in (0 , \pi),$ the sector $\Sec_{\vartheta}$ is contained in the resolvent set $\rho(- A_p)$ of $- A_p$. Moreover, there exists a constant $C > 0$ such that for all $\lambda \in \Sec_{\vartheta}$ and all $f \in \LL_{p , \sigma} (\IR^n_+)$ the function $u := (\lambda + A_p)^{-1} f$ and the associated pressure $P \in \widehat \W^1_p (\IR^n_+)$ satisfy the estimate
\begin{align}
\label{Eq: Resolvent estimate}
\begin{aligned}
 \lvert \lambda \rvert \| u \|_{\LL_{p , \sigma} (\IR^n_+)} + \lvert \lambda \rvert^{1 / 2} \| \nabla u \|_{\LL_p (\IR^n_+ ; \IC^{n^2})}  + \| \nabla^2 u \|_{\LL_p (\IR^n_+ ; \IC^{n^3})} + \| \nabla P &\|_{\LL_p (\IR^n_+ ; \IC^n)} \\
 &\leq C \| f \|_{\LL_{p , \sigma} (\IR^n_+)}.
\end{aligned}
\end{align}
In particular, $- A_p$ generates a bounded analytic semigroup on $\LL_{p , \sigma} (\IR^n_+)$.
\end{theorem} 

To calculate the interpolation spaces that appear in Chapter~\ref{Sec: Da Prato--Grisvard theorem} we will need a ``higher regularity version'' of the preceding theorem. Namely, we consider the 
Stokes resolvent problem with data in the space $\H^1_{p , \sigma} (\IR^n_+)$ and establish the validity of homogeneous solution estimates similar to the resolvent estimate in~\eqref{Eq: Resolvent estimate}.
Note that this higher regularity version does not follow directly from Amann's  interpolation scale method (see~\cite[Sec. V.2.2]{Ama95}), since $0 \in \sigma(A_p)$.   

\begin{proposition}
\label{Prop: Higher regularity resolvent estimate}
Let $1 < p < \infty$ and $\vartheta \in (0 , \pi)$. Then, there exists a constant $C > 0$ such that for all $\lambda \in \Sec_{\vartheta}$ and all $f \in \H^1_{p , \sigma} (\IR^n_+)$ the function $u = (\lambda + A_p)^{-1} f$ with associated pressure  $P \in \widehat{\W}^1_p (\IR^n)$ satisfies $u \in \H^3_{p , \sigma} (\IR^n_+)$, $\nabla P \in {\H}^1_p (\IR^n ; \IC^n)$ and
\begin{multline}
\label{Eq: First-order resolvent estimate}
 \lvert \lambda \rvert \| \nabla u \|_{\LL_p (\IR^n_+ ; \IC^{n^2})} + \lvert \lambda \rvert^{1 / 2} \| \nabla^2 u \|_{\LL_p (\IR^n_+ ; \IC^{n^3})} + \| \nabla^3 u \|_{\LL_p (\IR^n_+ ; \IC^{n^4})} + \| \nabla^2 P \|_{\LL_p (\IR^n_+ ; \IC^{n^2})} \\\leq C \| \nabla f \|_{\LL_p (\IR^n_+ ; \IC^{n^2})}.
\end{multline}
\end{proposition}

\begin{proof}
Let $1 \leq j \leq n - 1$ and $g$ be a function defined on $\IR^n_+$. Denote by 
$$\delta_j^h g (x) := h^{-1} (g(x + h \e_j) - g(x)),\quad x \in \IR^n_+,\quad h>0,$$ its difference quotient in the coordinate direction $x_j$.
\smallbreak
 Let $\lambda \in \Sec_{\vartheta}$, $f \in \H^1_{p , \sigma} (\IR^n_+)$ and $u \in \dom(A_p)$ be the unique solution to
\begin{align*}
 \lambda u + A_p u = f
\end{align*}
provided by Theorem~\ref{Thm: Shibata - Shimizu}. Since $\delta_j^h u \in \dom(A_p)$ and since $\delta_j^h A_p u = A_p \delta_j^h u$ we find
\begin{align*}
 \lambda \delta_j^h u + A_p \delta_j^h u = \delta_j^h f.
\end{align*}
Moreover, if $P \in \widehat \W^1_p (\IR^n_+)$ is the pressure  associated to $u$, then $\delta_j^h P$ is the pressure associated  to $\delta_j^h u$. Hence,~\eqref{Eq: Resolvent estimate} followed by a standard estimate for difference quotients implies that
$$\displaylines{\quad
 \lvert \lambda \rvert \| \delta_j^h u \|_{\LL_{p , \sigma} (\IR^n_+)} + \lvert \lambda \rvert^{1 / 2} \| \delta_j^h \nabla u \|_{\LL_p (\IR^n_+ ; \IC^{n^2})} + \| \delta_j^h \nabla^2 u \|_{\LL_p (\IR^n_+ ; \IC^{n^3})} + \| \delta_j^h \nabla P \|_{\LL_p (\IR^n_+ ; \IC^n)} \hfill\cr\hfill
 \leq C \| \partial_j f \|_{\LL_{p , \sigma} (\IR^n_+)}.\quad}$$
Because $h > 0$ is arbitrary, we deduce by having $h$ tend to $0,$
\begin{multline}\label{eq:findmeaname}
 \lvert \lambda \rvert \| \partial_j u \|_{\LL_{p , \sigma} (\IR^n_+)} + \lvert \lambda \rvert^{1 / 2} \| \partial_j \nabla u \|_{\LL_p (\IR^n_+ ; \IC^{n^2})} + \| \partial_j \nabla^2 u \|_{\LL_p (\IR^n_+ ; \IC^{n^3})} + \| \partial_j \nabla P \|_{\LL_p (\IR^n_+ ; \IC^n)} 
 \\\leq C \| \partial_j f \|_{\LL_{p , \sigma} (\IR^n_+)}.
\end{multline}
The previous inequality shows that it is possible to control all second derivatives of $P$ but  $\partial_n \partial_n P$. Notice that $P$ is harmonic since $u$ and $f$ are divergence free, so that, thanks to~\eqref{eq:findmeaname}, 
\begin{align*}
 \| \partial_n \partial_n P \|_{\LL_p (\IR^n_+)} \leq C \sum_{j = 1}^{n - 1} \| \partial_j f \|_{\LL_{p , \sigma} (\IR^n_+)}.
\end{align*}
It remains to estimate $\lvert \lambda \rvert \partial_n u$, $\lvert \lambda \rvert^{1 / 2} \partial_n \partial_n u$ and $\partial_n \partial_n \partial_n u$. Estimating  $\lvert \lambda \rvert \partial_n u$ readily follows from  the estimates above and the divergence free condition since 
\begin{align*}
 \partial_n u_n = - \sum_{j = 1}^{n - 1} \partial_j u_j.
\end{align*}
For the other two terms, consider the boundary condition
\begin{align*}
  \IT (u , P) \e_n = 0 \quad \Leftrightarrow \quad \partial_n u + \nabla u_n - P \e_n = 0\quad\hbox{on }
  \partial\IR^n_+.
\end{align*}
For $1 \leq j \leq n - 1$, 
using the information already established for the pressure, this shows that $u_j$ solves the following resolvent problem for the Laplacian with inhomogeneous Neumann data:
\begin{align*}
\left\{ \begin{aligned}
 \lambda u_j - \Delta u_j &= f_j - \partial_j P \in \H^1_p (\IR^n_+) \\
 \partial_n u_j &= - \partial_j u_n  \quad \text{on } \partial \IR^n_+ \quad \text{with} \quad \partial_j u_n \in \H^2_p (\IR^n_+).
\end{aligned}
\right.
\end{align*}
The corresponding higher regularity estimate for the Neumann Laplacian, see Proposition~\ref{Prop: Laplacian on half-space} in the appendix, delivers
\begin{align*}
 \lvert \lambda \rvert \| \nabla u_j \|_{\LL_p (\IR^n_+ ; \IC^n)} + \lvert \lambda \rvert^{1 / 2} &\| \nabla^2 u_j \|_{\LL_p (\IR^n_+ ; \IC^{n^2})} + \| \nabla^3 u_j \|_{\LL_p (\IR^n_+ ; \IC^{n^3})} \\
 &\leq C \Big( \| \nabla f_j \|_{\LL_p (\IR^n_+ ; \IC^n)} + \| \nabla \partial_j P \|_{\LL_p (\IR^n_+ ; \IC^n)} + \lvert \lambda \rvert \| \partial_j u_n \|_{\LL_p (\IR^n_+)} \\
 &\qquad\qquad+ \lvert \lambda \rvert^{1 / 2} \| \nabla \partial_j u_n \|_{\LL_p (\IR^n_+ ; \IC^n)} + \| \nabla^2 \partial_j u_n \|_{\LL_p (\IR^n ; \IC^{n^2})} \Big) \\
 &\leq C \| \nabla f \|_{\LL_p (\IR^n_+ ; \IC^{n^2})}.
\end{align*}
The divergence free condition implies again
\begin{align*}
 \partial_n \partial_n u_n = - \sum_{j = 1}^{n - 1} \partial_n \partial_j u_j \quad \text{and} \quad \partial_n \partial_n \partial_n u_n = - \sum_{j = 1}^{n - 1} \partial_n \partial_n \partial_j u_j,
\end{align*}
which shows that all estimates for $u$ hold.
\end{proof}

The estimate
\begin{align*}
 \lvert \lambda \rvert^{1 / 2} \| \nabla^2 u \|_{\LL_p (\IR^n_+ ; \IC^{n^3})} \leq C \| \nabla f \|_{\LL_p (\IR^n_+ ; \IC^{n^2})}
\end{align*}
in Proposition~\ref{Prop: Higher regularity resolvent estimate} is translated by Cauchy's integral formula into the following semigroup estimate (see, e.g.,~\cite[Prop.~3.7]{Tolksdorf}).
\begin{corollary}
\label{Corollary: Higher regularity semigroup estimate}
Let  $1 < p < \infty$. Then there exists a constant $C > 0$ such that for all $t > 0$ and all $f \in \H^1_{p , \sigma} (\IR^n_+),$ we have
\begin{align*}
 t^{1 / 2} \| \nabla^2 \e^{- t A_p} f \|_{\LL_p (\IR^n_+ ; \IC^{n^3})} \leq C\| \nabla f \|_{\LL_p (\IR^n_+ ; \IC^{n^2})}.
\end{align*}
\end{corollary}

Let us now  introduce the \textit{part} of the Stokes operator on $\H^1_{p , \sigma} (\IR^n_+)$. The corresponding domain is given by
\begin{align}\label{dom:AAp}
 \dom(\bA_p) := \{ u \in \dom(A_p) : A_p u \in \H^1_{p , \sigma} (\IR^n_+) \}, \quad \bA_p u := A_p u \quad \text{for } u \in \dom(\bA_p).
\end{align}
In the operator theoretical language, Proposition~\ref{Prop: Higher regularity resolvent estimate} states that for each $\vartheta \in (0 , \pi)$ the set $\Sec_{\vartheta}$ is contained in the resolvent set $\rho(- \bA_p)$. A combination of the estimates~\eqref{Eq: Resolvent estimate} and~\eqref{Eq: First-order resolvent estimate} shows the validity of the estimate
\begin{multline}
\label{Eq: Inhomogeneous first-order resolvent estimate}
 \| \lambda (\lambda + \bA_p)^{-1} f \|_{\H^1_{p , \sigma} (\IR^n_+)} + \lvert \lambda \rvert^{1 / 2} \| \nabla (\lambda + \bA_p)^{-1} f \|_{\H^1_p (\IR^n_+ ; \IC^{n^2})} \\+ \| \nabla^2 (\lambda + \bA_p)^{-1} f \|_{\H^1_p (\IR^n_+ ; \IC^{n^3})} \leq C \| f \|_{\H^1_{p , \sigma} (\IR^n_+)} 
\end{multline}
for $\lambda \in \Sec_{\vartheta} , f \in \H^1_{p , \sigma} (\IR^n_+)$ and some constant $C > 0$ independent of $\lambda$ and $f$. 
\medbreak
As a consequence, $- \bA_p$ generates a bounded analytic semigroup on $\H^1_{p , \sigma} (\IR^n_+)$.       

This semigroup is strongly continuous if and only if $\dom(\bA_p)$ is dense in $\H^1_{p , \sigma} (\IR^n_+)$, which is established in the following lemma.

\begin{proposition}
Let  $1 < p < \infty$. Then $\dom(A_p)$ is dense in $\LL_{p , \sigma} (\IR^n_+)$ and $\dom(\bA_p)$ is dense in $\H^1_{p , \sigma} (\IR^n_+)$.
\end{proposition}

\begin{proof}
Let $f \in \LL_{p^{\prime} , \sigma} (\IR^n_+)$ be such that
\begin{align*}
 \int_{\IR^n_+} f \cdot \overline{u} \; \d x = 0 \quad \hbox{for all }\ u \in \dom(A_p).
\end{align*}
Then
\begin{align*}
 0 =  \int_{\IR^n_+} f \cdot \overline{u} \; \d x = \int_{\IR^n_+} (\Id + A_p)^{-1} f \cdot \overline{(\Id + A_p)u} \; \d x.
\end{align*}
Since $-1 \in \rho(A_p),$ the range of the operator $\Id + A_p$ equals $\LL_{p , \sigma} (\IR^n_+)$ so that $(\Id + A_p)^{-1} f$ must be zero. Consequently, $f$ has to be zero. \par
Similarly, if $f \in (\H^1_{p , \sigma} (\IR^n_+))^*$ is a functional that vanishes on $\dom(\bA_p)$, then
\begin{align*}
 0 = \langle f , u \rangle_{(\H^1_{p , \sigma})^* , \H^1_{p , \sigma}} = \langle ((\Id + \bA_p)^{-1})^{\prime} f , (\Id + \bA_p) u \rangle_{(\H^1_{p , \sigma})^* , \H^1_{p , \sigma}} \ \hbox{ for all }\ u \in \dom(\bA_p).
\end{align*}
Since the image of $\Id + \bA_p$ under the set $\dom(\bA_p)$ is all of $\H^1_{p , \sigma} (\IR^n_+)$ by Proposition~\ref{Prop: Higher regularity resolvent estimate}, it follows that $((\Id + \bA_p)^{-1})^{\prime} f=0$ and hence $f$ must be zero.\end{proof}

\begin{corollary}
\label{Cor: Strong continuity}
For all $1 < p < \infty$ the (negative) Stokes operator $- \bA_p$ on $\H^1_{p , \sigma} (\IR^n_+)$ generates a bounded analytic semigroup on $\H^1_{p , \sigma} (\IR^n_+)$ which is strongly continuous. Moreover, the bounded analytic semigroup that is generated by $- A_p$ on $\LL_{p , \sigma} (\IR^n_+)$ is strongly continuous.
\end{corollary}

Let us emphasize at this point that the semigroup generated by $- \bA_p$ is a {\em bounded analytic} one.
The analyticity of $\e^{-t\bA_p}$ follows by combining the  arguments given in~\cite[Sec.V.2]{Ama95} with those in~\cite{NS03}. 

\section{Other homogeneous estimates of the Stokes operator}

The results from Section~\ref{Subsec: Generation properties of the Stokes operator} help to prove operator theoretic statements of the Stokes operator as for example its injectivity on $\LL_{p , \sigma} (\IR^n_+),$ which  is a necessary condition for having homogeneous estimates for the Stokes operator.

\begin{proposition}
\label{Prop: Injectivity of Stokes}
Let $1 < p < \infty$. Then the Stokes operator $A_p$ on $\LL_{p , \sigma} (\IR^n_+)$ is injective.
\end{proposition}

\begin{proof}
Consider the case $p = 2$ first. Since the boundary condition on $\partial\IR^n_+$ translates into
\begin{align*}
 \partial_n u + \nabla u_n - P \e_n = 0,
\end{align*}
one finds by the standard trace theorem that  $P|_{\partial \IR^n_+} \in \B^{1 / 2}_{2 , 2} (\partial \IR^n_+)$
as  $u \in \H^2_2 (\IR^n_+ ; \IC^n)$ (even though $P$ only lies in $\widehat \W^1_2 (\IR^n_+)$). Next, by Corollary~\ref{Cor: Density}, there exists $(\varphi_\ell)_{\ell \in \IN} \subset \C_{c , \sigma}^{\infty} (\IR^n)$ with $\varphi_\ell |_{\IR^n_+} \to u$ in $\H^1_{2 , \sigma} (\IR^n_+)$. Thus,  integration by parts and the boundary condition yield
$$\begin{aligned}
 \int_{\IR^n_+} A_2 u \cdot \overline{u} \; \d x &= - \lim_{\ell \to \infty} \int_{\IR^n_+} \divergence(\IT (u , P)) \cdot \overline{\varphi_\ell} \; \d x\\ &=  \lim_{\ell \to \infty}\sum_{j , k = 1}^n \int_{\IR^n_+} (\partial_j u^k \overline{\partial_j \varphi^k_\ell} + \partial_k u^j \overline{\partial_j \varphi^k_\ell}) \; \d x\\
 &= \sum_{j , k = 1}^n \int_{\IR^n_+} (\partial_j u^k \overline{\partial_j u^k} + \partial_k u^j \overline{\partial_j u^k}) \; \d x.
\end{aligned}$$
Moreover, a simple calculation ensures
\begin{align*}
 2 \sum_{j , k = 1}^n \int_{\IR^n_+} ( \partial_j u_k \overline{\partial_j u_k} + \partial_k u_j \overline{\partial_j u_k} ) \; \d x = \sum_{j , k = 1}^n \int_{\IR^n_+} (\partial_j u_k + \partial_k u_j) (\overline{\partial_j u_k + \partial_k u_j}) \; \d x.
\end{align*}
Thus, if $u \in \dom(A_2)$ with $A_2 u = 0$, then
\begin{align*}
 \sum_{j , k = 1}^n \int_{\IR^n_+} (\partial_j u_k + \partial_k u_j) (\overline{\partial_j u_k + \partial_k u_j}) \; \d x = 0.
\end{align*}
Taking the summands for $j = k$ yields that $\partial_k u_k$ is  zero  for each $1 \leq k \leq n$.  Consequently, for each $1 \leq k \leq n$ the function $u_k$ is constant in the $k$th coordinate direction what together with $u \in \LL_{2 , \sigma} (\IR^n_+)$ implies that $u$ is zero. \par
Now, consider the case $1 < p < 2$. Since $- 1 \in \rho(A_p)$ by Theorem~\ref{Thm: Shibata - Shimizu}, we find that if $u \in \dom(A_p)$ satisfies $A_p u = 0$ then $u$ also solves
\begin{align*}
 u + A_p u = f \with f=u.
\end{align*}
Since $f$ lies in particular in $\H^1_{p , \sigma} (\IR^n_+)$, Sobolev's embedding theorem implies that $f$ lies in $\LL_{r , \sigma} (\IR^n_+)$ for all $p \leq r < \infty$ with $1 / p - 1 / r \leq 1 / n $. Furthermore, Proposition~\ref{Prop: Higher regularity resolvent estimate} and Sobolev's embedding theorem imply that $u \in \H^2_{r , \sigma} (\IR^n_+)$ and $P \in \widehat{\W}^1_r (\IR^n_+)$ so that $u \in \dom(A_r)$. It follows that $A_r u = 0$. If it is possible to choose $r = 2$, then $u = 0$ follows from the first part of the proof. If this is not possible, iterate this argument until the choice $r = 2$ is possible. It follows that $A_p$ is injective. \par
If $2 < p < \infty$, we conclude by duality as follows. Let $p^{\prime}$ denote the H\"older conjugate exponent to $p$. Note that since $A_{p^{\prime}}$ is sectorial by Theorem~\ref{Thm: Shibata - Shimizu} and since $\LL_{p^{\prime} , \sigma} (\IR^n_+)$ is reflexive, the topological decomposition
\begin{align*}
 \LL_{p^{\prime} , \sigma} (\IR^n_+) = \ker(A_{p^{\prime}}) \oplus \overline{\Rg (A_{p^{\prime}})}
\end{align*}
holds, see~\cite[Prop.~2.1.1]{Haase}. Since $A_{p^{\prime}}$ is injective, it follows that $\LL_{p^{\prime} , \sigma} (\IR^n_+) = \overline{\Rg (A_{p^{\prime}})}$. Since $A_{p^{\prime}}$ is adjoint to $A_p$, the classical annihilator relations yields
\begin{align*}
 \ker(A_p) = \overline{\Rg (A_{p^{\prime}})}^{\perp} = \{ 0 \}. &\qedhere
\end{align*}
\end{proof}

Proposition~\ref{Prop: Injectivity of Stokes} and   the definition of  the part ${\bf A}_p$ imply:
\begin{corollary}
For all  $1 < p < \infty,$ the operator ${\bf A}_p$ is injective.
\end{corollary}

The injectivity combined with the estimates~\eqref{Eq: Resolvent estimate} and~\eqref{Eq: First-order resolvent estimate} results in a comparison of $A_p u$ with the second derivatives of $u,$ with ground space $\LL_{p , \sigma} (\IR^n)$ as well as a comparison of $A_p u$ with the second derivatives of $u$ but with ground space $\dot \H^1_{p , \sigma} (\IR^n_+)$. As a preparation, we record the following trace lemma.

\begin{lemma}
\label{Lem: Homogeneous trace and extension estimates}
For all $u \in \H^1_p (\IR^n_+)$ the trace of $u$ on $\partial \IR^n_+$ belongs to $\B^{1 - 1 / p}_{p , p} (\partial \IR^n_+)$. Moreover, there exists a constant $C > 0$ such that for all $u \in \H^1_p (\IR^n_+),$ we have  the homogeneous estimate
\begin{align}
\label{Eq: Homogeneous trace estimate}
 \| u|_{\partial \IR^n_+} \|_{\dot \B^{1 - 1 / p}_{p , p} (\partial \IR^n_+)} \leq C \| \nabla u \|_{\LL_p (\IR^n_+ ; \IC^n)}.
\end{align}
 Furthermore, there exists a constant $C > 0$ such that all $\eps > 0$ and for each function $g \in \B^{1 - 1 / p}_{p , p} (\partial \IR^n_+)$ there exists $\Psi_{\eps} \in \H^1_p (\IR^n_+)$ with $\Psi_{\eps}|_{\partial \IR^n_+} = g$ that satisfies the estimate
\begin{align}
\label{Eq: Homogeneous extension estimate}
 \| \nabla \Psi_{\eps} \|_{\LL_p (\IR^n_+ ; \IC^n)} \leq C \| g \|_{\dot \B^{1 - 1 / p}_{p , p} (\partial \IR^n_+)} + \eps \| g \|_{\LL_p (\partial \IR^n_+)}.
\end{align}
\end{lemma}
\begin{proof}
The corresponding statement in inhomogeneous spaces is classical: from~\cite[Thm.~7.39 and Lem.~7.40/7.41]{Adams_Fournier}, one knows that
\begin{equation}
\label{Eq: Inhomogeneous trace estimate}
 \| u|_{\partial \IR^n_+} \|_{\B^{1 - 1 / p}_{p , p} (\partial \IR^n_+)} \leq C \| u \|_{\H^1_p (\IR^n_+)} \andf \| \Psi \|_{\H^1_p (\IR^n_+)} \leq C \| g \|_{\B^{1 - 1 / p}_{p , p} (\partial \IR^n_+)},
\end{equation}
with $\Psi$ being an appropriate extension of $g$ to $\IR^n_+$. 
\medbreak
Let us also recall that, as a straightforward consequence of the definition of Besov norms, we have, whenever $\sigma>0,$
\begin{equation}\label{eq:normb}
\|z\|_{\B^\sigma_{q,r}(\IR^d)}\approx \|z\|_{\LL_q(\IR^d)}+ \|z\|_{\dot \B^\sigma_{q,r}(\IR^d)}.
\end{equation}
Let us apply the first inequality of~\eqref{Eq: Inhomogeneous trace estimate} to $u(2^\ell\cdot)$ for all $\ell\in\IZ.$ Since
\begin{align*}
 \| u (2^\ell \cdot) \|_{\LL_p (\IR^n_+)} = 2^{- \ell n / p} \| u \|_{\LL_p (\IR^n_+)} \andf \| \nabla u(2^\ell \cdot) \|_{\LL_p (\IR^n_+ ; \IC^n)} = 2^{\ell(1 - n / p)} \| \nabla u\|_{\LL_p (\IR^n_+ ; \IC^n)}
\end{align*}
and (identifying $\partial \IR^n_+$   with $\IR^{n - 1}$), 
$$ \begin{aligned}\| u (2^\ell \cdot) \|_{\dot\B^{1 - 1 / p}_{p , p} (\partial\IR^{n}_+)}&= 
 \bigg( \sum_{k\in\IZ} 2^{k (p - 1)} \| \dot \Delta_k u(2^\ell \cdot) \|_{\LL_p (\IR^{n - 1})}^p \bigg)^{1 / p}\\
&= 2^{\ell (1 - n / p)} \bigg( \sum_{k\in\IZ} 2^{(k-\ell) (p - 1)} \| \dot \Delta_{k-\ell} u \|_{\LL_p (\IR^{n - 1})}^p \bigg)^{1 / p}\\
&= 2^{\ell (1 - n / p)} \| u \|_{\dot\B^{1 - 1 / p}_{p , p} (\partial\IR^{n}_+)},\end{aligned}$$
we get, remembering~\eqref{Eq: Inhomogeneous trace estimate} and~\eqref{eq:normb},
$$
2^{\ell (1 - n / p)} \| u \|_{\dot\B^{1 - 1 / p}_{p , p} (\partial\IR^{n}_+)}\leq C2^{- \ell n / p}\Bigl( \| u \|_{\LL_p (\IR^n_+)}+2^\ell \| \nabla u\|_{\LL_p (\IR^n_+ ; \IC^n)}\Bigr)\cdotp
$$
Hence,  multiplying both sides by $2^{\ell (n / p-1)}$ and letting $\ell \to +\infty$ yields inequality~\eqref{Eq: Homogeneous trace estimate}.
 \smallbreak
To establish~\eqref{Eq: Homogeneous extension estimate}, denote the extension operator that maps $g$ to $\Psi$ by $E$. Because $E[g(2^\ell \cdot)] (x) = [E g] (2^\ell x)$ might not hold true, one cannot repeat the scaling argument of the first part of the proof. Instead, define a sequence of extensions of $g$ by $\Psi_\ell (x) := E [g (2^\ell \cdot)] (2^{- \ell} x)$ for $x \in \IR^n_+$. The scaling estimates above,  the second estimate in~\eqref{Eq: Inhomogeneous trace estimate} 
and~\eqref{eq:normb} imply
$$\begin{aligned}
 \| \nabla \Psi_\ell \|_{\LL_p (\IR^n_+)} &\leq C 2^{\ell (n / p-1)} \|E[g (2^\ell \cdot)] \|_{\B^{1 - 1 / p}_{p,p} (\partial \IR^n_+)} \\
 &\leq C2^{\ell (n / p-1)}\Bigl(\|g(2^\ell\cdot)\|_{\LL_p(\partial\IR^n_+)}+\|g(2^\ell\cdot)\|_{\dot \B^{1 - 1 / p}_{p , p} (\partial \IR^n_+)}\Bigr)\\
&\leq C \Bigl(2^{\ell(1/p-1)}\|g\|_{\LL_p(\partial\IR^n_+)}+ \| g \|_{\dot \B^{1 - 1 / p}_{p , p} (\partial \IR^n_+)}\Bigr),\end{aligned}$$
 with a constant $C > 0$ independent of $\ell$ and $g$. Since $p > 1$ the existence of an extension $\Psi_{\eps}$ that satisfies~\eqref{Eq: Homogeneous extension estimate} follows by taking $\ell$ large enough.
\end{proof}

Now, we are in the position to present a proof of the following proposition.

\begin{proposition}
\label{Prop: Comparison Stokes with second derivatives}
Let $1 < p < \infty$. There exist two constants $C , C_1 , C_2 > 0$ such that:
\begin{enumerate}
\item \label{It: Operator} For all $u \in \dom(A_p)$ with associated pressure $P \in \widehat \W^1_p (\IR^n_+),$ it holds
\begin{align*}
 C_1 \Big(\| \nabla^2 u \|_{\LL_p (\IR^n_+ ; \IC^{n^3})} + \| \nabla P \|_{\LL_p (\IR^n_+ ; \IC^n)} \Big) \leq \| A_p u \|_{\LL_{p , \sigma} (\IR^n_+;\IC^n)} \leq C_2 \| \nabla^2 u \|_{\LL_p (\IR^n_+ ; \IC^{n^3})}.
\end{align*}
\item \label{It: Item part} For all $u \in \dom({\bf A}_p)$ with associated pressure $P \!\in\! \widehat{\W}^1_p (\IR^n_+)$ such that  $\nabla P\! \in\! \H^1_p (\IR^n_+ ; \IC^n),$ it holds
\begin{align*}
 \| \nabla^2 u \|_{\dot \H^1_p (\IR^n_+ ; \IC^{n^3})} + \| \nabla P \|_{\dot \H^1_p (\IR^n_+ ; \IC^n)} \leq C \| {\bf A}_p u \|_{\dot \H^1_{p , \sigma} (\IR^n_+;\IC^n)}.
\end{align*}
\end{enumerate}
\end{proposition}

\begin{proof}
For the proof of~\eqref{It: Item part} and of  the first inequality in~\eqref{It: Operator}, we use the same symbol $\cA_p$ to designate  $A_p$ or  ${\bf A}_p$. Moreover, $\X$ either stands for the symbols $\LL_p$ or $\H^1_p$. The notation for homogeneous spaces is  $\dot \X,$
and we use $\X_\sigma$ or $\dot \X_\sigma$ to designate the solenoidal counterparts of those spaces.

Let $u \in \dom(\cA_p)$ and $P$ be the pressure associated to $u$. Denote by $f := \cA_p u \in \X_{\sigma} (\IR^n_+)$ the corresponding right-hand side to the Stokes equations. Since $\cA_p$ is a sectorial operator, functional analytic arguments show, see, e.g., Haase~\cite[Prop.~2.1.1]{Haase}, that for all $g \in \overline{\Rg(\cA_p)}$
\begin{align}
\label{Eq: Sectorial approximation}
 \cA_p (\lambda + \cA_p)^{-1} g \to g \quad \text{in} \quad \X_{\sigma} (\IR^n_+) \quad \text{as} \quad \lambda \searrow 0.
\end{align}
Moreover, for $\lambda , \mu > 0$ inequalities~\eqref{Eq: Resolvent estimate} and~\eqref{Eq: Inhomogeneous first-order resolvent estimate} combined with $\cA_p u = f$ yield
\begin{align*}
 \| \nabla^2 [(\mu + \cA_p)^{-1} f - (\lambda + \cA_p)^{-1} f] \|_{\X (\IR^n_+ ; \IC^{n^3})} &= \lvert \lambda - \mu \rvert \| \nabla^2 (\mu + \cA_p)^{-1} (\lambda + \cA_p)^{-1} f \|_{\X (\IR^n_+ ; \IC^{n^3})} \\
 &\leq C \lvert \lambda - \mu \rvert \| \cA_p (\lambda + \cA_p)^{-1} u \|_{\X_{\sigma} (\IR^n_+)} \\
 &\leq C \lvert \lambda - \mu \rvert \| u \|_{\X_{\sigma} (\IR^n_+)}.
\end{align*}
Consequently,
\begin{equation}\label{eq:rel1}
 \| \nabla^2 [(\mu + \cA_p)^{-1} f - (\lambda + \cA_p)^{-1} f] \|_{\X (\IR^n_+ ; \IC^{n^3})} \to 0 \quad \text{as} \quad \lambda , \mu \searrow 0.
\end{equation}
Now, since $\cA_p$ is injective, sectorial and since $\X_{\sigma}(\IR^n_+)$ is reflexive, $\cA_p$ has dense range, see~\cite[Prop.~2.2.1]{Haase}. In particular, $u \in \overline{\Rg (\cA_p)}$ holds true. This, together with the property $\cA_p u = f$ and~\eqref{Eq: Sectorial approximation}, implies
\begin{equation}\label{eq:rel2}
 \| (\lambda + \cA_p)^{-1} f - u \|_{\X_{\sigma} (\IR^n_+)} = \| \cA_p (\lambda + \cA_p)^{-1} u - u \|_{\X_{\sigma} (\IR^n_+)} \to 0 \quad \text{as} \quad \lambda \searrow 0.
\end{equation}
Note that~\eqref{eq:rel1} guarantees that the family $(\nabla^2(\lambda + \cA_p)^{-1}f)_{\lambda\in \Sec_\vartheta}$ 
satisfies the Cauchy criterion at $0$ in $\X(\IR^n_+;\IC^{n^3})$ while~\eqref{eq:rel2} shows that
the distributional limit should be $\nabla^2u.$ Therefore
\begin{align*}
 \nabla^2 (\lambda + \cA_p)^{-1} f \to \nabla^2 u \quad \text{in} \quad \X (\IR^n_+ ; \IC^{n^3}) \quad \text{as} \quad \lambda \searrow 0. 
\end{align*}
Now,~\eqref{Eq: Resolvent estimate} and~\eqref{Eq: First-order resolvent estimate} imply
\begin{align*}
 \| \nabla^2 u \|_{\dot \X (\IR^n_+ ; \IC^{n^3})} \leq \limsup_{\lambda \searrow 0} \| \nabla^2 (\lambda + \cA_p)^{-1} f \|_{\dot \X (\IR^n_+ ; \IC^{n^3})} \leq C \| f \|_{\dot \X_{\sigma} (\IR^n_+)} = C \| \cA_p u \|_{\dot \X_{\sigma} (\IR^n_+)}.
\end{align*}
Since $u$ and $P$ solve the Stokes equations
\begin{align*}
 \nabla P = \Delta u + \cA_p u,
\end{align*}
it follows that
\begin{align*}
 \| \nabla P \|_{\dot \X (\IR^n_+ ; \IC^n)} \leq C \| \cA_p u \|_{\dot \X_{\sigma} (\IR^n_+;\IC^n)}.
\end{align*}
\indent For proving the second inequality in~\eqref{It: Operator}, observe that the boundary condition $\IT(u , P) \e_n = 0$ implies that as in the proof of Proposition~\ref{Prop: Injectivity of Stokes}
we have $P|_{\partial \IR^n_+} \in \B^{1 - 1 / p}_{p , p} (\partial \IR^n_+)$,
see definitions~\eqref{dom:Ap} and~\eqref{dom:AAp}.
Let $\eps > 0$ and apply Lemma~\ref{Lem: Homogeneous trace and extension estimates} to obtain an extension $\Psi_{\eps} \in \H^1_p (\IR^n_+)$ with $\Psi_{\eps} |_{\partial \IR^n_+} = P|_{\partial \IR^n_+}$ that satisfies~\eqref{Eq: Homogeneous extension estimate}. Invoking~\eqref{Eq: Homogeneous extension estimate} and~\eqref{Eq: Homogeneous trace estimate} yields:
\begin{align*}
 \| \nabla \Psi_{\eps} \|_{\LL_p (\IR^n_+ ; \IC^n)} &\leq C \| P|_{\partial \IR^n_+} \|_{\dot \B^{1 -  1 / p}_{p , p} (\partial \IR^n_+)} + \eps \| P|_{\partial \IR^n_+} \|_{\LL_p (\partial \IR^n_+)} \\
 &\leq C \big( \| \partial_n u\cdot \e_n|_{\partial \IR^n_+} \|_{\dot \B^{1 - 1 / p}_{p , p} (\partial \IR^n_+)} \!+\! \| \nabla u_n \cdot \e_n|_{\partial \IR^n_+} \|_{\dot \B^{1 - 1 / p}_{p , p} (\partial \IR^n_+)} \big)  + \eps \| P|_{\partial \IR^n_+} \|_{\LL_p (\partial \IR^n_+)} \\
 &\leq C \| \nabla^2 u \|_{\LL_p (\IR^n_+ ; \IC^{n^3})}  + \eps \| P|_{\partial \IR^n_+} \|_{\LL_p (\partial \IR^n_+)},
\end{align*}
with $C > 0$ independent of $\eps$. Now, let $\IP : \LL_p (\IR^n_+ ; \IC^n) \to \LL_p (\IR^n_+ ; \IC^n)$ denote the Helmholtz projector onto the solenoidal vector fields $\LL_{p , \sigma} (\IR^n_+)$. Then, it follows since the trace of $P - \Psi_{\eps}$ is zero and thus $\IP \divergence (\IT(u ,P))
 =\IP \divergence (\IT(u , \Psi_{\eps})),$  that
\begin{align*}
 \| A_p u \|_{\LL_{p , \sigma} (\IR^n_+)} &= \| \IP \divergence (\IT(u , P)) \|_{\LL_{p , \sigma} (\IR^n_+)} = \| \IP \divergence (\IT(u , \Psi_{\eps})) \|_{\LL_{p , \sigma} (\IR^n_+)} \\
 &\leq C\Big( \| \nabla^2 u \|_{\LL_p (\IR^n_+ ; \IC^{n^3})} + \| \nabla \Psi_{\eps} \|_{\LL_p (\IR^n_+ ; \IC^n)} \Big) \\
 &\leq C \| \nabla^2 u \|_{\LL_p (\IR^n_+ ; \IC^{n^3})} + C \eps \| P \|_{\LL_p (\partial \IR^n_+)}.
\end{align*}
The result follows by letting $\eps \to 0$.
\end{proof}

This result will enable us to prove a  density statement. For its formulation, introduce
\begin{align*}
 \dom(A_p^{\infty}) := \bigcap_{k = 1}^{\infty} \dom(A_p^k)
\end{align*}
and observe that $\dom(A_p^{\infty}) \subset \dom(A_p) \cap \dom({\bf A}_p)$.

\begin{corollary}
\label{Cor: Density of domain in homogeneous Besov}
For all $1 < p < \infty$, $1 \leq q < \infty$ and  $s\in(0,1)$ satisfying~\eqref{Eq: Banach space conditions},
the space $\dom(A_p^{\infty})$ is dense in $\dot \B^s_{p , q , \sigma} (\IR^n_+)$.
\end{corollary}

\begin{proof}
First of all, recall that $\C_{c , \sigma}^{\infty} (\overline{\IR^n_+})$ is dense in $\dot \B^s_{p , q , \sigma} (\IR^n_+)$ by Corollary~\ref{Cor: Density of inhomogeneous spaces in homogeneous spaces}. Consequently, it suffices to approximate $\psi \in \C_{c , \sigma}^{\infty} (\overline{\IR^n_+})$ by elements in $\dom(A_p^{\infty})$ with respect to the $\dot \B^s_{p , q , \sigma} (\IR^n_+)$-norm. The analyticity of $\e^{- t A_p}$ and standard semigroup theory imply that $\e^{- t A_p} \psi \in \dom (A_p^{\infty})$. Moreover, since $\H^1_{p , \sigma} (\IR^n_+) \subset \LL_{p , \sigma} (\IR^n_+)$ the semigroup that is generated by $- {\bf A}_p$ satisfies $\e^{- t {\bf A}_p} = \e^{- t A_p}|_{\H^1_{p , \sigma} (\IR^n_+)}$. The strong continuity of the semigroups, see Corollary~\ref{Cor: Strong continuity}, and real interpolation, see Proposition~\ref{Prop: Interpolation of solenoidal spaces without boundary conditions}, then result in
\begin{align*}
 \| \e^{- t A_p} \psi - \psi \|_{\dot \B^s_{p , q , \sigma} (\IR^n_+)} \leq C \| \e^{- t A_p} \psi - \psi \|_{\LL_{p , \sigma} (\IR^n_+)}^s \| \e^{- t {\bf A}_p} \psi - \psi \|_{\dot \H^1_{p , \sigma} (\IR^n_+)}^{1 - s} \to 0 \quad \text{as} \quad t \searrow 0. &\qedhere
\end{align*}
\end{proof}

\section{The homogeneous Stokes operator and interpolation theory for its domain}
\label{Sec: The homogeneous Stokes operator and interpolation theory for its domain}

We introduce the homogeneous Stokes operator by following the approach described in Section~\ref{Sec: The homogeneous operator and homogeneous spaces}. As the ground space we choose $X := \LL_{p , \sigma} (\IR^n_+)$ and as an ambient space for $\dom(A_p)$ we take $Y := \dot \H^2_{p , \sigma} (\IR^n_+)$. Recall  that $A_p$ is injective by Proposition~\ref{Prop: Injectivity of Stokes} and that according to Proposition~\ref{Prop: Comparison Stokes with second derivatives},  there exist two constants $C_1 , C_2 > 0$ such that for all $u \in \dom(A_p)$ we have
\begin{align*}
 C_1 \| A_p u \|_{\LL_{p , \sigma} (\IR^n_+)} \leq \| u \|_{\dot \H^2_{p , \sigma} (\IR^n_+)} \leq C_2 \| A_p u \|_{\LL_{p , \sigma} (\IR^n_+)}.
\end{align*}
It follows that $A_p$ fulfills Assumption~\ref{Ass: Homogeneous operator}. According to Definition~\ref{Def: Definition homogeneous operator} we define the domain of the homogeneous Stokes operator by taking the closure of $\dom(A_p)$ with respect to the space $\dot \H^2_{p , \sigma} (\IR^n_+),$
namely,
\begin{align}\label{dom:HAp}
 \dom(\dot A_p) := \bigl\{ u \in \dot \H^2_{p , \sigma} (\IR^n_+) : \exists (u_k)_{k \in \IN} \subset \dom(A_p) \text{ s.t.\@ } u_k \to u \text{ in } \dot \H^2_{p , \sigma} (\IR^n_+) \bigr\}\cdotp
\end{align}
Note that this closure is taken with $\dot \H^2_{p , \sigma} (\IR^n_+)$ as an ambient space. Since, depending on $p$ and $n$, the space $\dot \H^2_{p , \sigma} (\IR^n_+)$ might not be complete, the domain of $\dot A_p$ might not be complete either with respect to the $\dot \H^2_{p , \sigma} (\IR^n_+)$-norm. \par
Following Section~\ref{Sec: The homogeneous operator and homogeneous spaces}, one can
 define $\dot A_p u$ for $u \in \dom(\dot A_p)$ to be  the following limit in $\LL_{p , \sigma} (\IR^n_+)$:
\begin{align*}
 \dot A_p u := \lim_{k \to \infty} A_p u_k\quad\hbox{in }\  \LL_{p , \sigma} (\IR^n_+).
\end{align*}
The space $\dom(\dot A_p)$ is endowed with the norm $\| u \|_{\dom(\dot A_p)} := \| \dot A_p u \|_{\LL_{p , \sigma} (\IR^n_+)}$ for $u \in \dom(\dot A_p)$ (that this defines a norm follows from Lemma~\ref{Lem: Injectivity of homogeneous abstract operator}). Finally, in order to fulfill Assumption~\ref{Ass: Intersection property} we need the following proposition.

\begin{proposition}\label{Prop: Homogeneous plus ground space equals operator}
Let $1 < p < \infty$. Then $\LL_{p , \sigma} (\IR^n_+) \cap \dom(\dot A_p) = \dom(A_p)$.
\end{proposition}


\begin{proof}
The inclusion $\dom(A_p) \subset \LL_{p , \sigma} (\IR^n_+) \cap \dom(\dot A_p)$ being trivial, we concentrate on the other inclusion. \par
Let $u \in \LL_{p , \sigma} (\IR^n_+) \cap \dom(\dot A_p)$. Since $\dom(\dot A_p) \subset \dot \H^2_{p , \sigma} (\IR^n_+)$ and since
\begin{align*}
 \LL_{p , \sigma} (\IR^n_+) \cap \dot \H^2_{p , \sigma} (\IR^n_+) \subset \H^2_{p , \sigma} (\IR^n_+),
\end{align*}
we only need to show the existence of $P \in \widehat{\W}^1_p (\IR^n_+)$ with $\divergence(\IT (u , P)) \in \LL_{p , \sigma} (\IR^n_+)$ and $\IT(u , P) \e_n = 0$ on $\partial \IR^n_+$. For this purpose, let $P \in \widehat{\W}^1_p (\IR^n_+)$ denote the solution to the Dirichlet problem
\begin{align} \label{Eq: wDP}
\left\{ \begin{aligned}
    - \Delta P &= 0 &&\text{in } \IR^n_+, \\
    P &= \partial_n u_n && \text{on } \partial \IR^n_+.
\end{aligned} \right.
\end{align}
Notice that such a $P$ exists if and only if $Q := P - \partial_n u_n \in \widehat{\W}^1_p (\IR^n_+)$ satisfies $Q|_{\partial \IR^n_+} = 0$ on $\partial \IR^n_+$ and
\begin{align*}
    \int_{\IR^n_+} \nabla Q \cdot \nabla \psi \, \d x = - \int_{\IR^n_+} F \cdot \nabla \psi \, \d x \quad \text{for all } \psi \in \C_c^{\infty} (\IR^n_+),
\end{align*}
with $F := \nabla \partial_n u_n$. This is known as the weak Dirichlet problem, and it is solvable for any $F \in \LL_p (\IR^n_+ ; \IC^n)$, by, e.g.,~\cite[Thm.~2.12]{Shi20}, with the estimate
\begin{align}
\label{Eq: Stability estimate weak Dirichlet problem}
    \| \nabla Q \|_{\LL_p (\IR^n_+ ; \IC^n)} \leq C \| F \|_{\LL_p (\IR^n ; \IC^n)}.
\end{align}
Now, we have to show that $u$ and $P$ solve the Stokes system with homogeneous Neumann boundary conditions and inhomogeneous right-hand side in $\LL_{p , \sigma} (\IR^n_+)$. Since $u \in \dom(\dot A_p)$ there exists $(u^k)_{k \in \IN} \subset \dom(A_p)$ such that $u^k \to u$ in $\dot \H^2_{p , \sigma} (\IR^n_+)$. Let $P^k\in \widehat{\W}^1_p (\IR^n_+)$ denote the associated pressure to $u^k$. Notice that $P^k$ satisfies~\eqref{Eq: wDP} with Dirichlet data being $\partial_n u^k_n$. Thus, by~\eqref{Eq: Stability estimate weak Dirichlet problem} one infers that
\begin{align*}
 \| \nabla (P - P^k) \|_{\LL_p (\IR^n_+ ; \IC^n)} &\leq \| \nabla [(P - \partial_n u_n) - (P^k - \partial_n u_n^k)] \|_{\LL_p (\IR^n_+ ; \IC^n)} + \| \nabla (\partial_n u_n - \partial_n u^k_n) \|_{\LL_p (\IR^n_+ ; \IC^n)} \\
 &\leq C \| \nabla \partial_n (u - u^k_n) \|_{\LL_p (\IR^n_+ ; \IC^n)} \to 0 \quad \text{as} \quad k \to \infty.
\end{align*}
Since $\divergence(\IT(u^k , P^k)) \in \LL_{p , \sigma} (\IR^n_+)$ for all $k \in \IN$ this implies that
\begin{align*}
 \divergence(\IT(u , P)) = \lim_{k \to \infty} \divergence(\IT(u^k , P^k)) \in \LL_{p , \sigma} (\IR^n_+).
\end{align*}
To conclude that $u$ and $P$ satisfy the correct boundary condition, notice first that 
we already know that 
$u \in \H^2_p (\IR^n_+ ; \IC^n).$ Hence, by the trace theorem, one has that $\nabla u|_{\partial \IR^n_+} \in \B^{1 - 1/p}_{p , p} (\IR^n_+ ; \IC^n)$. Now, taking into account the first $(n - 1)$ entries of the boundary condition $\IT (u^k , P^k) \e_n = 0$ on $\partial \IR^n_+$ one infers that for $j = 1 , \dots , n - 1,$ it holds
\begin{align*}
 \partial_j u^k_n + \partial_n u^k_j = 0 \quad \text{on} \quad \partial \IR^n_+.
\end{align*}
To obtain  the same condition for $u$, notice that by~\eqref{Eq: Homogeneous trace estimate}, one has that
\begin{align*}
 \|[\partial_j u_n + \partial_n u_j]|_{\partial \IR^n_+}\|_{\dot \B^{1 - 1/p}_{p , p} (\partial \IR^n_+)} &= \|[\partial_j u_n - \partial_j u^k_n + \partial_n u_j - \partial_n u^k_j]|_{\partial \IR^n_+}\|_{\dot \B^{1 - 1/p}_{p , p} (\partial \IR^n_+)}\\
 &\leq C \| \nabla^2 (u - u^k) \|_{\LL_p (\IR^n_+)} \to 0 \quad \text{as} \quad k \to \infty.
\end{align*}
Since $[\partial_j u_n + \partial_n u_j]|_{\partial \IR^n_+}$ lies in $\LL_p (\partial \IR^n_+)$ by the $\H^2_p$-regularity of $u$, it follows that $[\partial_j u_n + \partial_n u_j]|_{\partial \IR^n_+} = 0$. Finally, the last entry of the boundary condition $\IT (u , P) \e_n = 0$ on $\partial \IR^n_+$ reads
\begin{align*}
 \partial_n u_n - P = 0 \quad \text{on} \quad \IR^n_+
\end{align*}
and thus holds by construction of $P$.
\end{proof}

In the following, it is our aim to identify $\dot \B^{2 \theta}_{p , q , \sigma} (\IR^n_+)$ with the real interpolation space $$(\LL_{p , \sigma} (\IR^n_+) , \dom(\dot A_p))_{\theta , q}$$ for small values of $\theta$.

\begin{proposition}
\label{Prop: Besov domain interpolation}
For  $1 < p < \infty$, $1 \leq q < \infty$ and $0 < \theta < 1$ with
\begin{align*} 0 < \theta < \frac{n}{2 p} \ \hbox{ if }\ q>1,\quad\hbox{or}\quad  0 < \theta \leq \frac{n}{2 p}\ \hbox{ if }\ q=1,\end{align*}
one has
\begin{align*}
 \big( \LL_{p , \sigma} (\IR^n_+) , \dom(\dot A_p) \big)_{\theta , q} \hookrightarrow  \dot \B^{2 \theta}_{p , q , \sigma} (\IR^n_+).
\end{align*}
 If, furthermore, $0 < \theta < 1 / 2,$ then it holds with equivalent norms that
\begin{align}
\label{Eq: Embedding for smooth functions}
\big( \LL_{p , \sigma} (\IR^n_+) , \dom(\dot A_p) \big)_{\theta , q} =  \dot \B^{2 \theta}_{p , q , \sigma} (\IR^n_+).
 \end{align}

\end{proposition}

\begin{proof}
The first statement directly follows from Corollary~\ref{Cor: Lp-sigma in right scale} and Proposition~\ref{Prop: Interpolation of solenoidal spaces without boundary conditions} since
\begin{align*}
 \big(\LL_{p , \sigma} (\IR^n_+) , \dom(\dot A_p) \big)_{\theta , q} \hookrightarrow \big(\LL_{p , \sigma} (\IR^n_+) , \dot \H^2_{p , \sigma} (\IR^n_+) \big)_{\theta , q} = \dot \B^{2 \theta}_{p , q , \sigma} (\IR^n_+).
\end{align*}

\indent For the other direction, let $f \in \dot \B^{2 \theta}_{p , q , \sigma} (\IR^n_+)$. We first show that
\begin{align*}
 f \in \LL_{p , \sigma} (\IR^n_+) + \dom(\dot A_p).
\end{align*}
For this purpose, approximate $f$ in $\dot \B^{2 \theta}_{p , q , \sigma} (\IR^n_+)$ by $(f_k)_{k \in \IN} \subset \C_{c , \sigma}^{\infty} (\overline{\IR^n_+})$, which exists due to Corollary~\ref{Cor: Density of inhomogeneous spaces in homogeneous spaces}. Use~\eqref{Eq: Allerweltsformel 1} to write for $t > 0$
\begin{align}
\label{Eq: Representation of Besov space function}
 f_k = \e^{- t A_p} f_k + A_p \int_0^t \e^{- s A_p} f_k \, \d s.
\end{align}
In the following, we interpret $\e^{- t A_p} f$ as follows: since $\e^{- t A_p}$ is bounded on $\LL_{p , \sigma} (\IR^n_+)$ and since $\e^{- t A_p}|_{\H^1_{p , \sigma}(\IR^n_+)}$ satisfies a boundedness estimates on $\dot \H^1_{p , \sigma} (\IR^n_+)$ by Proposition~\ref{Prop: Higher regularity resolvent estimate}, the interpolation result of Proposition~\ref{Prop: Interpolation of solenoidal spaces without boundary conditions} yields that $\e^{- t A_p}|_{\H^1_{p , \sigma}} (\IR^n_+)$ satisfies a boundedness estimate on $\dot \B^{2 \theta}_{p , q , \sigma} (\IR^n_+)$. Since $\C_{c , \sigma}^{\infty} (\overline{\IR^n_+})$ is dense in this space by Corollary~\ref{Cor: Density of inhomogeneous spaces in homogeneous spaces}, we can regard the semigroup $\e^{- t A_p}$ as a bounded operator on this space. This gives sense of the expression $\e^{- t A_p} f$. A similar argument, but by interpolating the estimate from Corollary~\ref{Corollary: Higher regularity semigroup estimate} with~\eqref{Eq: Smoothing estimate} and Proposition~\ref{Prop: Comparison Stokes with second derivatives} yields
\begin{align*}
 \| \nabla^2 \e^{- t A_p} (f_k - f) \|_{\LL_p (\IR^n_+ ; \IC^{n^3})} \leq C t^{- (1 - \theta)} \| f_k - f \|_{\dot \B^{2 \theta}_{p , q , \sigma} (\IR^n_+)} \to 0 \quad \text{as} \quad k \to \infty.
\end{align*}
Since $\e^{- t A_p} f_k \in \dom(A_p)$ it follows that $\e^{- t A_p} f \in \dom(\dot A_p)$ and that $\e^{- t A_p} f_k$ converges to $\e^{- t A_p} f$ in $\dom(\dot A_p)$. The previous estimate further implies that
\begin{align*}
\Big\| \nabla^2 \int_0^t \e^{- s A_p} (f_k - f) \, \d s \Big\|_{\LL_{p , \sigma} (\IR^n_+)} \leq C \int_0^t\! s^{- (1 - \theta)} \, \d s \;\| f_k - f \|_{\dot \B^{2 \theta}_{p , q , \sigma} (\IR^n_+)} \to 0 \!\quad \text{as} \quad k \to \infty.
\end{align*}
Since $\int_0^t \e^{- s A_p} f_k \; \d s \in \dom(A_p)$ for each $k \in \IN$ it follows that $\int_0^t \e^{- s A_p} f \; \d s \in \dom(\dot A_p)$ and that
\begin{align*}
 A_p \int_0^t \e^{- s A_p} f_k \; \d s \to \dot A_p \int_0^t \e^{- s A_p} f \; \d s \quad \text{as} \quad k \to \infty \quad \text{in} \quad \LL_{p , \sigma} (\IR^n_+).
\end{align*}
Consequently, taking the limit $k \to \infty$ in~\eqref{Eq: Representation of Besov space function} with respect to the sum space $\LL_{p , \sigma} (\IR^n_+) + \dom(\dot A_p)$ reveals that
\begin{align*}
 f = \e^{- t A_p} f + \dot A_p \int_0^t s \e^{- s A_p} f \; \frac{\d s}{s} =: b + a
\end{align*}
with $a \in \LL_{p , \sigma} (\IR^n_+)$ and $b \in \dom(\dot A_p)$ so that the $K$-functional can be estimated by means of
\begin{align*}
 K(t , f ; \LL_{p , \sigma} (\IR^n_+) , \dom(\dot A_p)) \leq \| a \|_{\LL_{p , \sigma} (\IR^n_+)} + t \| \dot A_p b \|_{\LL_{p , \sigma} (\IR^n_+)}.
\end{align*}
Now, Lemma~\ref{Lem: Hardy} implies that
\begin{align}
\label{Eq: Real interpolation by semigroup}
 \| f \|_{(\LL_{p , \sigma} (\IR^n_+) , \dom(\dot A_p))_{\theta , q}} \leq \frac{1 + \theta}{\theta} \bigg( \int_0^{\infty} \| t^{1 - \theta} \dot A_p \e^{- t A_p} f \|^q_{\LL_{p , \sigma} (\IR^n_+)} \frac{\d t}{t} \bigg)^{1 / q}\cdotp
\end{align}
Let $(f_k)_{k \in \IN}$ again be as above. The lemma of Fatou implies that
\begin{align}
\label{Eq: Real interpolation Fatou}
 \bigg( \int_0^{\infty} \| t^{1 - \theta} \dot A_p \e^{- t A_p} f \|^q_{\LL_{p , \sigma} (\IR^n_+)} \frac{\d t}{t} \bigg)^{1 / q} \leq \liminf_{k \to \infty} \bigg( \int_0^{\infty} \| t^{1 - \theta} A_p \e^{- t A_p} f_k \|^q_{\LL_{p , \sigma} (\IR^n_+)} \frac{\d t}{t} \bigg)^{1 / q}
\end{align}
To conclude, it suffices to control the function $\| t^{1 - \theta} A_p \e^{- t A_p} f_k \|_{\LL_{p , \sigma} (\IR^n_+)}$ by the $K$-functional of the interpolation space $(\LL_{p , \sigma} (\IR^n_+) , \dot \H^1_{p , \sigma} (\IR^n_+))_{2 \theta , q}$. Since $\dot \B^{2 \theta}_{p , q , \sigma} (\IR^n_+) = (\LL_{p , \sigma} (\IR^n_+) , \dot \H^1_{p , \sigma} (\IR^n_+))_{2 \theta , q}$ by Proposition~\ref{Prop: Interpolation of solenoidal spaces without boundary conditions}, one can write $f_k = a_k + b_k$ with $a_k \in \LL_{p , \sigma} (\IR^n_+)$ and $b_k \in \dot \H^1_{p , \sigma} (\IR^n_+)$. Notice that $b_k = f_k - a_k \in \LL_{p , \sigma} (\IR^n_+)$ implies that $b_k \in \H^1_{p , \sigma} (\IR^n_+)$. Employing the smoothing estimate of $(\e^{- t A_p})_{t \geq 0}$ in~\eqref{Eq: Smoothing estimate}, Proposition~\ref{Prop: Comparison Stokes with second derivatives} and Corollary~\ref{Corollary: Higher regularity semigroup estimate} then yields
\begin{align*}
 \| t^{1 - \theta} A_p \e^{- t A_p} f_k \|_{\LL_{p , \sigma} (\IR^n_+)} &\leq \| t^{1 - \theta} A_p \e^{- t A_p} a_k \|_{\LL_{p , \sigma} (\IR^n_+)} + \| t^{1 - \theta} A_p \e^{- t A_p} b_k \|_{\LL_{p , \sigma} (\IR^n_+)} \\
 &\leq C \Big( t^{- \theta} \| a_k \|_{\LL_{p , \sigma} (\IR^n_+)} + \| t^{1 - \theta} \nabla^2 \e^{- t A_p} b_k \|_{\LL_p (\IR^n_+ ; \IC^{n^3})} \Big) \\
 &\leq C \Big( t^{- \theta} \| a_k \|_{\LL_{p , \sigma} (\IR^n_+)} + t^{1 / 2 - \theta} \| \nabla b_k \|_{\LL_p (\IR^n_+ ; \IC^{n^2})} \Big)\cdotp
\end{align*}
Taking the infimum over all such decompositions $f_k = a_k + b_k$ delivers
\begin{align*}
 \| t^{1 - \theta} A_p \e^{- t A_p} f_k \|_{\LL_{p , \sigma} (\IR^n_+)} \leq C t^{- \theta} K (t^{1 / 2} , f_k ; \LL_{p , \sigma} (\IR^n_+) , \dot \H^1_{p , \sigma} (\IR^n_+)).
\end{align*}
Combining this with~\eqref{Eq: Real interpolation by semigroup} and~\eqref{Eq: Real interpolation Fatou} and applying a substitution yields
\begin{align*}
 \| f \|_{(\LL_{p , \sigma} (\IR^n_+) , \dom(\dot A_p))_{\theta , q}} &\leq C \liminf_{k \to \infty} \bigg( \int_0^{\infty} \big\{ t^{- \theta} K (t^{1 / 2} , f_k ; \LL_{p , \sigma} (\IR^n_+) , \dot \H^1_{p , \sigma} (\IR^n_+)) \big\}^q \; \frac{\d t}{t} \bigg)^{1 / q} \\
 &= 2^{1 / q} C \liminf_{k \to \infty} \bigg( \int_0^{\infty} \big\{ s^{- 2 \theta} K (s , f_k ; \LL_{p , \sigma} (\IR^n_+) , \dot \H^1_{p , \sigma} (\IR^n_+)) \big\}^q \; \frac{\d s}{s} \bigg)^{1 / q} \\
 &\leq C \| f \|_{\dot \B^{2 \theta}_{p , q , \sigma} (\IR^n_+)}. \qedhere
\end{align*}
\end{proof}

Based on results of Theorem~\ref{Thm: Shibata - Shimizu} and Proposition~\ref{Prop: Higher regularity resolvent estimate}, one obtains by applying the real interpolation method and Proposition~\ref{Prop: Besov domain interpolation} the following result in Besov spaces.

\begin{proposition}
\label{Prop: Besov resolvent estimate}
Let  $1 < p < \infty$, $1 \leq q < \infty$ and $\vartheta \in (0 , \pi)$, and $s\in (0,1)$ 
fulfill~\eqref{Eq: Banach space conditions}.
Then, there exists a constant $C > 0$ such that for all $\lambda \in \Sec_{\vartheta}$ and all 
$f \in \B^s_{p, q, \sigma} (\IR^n_+)$ the solution $u := (\lambda + A_p)^{-1} f$ together with the associated pressure $P$ satisfy $u , \nabla P \in \dot \B^s_{p , q} (\IR^n_+ ; \IC^n)$, $\nabla u \in \dot \B^s_{p , q} (\IR^n_+ ; \IC^{n^2})$, $\nabla^2 u \in \dot \B^s_{p , q} (\IR^n_+ ; \IC^{n^3})$ and
$$
 \lvert \lambda \rvert \|  u \|_{\dot \B^s_{p , q , \sigma} (\IR^n_+)} + 
 \lvert \lambda \rvert^{1 / 2} \| \nabla u \|_{\dot \B^s_{p,q} (\IR^n_+ ; \IC^{n^2})} + \| \nabla^2 u \|_{\dot \B^s_{p,q} (\IR^n_+ ; \IC^{n^3})} + \| \nabla P \|_{\dot \B^s_{p,q} (\IR^n_+ ; \IC^n)} 
 \leq C \|  f \|_{\dot \B^s_{p , q , \sigma} (\IR^n_+)}.$$
\end{proposition}

Next, let ${\bf A}_{s , p , q}$ denote the part of $A_p$ on $\B^s_{p , q , \sigma} (\IR^n_+)$, i.e.,
\begin{align}\label{dom:Aspq}
 \dom({\bf A}_{s , p , q}) := \{ u \in \dom(A_p) : A_p u \in \B^s_{p , q , \sigma} (\IR^n_+) \}, \quad {\bf A}_{p , q , \sigma} u := A_p u \quad \text{for } u \in \dom({\bf A}_{s , p , q}).
\end{align}
With this definition, literally the same proof as the one of the first part of Proposition~\ref{Prop: Comparison Stokes with second derivatives} but relying on Proposition~\ref{Prop: Besov resolvent estimate} delivers the following statement.

\begin{proposition}
\label{Prop: Homogeneous control in Besov spaces}
Let $1 < p < \infty$, $1 \leq q < \infty$  and $s\in (0,1)$
 fulfill~\eqref{Eq: Banach space conditions}.
Then there exists a constant $C > 0$ such that for all $u \in \dom({\bf A}_{s , p , q})$ it holds
\begin{align*}
 \| \nabla^2 u \|_{\dot \B^s_{p , q} (\IR^n_+ ; \IC^{n^3})} + \| \nabla P \|_{\dot \B^s_{p , q} (\IR^n_+ ; \IC^n)} \leq C \| {\bf A}_{s , p , q} u \|_{\dot \B^s_{p , q , \sigma} (\IR^n_+)}.
\end{align*}
\end{proposition}

\section{Maximal regularity for the Stokes operator in homogeneous Besov spaces}

We are now in the position to apply the abstract results of Chapter~\ref{Sec: Da Prato--Grisvard theorem} to the Stokes operator on $\LL_{p , \sigma} (\IR^n_+)$. By  Propositions~\ref{Prop: Injectivity of Stokes},~\ref{Prop: Comparison Stokes with second derivatives} and~\ref{Prop: Homogeneous plus ground space equals operator}, the Stokes operator satisfies Assumptions~\ref{Ass: Homogeneous operator} and~\ref{Ass: Intersection property}
with the choice:
\begin{align*}
X = \LL_{p , \sigma}(\IR^n_+) \quad \text{and} \quad Y = \dot \H^2_{p , \sigma} (\IR^n_+).
\end{align*}

\indent 
Let $\e_n = (0 , \dots , 0 , 1)$ is the $n$-th standard basis vector of $\IR^n$. 
Consider the evolutionary Stokes system
\begin{align}
\label{Eq: Ev Stokes}
\left\{ \begin{aligned}
 \partial_t u - \divergence \IT (u, P) &= f && t > 0 , \quad x \in \IR^n_+ \\
 \divergence u &=0  && t > 0 , \quad x \in \IR^n_+ \\
 \IT (u , P) \e_n &= 0  && t > 0 , \quad x \in \partial \IR^n_+\\
 u|_{t = 0} &= u_0  && x \in \IR^n_+
\end{aligned} \right.
\end{align}
Then, a direct application of the Da Prato -- Grisvard theorem stated by Theorem~\ref{Thm: Full Da Prato - Grisvard} 
(with $\theta=s/2$) together with Propositions~\ref{Prop: Equivalent norms},~\ref{Prop: Besov domain interpolation} and~\ref{Prop: Homogeneous control in Besov spaces} and the fact that $\B^s_{p, q, \sigma} (\IR^n_+)$ is dense in $\dot\B^s_{p, q, \sigma} (\IR^n_+)$
(a consequence of Lemma~\ref{Lem: Density in abstract spaces} and of~\eqref{Eq: Embedding for smooth functions})
 leads to the following theorem.

\begin{theorem}
\label{Thm: Da Prato - Grisvard for Stokes}
Let $p\in (1,\infty)$ and $q\in [1 , 2)$ satisfy
\begin{align*}
 \frac{n}{p} + \frac{2}{q} > 2
\end{align*}
and let $s \in (0 , 2 /q - 1)$ satisfy 
\begin{align*}
  0 < s < \frac{n}{p} + \frac{2}{q} - 2,\  \mbox{ \ or \ } 0< s \leq \frac{n}{p} + \frac{2}{q} - 2 \mbox{ \ if \ } q = 1.
\end{align*}
Let $f \in \LL_q(\IR_+  ; \dot \B^s_{p , q , \sigma} (\IR^n_+))$ and $u_0 \in \dot \B^{2 + s - 2 / q}_{p , q , \sigma} (\IR^n_+).$ Then, there exists a unique mild solution $(u , P)$ to system~\eqref{Eq: Ev Stokes} with $u \in \cC(\IR_+  ; \dot \B^{2+s-2/q}_{p , q , \sigma} (\IR^n_+))$ such that
\begin{align*}
 u_t \in \LL_q(\IR_+ ; \dot \B^{s}_{p , q , \sigma} (\IR^n_+;\IC^n)) , \quad \nabla^2 u \in \LL_q(\IR_+ ; \dot \B^{s}_{p , q} (\IR^n_+ ; \IC^{n^3})), \quad \nabla P \in \LL_q(\IR_+ ; \dot \B^{s}_{p , q} (\IR^n_+ ; \IC^n))
\end{align*}
 and the following inequality holds
 $$\displaylines{\quad
  \|u\|_{\LL_\infty(\IR_+;\dot \B^{2+s-2/q}_{p,q,\sigma}(\IR^n_+))}+
  \|u_t\|_{\LL_q(\IR_+ ; \dot \B^{s}_{p,q,\sigma}(\IR^n_+))} + \| \nabla^2 u\|_{\LL_q(\IR_+ ; \dot \B^{s}_{p,q}(\IR^n_+ ; \IC^{n^3}))} \hfill\cr\hfill+ \| \nabla P \|_{\LL_q(\IR_+ ; \dot \B^{s}_{p,q}(\IR^n_+ ; \IC^n))} \leq C\Big( \| f\|_{\LL_q(\IR_+ ; \dot \B^{s}_{p,q , \sigma}(\IR^n_+))} + \|u_0\|_{ \dot \B^{2+s-2/q}_{p,q,\sigma}(\IR^n_+)}\Big)\cdotp}$$
\end{theorem}
\begin{remark}
Note that the conditions on $p$, $q$ and $s$ come from the conditions imposed in Theorem~\ref{Thm: Full Da Prato - Grisvard} and in Proposition~\ref{Prop: Besov domain interpolation} to ensure~\eqref{Eq: Embedding for smooth functions}. The most restrictive condition arises in order to compute $(\LL_{p , \sigma} (\IR^n_+) , \dom(\dot \cA_p))_{1 + s/2 - 1/q , q}$ to characterize the initial data space. These are exactly the conditions imposed in Theorem~\ref{Thm: Da Prato - Grisvard for Stokes}. All other conditions that arise in Theorem~\ref{Thm: Full Da Prato - Grisvard} and in Proposition~\ref{Prop: Besov domain interpolation} are already implied by these ones. 
\end{remark}





\chapter{The free boundary problem for incompressible fluids with infinite depth}
\label{Sec: The free boundary problem for incompressible fluids with infinite depth}

As an   application of the linear theory developed in the previous parts, we investigate the free boundary problem for the incompressible Navier-Stokes equations 
in the case where the initial fluid domain is the half-space $\IR_+^n$ with $n\geq3.$
The governing equations read 
\begin{align}
\label{eq:FBNS}
 \left\{
  \begin{aligned}
   \partial_t v + (v \cdot \nabla) v - \divergence \IT (v , Q) &= 0, \qquad && t \in\IR_+ ,\  x \in \Omega_t, \\
   \divergence v &= 0, \qquad &&  t \in\IR_+ ,\  x \in \Omega_t, \\
   \IT (v , Q) \overline{n} &= 0, \qquad && t \in \IR_+,\  x \in \partial \Omega_t, \\
      v \cdot \overline{n} &= - (\partial_t \eta) / \lvert \nabla_x \eta \rvert ,\qquad && t \in\IR_+,\  x \in \partial \Omega_t,\\
   v|_{t = 0} &= v_0, \qquad && x \in \IR_+^n, \\
   \Omega_t |_{t = 0} &= \IR_+^n.\end{aligned} \right.
\end{align}
Here the unknowns are the velocity field $v=v(t,x),$ the 
pressure $Q=Q(t,x)$ and the time-dependent fluid domain $\Omega_t.$
Recall that $\IT (v , Q):= D(v)-Q \Id$ denotes the stress tensor
with $ D(v) := \nabla v + [\nabla v]^{\top}$ and that  the boundary of $\Omega_t$ is represented by 
some (unknown) nondegenerate function $\eta=\eta(t,x)$ through
$$\partial \Omega_t=\bigl\{x\in\IR^n\,:\, \eta(t,x)=0\bigr\}\cdotp$$ 
 We denote  by $\overline{n}$ the outward unit normal vector
to $\partial \Omega_t,$  that is $\overline{n} = \nabla_x \eta / \lvert \nabla_x \eta \rvert$.  
Hence, initially, $\overline n=-\e_n$ with $\e_n:=(0,\cdots,0,1).$ 
\smallbreak
One of the difficulties is that the free interface is unbounded,
so that  the  classical approaches based on fast time decay 
of the solutions to the Stokes system do not work.
To overcome the difficulty, we will take advantage 
of $\LL_1$ maximal regularity estimates so as  to control 
the geometrical change of the free boundary globally in time.



\section{The free boundary problem for incompressible fluids}

In order to solve~\eqref{eq:FBNS}, we shall perform a Lagrangian change of variables  to  reduce 
the problem to a system of equations in the \emph{time independent domain} $\IR_+^n.$ 

Assume that the velocity field $v$ fulfills the conditions 
of the Cauchy-Lipschitz theorem, i.e., $\nabla v \in \LL_{1,(loc)}(0,\infty;\LL_\infty (\IR^n_+))$ and define the flow $X$ of $v$  according to 
\begin{equation}\label{def:Xv}
  X (t, y) :=  y + \int_0^t  v (t^\prime,X (t^\prime, y))\,\d t^{\prime},\qquad y\in\IR^n_+.
\end{equation}
Then, we change  Eulerian coordinates $(t,x)$
to Lagrangian coordinates $(t,y)$,  where  $x$ and $y$ are interrelated through 
$x=X(t,y),$ that is to say we consider the   following new unknown velocity field $u$ and pressure $P$ defined by
$$u(t,y):= v (t,X (t, y))\andf P(t,y):=Q(t,X (t, y)).$$
The `Lagrangian' flow,  now denoted by $X_u,$ may  be directly
computed  from $u$ thanks to:
\begin{equation}\label{def:Xu}
X_u (t, y) :=  y + \int_0^t  u(t^\prime,y)\,\d t^\prime.
\end{equation}
Besides,  we have 
$$\Omega_t = \bigl\{ x \in \IR^n : x = X_u(t, y),\;  y \in \IR^n_+ \bigr\}
\andf \partial \Omega_t = \{ x \in \IR^n : x = X_u (t, y) ,\;  y \in \partial \IR^n_+ \bigr\}\cdotp $$
In other words, when performing the change of coordinates $(t,x)\leadsto (t,y),$ 
the \emph{unknown} time dependent domain $\Omega_t$ is changed
into the \emph{time independent domain} $\IR^n_+,$ and the boundary $\partial \Omega_t$ corresponds  to 
the horizontal (hyper-)plane $\partial\IR_+^n.$
\medbreak
{In what follows, for all vector fields $Z,$ we denote 
$$(\nabla_yZ)_{ij}:=\partial_{y_i}  Z^j \mbox{ \ \ 
and  \ \ } (D_y Z)_{ij}:=\partial_{y_j} Z^i$$
and,  for a matrix valued function $F,$ we set $(\divergence_y F)^j:=\sum_i\partial_{y_i} F_{i}^j$.} In order to recast the  equations of~\eqref{eq:FBNS} in Lagrangian coordinates, let us introduce the matrix
$A_u(t,y):=(D_yX_u(t,y))^{-1}.$  As pointed out in, e.g., the appendix of~\cite{DM-CPAM}, in the $(t,y)$ coordinates system, operators $\nabla$  and $\divergence$ 
translate, assuming that the diffeomorphism $X_v$ is measure preserving, into 
\begin{equation}\label{eq:lagop}
\nabla_u:=A_u^\top\nabla_y\andf
\divergence_u:=A_u^\top:\nabla_y=\divergence_y(A_u\cdot).
\end{equation}
Hence,
 the stress tensor $\IT$ becomes 
\begin{equation}\label{eq:Tlag}
\IT_u(w,\vartheta):= \nabla_uw+(\nabla_uw)^\top-\vartheta\Id=
A_u^\top\cdot\nabla_y w+D_yw\cdot A_u-\vartheta\Id,
\end{equation}
and the unit normal vector is transformed into $\overline{n}_u (t, y) := \overline{n}(t,X_u(t, y)).$
\medbreak
With the above notation,  system~\eqref{eq:FBNS}  in Lagrangian coordinates thus reads
\begin{align}
\label{eq:FBNSLag}
 \left\{
  \begin{aligned}
   \partial_t u - \divergence_u \IT_u (u , P) &= 0 \quad &\hbox{in }\ & \IR_+\times\IR^n_+, \\
   \divergence_uu &= 0 \quad &\hbox{in }\ & \IR_+\times\IR^n_+, \\
   \IT_u (u , P) \overline{n}_u &= 0 \quad &\hbox{on }\ & \IR_+\times\partial\IR^n_+,\\
   u|_{t = 0} &= v_0  \quad &\hbox{in }\ & \IR^n_+.
  \end{aligned}
 \right.
\end{align}
To close the system without referring  to the initial Eulerian framework, we add
the following relation which is the consequence of~\eqref{def:Xu}:
\begin{equation}\label{eq:Au}
A_u(t,y)=(D_yX_u(t,y))^{-1} \with D_yX_u(t,y)=\Id+\int_0^t D_yu(t^\prime, y)\, \d t^\prime.
\end{equation}
Clearly, in order to justify that the Eulerian and Lagrangian formulations are equivalent, 
we need  a control on the supremum norm of  the right-hand side and, since  we strive for a \emph{global-in-time} existence result, 
it is somehow unavoidable to solve the system in a functional framework ensuring  a bound 
on $\nabla u$ in the space $\LL_1(\IR_+;\LL_\infty(\IR_+^n)).$ 

Our functional framework should also  be a relevant one for  the   linearization of~\eqref{eq:FBNSLag}, that is 
the  following (nonhomogeneous) Stokes system:
\begin{align}
\label{Eq: Stokes general}
 \left\{
  \begin{aligned}
   \partial_t u - \divergence \IT (u , P) &= f \quad &\hbox{in }\ & \IR_+\times\IR^n_+, \\
   \divergence u &= g \quad &\hbox{in }\ & \IR_+\times\IR^n_+, \\
   \IT (u , P) \e_n|_{\partial\IR^n_+}&= h \quad &\hbox{on }\ & \IR_+\times\partial\IR^n_+,\\
   u|_{t = 0} &= u_0  \quad &\hbox{in }\ & \IR^n_+.
  \end{aligned}\right.\end{align}
  The situation we have to consider is 
  $$\begin{aligned} f&=\divergence_u\IT_u(u,P)-\divergence \IT(u,P),\\
g&=\divergence u-\divergence_uu,\\
h&= \IT (u , P) \e_n-  \IT_u (u , P) \overline{n}_u.\end{aligned}$$
Even for $g=0$ and $h=0$, and adopting  the classical maximal regularity framework for the Stokes system, the only available 
global-in-time estimates for $u$ are formulated in spaces of type $\LL_r(\IR_+;X)$ with $r>1$
and $X$ being, e.g., a Lebesgue or Sobolev potential space.  
In the  half-space case, since both the domain and the boundary are unbounded, 
there is no control on `low frequencies' whatsoever, and thus no hope to get enough time decay 
for the Stokes semigroup to deduce a $\LL_1(\IR_+;X)$  control from $\LL_r(\IR_+;X)$ bounds with $r>1.$ 
 
Our strategy  to achieve a global-in-time $\LL_1$ control  is 
based on the Da Prato -- Grisvard theorem presented in Chapter~\ref{Sec: Da Prato--Grisvard theorem}.  
Of course,  those estimates will have to be extended to the case  where  the divergence and the trace of
the normal component of the stress tensor in~\eqref{Eq: Stokes general} are nonzero, and the fine structure of  $g_t:=\partial_tg$ and $h_t:=\partial_th$  will play an important role.

As regards the choice of a functional space $X,$ our previous analysis pushes  us to choose a  Besov space 
$\dot\B^s_{p,1,\sigma}(\IR_+^n)$ with  $s\in(0 , 1 )$ and $s \leq n / p$.  
Hence, we  expect $D^2u$ to be in  $\LL_1(\IR_+;\dot\B^s_{p,1}(\IR_+^n)),$
or (almost) equivalently $Du$  to be in  $\LL_1(\IR_+;\dot\B^{s+1}_{p,1}(\IR_+^n)),$ and  the only way 
of having eventually a control of  $Du$ in $\LL_1(\IR_+;\LL_\infty(\IR_+^n))$ is thus to assume 
that $s=n/p-1$ for some $1<p<\infty,$  owing to the (critical) embedding $\dot\B^{n/p}_{p,1}(\IR_+^n)\hookrightarrow \LL_\infty(\IR_+^n).$  
In a nutshell, our candidate for a suitable space for the initial velocity is a  homogeneous Besov space
of type $\dot\B^{n/p-1}_{p,1}(\IR^n_+)$ with $1<p<\infty.$


In Lagrangian variables, the global existence result to~\eqref{eq:FBNSLag}, that is proven in this chapter reads as follows.

\begin{theorem}  \label{Thm:NS3D}
 Let $v_0$ be in $\dot\B^{n/p-1}_{p,1,\sigma}(\IR^n_+)$ for some $p\in(n-1,n).$  
 There exists $c>0$ such that if 
 \begin{equation}
  \|v_0\|_{\dot \B^{n/p-1}_{p,1}(\IR^n_+)}\leq c,
 \end{equation}
then  system~\eqref{eq:FBNSLag} 
has a unique global solution $(u , P)$ such that
\begin{align*}
 (u,P) \in \cC_b ([0 , \infty) ; \dot \B^{n/p - 1}_{p , 1 , \sigma} (\IR^n_+)) \times \LL_1 (0 , \infty ; \dot \B^{n / p}_{p , 1} (\IR^n_+))
\end{align*} 
 satisfying 
$$\sup_{t\geq0} \|u(t)\|_{\dot \B^{n/p-1}_{p,1}(\IR^n_+)}<\infty\andf
\int_0^\infty \Bigl( \|\nabla u(t)\|_{\dot \B^{n/p}_{p,1}(\IR^n_+)} + \| \partial_t u , \nabla^2 u , \nabla P \|_{\dot \B^{n/p- 1}_{p , 1} (\IR^n_+)} \Bigr) \,\d t<\infty.$$
\end{theorem}

Finally, the product law proven in~\cite[Lem.~A.5]{Danchin14} together with Proposition~\ref{Prop: Change of half} and Remark~\ref{Rem: Change of half} allow to return to Eulerian coordinates, so that one obtains the global existence result stated in Theorem~\ref{Thm:NS3D, intro}.

\begin{remark}
Some comments about the above answer to~\eqref{eq:FBNS} are in order.  
 First, since 
 \begin{equation}\label{eq:key-embedding}
 \dot\B^{n/p}_{p,1}(\IR_+^n)\hookrightarrow \cC_b(\IR_+^n),\end{equation}
 the  flow  $X$ is $\cC^1$ and the $\cC^1$ regularity of the interface is  thus preserved.
 Second,~\eqref{eq:key-embedding} and the smallness condition on $v_0$ imply that the  free boundary is  a perturbation of a plane, namely
 \begin{equation}
\sup_{t > 0}  \|\overline{n} - (0,0,-1)\|_{\LL_{\infty}(\partial\Omega_t)}\leq Cc.
 \end{equation}
However, the motion of the domain may be unbounded  in time: whether 
 $\int_0^t \|u(t^\prime)\|_{\LL_{\infty}(\partial\IR_+^n)} \; \d t^\prime$  is uniformly bounded in time if no additional hypotheses on the data are assumed, remains  unclear.
\end{remark}
The existence part  of the above theorem will be obtained 
by a fixed point argument in a suitable `solution space' $\Xi_p$  that will be modelled from our analysis of
the nonhomogeneous Stokes system~\eqref{Eq: Stokes general}.
More precisely,  a time-dependent vector field $v$
being given,  we shall exhibit a fixed point of the map 
 $(v,Q)\mapsto (u,P)$,   where 
 $(u,P)$ stands for the solution to the following 
 Stokes-like system with \emph{variable coefficients}:
 \begin{align}
\label{Eq: Lin3}
 \left\{\begin{aligned}
 \partial_t u -\divergence_{v} \IT_{v} (u , P) &=0 \quad &\hbox{in }\ & \IR_+\times\IR^n_+,  \\
 \divergence_v u &=0 \quad &\hbox{in }\ & \IR_+\times\IR^n_+,  \\
 \IT_{v} (u , P) \overline{n}_{v}& =0 \quad &\hbox{on }\ & \IR_+\times\partial\IR^n_+, \\
 u|_{t = 0} &= v_0 \quad &\hbox{in }\ & \IR^n_+.
\end{aligned}\right.
\end{align} 
The question now is to produce a solution to~\eqref{Eq: Lin3} in the space $\Xi_p.$ 
The  difficulty is that  the differential 
operators $\divergence_v$ and $\IT_v$ have \emph{variable} and rough coefficients.
However, at time $0,$  they coincide with  $\divergence$ and $\IT,$ respectively, 
so that one may expect  to solve the system for all time if $v$ is  small enough. More precisely, 
since $$DX_v(t)-\Id =\int_0^t Dv(t',y)\,\d t',$$
one can use the following Neumann expansion to compute $A_v$:
\begin{equation}\label{eq:Neumann-expansion}A_v(t,y)= \sum_{k=0}^\infty (-1)^k\biggl(\int_0^t Dv(t',y)\,\d 
 t'\biggr)^k,
 \end{equation}
provided the integral in the sum  is small enough. This will be the case for all $t\geq0,$ 
if $v$ satisfies  
$$\int_0^\infty \| \nabla v\|_{\dot \B^{n / p}_{p , 1} (\IR^n_+)} \, \d t^{\prime} \ll 1.$$

Now, in order to solve system~\eqref{Eq: Lin3}, we shall apply the contraction mapping argument 
in the solution space  $\Xi_p$   to 
 the map $(w,\vartheta)\to (u,P),$  where $(u,P)$ is a solution to the system 
\begin{align}
\label{Eq: Linearized2}
\left\{\begin{aligned}
 \partial_t u - \divergence \IT (u , P) &= f_0+\divergence_{v} \IT_{v} (w , \vartheta) - \divergence \IT (w , \vartheta) \quad &\hbox{in }\ & \IR_+\times\IR^n_+ \\
 \divergence u &= g_0-\divergence_v w + \divergence  w \quad &\hbox{in }\ & \IR_+\times\IR^n_+, \\
  \IT (u , P) \e_n &= h_0+ \IT (w, \vartheta) \e_n - \IT_{v} (w , \vartheta) \overline{n}_{v}  \quad &\hbox{on }\ & \IR_+\times\partial\IR^n_+,\\
 u|_{t = 0} &= v_0 \quad &\hbox{in }\ & \IR^n_+,
\end{aligned}\right.
\end{align}
where $v_0,$ $f_0,$ $g_0$ and $h_0$ are given (and satisfy suitable structure properties 
that we shall specify later on).
\medbreak
The two main ingredients will be quadratic type estimates in Besov spaces
to handle the right-hand sides, 
and a refined analysis of the nonhomogeneous Stokes system~\eqref{Eq: Stokes general}
that takes into account the particular structure of the right-hand side. 
In order to find precise conditions, we shall first carry out the analysis of the right-hand side of~\eqref{Eq: Linearized2}
and prove estimates, then go to the linear analysis and finally perform the two contraction mapping arguments  and prove uniqueness.


\section{Nonlinear estimates}\label{ss:NL}

Let us denote by $f,$ $g$ and $h$ the right-hand sides of the first three equations of~\eqref{Eq: Linearized2}. According to~\eqref{eq:lagop}, we have
$$\begin{aligned} f
&=f_0+\divergence\bigl((A_vA_v^\top-\Id)\nabla w+(A_v\cdot Dw\cdot A_v-Dw)\bigr)+(\Id-A_v^\top)\cdot\nabla\vartheta,\\
g&
=g_0+(\Id-A_v^\top):\nabla w=g_0+\divergence\bigl((\Id-A_v)w\bigr),\\
h&
=h_0+\bigl(\bigl((\Id-A_v^\top)\cdot\nabla w+Dw\cdot(\Id-A_v)\bigr)\cdot \e_n\bigr)|_{\partial\IR^n_+}
\!-\!\bigl((A_v^\top\cdot\nabla w+Dw\cdot A_v-\vartheta)\cdot\wt n_v\bigr)|_{\partial\IR^n_+},
\end{aligned}$$
where $\wt n_v:=\overline n_v+\e_n$ stands for the perturbation of the initial outward unit normal vector
at the boundary of the fluid domain.
\medbreak
Our goal here is to analyze  $f,$ $g$ and $h$ in the spaces that we will use for proving our main theorem. 
We assume that $(v,0)$ and $(w,\vartheta)$ belong 
to our  solution space  $\Xi_p,$ that is the set of couples $(z,\theta)$ satisfying:
\begin{itemize}
\item 
$z$ is a time-dependent vector-field  with coefficients in $\cC_b(\IR_+;\dot\B^{n/p-1}_{p,1}(\IR^n_+))$ such that 
$\partial_tz,\, \nabla^2z$ have coefficients in $\LL_1(\IR_+;\dot\B^{n/p-1}_{p,1}(\IR^n_+))$;
\item  $\nabla\theta$ has  coefficients in $\LL_1(\IR_+;\dot\B^{n/p-1}_{p,1}(\IR^n_+))$;
\item $\theta=\theta_I+\theta_B$ with:
\begin{itemize} 
\item $\theta_I$ such that $\theta_I|_{\partial\IR_+^n}\equiv0$ and 
$\nabla\theta_I\in \LL_1(\IR_+;\dot\B^{n/p-1}_{p,1}(\IR^n_+));$
\item  $\theta_B|_{\partial\IR_+^3}$ admitting an extension $E\theta_B$ to $\IR^n_+$ 
with $E\theta_B\in\LL_1(\IR_+;\dot\B^{n/p}_{p,1}(\IR^n_+)),$
$\partial_t(E\theta_B)\in\LL_1(\IR_+;\dot\B^{n/p-2}_{p,1}(\IR^n_+))$
satisfying
$$\partial_t(E\theta_B)=\divergence k^1+ m\divergence k^2+ m'\divergence k'^2$$
with $m\in \LL_\infty(\IR_+;\dot\B^{n/p}_{p,1}(\IR^n_+)),$
$m'\in \LL_1(\IR_+;\dot\B^{n/p}_{p,1}(\IR^n_+)),$
$k^1 \in \LL_1 (\IR_+ ; \dot \B^{n/ p - 1}_{p , 1} (\IR^n_+))$,
$k^2\in \LL_1(\IR_+;\dot\B^{n/p-1}_{p,1}(\IR^n_+))$ and
$k'^2\in \LL_\infty(\IR_+;\dot\B^{n/p-1}_{p,1}(\IR^n_+)).$
\end{itemize}
\end{itemize}
We denote by $\Xi_p(z,\theta)$ the norm of an element $(z,\theta)$ of $\Xi_p,$ namely
\begin{multline}\label{def:Xi_p}
\Xi_p(z,\theta):=\|z\|_{\LL_\infty(\IR_+;\dot\B^{n/p-1}_{p,1}(\IR^n_+))}
+\|z_t,\nabla^2z,\nabla\theta_I,\nabla\theta_B\|_{\LL_1(\IR_+;\dot\B^{n/p-1}_{p,1}(\IR^n_+))}\\
+\inf\bigl(\|E\theta_B\|_{\LL_1(\IR_+;\dot\B^{n/p}_{p,1}(\IR^n_+))}
+\|\partial_t(E\theta_B)\|_{\LL_1(\IR_+;\dot\B^{n/p-2}_{p,1}(\IR^n_+))}\bigr)\cdotp
\end{multline}
The infimum is taken on the set of  all extensions having the correct regularity and structure. 
Let us emphasize that, owing to $p<n,$  the 
products in the expansion  of $\partial_t(E\theta_B)$  are  well defined  (see the product laws recalled in Proposition~\ref{prod-Bes}).
\medbreak
The rest of this subsection is devoted to proving estimates
for $f,$ $g$ and $h,$ and to checking that, indeed, $h$ has an extension $Eh$ having the structure
of  $\theta_B$ in the definition of $\Xi_p.$ 
\medbreak
We shall  assume throughout that the given velocity field $v\in\cC_b(\IR_+;\dot\B^{n/p-1}_{p,1}(\IR_+^n))$ 
is such that $X_v$ is measure preserving and 
 \begin{equation}\label{eq:smallD2v} \int_0^t\!\|\nabla  v \|_{\dot \B^{n / p }_{p , 1} (\IR^n_+)}\d\tau
 \leq c\ll1,\quad\hbox{for all }\ t\geq0.  \end{equation}
According to the  expansion~\eqref{eq:Neumann-expansion} and the fact that 
 $\dot \B^{n / p}_{p , 1} (\IR^n_+)$ is an algebra,   this implies that 
   \begin{equation}\label{eq:key}
\sup_{t'\in[0,t]}\Bigl(\|A_v(t')-\Id\|_{\dot\B^{n/p}_{p,1}(\IR_+^n)}
+\|A_vA_v^\top (t')-\Id\|_{\dot\B^{n/p}_{p,1}(\IR_+^n)}\Bigr) \leq 
 C \int_0^t\!\| \nabla v \|_{\dot \B^{n/ p}_{p , 1} (\IR^n_+)}\d\tau.\end{equation}

\noindent
\subsection{Estimate of the right-hand side of~$(\ref{Eq: Linearized2})_1$} 
 \begin{lemma}\label{Lem: Control of right-hand side momentum}  If~\eqref{eq:smallD2v} is fulfilled and $1<p<2n,$ then 
  \begin{multline}\label{Eq: Complete estimate of perturbed right-hand side}
   \|\divergence_{v} \IT_{v} (w , \vartheta) - \divergence \IT (w , \vartheta)\|_{\LL_1(\IR_+ ;\dot \B^{n/p-1}_{p,1}(\IR^n_+))} \\
   \leq     C\|\nabla v\|_{\LL_1 (\IR_+  ; \dot \B^{n / p }_{p , 1} (\IR^n_+))} 
 \bigl( \|\nabla w\|_{\LL_1 (\IR_+  ; \dot \B^{n / p }_{p , 1} (\IR^n_+))} + \| \nabla \vartheta\|_{\LL_1 (\IR_+  ; \dot \B^{n / p - 1}_{p , 1} (\IR^n_+))}\bigr)\cdotp\end{multline}
   \end{lemma}\begin{proof}
 {}From  the algebraic relations~\eqref{eq:lagop} and~\eqref{eq:Tlag}, we gather that
\begin{multline}\label{Eq: First right-hand side}
 \divergence_{v} \IT_{v} (w , \vartheta) - \divergence \IT (w, \vartheta)\\
 =\divergence\bigl((A_vA_v^\top-\Id)\nabla w+(A_v\cdot Dw\cdot A_v-Dw)\bigr)
 +(\Id-A_v^\top)\cdot\nabla\vartheta.\end{multline}
Now, using Inequality~\eqref{eq:key} and the fact that $\dot\B^{n/p}_{p,1}(\IR^n_+)$ is an algebra, we  get
$$  \| (A_vA_v^\top-\Id)\nabla w \|_{\LL_1 (\IR_+; \dot \B^{n /p}_{p , 1} (\IR^n_+))} \lesssim
   \|\nabla v\|_{\LL_1 (\IR_+ ; \dot \B^{n / p}_{p , 1} (\IR^n_+))}\| \nabla  w\|_{\LL_1 (\IR_+ ; \dot \B^{n / p }_{p , 1} (\IR^n_+))}.$$
   Similarly, since 
   $$   A_v\cdot Dw\cdot A_v-Dw =(A_v-\Id)\cdot Dw\cdot A_v + Dw\cdot (A_v-\Id),$$
   we obtain
   $$  \|A_v\cdot Dw\cdot A_v-Dw \|_{\LL_1 (\IR_+; \dot \B^{n / p}_{p , 1} (\IR^n_+))} \lesssim
   \|\nabla v\|_{\LL_1 (\IR_+ ; \dot \B^{n / p }_{p , 1} (\IR^n_+))}\| \nabla  w\|_{\LL_1 (\IR_+ ; \dot \B^{n / p} _{p , 1} (\IR^n_+))}.$$
Finally, applying Inequality~\eqref{eq:prod1} with $s_1=n/p,$ $s_2=n/p-1$ in the case $2\leq p<2n$
(the slightly different case $1<p<2$ can be found in~\cite[Lem.~A.5]{Danchin14}),  we find that 
$$\|(\Id-A_v^\top)\cdot\nabla\vartheta  \|_{\LL_1 (\IR_+; \dot \B^{n / p-1}_{p , 1} (\IR^n_+))} \lesssim
  \|\nabla v\|_{\LL_1 (\IR_+ ; \dot \B^{n / p}_{p , 1} (\IR^n_+))}\| \nabla\vartheta\|_{\LL_1 (\IR_+ ; \dot \B^{n / p - 1}_{p , 1} (\IR^n_+))}.$$
  Putting the above inequalities together yields~\eqref{Eq: Complete estimate of perturbed right-hand side}.
\end{proof}

\subsection{Estimate for  the right-hand side of~$(\ref{Eq: Linearized2})_2$} 

\begin{lemma}\label{l:divW}  Assume that $1<p<2n.$  Then  
$$\| \divergence w - \divergence_{v} w\|_{ \LL_1 (\IR_+ ; \dot \B^{n / p}_{p , 1} (\IR^n_+))} 
\lesssim \|\nabla v\|_{ \LL_1(\IR_+;\dot \B^{n / p}_{p , 1} (\IR^n_+))}\|\nabla w\|_{\LL_1(\IR_+;\dot \B^{n / p}_{p , 1} (\IR^n_+))}$$
and $\partial_t( \divergence w - \divergence_{v} w)=\divergence R_t$ with 
$$\|R_t\|_{ \LL_1 (\IR_+ ; \dot \B^{n / p-1}_{p , 1} (\IR^n_+))} \lesssim  
\| \nabla v\|_{\LL_1(\IR_+;\dot \B^{n / p}_{p , 1} (\IR^n_+))}
\bigl(\|w_t\|_{\LL_1(\IR_+;\dot \B^{n / p - 1}_{p , 1} (\IR^n_+))} 
 + \|w\|_{\LL_\infty(\IR_+;\dot \B^{n / p - 1}_{p , 1} (\IR^n_+))}\bigr)\cdotp$$
\end{lemma}
\begin{proof}  Since $X_v$ is measure preserving, we have the following
relations:
$$ \divergence w - \divergence_{v} w = (\Id - A_v^\top): \nabla w = \divergence ((\Id - A_v) w).$$
Using the first relation,~\eqref{eq:key} and  that 
$\dot\B^{n/p}_{p,1}(\IR_+^n)$ is an algebra readily yields the first part 
of Lemma~\ref{l:divW}. 
\smallbreak
For proving the second part,  we  use the fact that
$$\partial_t( \divergence w - \divergence_{v} w)=\divergence R_t\with
R_t:= \partial_t [  w- A_v w].$$
Now, Leibniz rule implies that 
$$R_t= -\partial_t A_{v}\,w + (\Id - A_v)w_t.$$
Bounding the last term  according to~\eqref{eq:key} and~\eqref{eq:prod1} or~\cite[Lem.~A.5]{Danchin14} gives 
$$\|(\Id - A_v)w_t\|_{\LL_1(\IR_+;\dot \B^{n / p - 1}_{p , 1} (\IR^n_+))} \leq C\| \nabla v\|_{\LL_1(\IR_+;\dot \B^{n / p}_{p , 1} (\IR^n_+))}
 \|w_t\|_{\LL_1(\IR_+;\dot \B^{n / p - 1}_{p , 1} (\IR^n_+))}.
$$
To handle  the first term, we  differentiate~\eqref{eq:Neumann-expansion} with respect to $t$ and get:
\begin{align}
\label{Eq: Time derivative of transformation matrix}
 \partial_t A_v(t)=\sum_{k\geq1} (-1)^k\sum_{j=1}^k  \biggl(\int_0^tDv\,\d t'\biggr)^{k-j}
\cdot (Dv(t))\cdot\biggl(\int_0^tDv\,\d t'\biggr)^{j-1}.
\end{align}
Hence, using once again~\eqref{eq:prod1}, one may conclude that 
$$\|\partial_t A_{v}\,w\|_{\LL_1(\IR_+;\dot \B^{n / p - 1}_{p , 1} (\IR^n_+))} \leq C\| \nabla v\|_{\LL_1(\IR_+;\dot \B^{n / p}_{p , 1} (\IR^n_+))}
 \|w\|_{\LL_\infty(\IR_+;\dot \B^{n / p - 1}_{p , 1} (\IR^n_+))}.$$
This completes the proof. 
\end{proof}

\subsection{The normal vector}
\label{Sec: The normal vector}
\begin{lemma}
\label{Lem: The normal vector}
Let $\wt n_v:= \ov n_v+\e_n.$   We have
 \begin{equation}\label{eq:normal}
  \|\wt n_v\|_{\LL_\infty(\IR_+ ;\dot \B^{(n-1)/p}_{p,1}(\partial\IR_+^n))} \leq C\, 
  \|\nabla v\|_{ \LL_1(\IR_+;\dot \B^{n/ p}_{p , 1} (\IR^n_+))}.
 \end{equation}
\end{lemma} 
\begin{proof}
To avoid technicalities, we focus on the case 
$n=3,$ just indicating at the end how to modify the proof in higher dimension. 

Now, in the 3D case,  let us note that at time  $t=0$, 
the normal vector is $-e_3$ since the fluid 
domain is just $\IR_+^3.$
Denote by $y'$ the first two components of a generic element  $y$ of $\IR_+^3.$ 
By the definition of the Lagrangian coordinates, 
we have the following  parameterization of the free boundary:
$$ X_v(t, y' , 0) = \begin{pmatrix}  y' \\ 0 \end{pmatrix} + \int_0^t v (\tau, y' , 0) \; \d \tau.$$
Hence, the  tangent plane to the boundary is generated by 
$$ \partial_{y_1} X_v(t, y' , 0) = \left(\begin{smallmatrix} 1\\0 \\ 0 \end{smallmatrix}\right) + \int_0^t \partial_{y_1} v (\tau, y', 0 ) \; \d \tau
 \andf  \partial_{y_2} X_v(t, y' , 0) = \left(\begin{smallmatrix} 0\\1 \\ 0 \end{smallmatrix}\right) + \int_0^t \partial_{y_2} v (\tau, y', 0 ) \; \d \tau.$$
As $\nabla v$ is small in  $\LL_{1} (\IR_+  ; \dot \B^{3 / p}_{p , 1} (\IR_+^3))$ and $\dot \B^{2 / p}_{p , 1} (\partial\IR_+^3)\hookrightarrow
 \LL_{\infty} (\partial\IR_+^3),$    the outward unit normal is given by
$$ \overline{n}_{v}(t,y',0) := \frac{ (\partial_{y_2} X_v(t, y' , 0))\times (\partial_{y_1} X_v(t, y' , 0))}{| (\partial_{y_2} X_v(t, y' , 0))\times (\partial_{y_1} X_v(t, y' , 0)) |}\cdotp$$
 Denoting $ \displaystyle V_i(t,y',0):= \int_0^t\partial_{y_i}v(\tau,y',0)\,\d\tau$ for $i=1,2,$ we thus have (omitting the dependence in $t$):
 $$\overline{n}_v = \frac{ V-\e_3 }{| V-\e_3|}\with  V:= V_2\times \e_1+ \e_2\times V_1 +V_2\times V_1.    $$
 Since   $\dot \B^{2/p}_{p,1}(\partial\IR_+^3)$ is an algebra and the vector field $V$ is small in  $\dot \B^{2/p}_{p,1}(\partial\IR_+^3)$, we compute $\overline{n}_v$ by expanding the term $(1 + x)^{-1 / 2}$ in the right-hand side below into its Taylor series:
 $$ \overline{n}_v = (V-\e_3)(1+|V|^2 -2V\cdot \e_3)^{-1/2}.$$
 Using again that  $\dot \B^{2/p}_{p,1}(\partial\IR_+^3)$ is an algebra completes the proof. 
 \end{proof}

At this stage, the fundamental observation is that, setting 
 $$V_i(t,y):= \int_0^t\partial_{y_i}v(\tau,y)\,\d\tau  \mbox{ \  for  \ } i=1,2, \quad t \geq0
 \mbox{ and  } y\in\IR^3_+,
 $$
 and finally  
 $$E\bar n_v:= \frac{V-\e_3}{|V-\e_3|}\andf E\wt n_v:= E\bar n_v+\e_3,$$ ensures that 
$E\wt n_v$ is an extension of $\wt n_v$ on $\IR_+^3$ that 
belongs to $\LL_\infty(\IR_+;\dot\B^{3/p}_{p,1}(\IR_+^3))$ and satisfies
 \begin{equation}\label{eq:normalbis}
  \|E\wt n_v\|_{\LL_\infty(\IR_+ ;\dot \B^{3/p}_{p,1}(\IR_+^3))} \leq C\, 
  \|\nabla v\|_{ \LL_1(\IR_+;\dot \B^{3 / p}_{p , 1} (\IR^3_+))}.
 \end{equation}
 Finally,  expanding $E\wt n_v$ and  taking the time derivative, we 
prove that $\partial_t(E\wt n_v)$ is in $\LL_1(\IR_+;\dot\B^{3/p}_{p,1}(\IR_+^3))$ and 
satisfies 
\begin{equation}\label{eq:normalt}
\|\partial_t(E\wt n_v)\|_{\LL_1(\IR_+;\dot\B^{3/p}_{p,1}(\IR_+^3))}\lesssim 
\|\nabla v\|_{\LL_1(\IR_+;\dot\B^{3/p}_{p,1}(\IR^3_+))}.
\end{equation}
Let us briefly explain how to adapt the proof 
to higher dimension $n.$ 
Then, the outward unit normal vector may 
be expressed in terms of a cross product
of $n-1$ vectors as follows:
$$ \overline{n}_{v}(t,y',0) := -
\frac{ (\partial_{y_1} X_v(t, y' , 0))\times\dotsm\times (\partial_{y_{n-1}} X_v(t, y' , 0))}
{|(\partial_{y_1} X_v(t, y' , 0))\times\dotsm\times (\partial_{y_{n-1}} X_v(t, y' , 0))|}\cdotp$$
Denoting $$V_i(t,y):= \int_0^t\partial_{y_i}v(\tau,y)\,\d\tau  \mbox{ \  for  \ } i=1,\cdots,n-1, \quad t \geq0
 \mbox{ and  } y\in\IR^n_+,
 $$
we extend $ \overline{n}_{v}(t,y',0)$  on $\IR^n_+$ 
by setting
$$
E\bar n_v:= -\frac{(e_1+V_1)\times\dotsm\times(e_{n-1}+V_{n-1}))}
{|(e_1+V_1)\times\dotsm\times(e_{n-1}+V_{n-1}))|}\cdotp$$
By  using Taylor series expansion, it 
is easy to complete the proof of the lemma, as 
in the three-dimensional case.

\subsection{Estimate for  the right-hand side of~$(\ref{Eq: Linearized2})_3$} 
The particular structure of that term (denoted by $h$ in what follows) 
plays a fundamental role.  Recall that, by 
assumption, $\vartheta=\vartheta_I+\vartheta_B$ with $\vartheta_I|_{\partial\IR_+^n}\equiv0.$
Hence $h$ admits the following extension $Eh$ on  $\IR_+\times\IR^n_+$:   
$$Eh:= \bigl((\Id-A_v^\top)\cdot\nabla w+Dw\cdot(\Id-A_v)\bigr)\cdot \e_n
-\bigl((A_v^\top\cdot\nabla w+Dw\cdot A_v-E\vartheta_B \Id)\bigr)\cdot E\wt n_v,$$
and the time derivative of $Eh$  reads 
 $Eh_t^1-Eh_t^2-Eh_t^3$ with
  $$ \begin{aligned}
E h_t^1&:=\bigl((\Id-A_v^\top)\cdot\nabla w_t-A_{v,t}^\top \cdot \nabla w
 +Dw_t\cdot(\Id-A_v)- Dw\cdot A_{v,t}\bigr)\cdot \e_3,\\
 Eh_t^2&:= \bigl(A_{v,t}^\top\!\cdot\!\nabla w+ A_v^\top\!\cdot\!\nabla w_t
  +Dw_t\cdot A_v+Dw\cdot A_{v,t} -\partial_t(E\vartheta_B)\Id\bigr)\cdot E\wt n_v,\\
 Eh_t^3&:=\bigl(A_v^\top\cdot\nabla w+Dw\cdot A_v-E\vartheta_B\Id\bigr) \cdot \partial_t(E\wt n_{v}).
  \end{aligned}$$
 Taking advantage of the previous subsection and, 
once more, of~\eqref{eq:key} and of the fact that
$\dot\B^{n/p}_{p,1}(\IR_+^n)$ is an algebra, we
see that $Eh$ belongs to  $\LL_1(\IR_+;\dot\B^{n/p}_{p,1}(\IR_+^n))$ and satisfies:
\begin{equation}\label{eq:Eh}
\|Eh\|_{\LL_1(\IR_+;\dot\B^{n/p}_{p,1}(\IR_+^n))} \lesssim
 \|\nabla v\|_{ \LL_1(\IR_+;\dot \B^{n / p}_{p , 1} (\IR^n_+))}
 \|\nabla w,E\vartheta_B\|_{\LL_1(\IR_+;\dot \B^{n / p}_{p , 1} (\IR^n_+))}.
\end{equation}
 
The components of the four terms constituting $E h_t^1$ are linear combinations of products of type  $m\divergence k$ with (use~\eqref{eq:prod1},~\eqref{eq:key} and~\eqref{Eq: Time derivative of transformation matrix}) either 
$$\begin{aligned}&\|m\|_{\LL_\infty(\IR_+;\dot\B^{n/p}_{p,1}(\IR^n_+))}\!\lesssim\! 
\|\nabla v\|_{\LL_1(\IR_+;\dot\B^{n/p}_{p,1}(\IR^n_+))}
\!\!\andf\!  \|k\|_{\LL_1(\IR_+;\dot\B^{n/p-1}_{p,1}(\IR^n_+))}\!\lesssim\! 
\|w_t\|_{\LL_1(\IR_+;\dot\B^{n/p-1}_{p,1}(\IR^n_+))},\\
\hbox{or }\  \!&\|m\|_{\LL_1(\IR_+;\dot\B^{n/p}_{p,1}(\IR^n_+))}\!\lesssim\! 
\|\nabla v\|_{\LL_1(\IR_+;\dot\B^{n/p}_{p,1}(\IR^n_+))}
\!\andf\!  \|k\|_{\LL_\infty(\IR_+;\dot\B^{n/p-1}_{p,1}(\IR^n_+))}\!\lesssim\! 
\|w\|_{\LL_\infty(\IR_+;\dot\B^{n/p-1}_{p,1}(\IR^n_+))}.\end{aligned}$$
Next, owing to~\eqref{eq:normalbis}, the first four terms of  $E h_t^2$ are similar to those of $E h_t^1.$
The last term also has a structure of type $m\divergence k^2 + m'\divergence k'^2$
since we made that assumption for $\partial_t(E\vartheta_B),$ and we can keep the same definition for $k^2$
and $k'^2$ and just multiply $m$ and $m'$ by $E\wt n_v.$ 
Hence, the new $m$ and $m'$ satisfy  the original  bounds, multiplied by $\|\nabla v\|_{\LL_1(\IR_+;\dot\B^{n/p}_{p,1}(\IR^n_+))}.$  
Finally, combining~\eqref{eq:normalt},~\eqref{eq:key} and the fact that $\dot\B^{n/p}_{p,1}(\IR_+^n)$ is an algebra ensures that the first  two terms of $E h_t^3$ are  like  $Eh_t^1.$
As for the last term, it looks like $m\divergence k$ with 
$$
\|m\|_{\LL_1(\IR_+;\dot\B^{n/p}_{p,1}(\IR^n_+))}\lesssim \|E\vartheta_B\|_{\LL_1(\IR_+;\dot\B^{n/p}_{p,1}(\IR^n_+))}\andf
\|k\|_{\LL_\infty(\IR_+;\dot\B^{n/p-1}_{p,1}(\IR^n_+))} \lesssim 
\|v\|_{\LL_\infty(\IR_+;\dot\B^{n/p-1}_{p,1}(\IR^n_+))}.$$

Note that we have  a direct control of $\nabla^2v$ in $\LL_1(\IR_+;\dot\B^{n/p-1}_{p,1}(\IR^n_+))$
from $\Xi_p(v,0).$ 
In  the above  computations however, this is 
the norm of $\nabla v$ in $\LL_1(\IR_+;\dot\B^{n/p}_{p,1}(\IR^n_+))$
that came  into play.  The following lemma guarantees that this latter norm is  bounded  by $\Xi_p(v,0).$ 
\begin{lemma}\label{l:nablau}
Let $2 < p < \infty$ and $z$ be a function defined on $\IR_+ $ such that for almost every $t > 0$, $z(t)$ is the restriction of an element in 
$\cS_h^{\prime} (\IR^n)$ with  space derivatives of second order in $\LL_1 (\IR_+ ; \dot \B^{n / p - 1}_{p , 1} (\IR^n_+))$. Then, 
there exists some constant $C>0$ such that
$$\int_0^t \| \nabla z\|_{\dot \B^{n / p}_{p , 1} (\IR^n_+)}  \, \d \tau\leq C \int_0^t\| \nabla^2 z \|_{\dot \B^{n / p - 1}_{p , 1} (\IR^n_+)}\,\d\tau
\quad\hbox{for all }\ t\geq0.$$ 
\end{lemma}
\begin{proof} It stems from  Corollary~\ref{Cor: Curl free vector fields} with $s=n/p-1$ and $q=1.$
 \end{proof}
 \smallbreak


\section{Study of  the nonhomogeneous Stokes system} 

This subsection is devoted to solving~\eqref{Eq: Stokes general}  
supplemented with data $u_0,$ $f,$ $g$ and $h$ such  that:
\begin{itemize}
\item  $u_0\in \dot \B^{s}_{p,1}(\IR_+^n)$; 
\item $f\in \LL_1(\IR_+; \dot \B^{s}_{p,1}(\IR_+^n))$;
\item $g\in \LL_1(\IR_+; \dot \B^{s+1}_{p,1}(\IR_+^n))$ with $g|_{t=0}=0$
and $g=\divergence R$ with $R_t\in  \LL_1(\IR_+; \dot \B^{s}_{p,1}(\IR_+^n))$;
\item $h|_{t=0}=0$ and  $h$ admits an extension $Eh$ on $\IR_+\times\IR^n_+$ 
belonging to $\LL_1(\IR_+; \dot \B^{s+1}_{p,1}(\IR_+^n))$
and such that $\partial_t(Eh)\in \LL_1(\IR_+; \dot \B^{s-1}_{p,1}(\IR_+^n))$
with:
\begin{itemize}
\item Case $n=2$: 
\begin{equation}\label{eq:decompo-n2}
\partial_t(Eh)= \divergence k,
\with k\in \LL_1(\IR_+; \dot \B^{s}_{p,1}(\IR_+^2)).
\end{equation}
\item Case $n\geq 3$:
 \begin{equation}\label{eq:decompo}
 \partial_t(Eh)= \divergence k^1+m\divergence k^2+m'\divergence k'^2,
\end{equation}
with  $m \in \LL_\infty(\IR_+; \dot \B^{n/p}_{p,1}(\IR_+^n)),$ 
$m' \!\in\! \LL_1(\IR_+; \dot \B^{n/p}_{p,1}(\IR_+^n)),$ 
$k^1,k^2\in \LL_1(\IR_+; \dot \B^{s}_{p,1}(\IR_+^n))$ 
and $k'^2 \in \LL_\infty(\IR_+; \dot \B^{s}_{p,1}(\IR_+^n)).$
\end{itemize}
\end{itemize}
The much stronger assumption in the case $n=2$ comes from the fact that 
the restriction $s\leq n/p-1$ that will come from, e.g., the first step of the proof, will preclude
us to have also   $s>1-\min(n/p,n/p'),$ an \emph{optimal}
condition that is required when considering the products in the   right-hand side of~\eqref{eq:decompo}
(see Proposition~\ref{prod-Bes}). This is actually the only reason why, so far,  our method does not enable  us
to solve our free boundary problem in the case $n=2.$  

\begin{proposition}\label{p:fight}  Assume that $n\geq2.$ Let $n-1<p<\infty$ and $s\in(0,1/p)$  
satisfy $s>1-\min(n/p,n/p')$ and $s\leq n/p -1.$ 
Then, system~\eqref{Eq: Stokes general} admits a unique solution $(u,P)$ 
with  
\begin{equation}\label{eq:regularity}
u\in\cC_b(\IR_+; \dot \B^{s}_{p,1}(\IR_+^n))\andf u_t,\nabla^2u,\nabla P\in \LL_1(\IR_+; \dot \B^{s}_{p,1}(\IR_+^n)),\end{equation}
with  $P=P_I+P_B$ such that 
 $P_I|_{\partial\IR_+^n}\equiv0,$ $\nabla P_I\in \LL_1(\IR_+; \dot \B^{s}_{p,1}(\IR_+^n))$
and  $P_B|_{\partial\IR_+^n}$ has  an extension $EP_B$ on $\IR_+^n$ 
that belongs to   $\LL_1(\IR_+; \dot \B^{s+1}_{p,1}(\IR_+^n))$ 
and such that $\partial_t(EP_B)\in \LL_1(\IR_+; \dot \B^{s-1}_{p,1}(\IR_+^n))$ 
and satisfies
$$\partial_t(EP_B)= k^1_B+m_B\divergence k^2_B+m'_B\divergence k'^2_B$$
for some distributions $k^1_B,$ $k^2_B,$ $k'^2_B,$ $m_B$ and $m'_B$
as in~\eqref{eq:decompo} in the case $n\geq3,$ and~\eqref{eq:decompo-n2} in the case $n=2$.
\medbreak
 Furthermore, the extension operator $E$ may be chosen linear and continuous
 with respect to $f,$ $R$ and the functions in $Eh,$ and the  
 following inequality is satisfied:
\begin{multline}\label{eq:Stokes general}
\|u\|_{\LL_\infty(\IR_+; \dot \B^{s}_{p,1}(\IR_+^n))}+\|u_t,\nabla^2u,\nabla P_I,\nabla P_B\|_{ \LL_1(\IR_+; \dot \B^{s}_{p,1}(\IR_+^n))}+\|EP_B\|_{ \LL_1(\IR_+; \dot \B^{s+1}_{p,1}(\IR_+^n))}\\
+\|\partial_t(EP_B)\|_{ \LL_1(\IR_+; \dot \B^{s-1}_{p,1}(\IR_+^n))}
\lesssim \|u_0\|_{ \dot \B^{s}_{p,1}(\IR_+^n)} 
\\+ \|f,R_t,\nabla Eh\|_{ \LL_1(\IR_+; \dot \B^{s}_{p,1}(\IR_+^n))}
+\|g\|_{ \LL_1(\IR_+; \dot \B^{s+1}_{p,1}(\IR_+^n))}
+\|\partial_t(Eh)\|_{ \LL_1(\IR_+; \dot \B^{s-1}_{p,1}(\IR_+^n))}
.\end{multline}
\end{proposition}

\begin{proof} 
 Solving the system  requires the following  four steps:
\begin{enumerate}
\item[1.] removing the potential part of the velocity (second line of~\eqref{Eq: Stokes general});
\item[2.] removing the boundary term  (third line of~\eqref{Eq: Stokes general});
\item[3.] determining the  pressure corresponding
to the potential part of the source term of the first line of~\eqref{Eq: Stokes general}
(after modification according to the first two steps);
\item[4.] solving the (homogeneous) Stokes system according to the Da Prato -- Grisvard theory, 
and checking that the corresponding pressure indeed has the desired structure.
\end{enumerate}

\subsection*{Step 1: Removing the potential part of the velocity}
We claim that there exists  a vector-field $W$ such that
\begin{equation}\label{eq:Wt} W_t  \in \LL_1(\IR_+ ;\dot \B^{s}_{p,1}(\IR^n_+))\andf \nabla W \in \LL_1(\IR_+ ;\dot \B^{s+1}_{p,1}(\IR^n_+)),\end{equation}
satisfying 
$$  \divergence W = g = \divergence(R) \quad\hbox{in}\quad \IR_+\times\IR^n_+$$ and 
\begin{equation}\label{eq:Wtt} \|W_t\|_{\LL_1(\IR_+;\dot \B^{s}_{p,1}(\IR^n_+))} + \|\nabla W\|_{\LL_1(\IR_+ ;\dot \B^{s+1}_{p,1}(\IR^n_+))} \lesssim   \|g\|_{ \LL_1(\IR_+; \dot \B^{s+1}_{p,1}(\IR_+^n))}
+ \|R_t\|_{ \LL_1(\IR_+; \dot \B^{s}_{p,1}(\IR_+^n))}.\end{equation}

For fixed $t>0,$  the idea is to set   $W=\nabla\Psi|_{\IR^n_+}$ with $\Psi$ a solution of
\begin{equation}\label{eq:Psi} \Delta\Psi= E'g=E'\divergence R\quad\hbox{in }\ \IR_+\times\IR^n,\end{equation}
where $E'$ stands for  a suitable linear extension operator from $\IR^n_+$ to $\IR^n.$
\medbreak
Clearly, we have
$$\Delta\Psi_t= E'\divergence  R_t.$$
The difficulty   is that we want $E'$ to map both
$\dot\B^{s+1}_{p,1}(\IR^n_+)$ to $\dot\B^{s+1}_{p,1}(\IR^n)$ and 
$\dot\B^{s-1}_{p,1}(\IR^n_+)$ to $\dot\B^{s-1}_{p,1}(\IR^n)$
while $s$ is close to $0$ (namely, in the range $(-1+1/p,1/p)$). 
To this end, we take $m\in\IN$ large enough and consider the extension operator $E_{\sigma}$ which was constructed in Lemma~\ref{Lem: Extension operators} and was defined by
\begin{align*}
 [E_{\sigma} k] (x) := \bigg( - \sum_{j = 0}^m \frac{\beta^{\prime}_j}{j + 1} k^{\prime} \Big( x^{\prime} , - \frac{x_n}{j + 1} \Big) , \sum_{j = 0}^m \beta^{\prime}_j k_n \Big(x^{\prime} , - \frac{x_n}{j + 1} \Big) \bigg) \qquad (x \in \IR^n \setminus \IR^n_+),
\end{align*}
where the numbers $\beta_j^{\prime}$ satisfy
\begin{align}
\label{Eq: Vandermonde condition for betas}
 \sum_{j = 0}^m \Big( - \frac{1}{j + 1} \Big)^{\iota} \beta_j^{\prime} = 1
 \quad\hbox{for }\  \iota = 0 , \cdots , m. 
\end{align}
Taking the divergence of $E_{\sigma} k$ on the lower half-space, we obtain
\begin{align}
\label{Eq: Commutator of solenoidal extension}
 \divergence(E_{\sigma} k) = - \sum_{j = 0}^m \frac{\beta_j^{\prime}}{j + 1} \divergence k \Big( x^{\prime} , - \frac{x_n}{j + 1} \Big) =: E^{\prime} \divergence k.
\end{align}
If we let $E^{\prime}$ and $E_\sigma$ act on the upper half-space as the identity, then both $E^{\prime}$ and $E_\sigma$ define extension operators that, owing to~\eqref{Eq: Vandermonde condition for betas},
are bounded from  $\dot \B^a_{p , 1} (\IR^n_+)$ to $\dot \B^a_{p , 1} (\IR^n)$ whenever  $1 < p < \infty$,  
 $-1+1/p < a < m + 1 / p$  and $a\leq n/p$ (one may argue as for proving Proposition~\ref{Prop: Proper boundedness extension operators}). 
 \medbreak
 In order to solve~\eqref{eq:Psi}, let us fix some compactly supported cut-off function $\chi$ with 
 value~$1$ near the origin, and  define $\Psi_N$ for $N\geq1$ by
$$\cF_x\Psi_N(t,\xi):=-\biggl(\frac{1-\chi(2^N\xi)}{|\xi|^2}\biggr)\cF_x(E'\divergence R)(t,\xi).$$
It is obvious that there exists a constant $C$ such that for all $N\geq1,$ we have
 $$ \|\nabla^2 \Psi_N\|_{\LL_1(\IR_+ ; \dot \B^{s+1}_{p , 1} (\IR^n))}\leq C\|E'\divergence R\|_{\LL_1(\IR_+ ; \dot \B^{s+1}_{p , 1} (\IR^n))}\leq  C\|g\|_{\LL_1(\IR_+ ; \dot \B^{s+1}_{p , 1} (\IR^n_+))},$$
 since,  for  $s+1\leq n/p$, we have
 $E^{\prime}\divergence R \in  \LL_1(\IR_+ ; \dot \B^{s+1}_{p , 1} (\IR^n)).$
Furthermore, from the definition of the Fourier transform, one can see that for a.e. $t\in\IR_+$ and $\varphi\in \cS_c(\IR^n),$ we have for large enough $N,$
\begin{equation}\label{eq:Psirel}\int_{\IR^n}\Delta\Psi_N\,\varphi\, \d x = \int_{\IR^n} E'\divergence R\:\varphi\, \d x.\end{equation}
Since  $s+1\leq n/p$ and the property of being in $\cS'_h(\IR^n)$ is preserved 
by the construction, one can conclude that $\nabla^2\Psi_N$ tends to some 
$\nabla^2\Psi$ in $\LL_1(\IR_+ ; \dot \B^{s+1}_{p , 1} (\IR^n)),$ and that we have
a distributional (modulo polynomials of degree $1$) solution $\Psi$ to~\eqref{eq:Psirel} satisfying 
 $$ \|\nabla^2 \Psi\|_{\LL_1(\IR_+ ; \dot \B^{s+1}_{p , 1} (\IR^n))}
 \lesssim  \|g\|_{\LL_1(\IR_+ ; \dot \B^{s+1}_{p , 1} (\IR^n_+))}.$$
Since $W=\nabla\Psi|_{\IR^n_+},$ it is clear that the second half of~\eqref{eq:Wt} and~\eqref{eq:Wtt} is
satisfied. 

Next,  differentiating~\eqref{eq:Psirel} with respect to time
and using~\eqref{Eq: Commutator of solenoidal extension}
ensures in addition that  
$$ \Delta \Psi_t =  E'\divergence R_t= \divergence(E_{\sigma} R_t). $$
Given the definition of $\Psi,$ we have
$$\nabla\Psi_t=-\cF_x^{-1}\biggl(\frac {i\xi}{|\xi|^2}\,\xi\cdot\cF_x(E_\sigma R_t)\biggr)\cdotp$$
As $E_\sigma R_t \in \LL_1 (\IR_+ ; \dot \B^{s}_{p , 1} (\IR^n))$  (use Proposition~\ref{Prop: Proper boundedness extension operators}) and $W_t$ is the restriction of $\nabla\Psi_t$ to the half-space, we 
eventually deduce that $$ \|  W_t \|_{\LL_1 (\IR_+ ; \dot \B^{s}_{p , 1} (\IR^n_+))}\lesssim  
 \| E_{\sigma} R_t \|_{\LL_1 (\IR_+; \dot \B^{s}_{p , 1} (\IR^n))}
 \lesssim   \|R_t \|_{\LL_1 (\IR_+; \dot \B^{s}_{p , 1} (\IR^n_+))}.$$

Now, setting 
\begin{equation}\label{eq:stokesbis}
v:= u-W,\quad \wt f:= f-W_t+\divergence\IT(W,0)\andf
\wt h:=(h-\IT(W,0)\cdot \e_n)|_{\partial\IR^n_+},
\end{equation}
the initial problem reduces to finding $(v,P)$ such that 
$$ \left\{
  \begin{aligned}
   \partial_t v - \divergence \IT (v , P) &= \wt f \quad &\hbox{in }\ & \IR_+\times\IR^n_+, \\
   \divergence v&= 0 \quad &\hbox{in }\ & \IR_+\times\IR^n_+, \\
   \IT (v , P) \cdot \e_n|_{\partial\IR^n_+}&= \wt h \quad &\hbox{on }\ & \IR_+\times\partial\IR^n_+,\\
   v|_{t = 0} &= u_0  \quad &\hbox{in }\ & \IR^n_+.
  \end{aligned}
 \right.$$
Since, by construction,  
$$
\|\IT(W,0)\|_{\LL_1(\IR_+;\dot\B^{s+1}_{p,1}(\IR^n_+))}\lesssim
\|DW\|_{\LL_1(\IR_+;\dot\B^{s+1}_{p,1}(\IR^n_+))}
\lesssim\|D^2\Psi\|_{\LL_1(\IR_+;\dot\B^{s+1}_{p,1}(\IR^n))}
\lesssim \|g\|_{\LL_1(\IR_+;\dot\B^{s+1}_{p,1}(\IR^n_+))},$$
it is natural to extend $\wt h$ by 
$$E\wt h:=Eh-\IT(W,0)\cdot \e_n,$$ 
so that we have
$$\|E\wt h\|_{\LL_1(\IR_+;\dot\B^{s+1}_{p,1}(\IR^n_+))}\lesssim
 \|Eh\|_{\LL_1(\IR_+;\dot\B^{s+1}_{p,1}(\IR^n_+))}
 +\|g\|_{\LL_1(\IR_+;\dot\B^{s+1}_{p,1}(\IR^n_+))},$$
 and  $\partial_t(E\wt h)= \partial_t(Eh) - \IT(W_t,0)\!\cdot\! \e_n$
 is thus in $\LL_1(\IR_+;\dot\B^{s-1}_{p,1}(\IR^n_+))$ and satisfies
 $$\|\partial_t(E\wt h))\|_{\LL_1(\IR_+;\dot\B^{s-1}_{p,1}(\IR^n_+))}
 \leq \|\partial_t(E\wt h)\|_{\LL_1(\IR_+;\dot\B^{s-1}_{p,1}(\IR^n_+))}+
 C\|R_t\|_{\LL_1(\IR_+;\dot\B^{s}_{p,1}(\IR^n_+))}.$$
Clearly, $\partial_t(E\wt h)$ has the required structure, since 
all components of $\IT(W_t,0)\!\cdot\!\e_n$ may be written 
in the form $\divergence k^1$ for some $k^1$ in $\LL_1(\IR_+;\dot\B^{s}_{p,1}(\IR^n_+))$
that depends linearly on  $W_t.$ 
\medbreak
Note that, owing to $g|_{t=0}=0$ and  $h|_{t=0}=0,$ 
the regularities of  $g,$ $h,$ $g_t$ and $h_t$ on $\IR$ are the same as on $\IR_+,$
if extending  $g$ and $h$ by $0$ for $t<0.$ 
The above  construction also ensures that 
 $\wt h(0)=0,$  and  we shall thus extend that function by $0$ for $t<0.$
 Its  regularity (as well as estimates)  will  be conserved  on the whole $\IR.$ 
 We shall keep the same notation for the extension of $\wt h$ to $\IR.$

\subsection*{Step 2: The fight with boundary terms}
 Our aim  is to solve  
 \begin{align}\label{Eq: Equation with right boundary condition}
 \left\{\begin{aligned}
 \divergence w &= 0 \quad &\hbox{in }\   \IR\times\IR^n_+ \\
 \IT (w , \vartheta) \cdot \e_n &= \wt h \quad &\hbox{on }\ \IR\times\partial \IR^n_+.
\end{aligned}
\right.\end{align}
By a direct calculation, we find, denoting $w':=(w_1,\cdots,w_{n-1})$ and
$\nabla':=(\partial_1,\cdots,\partial_{n-1}),$ 
\begin{align*}
  \IT (w , \vartheta) \cdot \e_n = \begin{pmatrix} \partial_n w' + \nabla' w_n \\ 2 \partial_n w_n \end{pmatrix} - \begin{pmatrix} 0 \\ \vartheta \end{pmatrix}\cdotp
\end{align*}
Hence,  the boundary condition translates into 
\begin{equation}\label{eq: BC}
\left\{\begin{aligned}
\partial_nw' +\nabla' w_n  &= \wt h' \quad &\hbox{in }\   \IR\times\partial\IR^n_+ \\
 2\partial_nw_n &=\vartheta+ \wt h_n \quad &\hbox{in }\ \IR\times\partial \IR^n_+.
\end{aligned}\right.\end{equation}
The natural idea is to first construct a suitable divergence free vector-field 
$w$ satisfying the first line, then  to define  $\vartheta$  according to the last line.

\subsubsection*{Substep I}
Let us start with $n=2.$ Then, we  look for 
$w$  in the form of a stream function $w = (- \partial_2 \varphi , \partial_1 \varphi)$ with
\begin{equation}\label{eq:varphi}
 \partial_t \varphi - \Delta \varphi = 0 \quad\hbox{in }\  \IR\times\IR^2_+.
\end{equation}

By construction, $w$ is   divergence free and, denoting by $\Phi$ the trace of $\varphi$ 
on $ \IR\times\partial\IR^2_+,$  the  boundary condition in~\eqref{eq: BC} translates into
\begin{align}
\label{Eq: Boundary equation for stream function}
 \partial_1 \partial_1 \Phi - \partial_2 \partial_2 \Phi = \wt h_1 \quad\hbox{in }\  \IR\times\partial\IR^2_+,
\end{align}
that is to say, remembering~\eqref{eq:varphi}, 
\begin{align}
\label{Eq: Heat equation on the boundary}
 \partial_t \Phi - 2 \partial_1 \partial_1 \Phi =  - \wt h_1 \quad  \hbox{in }\  \IR \times \IR.
\end{align}
In other words, we  take for   $\Phi$ 
the unique solution in $\cC(\IR;  \dot \B^{s+1-1 / p}_{p , 1} (\IR))$ of the above  one-dimensional heat equation, 
and, given the properties of $\wt h_1,$ we have   $\partial_1\partial_1\Phi, \Phi_t \in \LL_1(\IR ; \dot \B^{s+1-1 / p}_{p , 1} (\IR))$  and
 \begin{equation}\label{eq:estPhi}
 \|\partial_1\partial_1\Phi\|_{\LL_1(\IR ; \dot \B^{s+1-1 / p}_{p , 1} (\IR))}+\|\Phi_t\|_{\LL_1(\IR ; \dot \B^{s+1-1 / p}_{p , 1} (\IR))}
\lesssim \|\wt h_1\|_{\LL_1(\IR ; \dot \B^{s+1-1 / p}_{p , 1} (\IR))}.
\end{equation}

In the case $n=3,$ we look for $w$ under the form $w=\nabla\times \varphi$
with $\varphi=(\varphi^1,\varphi^2,0)$ still satisfying~\eqref{eq:varphi}.
  The boundary condition in~\eqref{eq: BC} now translates into
$$\begin{aligned}
\partial_1\partial_1\varphi_2-\partial_3\partial_3\varphi_2-\partial_1\partial_2\varphi_1&=\wt h_1,\\
\partial_3\partial_3\varphi_1-\partial_2\partial_2\varphi_1+\partial_1\partial_2\varphi_2&=\wt h_2,
\end{aligned}$$
and the trace $\Phi$ of $\varphi$ 
on $ \IR\times\partial\IR^3_+$ thus has to satisfy (denoting
$\Delta':=\partial_1\partial_1+\partial_2\partial_2$),
$$\begin{aligned}
\partial_t\Phi_1-\Delta'\Phi_1 -\partial_2\partial_2\Phi_1+\partial_1\partial_2\Phi_2&=\wt h_2,\\
\partial_t\Phi_2-\Delta'\Phi_2-\partial_1\partial_1\Phi_2+\partial_1\partial_2\Phi_1&=-\wt h_1.
\end{aligned}$$
This is a Lam\'e system in $\IR\times\IR^2$  with coefficients $\mu=2$ and $\mu'=-1.$ 
By using the Helmholtz projectors, it may be reduced to two heat equations. 
Hence, as in the case $n=2,$ since   our assumptions and Step 1 guarantee that $\wt h'=(\wt h_1,\wt h_2)$ is in $\LL_1 (\IR ; \dot \B^{s+1-1 / p}_{p , 1} (\partial\IR^3_+)),$  we find that 
 $\nabla'^2\Phi, \Phi_t \in \LL_1(\IR ; \dot \B^{s+1-1 / p}_{p , 1} (\IR^2))$ and
  \begin{equation}\label{eq:estPhi3}
 \|\nabla'^2\Phi\|_{\LL_1(\IR ; \dot \B^{s+1-1 / p}_{p , 1} (\IR^2))}+\|\Phi_t\|_{\LL_1(\IR ; \dot \B^{s+1-1 / p}_{p , 1} (\IR^2))}
\lesssim \|\wt h'\|_{\LL_1(\IR ; \dot \B^{s+1-1 / p}_{p , 1} (\IR^2))}.
\end{equation}
In higher dimension,  we can look for $w$ under the form 
$w=(\partial_n \varphi', -\divergence' \varphi')$ with 
$\varphi'=(\varphi_1,\cdots,\varphi_{n-1}).$ Then, using the boundary conditions, 
we discover that the trace $\Phi'$  of $\varphi'$ 
on $ \IR\times\partial\IR^n_+$  has to satisfy the Lam\'e system
$$\partial_t\Phi'-\Delta'\Phi'-\nabla'\divergence'\Phi' = \wt h'. $$
At this stage,  one may conclude exactly as in the cases $n=2,3.$
\medbreak
 Whatever the dimension is, one can 
easily compute $\varphi$ from $\Phi.$ 
Indeed,  denoting  by  $\,\widehat{~}\,$ or $\cF$
  the Fourier transform in the $t$ and $x^\prime$ directions, 
and by $\xi_0$ and $\xi^\prime$, respectively, the corresponding  Fourier variables, 
we obtain the following linear ODE for $\widehat\varphi$:
$$(i\xi_0+|\xi^\prime|^2)\widehat\varphi-\partial_{x_n}^2\widehat\varphi=0.$$
Hence, prescribing that $\wh\varphi\to0$ for $x_n\to+\infty,$ we obtain
$$ \widehat{\varphi} (\xi_0 , \xi^\prime , x_n) = \wh\Phi(\xi_0 , \xi^\prime) \e^{- r x_n},$$
 where $r$ is chosen such that $r^2 = \ii \xi_0 + |\xi^\prime|^2$ and $\Re r\geq0.$
\medbreak
A systematic study of symbols of such a form is presented in Section~\ref{Sec: A free boundary problem for pressureless gases}. In the following, however, we present a shorter proof exploiting the underlying equations.

\subsubsection*{Substep II}

Consider,  for all $t \in \IR,$  the  harmonic extension 
 $G(t,\cdot) : \IR^n_+ \to \IR$ of  $\Phi$ defined by
\begin{align*}
 G(t,x^\prime , x_n) := \cF^{-1} ( \wh\Phi (\xi_0 , \xi^\prime) \e^{- \lvert \xi^\prime \rvert x_n}) (t , x').
\end{align*}
By construction,   for all $t\in\IR,$  both  $G(t,\cdot)$ and $\partial_t G(t,\cdot)$ are harmonic in $\IR^n_+,$  and  we have
\begin{align}
\label{Eq: Harmonic extension on the boundary}
 G(t,x^\prime , 0) = \Phi (t , x^\prime), \quad \partial_t G(t,x^\prime , 0) = \partial_t \Phi (t , x^\prime).
\end{align}
Since 
$$\begin{aligned}\nabla'^2 G(t,x^\prime , x_n) &= \cF^{-1} ( \wh{\nabla'^2\Phi}(\xi_0 , \xi^\prime) \e^{- \lvert \xi^\prime \rvert x_n}) (t , x')\\
\andf\partial_t G(t,x^\prime , x_n) 
&= \cF^{-1} ( \wh{\partial_t\Phi}(\xi_0 , \xi^\prime) \e^{- \lvert \xi^\prime \rvert x_n}) (t , x'),\end{aligned}$$
  employing~\cite[Lem.~2]{Danchin_Mucha} and the fact that $\partial^2_n G=-\Delta' G$ yields, if $s+1\leq n/p,$
$$\partial_n^2G,\,\nabla'^2G \in \LL_1(\IR ; \dot \B^{ s+1}_{p , 1} (\IR^n_+)) \andf \partial_t G \in \LL_1 (\IR ; \dot \B^{s+1}_{p , 1} (\IR^n_+))$$
with 
\begin{equation}\label{Eq: Regularity of harmonic extension}
 \|\partial_tG,\partial^2_nG,\nabla'^2G\|_{\LL_1(\IR ; \dot \B^{ s+1}_{p , 1} (\IR^n_+))}
 \lesssim \|\wt h'\|_{\LL_1(\IR ; \dot \B^{ s+1-1/p}_{p , 1} (\partial\IR^n_+))}.
\end{equation}

\subsubsection*{Substep III}
We observe that   $\psi:=\varphi-G$ satisfies: 
$$\left\{\begin{array}{rl}
 \partial_t \psi- \Delta\psi &= - \partial_t G \in \LL_1 (\IR ; \dot \B^{s+1}_{p , 1} (\IR^n_+)) \\
 \psi|_{x_n = 0} &= 0.
\end{array}\right.$$
Differentiating  in the horizontal directions preserves the boundary condition and delivers
$$\left\{\begin{array}{rl}
 \partial_t  \nabla'\psi - \Delta \nabla'\psi &= - \partial_t \nabla'G \in \LL_1 (\IR ; \dot \B^{s}_{p , 1} (\IR^n_+)) \\ \nabla'\psi|_{\partial\IR^n_+} &= 0.
\end{array}\right.$$
As a consequence of the result stated in~\cite[Prop.~6]{Danchin_Mucha} and of~\eqref{Eq: Regularity of harmonic extension}, we  have
\begin{align}
\label{Eq: Properties of derivatives of phi 1}
 \partial_t \nabla'  \psi \in \LL_1 (\IR ; \dot \B^{s}_{p , 1} (\IR^n_+))\andf
 \nabla^2\nabla'  \psi \in \LL_1 (\IR ; \dot \B^{s}_{p , 1} (\IR^n_+))
\end{align}
and, since $\varphi=G+\psi,$ Inequality~\eqref{Eq: Regularity of harmonic extension} implies that 
\begin{equation}
\| \partial_t \nabla'  \varphi, \nabla\nabla'^2\varphi\|_{ \LL_1 (\IR ; \dot \B^{s}_{p , 1} (\IR^n_+))}
\lesssim \|\wt h'\|_ {\LL_1 (\IR ; \dot \B^{s+1-1/p}_{p , 1} (\partial\IR^n_+))}.
\end{equation}
Hence, we miss only the information on the regularity of 
$\partial_t\partial_n\varphi$ and $\partial_n\partial_n\nabla\varphi$ 
(the latter may be deduced from the former since $\partial_{n}\partial_{n}\nabla\varphi=\partial_t\nabla\varphi-\Delta'\nabla\varphi$).

\subsubsection*{Substep IV}  Our goal now is to recover the regularity of $\partial_n\varphi_t$.

To proceed in the case $n=2,$
  consider the function $\zeta := \partial_t \varphi - 2 \partial_1 \partial_1 \varphi + E \wt h_1,$
  that obviously satisfies 
  $$\left\{\begin{aligned}
 \partial_t \zeta - \Delta \zeta &= 
 \partial_t(E\wt h_1)-\divergence (\nabla E \wt h_1)\\
  \zeta|_{x_2 = 0} &= 0.
\end{aligned}\right.$$
The (homogeneous) boundary condition is guaranteed by~\eqref{Eq: Heat equation on the boundary} and $\Phi = \varphi|_{x_2 = 0}$.
\smallbreak
According to Lemma~\ref{eq:extheat} below, although 
 the regularity of the right-hand side is only given in $\LL_1(\IR;\dot\B^{s-1}_{p,1}(\IR^2_+)),$ 
its particular structure ensures that $\nabla\zeta \in \LL_1(\IR ; \dot \B^{s}_{p , 1} (\IR^2_+))$ and 
that
$$
\|\nabla\zeta\|_{\LL_1(\IR ; \dot \B^{s}_{p , 1} (\IR^2_+))}\lesssim
\|\nabla E\wt h_1\|_{\LL_1(\IR ; \dot \B^{s}_{p , 1} (\IR^2_+))}+
\|\partial_t(E\wt h_1)\|_{\LL_1(\IR ; \dot \B^{s-1}_{p , 1} (\IR^2_+))}.
$$
Since $\partial_2\varphi_t=\partial_2\zeta+2\partial_1\partial_1\partial_2\varphi
-\partial_2E\wt h_1,$ the property~\eqref{Eq: Properties of derivatives of phi 1}
combined with the assumption 
on $h_1$ ensures that $\partial_2\varphi_t$ is in 
$\LL_1(\IR ; \dot \B^{s+1}_{p , 1} (\IR^2_+)),$ with the estimate
$$\displaylines{\|\partial_2\varphi_t\|_{\LL_1(\IR ; \dot \B^{s}_{p , 1} (\IR^2_+))}\lesssim
\|\nabla E\wt h_1\|_{\LL_1(\IR ; \dot \B^{s}_{p , 1} (\IR^2_+))}
+\|\partial_t(E\wt h_1)\|_{\LL_1(\IR ; \dot \B^{s-1}_{p , 1} (\IR^2_+))}+\|\wt h_1\|_{\LL_1 (\IR ; \dot \B^{s+1-1/p}_{p , 1} (\partial\IR^n_+))}.}$$
As $\nabla\partial_2\partial_2\varphi=\nabla\varphi_t-\nabla\partial_1\partial_1\varphi,$
putting together with the previous substep gives
$$\nabla\partial_2\partial_2\varphi\in\LL_1(\IR ; \dot \B^{s+1}_{p , 1} (\IR^2_+)).$$
In the 2D case, one can thus conclude that there exists a divergence free vector-field  $w$ satisfying~\eqref{eq: BC} which has the regularity properties
$$
w_t \in \LL_1(\IR ; \dot \B^{s}_{p , 1} (\IR^2_+)) \andf \nabla w \in \LL_1 (\IR ; \dot \B^{s+1}_{p , 1} (\IR^2_+)).
$$
and fulfills 
$$\begin{aligned}
\|w_t\|_{ \LL_1(\IR ; \dot \B^{s}_{p , 1} (\IR^2_+))} \!+\!\|\nabla w&\|_{ \LL_1(\IR ; \dot \B^{s+1}_{p , 1} (\IR^2_+))}
\lesssim \|\nabla E\wt h\|_{\LL_1(\IR ; \dot \B^{s}_{p , 1} (\IR^2_+))}
+\|\partial_t(E\wt h)\|_{\LL_1(\IR ; \dot \B^{s-1}_{p , 1} (\IR^2_+))}\\
&\lesssim \|R_t,\nabla E h\|_{\LL_1(\IR ; \dot \B^{s}_{p , 1} (\IR^2_+))}
\!+\!\|g\|_{\LL_1(\IR ; \dot \B^{s+1}_{p , 1} (\IR^2_+))}
\!+\!\|\partial_t(Eh)\|_{\LL_1(\IR ; \dot \B^{s-1}_{p , 1} (\IR^2_+))}.
\end{aligned}$$
Let us shortly explain how to modify the above arguments if $n=3.$
Then, we define
$$\begin{aligned}
\zeta_1&:=\partial_t\phi_1-\Delta'\phi_1-\partial_2\partial_2\phi_1+\partial_1\partial_2\phi_2-E\wt h_2,\\
\zeta_2&:=\partial_t\phi_2-\Delta'\phi_2-\partial_1\partial_1\phi_2+\partial_1\partial_2\phi_1+E\wt h_1
\end{aligned}$$
so that, by construction, both $\zeta_1$ and $\zeta_2$ vanish at $\partial\IR_+^3.$
Furthermore
$$\begin{aligned}
\partial_t\zeta_1-\Delta\zeta_1&=\divergence(\nabla E\wt h_2)-\partial_t(E\wt h_2),\\
\partial_t\zeta_2-\Delta\zeta_2&=-\divergence(\nabla E\wt h_1)+\partial_t(E\wt h_1).\end{aligned}$$
{}From those relations, one can conclude as in the 2D case.
The higher dimensional case is similar.

\subsubsection*{Substep V} Construction of the boundary pressure.

It suffices to exhibit  a  function $\vartheta$  in   $\LL_1 (\IR_+  ; \dot \B^{s+1}_{p , 1} (\IR^n_+))$
fulfilling the desired boundary condition, namely 
\begin{align*} \vartheta|_{x_n = 0} = 2 \partial_n w_n - \wt h_n\end{align*}
and admitting an extension on $\IR^n_+$ with the desired structure and regularity. 
\medbreak
Since $\divergence w=0,$ the simplest choice is 
$$ E\vartheta=: -2\divergence' w' - E\wt h_{n},$$
so that
$$\partial_t(E\vartheta):=-2\divergence' w'_t -\partial_t(E\wt h_{n}).$$ 
That definition is linear with respect to the data and, according to the previous steps,
$$\begin{aligned}
\|E\vartheta\|_{\LL_1(\IR ; \dot \B^{s}_{p , 1} (\IR^n_+))}&\lesssim \|g,Eh\|_{\LL_1(\IR ; \dot \B^{s+1}_{p , 1} (\IR^n_+))},\\\|\partial_t(E\vartheta)\|_{\LL_1(\IR ; \dot \B^{s}_{p , 1} (\IR^n_+))}
&\lesssim \|R_t\|_{\LL_1(\IR ; \dot \B^{s}_{p , 1} (\IR^n_+))}
+\|\partial_t(Eh)\|_{\LL_1(\IR ; \dot \B^{s-1}_{p , 1} (\IR^n_+))}.\end{aligned}$$
Furthermore, putting together all the steps of the construction, we see that
$\partial_t(E\vartheta)$ possesses the desired structure.

\subsection*{Step 3: Removing the potential part of the (modified) source term}

Let us set $v=w+z$ and $P=\vartheta +Q.$ Then, $(z,Q)$ has to fulfill:
\begin{equation}\label{eq:DPG} \left\{
  \begin{aligned}
   \partial_t z - \divergence \IT (z , Q) &= \divergence \IT (w,\vartheta)-\partial_tw+ \wt f=:\check f \quad &\hbox{in }\ & \IR_+\times\IR^n_+, \\
   \divergence z&= 0 \quad &\hbox{in }\ & \IR_+\times\IR^n_+, \\
   \IT (z , Q) \cdot \e_n&= 0 \quad &\hbox{on }\ & \IR_+\times\partial\IR^n_+,\\
   v|_{t = 0} &= u_0  \quad &\hbox{in }\ & \IR^n_+.
  \end{aligned}
 \right. \end{equation}
 At this stage, one may apply the Da Prato -- Grisvard  result
 and get the desired regularity and estimates for $u_t,\nabla^2u,\nabla Q.$ 
 However, in doing that, we miss the  needed  information on $\partial_t(Q|_{\partial\IR_+^n}).$ 
 To achieve what we need, we first have  to remove the potential part of the right-hand side of the first equation, 
 solving
$$ \left\{  \begin{aligned}\Delta P_I&=\divergence\check f  \quad &\hbox{in }\ & \IR_+\times\IR^n_+, \\
P_I|_{\partial\IR^n_+}&=0 \quad &\hbox{on }\ & \IR_+\times\partial\IR^n_+.\end{aligned}\right.$$
Using an  antisymmetric/symmetric  extension for $\check f$ that is, setting 
$E_{a,s}\check f:=(E_a\check f',E_s\check f_n)$ so that 
$\divergence E_{a,s}\check f=E_a\divergence \check f$
as in  Lemma~1 in~\cite{Danchin_Mucha}, we obtain
a skewsymmetric solution of $\Delta P_I= \divergence E_{a,s}\check f$ on the whole space
which, after restriction to the half-space eventually yields 
$\nabla P_I\in \LL_1(\IR_+;\dot\B^s_{p,1}(\IR^n))$ with 
\begin{equation}\|\nabla P_I\|_{\LL_1(\IR_+;\dot\B^s_{p,1}(\IR^n_+))}
\lesssim \|\check f\|_{\LL_1(\IR_+;\dot\B^s_{p,1}(\IR^n_+))}.\end{equation}

\subsection*{Step 4: Back to the homogeneous Stokes system}

Let us set $Q=P_I+\wt P_B.$ Then $(z,\wt P_B)$  has to satisfy~\eqref{eq:DPG}
with source term $\bar f:= \check f-\nabla P_I$ instead of $\check f.$
  Da Prato and Grisvard theory (see Theorem~\ref{Thm: Da Prato - Grisvard for Stokes})  guarantees the existence of a solution $(z,\wt P_B)$ with 
$$
z\in\cC_b(\IR_+;\dot\B^s_{p,1}(\IR^n_+)),\qquad
\partial_tz,\: \nabla^2z,\: \nabla \wt P_B \in \LL_1(\IR_+;\dot\B^s_{p,1}(\IR^n_+)).$$
  The gain compared to the previous step is that, now, we have
 $\divergence \bar f=0$ so that  
 $$ \left\{  \begin{aligned}\Delta \wt P_B&=0  \quad &\hbox{in }\ & \IR_+\times\IR^n_+, \\
\wt P_B|_{\partial\IR^n_+}&=-2\divergence^\prime z^\prime  \quad &\hbox{on }\ & \IR_+\times\partial\IR^n_+.\end{aligned}\right.$$
Hence, if one extends $\wt P_B|_{\partial\IR^n_+}$ by 
$-2\divergence' z',$ then we have $\partial_t(E\wt P_B)= 
-2\divergence^\prime z_t^\prime$ so that 
both $E\wt P_B$ and its time derivative 
satisfy the desired regularity, bounds and structure.


\subsection*{Step 5: End of the proof}

Let us set
$$u:=W+w+z\andf P:=P_I+P_B\with P_B:=\vartheta+\wt P_B.$$
Putting all the previous steps together, 
we see that $\nabla^2u$ and $\partial_tu$ are in  $\LL_1(\IR_+;\dot\B^s_{p,1}(\IR^n_+))$
and fulfill the announced estimates, and
that $u\in\cC_b(\IR_+; \dot\B^s_{p,1}(\IR^n_+))$ (integrate
with respect to time and use the fact that $W|_{t=0}=w|_{t=0}=0$
owing to $g|_{t=0}=0$ and $h|_{t=0}=0$).

As regards the pressure, it is clear that $P_I$ fulfills what we want
(in particular it vanishes at the boundary) and 
that $(P_B)|_{\partial\IR_+^n}$ admits an extension $EP_B$ with the 
required structure and regularity. 
\end{proof}
\medbreak
The following lemma was decisive in Step 2. 

\begin{lemma}\label{eq:extheat} Consider the system 
\begin{equation}\label{eq:model}\left\{\begin{aligned}\partial_t\phi-\Delta\phi&= \divergence k^1 + m\divergence k^2&\hbox{in  }&\ \IR\times\IR^n_+,\\
\phi|_{\partial\IR^n_+}&=0&\hbox{on }&\ \IR\times\partial\IR^n_+,\end{aligned}\right.\end{equation}
with $m\in \LL_q(\IR;\dot\B^{n/p}_{p,1}(\IR^n_+)),$ $k^1$ in $\LL_1(\IR;\dot\B^s_{p,1}(\IR^n_+))$
and $k^2$ in $\LL_{q'}(\IR;\dot\B^s_{p,1}(\IR^n_+))$ for some $1<p<\infty,$ $-1+1/p <s <1/p$ and  $1\leq q\leq \infty.$

Then,  there exists a unique solution $\phi\in\cC_b(\IR;\dot \B^{s-1}_{p,1}(\IR^n_+))$  of~\eqref{eq:model} 
in the following cases:
\begin{itemize}
\item $m=0$ and $n\geq2$;
\item $m\not=0,$  $n\geq3$ and $(p,s)$ satisfy in addition 
 $n-1<p<\infty$  and $s>1-\min(n/p,n/p').$
 \end{itemize}
 Furthermore, in all situations, we have the following inequality:
 $$\|\nabla\phi\|_{\LL_1(\IR;\dot\B^s_{p,1}(\IR^n_+))}\lesssim
\|k^1\|_{\LL_1(\IR;\dot\B^s_{p,1}(\IR^n_+))}
+ \|m\|_{\LL_q(\IR;\dot\B^{n/p}_{p,1}(\IR^n_+))}\|k^2\|_{\LL_{q'}(\IR;\dot\B^s_{p,1}(\IR^n_+))}.$$
\end{lemma}
\begin{proof}
We split $\phi$ into $\phi=\phi^1+\phi^2$ where $\phi^1$ corresponds to the source
term $\divergence k^1$, and $\phi^2$ to $m\divergence k^2,$ respectively.

To define  $\phi^1,$  we argue as in Lemma~1 of~\cite{Danchin_Mucha}, considering 
the antisymmetric/symmetric extension of $k^1,$ namely
$E_{a,s} k^1:=(E_a{k^1}',E_sk^1_n)$ so that $\divergence E_{a,s} k^1=E_a\divergence k^1.$
Note that $E_{a,s}k^1\in \LL_1(\IR;\dot\B^s_{p,1}(\IR^n))$ since $-1+1/p<s<1/p.$ 
Then, we solve 
$$\partial_t\check\phi^1-\Delta\check\phi^1=\divergence E_{a,s}k^1
\quad\hbox{in}\quad\IR\times\IR^n$$
and find a unique antisymmetric solution (as  the source term
is antisymmetric) $\check\phi^1$
in $\cC_b(\IR;\dot \B^{s-1}_{p,1}(\IR^n))$ satisfying in addition
$$\|\nabla\check\phi^1\|_{ \LL_1(\IR;\dot \B^s_{p,1}(\IR^n))}\lesssim 
\|E_{a,s}k^1\|_{ \LL_1(\IR;\dot \B^s_{p,1}(\IR^n))}\lesssim 
\|k^1\|_{ \LL_1(\IR;\dot \B^s_{p,1}(\IR^n_+))}.$$
Clearly, $\phi^1:=\check\phi^1|_{x_n>0}$ satisfies what we want, which completes the proof of the first case.
\medbreak
To complete the proof of the second case, we also need to construct $\phi_2.$ 
So we still consider $E_{a,s}k^2$ (that belongs to $ \LL_1(\IR;\dot\B^s_{p,1}(\IR^n))$),
and use, in addition, the \emph{symmetric} extension $E_s m$ of $m,$ the
regularity of which is conserved provided $n/p < 1+ 1/p.$
Obviously, $E_sm\,\divergence E_{a,s}k^2$ is antisymmetric and belongs
to $\LL_1(\IR;\dot\B^{s-1}_{p,1}(\IR^n))$ provided $s-1>-\min(n/p,n/p')$ and $s-1\leq n/p$
 (see~\cite[Lem.~A.5]{Danchin14})\footnote{Here again we use that $n \geq 3$.}. 
Hence, one can solve 
$$\partial_t\check\phi^2-\Delta\check\phi^2= E_s m\,\divergence E_{a,s}k^2
\quad\hbox{in}\quad\IR\times\IR^n$$
and get a skewsymmetric solution $\check\phi^2$ with $\nabla\check\phi^2$ 
in $ \LL_1(\IR;\dot\B^s_{p,1}(\IR^n)).$ Setting 
$\phi^2:=\check\phi^2|_{x_n>0},$ we eventually get 
a solution to 
$$
\partial_t\phi^2-\Delta\phi^2= m\divergence k^2\quad\hbox{in }\ \IR\times\IR^n_+
$$
vanishing at the boundary and satisfying the desired inequality. 
\end{proof}

\section{The fixed point procedures} 
The first step is to solve system~\eqref{Eq: Linearized2}. 
We aim at proving:
\begin{proposition}\label{p:fixedpoint} Assume that   $n-1<p<n.$ 
Let $v$ be  a  vector-field  in $\cC_b(\IR_+;\dot\B^{n/p-1}_{p,1}(\IR_+^n))$
such that $X_v$ is measure preserving
and that, for a small enough $c_0>0,$ 
\begin{equation}\label{eq:small2v} 
 \|\nabla  v \|_{\LL_1(\IR_+;\dot \B^{n / p }_{p , 1} (\IR^n_+))} \leq c_0.
\end{equation}
Then, for any data $v_0,$ $f_0,$ $g_0$ and $h_0$ satisfying the same conditions as the data of 
Proposition~\ref{p:fight}, 
there exists a unique solution $(u,P)$ in $\Xi_p$ to system~\eqref{Eq: Linearized2}
 and  we have for some constant $C$ depending only on $p,$
\begin{multline}\label{eq:Xi_p}  \Xi_p(u,P) \leq 
C\Bigl(\|v_0\|_{\dot \B^{n/p-1}_{p,1}(\IR^n_+)}+ \|f_0,\partial_tR_0,\nabla Eh_0\|_{ \LL_1(\IR_+; \dot \B^{n/p-1}_{p,1}(\IR_+^n))}\\
+\|g_0\|_{ \LL_1(\IR_+; \dot \B^{n/p}_{p,1}(\IR_+^n))}+\|\partial_t(Eh_0)\|_{ \LL_1(\IR_+; \dot \B^{n/p-2}_{p,1}(\IR_+^n))}\Bigr)\cdotp\end{multline}
In the particular case $g_0\equiv0,$ then  the flow defined from  $u$  by~\eqref{def:Xu} is measure preserving.
\end{proposition}
\begin{proof} It is based on  the  standard fixed point theorem. 
However, since the structure of the solution space  $\Xi_p$ is rather 
complicated, it is more informative to do the proof `by hand'. 

\subsubsection*{Preliminary step}  We solve
 \begin{equation}\label{Eq: Linearized0}\left\{
  \begin{aligned}
   \partial_t u^0 - \divergence \IT (u^0 , P^0) &= f_0 \quad &\hbox{in }\ & \IR_+\times\IR^n_+, \\
   \divergence u^0 &= g_0 \quad &\hbox{in }\ & \IR_+\times\IR^n_+, \\
   \IT (u^0 , P^0) \cdot \e_n|_{\partial\IR^n_+}&= h_0 \quad &\hbox{on }\ & \IR_+\times\partial\IR^n_+,\\
   u^0|_{t = 0} &= v_0  \quad &\hbox{in }\ & \IR^n_+.
  \end{aligned}\right.\end{equation}
  As $n-1<p<n,$    Proposition~\ref{p:fight}  enables us to get a solution $(u^0,P^0)$
  in  the space $\Xi_p,$ that satisfies
  $$\displaylines{  \Xi_p(u^0,P^0)\leq CA_0  \with\cr
   A_0:=\|v_0\|_{ \dot \B^{n/p-1}_{p,1}(\IR_+^n)} 
+ \|f_0,\partial_tR_0,\nabla Eh_0\|_{ \LL_1(\IR_+; \dot \B^{n/p-1}_{p,1}(\IR_+^n))}\hfill\cr\hfill
+\|g_0\|_{ \LL_1(\IR_+; \dot \B^{n/p}_{p,1}(\IR_+^n))}+\|\partial_t(Eh_0)\|_{ \LL_1(\IR_+; \dot \B^{n/p-2}_{p,1}(\IR_+^n))}.}$$

  \subsubsection*{Generic step} Assuming that $(u^\ell,P^\ell)$ has 
  been constructed in $\Xi_p$ and satisfies~\eqref{eq:Xi_p}, we want to solve 
  \begin{equation}
\label{Eq: Linearizedn}
\left\{\begin{aligned}
 \!\partial_t u^{\ell+1} \!-\! \divergence \IT (u^{\ell+1} , P^{\ell+1}) &= f_0\!+\!\divergence_{v} \IT_{v} (u^\ell , P^\ell) - \divergence \IT (u^\ell , P^\ell) \  &\hbox{in }\ & \IR_+\times\IR^n_+ \\
 \divergence u^{\ell+1} &= g_0-\divergence_v u^\ell + \divergence  u^\ell \ &\hbox{in }\ & \IR_+\times\IR^n_+, \\
 \IT (u^{\ell+1} , P^{\ell+1}) \cdot \e_n &=  h_0+\IT (u^\ell, P^\ell) \cdot \e_n - \IT_{v} (u^\ell , P^\ell) 
 \overline{n}_{v}  \ &\hbox{on }\ & \IR_+\!\times\!\partial\IR^n_+,\\
 u^{\ell + 1}|_{t = 0} &= v_0 \ &\hbox{in }\ & \IR^n_+.
\end{aligned}\right.
\end{equation}
Let us denote by $f^\ell,$ $g^\ell$ and $h^\ell$ the second part of the right-hand sides of the first three
equations of~\eqref{Eq: Linearizedn}.
Taking advantage of Lemmas~\ref{Lem: Control of right-hand side momentum},~\ref{l:divW}, 
Inequality~\eqref{eq:Eh} and the computations that follow below this inequality, and of Lemma~\ref{l:nablau}, we see that
defining $R_t^\ell$ by $\partial_tg^\ell=\divergence R_t^\ell,$ 
 $$\begin{aligned}
 \|f^\ell\|_{\LL_1(\IR_+;\dot\B^{n/p-1}_{p,1}(\IR^n_+))}&\lesssim
 \|\nabla v\|_{\LL_1(\IR_+;\dot\B^{n/p}_{p,1}(\IR^n_+))}\:\Xi_p(u^\ell,P^\ell),\\
 \|g^\ell\|_{\LL_1(\IR_+;\dot\B^{n/p}_{p,1}(\IR^n_+))}&\lesssim
 \|\nabla v\|_{\LL_1(\IR_+;\dot\B^{n/p}_{p,1}(\IR^n_+))}\:\Xi_p(u^\ell,P^\ell),\\
 \|R^\ell_t\|_{\LL_1(\IR_+;\dot\B^{n/p-1}_{p,1}(\IR^n_+))}&\lesssim
 \|\nabla v\|_{\LL_1(\IR_+;\dot\B^{n/p}_{p,1}(\IR^n_+))}\:\Xi_p(u^\ell,P^\ell),\\
 \|Eh^\ell\|_{\LL_1(\IR_+;\dot\B^{n/p}_{p,1}(\IR^n_+))}&\lesssim \|\nabla v\|_{\LL_1(\IR_+;\dot\B^{n/p}_{p,1}(\IR^n_+))}\:\Xi_p(u^\ell,P^\ell),\\
 \|\partial_t(Eh^\ell)\|_{\LL_1(\IR_+;\dot\B^{n/p-2}_{p,1}(\IR^n_+))}&\lesssim \|\nabla v\|_{\LL_1(\IR_+;\dot\B^{n/p}_{p,1}(\IR^n_+))}\:\Xi_p(u^\ell,P^\ell).\end{aligned}$$
 Furthermore, as pointed out before below~\eqref{eq:Eh}, 
 the term of $\partial_t(Eh^\ell)$ has the  structure that is required in  Proposition~\ref{p:fight} (here we use the fact that $(u^\ell,P^\ell)$ is in $\Xi_p$).  
 Hence, applying that proposition provides us with a solution 
 $(u^{\ell+1},P^{\ell+1})$ in $\Xi_p$ such that, according to the above inequalities,
 $$\Xi_p(u^{\ell+1},P^{\ell+1}) \leq C\bigl( A_0
 +  \|\nabla v\|_{\LL_1(\IR_+;\dot\B^{n/p}_{p,1}(\IR^n_+))}\:\Xi_p(u^\ell,P^\ell)\bigr)\cdotp$$
 It is now clear that if $v$ is chosen so that 
 \begin{equation}\label{eq:smallv}
 2C \|\nabla v\|_{\LL_1(\IR_+;\dot\B^{n/p}_{p,1}(\IR^n_+))}\leq 1,\end{equation} 
 then we have 
 \begin{equation}\label{eq:uniform}
 \Xi_p(u^\ell,P^\ell)\leq 2C A_0\quad\hbox{for all }\ \ell\in\IN.
   \end{equation}
   \subsubsection*{Convergence of the sequence}
   We just have to observe that, for all $\ell\geq1,$ the couple $(u^{\ell+1}-u^\ell,P^{\ell+1}-P^\ell)$ satisfies
   system~\eqref{Eq: Linearizedn} with null initial data and 
   right-hand sides $f^\ell-f^{\ell-1},$ $g^\ell-g^{\ell-1}$ and $h^\ell-h^{\ell-1}.$ 
   Hence, Proposition~\ref{p:fight}  and hypothesis~\eqref{eq:smallv} guarantee  that 
   $$\Xi_p(u^{\ell+1}-u^\ell, P^{\ell+1}-P^\ell) \leq\frac12 \,   \Xi_p(u^{\ell}-u^{\ell-1}, P^{\ell}-P^{\ell-1}), $$
   and we thus have a Cauchy sequence. 
   \medbreak
  Using the standard completeness properties
  of the Besov spaces, this already ensures that $(u^\ell,P^\ell)$ has a limit $(u,P)$ fulfilling~\eqref{eq:regularity} with $s=n/p-1$ and   
   such that 
  $$  \|u\|_{\LL_\infty(\IR_+; \dot \B^{n/p-1}_{p,1}(\IR_+^n))}+\|u_t,\nabla^2u,\nabla P\|_{ \LL_1(\IR_+; \dot \B^{n/p-1}_{p,1}(\IR_+^n))}
  \leq 2CA_0. $$
  Since the above convergence also controls the boundary value of the pressure, 
   this regularity is enough to pass to the limit in  all 
   the equations of~\eqref{Eq: Linearizedn}, and to see that $(u,P)$ satisfies~\eqref{Eq: Linearized2}. In particular, as explained in, e.g.,~\cite{DM-CPAM}, 
   the flow $X_u$ associated to $u$ through~\eqref{def:Xu} is measure preserving. 
   As the extension operator $E$ of Proposition~\ref{p:fight} can be chosen linear
   and continuous, the pressure $P$ has the desired structure and $(u,P)$ is thus in $\Xi_p.$
   Finally, the reason why the flow associated to $u$ through~\eqref{def:Xu} is measure preserving is 
   explained in, e.g.,~\cite{DM-CPAM}.  
      \end{proof}
\bigbreak
We are now ready to prove the existence part of our main theorem. 
After recasting the system in Lagrangian coordinates. 
the problem amounts to finding a fixed point in the space $\Xi_p$ 
for the map $(v,Q)\mapsto (u,P)$ defined by system~\eqref{Eq: Lin3}.
To this end, we fix an initial data $v_0$ satisfying the condition of Theorem~\ref{Thm:NS3D}, 
then we argue by induction. 
The preliminary step is as before : we define $(u^0,P^0)$ to be the solution 
of~\eqref{Eq: Linearized0} with null source terms.
Then, once $(u^\ell,P^\ell)$ has been constructed in $\Xi_p,$ we 
set,  according to Proposition~\ref{p:fixedpoint}, $(u^{\ell+1},P^{\ell+1})$ to be the solution 
in $\Xi_p$ of 
  \begin{equation}
\label{Eq: Linearizednn}
\left\{\begin{aligned}
 \partial_t u^{\ell+1} - \divergence_{u^\ell} \IT_{u^\ell} (u^{\ell+1} , P^{\ell+1}) &= 0 \quad &\hbox{in }\ & \IR_+\times\IR^n_+ \\
 \divergence_{u^\ell}u^{\ell+1} &= 0\quad &\hbox{in }\ & \IR_+\times\IR^n_+, \\
 \IT_{u^\ell} (u^{\ell+1} , P^{\ell+1}) \cdot \e_n &= 0  \quad &\hbox{on }\ & \IR_+\times\partial\IR^n_+,\\
 u^{\ell+1}|_{t = 0} &= v_0 \quad &\hbox{in }\ & \IR^n_+.
\end{aligned}\right.
\end{equation}
Note that solving~\eqref{Eq: Linearizednn} requires that 
$$ \|u^\ell\|_{\LL_\infty(\IR_+; \dot \B^{n/p-1}_{p,1}(\IR_+^n))}+\|\nabla u^\ell\|_{\LL_1(\IR_+; \dot \B^{n/p}_{p,1}(\IR_+^n))}\leq c_0.$$
Given the definition of the norm in $\Xi_p,$ Lemma~\ref{l:nablau}
guarantees that there exists $c>0$ such that 
the above  condition is fulfilled whenever 
\begin{equation}\label{eq:smallun}\Xi_p(u^\ell,P^\ell)\leq c.\end{equation}
Assuming that it is the case, one can apply  Proposition~\ref{p:fixedpoint}
and get a unique solution $(u^{\ell+1},P^{\ell+1})$ to~\eqref{Eq: Linearizednn} such that 
 \begin{equation}\label{eq:boundn} \Xi_p(u^{\ell+1},P^{\ell+1}) \leq C\|v_0\|_{\dot \B^{n/p-1}_{p,1}(\IR^n_+)}.
 \end{equation}
 Hence, in order to be able to solve iteratively system~\eqref{Eq: Linearizednn}
 for all $\ell\in\IN,$ it suffices to assume that $v_0$ has been chosen so that
 \begin{equation}\label{eq:smallv0}  C\|v_0\|_{\dot \B^{n/p-1}_{p,1}(\IR^n_+)}\leq c.\end{equation}
 Then, we get a sequence of $\Xi_p$ that satisfies~\eqref{eq:boundn} for all $\ell\in\IN.$
 \medbreak
 To prove the convergence, we are going to show that $(u^\ell,P^\ell)_{\ell\in\IN}$ is a Cauchy 
 sequence of $\Xi_p.$ 
 Let $\du^\ell:=u^{\ell+1}-u^\ell$ and $\dP^\ell:=P^{\ell+1}-P^\ell.$ Then, we have
 $$
 \left\{\begin{aligned}
 \partial_t \du^{\ell} - \divergence_{u^\ell} \IT_{u^\ell} (\du^{\ell} , \dP^{\ell}) &=
 \df^\ell:=\bigl(\divergence_{u^{\ell-1}}\IT_{u^{\ell-1}}-\divergence_{u^\ell}\IT_{u^\ell}\bigr)(u^\ell,P^\ell) \  &\hbox{in }\ & \IR_+\times\IR^n_+ \\
 \divergence_{u^\ell}\du^{\ell} &= \dg^\ell:= \bigl(\divergence_{u^\ell}-\divergence_{u^{\ell-1}}\bigr)u^{\ell}\  &\hbox{in }\ & \IR_+\times\IR^n_+, \\
 \IT_{u^\ell} (\du^{\ell} , \dP^{\ell}) \cdot \e_n &= \dh^\ell:=\bigl(\IT_{u^\ell} -\IT_{u^{\ell-1}}\bigr) (u^{\ell} , P^{\ell}) \cdot \e_n   \  &\hbox{on }\ & \IR_+\!\times\!\partial\IR^n_+.
\end{aligned}\right.$$
We already know that $(u^\ell,P^\ell)$ is small in $\Xi_p$ for all $\ell\in\IN$ 
(in particular Conditions~\eqref{eq:small2v} and~\eqref{eq:smallD2v} are satisfied). 
Hence, in order to bound $(\du^\ell,\dP^\ell)$ in $\Xi_p$ by means of  Proposition~\ref{p:fixedpoint},
it suffices to show that $\df^\ell,$ $\dg^\ell$ and $\dh^\ell$  fulfill the conditions therein. 
Thanks to~\eqref{eq:lagop} and~\eqref{eq:Tlag}, we have
$$
\df^\ell=\divergence\Bigl(\bigl(A_{u^{\ell-1}} A_{u^{\ell-1}}^\top-A_{u^\ell} A_{u^\ell}^\top\bigr)\nabla u^\ell
+\bigl(A_{u^{\ell-1}}Du^\ell A_{u^{\ell-1}}-A_{u^\ell} Du^\ell A_{u^\ell}\bigr)\Bigr)
+\bigl(A_{u^\ell}-A_{u^{\ell-1}}\bigr)\cdot\nabla P^\ell.$$
Note that the bound~\eqref{eq:boundn} and the definition of $A_{u^\ell}$ in~\eqref{eq:Au} ensure that we have
$$\|A_{u^\ell}-\Id\|_{\LL_\infty(\IR_+;\dot\B^{n/p}_{p,1}(\IR_+^n))} \ll1.$$
Hence, one can take advantage of the the series expansion~\eqref{eq:Neumann-expansion} to bound
all the terms involving $A_{u^\ell}$ and $A_{u^{\ell-1}},$ and mimicking the proof of inequality~\eqref{Eq: Complete estimate of perturbed right-hand side}, 
we end up with 
$$\|\df^\ell \|_{\LL_1(\IR_+;\dot\B^{n/p-1}_{p,1}(\IR_+^n))}
\!\lesssim\! \|\nabla\du^{\ell-1} \|_{\LL_1(\IR_+;\dot\B^{n/p}_{p,1}(\IR_+^n))} 
\bigl(\|\nabla u^\ell \|_{\LL_1(\IR_+;\dot\B^{n/p}_{p,1}(\IR_+^n))} +\|\nabla\! P^\ell \|_{\LL_1(\IR_+;\dot\B^{n/p-1}_{p,1}(\IR_+^n))}\bigr)\cdotp$$
 Using~\eqref{eq:boundn}, we thus get 
 $$\|\df^\ell \|_{\LL_1(\IR_+;\dot\B^{n/p-1}_{p,1}(\IR_+^n))}\leq C \|v_0\|_{\dot \B^{n/p-1}_{p,1}(\IR^n_+)}\Xi_p(\du^{\ell-1},\dP^{\ell-1}).$$
Performing the (cumbersome and tedious) checks and estimates for $\dg^\ell$ and $\dh^\ell$  is left to the reader.
This is essentially a matter of repeating the computations 
 of Subsection~\ref{ss:NL}.

In the end, we get, if $c$ in~\eqref{eq:smallv0} is small enough, 
\begin{equation}\label{eq:cauchy}\Xi_p(\du^\ell,\dP^\ell)\leq\frac12\Xi_p(\du^{\ell-1},\dP^{\ell-1}),\qquad \ell\geq1.\end{equation}
Hence  $(u^\ell,P^\ell)_{\ell\in\IN}$ is a Cauchy 
 sequence of $\Xi_p,$ and one can conclude as in the proof of Proposition~\ref{p:fixedpoint}
 that it has  a limit $(u,P)$ in $\Xi_p,$ that is bounded by the right-hand side of~\eqref{eq:boundn}
 and satisfies system~\eqref{eq:FBNSLag}. 
 
 As for the proof of uniqueness, it is almost the same as the proof of inequality~\eqref{eq:cauchy}.

\chapter{The Lam\'e operator in the upper half-space}
\label{Sec: The Lame operator in the upper half-plane}

This chapter provides the basic properties of the Lam\'e operator in the upper half-plane that are needed to apply the Da Prato -- Grisvard theory of Chapter~\ref{Sec: Da Prato--Grisvard theorem}. To introduce the Lam\'e system, recall that  the viscous stress tensor  $\IS (v)$  is given by
\begin{align*}
 \IS (u) := \mu D(u) + \lambda \divergence(u) \Id \quad \text{with} \quad D(u) := \nabla u + [\nabla u]^{\top}.
\end{align*}
We make the ellipticity assumption
\begin{align}
\label{eq: ellipticity}
 \mu > 0 \quad \text{and} \quad 2 \mu + n \lambda > 0
\end{align}
and define the Lam\'e operator $L$ subject to homogeneous Neumann boundary conditions
\begin{align}
\label{Eq: Lame system}
\left\{
\begin{aligned}
 - \divergence(\IS (u)) &= f  && \text{in } \IR^n_+ \\
 \IS (u) \cdot \e_n &= 0 && \text{on } \partial \IR^n_+
 \end{aligned} \right.
\end{align}
via the form method. To this end, define the coefficients
\begin{align*}
 a_{i j}^{\alpha \beta} := \mu (\delta_{i j} \delta_{\alpha \beta} + \delta_{j \alpha} \delta_{i \beta}) + \lambda \delta_{i \alpha} \delta_{j \beta}
\end{align*}
and  the sesquilinear form
\begin{align*}
 \mathfrak{a} : \H^1 (\IR^n_+ ; \IC^n) \times \H^1 (\IR^n_+ ; \IC^n) \to \IC, \quad (u , v) \mapsto \sum_{\alpha , \beta = 1}^n \sum_{i , j = 1}^n \int_{\IR^n_+} a_{i j}^{\alpha \beta} \partial_j u^{\beta} \, \overline{\partial_i v^{\alpha}} \; \d x.
\end{align*}
For $f \in \LL_2 (\IR^n_+ ; \IC^n)$, an integration by parts shows that~\eqref{Eq: Lame system} is equivalent to the weak formulation
\begin{align*}
 \mathfrak{a} (u , v) = \int_{\IR^n_+} f \cdot \overline{v} \; \d x \qquad (v \in \H^1 (\IR^n_+ ; \IC^n)).
\end{align*}
Clearly, $\mathfrak{a}$ defines a symmetric and bounded sesquilinear form. We claim that the ellipticity assumption~\eqref{eq: ellipticity} implies the existence of a constant $C > 0$ such that the following G\r{a}rding inequality:
\begin{align}
\label{Eq: Garding}
 \mathfrak{a} (u , u) \geq C \| \nabla u \|_{\LL_2 (\IR^n_+ ; \IC^{n \times n})}^2 \qquad (u \in \H^1 (\IR^n_+ ; \IC^n)).
\end{align}
 Indeed, notice that first of all, the inequality $2 a b \leq a^2 + b^2$ implies that
\begin{align}
\label{Eq: Estimate of divergence}
 \lvert \divergence(u) \rvert^2 = \sum_{i , j = 1}^n \partial_j u^j \, \overline{\partial_i u^i} \leq n \sum_{j = 1}^n \lvert \partial_j u^j \rvert^2.
\end{align}
Second, let $0 \leq \mu^{\prime} < \mu$ be such that $2 \mu^{\prime} + n \lambda \geq 0$. Then, we find by virtue of Korn's inequality a constant $C > 0$ such that
\begin{align*}
 \mathfrak{a}(u , u) &= \sum_{\substack{i , j = 1 \\ \alpha , \beta = 1}}^n \int_{\IR^n_+} a_{i j}^{\alpha \beta} \partial_j u^{\beta} \, \overline{\partial_i u^{\alpha}} \, \d x \\
  &= \sum_{\substack{i , j = 1 \\ \alpha , \beta = 1}}^n \bigg( \int_{\IR^n_+} \mu \delta_{i j} \delta_{\alpha \beta} \partial_j u^{\beta} \, \overline{\partial_i u^{\alpha}} \; \d x + \int_{\IR^n_+} \mu \delta_{j \alpha} \delta_{i \beta} \partial_j u^{\beta} \, \overline{\partial_i u^{\alpha}} \, \d x + \int_{\IR^n_+} \lambda \delta_{i \alpha} \delta_{j \beta} \partial_j u^{\beta} \, \overline{\partial_i u^{\alpha}} \, \d x \bigg) \\
  &= \sum_{i , \alpha = 1}^n \int_{\IR^n_+} (\mu - \mu^{\prime}) \{\partial_i u^{\alpha} \, \overline{\partial_i u^{\alpha}} + \partial_{\alpha} u^i  \, \overline{\partial_i u^{\alpha}} \} + \mu^{\prime} \{ \partial_i u^{\alpha} \, \overline{\partial_i u^{\alpha}} + \partial_{\alpha} u^i  \, \overline{\partial_i u^{\alpha}} \} + \lambda \partial_{\alpha} u^{\alpha} \, \overline{\partial_i u^i} \, \d x \\
  &\geq C \int_{\IR^n_+} \lvert \nabla u \rvert^2 \, \d x + \sum_{i , \alpha = 1}^n \int_{\IR^n_+} \mu^{\prime} \{ \partial_i u^{\alpha} \, \overline{\partial_i u^{\alpha}} + \partial_{\alpha} u^i \, \overline{\partial_i u^{\alpha}} \} + \lambda \partial_{\alpha} u^{\alpha} \, \overline{\partial_i u^i} \, \d x.
\end{align*}
Finally,~\eqref{Eq: Estimate of divergence} implies that
\begin{align*}
 &\sum_{i , \alpha = 1}^n \Big( \mu^{\prime} \{ \partial_i u^{\alpha} \, \overline{\partial_i u^{\alpha}} + \partial_{\alpha} u^i \,  \overline{\partial_i u^{\alpha}} \} + \lambda \partial_{\alpha} u^{\alpha} \, \overline{\partial_i u^i} \Big) \\
 &\qquad = 2 \mu^{\prime} \sum_{i = 1}^n \lvert \partial_i u^i \rvert^2 + \mu^{\prime} \sum_{i = 1}^n \sum_{\substack{\alpha = 1 \\ \alpha \neq i}}^n \lvert \partial_i u^{\alpha} \rvert^2 + \mu^{\prime} \sum_{i = 1}^n \sum_{\substack{\alpha = 1 \\ \alpha \neq i}}^n \partial_{\alpha} u^i \, \overline{\partial_i u^{\alpha}} + \lambda \lvert \divergence(u) \rvert^2 \\
 &\qquad \geq \mu^{\prime} \sum_{i = 1}^n \sum_{\substack{\alpha = 1 \\ \alpha \neq i}}^n \lvert \partial_i u^{\alpha} \rvert^2 + \mu^{\prime} \sum_{i = 1}^n \sum_{\substack{\alpha = 1 \\ \alpha \neq i}}^n \partial_{\alpha} u^i \, \overline{\partial_i u^{\alpha}} + \Big(\frac{2 \mu^{\prime}}{n} + \lambda\Big) \lvert \divergence(u) \rvert^2 \\
 & \qquad \geq 0.
\end{align*}
Consequently, under condition~\eqref{eq: ellipticity} there exists a constant $C > 0$ such that~\eqref{Eq: Garding} holds. \medbreak
To proceed, we prove the validity of a Caccioppoli inequality for solutions that locally solve the homogeneous resolvent 
equation.

\begin{lemma}
\label{Lem: Caccioppoli}
Let $r > 0$ and $x_0 \in \overline{\IR^n_+}$. Let $\theta \in (0 , \pi)$ and $z \in \Sec_{\theta} = \{ z \in \IC \setminus \{0\} : \lvert \arg(z) \rvert < \theta \}$. Moreover, let $u \in \H^1_2 (B(x_0 , 2 r) \cap \overline{\IR^n_+} ; \IC^n)$ be such that
\begin{align}
\label{Eq: Variational}
 z \int_{\IR^n_+} u \cdot \overline{v} \; \d x + \mathfrak{a} (u , v) = 0
\end{align}
for all $v \in \H^1_2 (\IR^n_+ ; \IC^n)$ with $\supp (v) \subset B(x_0 , 2 r) \cap \overline{\IR^n_+}$. Then there exists a constant $C > 0$ such that
\begin{align}
\label{Eq: Caccioppoli}
 \lvert z \rvert \int_{B(x_0 , r) \cap \IR^n_+} \lvert u \rvert^2 \; \d x + \int_{B(x_0 , r) \cap \IR^n_+} \lvert \nabla u \rvert^2 \; \d x \leq \frac{C}{r^2} \int_{B(x_0 , 2 r) \cap \IR^n_+} \lvert u \rvert^2 \; \d x.
\end{align}
\end{lemma}

\begin{proof}
The classical proof is as follows. Let $\eta : \IR^n \to \IR$ be a smooth function with $\eta \equiv 1$ in $B(x_0 , r)$, $\eta \equiv 0$ in $\IR^n \setminus B(x_0 , 2 r)$ and $\lvert \nabla \eta \rvert \leq C / r$. Note that $C > 0$ can be chosen to just depend on the dimension. Take $v := u \eta^2$ in~\eqref{Eq: Variational}. Then
\begin{align*}
 0 &= z \int_{\IR^n_+} \lvert u \rvert^2 \eta^2 \; \d x + \sum_{\substack{i , j = 1 \\ \alpha , \beta = 1}}^n \int_{\IR^n_+} a_{i j}^{\alpha \beta} \partial_j u^{\beta} \, \overline{\partial_i (u^{\alpha} \eta^2)} \; \d x \\
  &= z \int_{\IR^n_+} \lvert u \rvert^2 \eta^2 \; \d x + \sum_{\substack{i , j = 1 \\ \alpha , \beta = 1}}^n \int_{\IR^n_+} a_{i j}^{\alpha \beta} \partial_j u^{\beta} \, \overline{\partial_i (u^{\alpha} \eta)} \eta \; \d x + \sum_{\substack{i , j = 1 \\ \alpha , \beta = 1}}^n \int_{\IR^n_+} a_{i j}^{\alpha \beta} \partial_j u^{\beta} (\partial_i \eta) \, \overline{u^{\alpha}} \eta \; \d x \\
  &= z \int_{\IR^n_+} \lvert u \rvert^2 \eta^2 \; \d x + \sum_{\substack{i , j = 1 \\ \alpha , \beta = 1}}^n \int_{\IR^n_+} a_{i j}^{\alpha \beta} \partial_j (u^{\beta} \eta) \, \overline{\partial_i (u^{\alpha} \eta)} \; \d x + \sum_{\substack{i , j = 1 \\ \alpha , \beta = 1}}^n \int_{\IR^n_+} a_{i j}^{\alpha \beta} \partial_j (u^{\beta} \eta) (\partial_i \eta) \, \overline{u^{\alpha}} \; \d x \\
  &\qquad - \sum_{\substack{i , j = 1 \\ \alpha , \beta = 1}}^n \int_{\IR^n_+} a_{i j}^{\alpha \beta} (\partial_j \eta) (\partial_i \eta) u^{\beta} \, \overline{u^{\alpha}} \; \d x - \sum_{\substack{i , j = 1 \\ \alpha , \beta = 1}}^n \int_{\IR^n_+} a_{i j}^{\alpha \beta} (\partial_j \eta) \, \overline{\partial_i (u^{\alpha} \eta)} u^{\beta} \; \d x.
\end{align*}
By~\eqref{Eq: Garding}, we find by elementary trigonometry a constant $C > 0$ depending only on $\theta$ and on the constant in~\eqref{Eq: Garding} such that
\begin{align*}
 \lvert z \rvert \int_{\IR^n_+} \lvert u \rvert^2 \eta^2 \; \d x + \int_{\IR^n_+} \lvert \nabla (u \eta) \rvert^2 \; \d x \leq C \bigg\lvert z \int_{\IR^n_+} \lvert u \rvert^2 \eta^2 \; \d x + \int_{\IR^n_+} a_{i j}^{\alpha \beta} \partial_j (u^{\beta} \eta) \, \overline{\partial_i (u^{\alpha} \eta)} \; \d x \bigg \rvert \cdotp
\end{align*}
Thus, by the equality above, we conclude to the existence of a constant depending additionally on $\mu$ and $\lambda$ such that
\begin{align*}
 \lvert z \rvert \int_{\IR^n_+} \lvert u \rvert^2 \eta^2 \; \d x + \int_{\IR^n_+} \lvert \nabla (u \eta) \rvert^2 \; \d x \leq C \bigg( \int_{\IR^n_+} \lvert \nabla (u \eta) \rvert \lvert \nabla \eta u \rvert \; \d x + \int_{\IR^n_+} \lvert \nabla \eta \rvert^2 \lvert u \rvert^2 \; \d x \bigg).
\end{align*}
Use the inequality $2 a b \leq a^2 + b^2$ to absorb the term $\lvert \nabla (u \eta) \rvert$ from the right-hand side to the left-hand side. Use afterwards the properties of $\eta$ to conclude that~\eqref{Eq: Caccioppoli} holds.
\end{proof}

The following lemma is a higher-order version of Lemma~\ref{Lem: Caccioppoli}.

\begin{lemma}
\label{Lem: Higher-order Caccioppoli}
Let $r > 0$ and $x_0 \in \overline{\IR^n_+}$. Let $\theta \in (0 , \pi)$ and $z \in \Sec_{\theta} = \{ z \in \IC \setminus \{0\} : \lvert \arg(z) \rvert < \theta \}$. Let $u \in \H^1_2 (B(x_0 , 2 r) \cap \IR^n_+ ; \IC^n)$ be such that
\begin{align}
\label{Eq: Variational higher}
 z \int_{\IR^n_+} u \cdot \overline{v} \; \d x + \mathfrak{a} (u , v) = 0
\end{align}
for all $v \in \H^1_2 (\IR^n_+ ; \IC^n)$ with $\supp (v) \subset B(x_0 , 2 r) \cap \overline{\IR^n_+}$. Then $\partial^{\alpha} u \in \H^1_2 (B(x_0 , 4^{- \lvert \alpha \rvert} r) \cap \Omega ; \IC^n)$ for all multi-indices $\alpha \in \IN_0^{n - 1}$ and there exists a constant $C > 0$ such that
\begin{align}
\label{Eq: Caccioppoli higher order}
 \lvert z \rvert \int_{B(x_0 , 4^{- \lvert \alpha \rvert} r) \cap \IR^n_+} \lvert \partial^{\alpha} u \rvert^2 \; \d x + \int_{B(x_0 , 4^{- \lvert \alpha \rvert} r) \cap \IR^n_+} \lvert \nabla \partial^{\alpha} u \rvert^2 \; \d x \leq \frac{C}{r^{2 (\lvert \alpha \rvert + 1)}} \int_{B(x_0 , 2 r) \cap \IR^n_+} \lvert u \rvert^2 \; \d x.
\end{align}
\end{lemma}

\begin{proof}
We only prove the case $\lvert \alpha \rvert = 1$, the rest follows by induction. For $j = 1 , \cdots , n - 1$ and $h > 0$ denote $\delta_j^h g (x) = h^{-1} (g(x + h \e_j) - g(x))$ the difference quotient of a function $g$ in direction $\e_j$. \par
A short calculation shows, that for all $0 < \lvert h \rvert < r/2$ the function $\delta_j^h u$ solves
\begin{align*}
 z \int_{\IR^n_+} \delta_j^h u \cdot \overline{v} \; \d x + \mathfrak{a} (\delta_j^h u , v) = 0
\end{align*}
for all $v \in \H^1 (\IR^n_+ ; \IC^n)$ with $\supp(v) \subset B(x_0 , r/2) \cap \overline{\IR^n_+}$. Lemma~\ref{Lem: Caccioppoli} now yields that
\begin{align*}
 \lvert z \rvert \int_{B(x_0 , r/4) \cap \IR^n_+} \lvert \delta_j^h u \rvert^2 \; \d x + \int_{B(x_0 , r/4) \cap \IR^n_+} \lvert \nabla \delta_j^h u \rvert^2 \; \d x \leq \frac{C}{r^2} \int_{B(x_0 , r/2) \cap \IR^n_+} \lvert \delta_j^h u \rvert^2 \; \d x,
\end{align*}
with $C > 0$ being independent of $h$. The method of difference quotients now implies that
\begin{align*}
 \lvert z \rvert \int_{B(x_0 , r/4) \cap \IR^n_+} \lvert \partial_j u \rvert^2 \; \d x + \int_{B(x_0 , r/4) \cap \IR^n_+} \lvert \nabla \partial_j u \rvert^2 \; \d x \leq \frac{C}{r^2} \int_{B(x_0 , r) \cap \IR^n_+} \lvert \partial_j u \rvert^2 \; \d x.
\end{align*}
Employing Lemma~\ref{Lem: Caccioppoli} again delivers that
\begin{align*}
 \lvert z \rvert \int_{B(x_0 , r/4) \cap \IR^n_+} \lvert \partial_j u \rvert^2 \; \d x + \int_{B(x_0 , r/4) \cap \IR^n_+} \lvert \nabla \partial_j u \rvert^2 \; \d x \leq \frac{C}{r^4} \int_{B(x_0 , 2 r) \cap \IR^n_+} \lvert u \rvert^2 \; \d x. &\qedhere
\end{align*}
\end{proof}

In the following proposition, Lemma~\ref{Lem: Higher-order Caccioppoli} is combined with Sobolev's inequality in order to derive the validity of weak reverse H\"older estimates. We will see afterwards, that such estimates already imply optimal resolvent bounds for the Lam\'e system in $\LL_p$-spaces.

\begin{proposition}
\label{Prop: Weak reverse Holder}
Let $r > 0$ and $x_0 \in \overline{\IR^n_+}$. Let $\theta \in (0 , \pi)$ and $z \in \Sec_{\theta} = \{ z \in \IC \setminus \{0\} : \lvert \arg(z) \rvert < \theta \}$ and let $p > 2$. Let $k \in \IN$ be a number that satisfies
\begin{align*}
 \frac{1}{2} - \frac{1}{p} \leq \frac{k}{n - 1}\cdotp
\end{align*}
Let $u \in \H^1_2 (\IR^n_+ ; \IC^n)$ be such that
for all $v \in \H^1_2 (\IR^n_+ ; \IC^n)$ with $\supp (v) \subset B(x_0 , 2 \cdot 4^k \sqrt{n} r) \cap \overline{\IR^n_+},$ we have
\begin{align*}
 z \int_{\IR^n_+} u \cdot \overline{v} \; \d x + \mathfrak{a} (u , v) = 0.
\end{align*}
  Then, $u \in \LL_p (B(x_0 , r) \cap \IR^n_+ ; \IC^n)$ holds together with the weak reverse H\"older inequality
\begin{align*}
 \bigg( \frac{1}{r^n} \int_{B(x_0 , r) \cap \IR^n_+} \lvert u \rvert^p \; \d x \bigg)^{\frac{1}{p}} \leq C \bigg( \frac{1}{r^n} \int_{B(x_0 , 2 r) \cap \IR^n_+} \lvert u \rvert^2 \; \d x \bigg)^{\frac{1}{2}}.
\end{align*}
The constant $C > 0$ only depends on $n$, $\mu$, $\lambda$, $\theta$, $p$, and $k$.
\end{proposition}

\begin{proof}
Let $Q (x_0 , r) \subset \IR^n$ denote the open cube with center $x_0$ and side-length $2 r$. Notice that $B(x_0 , r) \subset Q(x_0 , r)$. Write $Q(x_0 , r) = Q^{\prime} \times (x_0^n - r , x_0^n + r)$, where $x_0^n$ denotes the $n$'th component of the vector $x_0$. Let $k \in \IN$ be such that
\begin{align*}
 \frac{1}{2} - \frac{1}{p} \leq \frac{k}{n - 1}\cdotp
\end{align*}
Applying Sobolev's embedding in the $x^{\prime}$-variables yields for some constant $C > 0$ independent of $x_0$ and $r$ that
\begin{align*}
 \bigg( \frac{1}{r^n} \int_{B(x_0 , r) \cap \IR^n_+} \lvert u \rvert^p \; \d x \bigg)^{\frac{1}{p}} &\leq \bigg( \frac{1}{r^n} \int_{Q(x_0 , r) \cap \IR^n_+} \lvert u \rvert^p \; \d x \bigg)^{\frac{1}{p}} \allowdisplaybreaks \\
 &= \bigg( \frac{1}{r^n} \int_{(x_0^n - r , x_0^n + r) \cap (0 , \infty)} \int_{Q^{\prime}} \lvert u \rvert^p \; \d x^{\prime} \; \d x_n \bigg)^{\frac{1}{p}} \\
 &\leq C \bigg( \frac{1}{r} \int_{(x_0^n - r , x_0^n + r) \cap (0 , \infty)} \bigg( \sum_{\ell = 0}^k \frac{r^{2 \ell}}{r^{n - 1}} \int_{Q^{\prime}} \lvert \nabla_{x^{\prime}}^{\ell} u \rvert^2 \; \d x^{\prime} \bigg)^{\frac{p}{2}} \; \d x_n \bigg)^{\frac{1}{p}}\cdotp
\end{align*}
Notice that the factors of $r$ appear so that Sobolev's inequality becomes scaling invariant. Now, employ Sobolev's inequality in the $x_n$-variable to the function
\begin{align*}
 x_n \mapsto \bigg( \sum_{\ell = 0}^k \frac{r^{2 \ell}}{r^{n - 1}} \int_{Q^{\prime}} \lvert \nabla_{x^{\prime}}^{\ell} u (x^{\prime} , x_n) \rvert^2 \; \d x^{\prime} \bigg)^{\frac{1}{2}}\cdotp
\end{align*}
Notice that by virtue of the chain rule and  of 
the Cauchy--Schwarz inequality, we have
\begin{align*}
 \bigg\lvert \frac{\d}{\d x_n} &\bigg( \sum_{\ell = 0}^k \frac{r^{2 \ell}}{r^{n - 1}} \int_{Q^{\prime}} \lvert \nabla_{x^{\prime}}^{\ell} u (x^{\prime} , x_n) \rvert^2 \; \d x^{\prime} \bigg)^{\frac{1}{2}} \bigg\rvert \\
 &\simeq \bigg\lvert \bigg( \sum_{\ell = 0}^k \frac{r^{2 \ell}}{r^{n - 1}} \int_{Q^{\prime}} \lvert \nabla_{x^{\prime}}^{\ell} u (x^{\prime} , x_n) \rvert^2 \; \d x^{\prime} \bigg)^{- \frac{1}{2}} \cdot \sum_{\ell = 0}^k \frac{r^{2 \ell}}{r^{n - 1}} \int_{Q^{\prime}} \nabla_{x^{\prime}}^{\ell} \partial_n u (x^{\prime} , x_n) \cdot \nabla_{x^{\prime}}^{\ell} u(x^{\prime} , x_n) \; \d x^{\prime} \bigg\rvert \\
 &\lesssim \bigg( \sum_{\ell = 0}^k \frac{r^{2 \ell}}{r^{n - 1}} \int_{Q^{\prime}} \lvert \nabla_{x^{\prime}}^{\ell} \partial_n u (x^{\prime} , x_n) \rvert^2 \; \d x^{\prime} \bigg)^{\frac{1}{2}}\cdotp
\end{align*}
Thus, an application of Sobolev's inequality in the $x_n$-variable implies that
\begin{align*}
 \bigg( \frac{1}{r^n} \int_{B(x_0 , r) \cap \IR^n_+} \lvert u \rvert^p \; \d x \bigg)^{\frac{1}{p}} &\leq C \bigg( \frac{1}{r} \int_{(x_0^n - r , x_0^n + r) \cap (0 , \infty)} \bigg( \sum_{\ell = 0}^k \frac{r^{2 \ell}}{r^{n - 1}} \int_{Q^{\prime}} \lvert \nabla_{x^{\prime}}^{\ell} u \rvert^2 \; \d x^{\prime} \bigg)^{\frac{p}{2}} \; \d x_n \bigg)^{\frac{1}{p}} \\
 &\leq C \bigg( \frac{1}{r^n} \sum_{j = 0}^1 \sum_{\ell = 0}^k r^{2 (\ell + j)} \int_{Q(x_0 , r) \cap \IR^n_+} \lvert \nabla_{x^{\prime}}^{\ell} \partial_n^j u \rvert^2 \; \d x \bigg)^{\frac{1}{2}}\cdotp
\end{align*}
Finally, using that $Q(x_0 , r) \subset B(x_0 , \sqrt{n} r)$ followed by Lemma~\ref{Lem: Higher-order Caccioppoli} one obtains that
\begin{align*}
 \bigg( \frac{1}{r^n} \sum_{j = 0}^1 \sum_{\ell = 0}^k r^{2 (\ell + j)} \int_{Q(x_0 , r) \cap \IR^n_+} \lvert \nabla_{x^{\prime}}^{\ell} \partial_n^j u \rvert^2 \; \d x \bigg)^{\frac{1}{2}} \leq C \bigg( \frac{1}{r^n} \int_{B(x_0 , 2 \cdot 4^k \sqrt{n} r) \cap \IR^n_+} \lvert u \rvert^2 \; \d x \bigg)^{\frac{1}{2}}
\end{align*}
and thus that
\begin{align*}
 \bigg( \frac{1}{r^n} \int_{B(x_0 , r) \cap \IR^n_+} \lvert u \rvert^p \; \d x \bigg)^{\frac{1}{p}} &\leq C \bigg( \frac{1}{r^n} \int_{B(x_0 , 2 \cdot 4^k \sqrt{n} r) \cap \IR^n_+} \lvert u \rvert^2 \; \d x \bigg)^{\frac{1}{2}}
\end{align*}
which is the desired weak reverse H\"older inequality.
\end{proof}

Let $L$ denote the Lam\'e operator associated with the sesquilinear form $\mathfrak{a}$ and let $z \in \Sec_{\theta}$ for some $\theta \in (0 , \pi)$.   Lax--Milgram lemma implies that $z \in \rho(- L)$ and testing the resolvent equation by the solution $u$ shows that
\begin{align*}
 \| z (z + L)^{-1} f \|_{\LL_2 (\IR^n_+ ; \IC^n)} \leq C \| f \|_{\LL_2 (\IR^n_+ ; \IC^n)} \qquad (f \in \LL_2 (\IR^n_+ ; \IC^n)),
\end{align*}
with some constant $C > 0$ depending only on $\theta,$ $\lambda$ and $\mu.$ Define for fixed $z \in \Sec_{\theta}$ the operator $T_z := z (z + L)^{-1}$. As was just remarked, its $\LL_2$ operator norm is bounded by a constant independent of $z$. This fact, together with Proposition~\ref{Prop: Weak reverse Holder} allows us to employ the following version of the extrapolation theorem of Shen~\cite[Thm.~4.1]{Tolksdorf_elliptic}, see also~\cite{Shen}.

\begin{theorem}
Let $\Omega \subset \IR^n$ be measurable. Let $T \in \Lop(\LL_2 (\Omega))$ with $\| T \|_{\Lop(\LL_2 (\Omega))} \leq M$ for some constant $M > 0$. Assume that there exist  $N > 0$, $p > 2$, and $\alpha_2 > \alpha_1 > 1$ such that for all $x_0 \in \overline{\Omega}$, $r > 0$ and $f \in \LL_2 (\Omega)$ with $f \equiv 0$ in $\Omega \setminus B(x_0 , \alpha_2 r)$ the following estimate holds true:
\begin{align*}
 \bigg( \frac{1}{r^n} \int_{B(x_0 , r) \cap \Omega} \lvert T f \rvert^p \; \d x \bigg)^{\frac{1}{p}} \leq N \bigg( \frac{1}{r^n} \int_{B(x_0 , \alpha_1 r) \cap \Omega} \lvert T f \rvert^2 \; \d x \bigg)^{\frac{1}{2}}\cdotp
\end{align*}
Then, for all $2 \leq q < p,$ the operator $T$ is bounded from $\LL_q (\Omega)$ to $\LL_q (\Omega)$. Moreover, the operator norm of $T$ depends only on $p$, $q$, $n$, $N$ and $M$.
\end{theorem}

As the constant in Proposition~\ref{Prop: Weak reverse Holder} is independent of $z$, we conclude that $T_z$ is bounded on $\LL_q (\IR^n_+ ; \IC^n)$ with  bounds independent of $z.$
Besides, since $L$ is self-adjoint on $\LL_2 (\IR^n_+ ; \IC^n)$ by duality, we obtain   the boundedness of $T_z$ on $\LL_{q^{\prime}} (\IR^n_+ ; \IC^n)$. 
As a conclusion, we get that
\begin{align*}
 \{ z (z + L)^{-1} \}_{z \in \Sec_{\theta}}
\end{align*}
defines a uniformly bounded family of operators on $\LL_p (\IR^n_+ ; \IC^n)$ \emph{for any $1 < p < \infty$}. Thus, $- L$ generates a bounded analytic semigroup on $\LL_p (\IR^n_+ ; \IC^n)$ and consequently, in order to show that it fulfills the assumptions of Theorem~\ref{Thm: Full Da Prato - Grisvard}, 
it suffices  to verify  Assumptions~\ref{Ass: Homogeneous operator} and~\ref{Ass: Intersection property}. \par
To verify Assumption~\ref{Ass: Homogeneous operator}, we show that
\begin{align}
\label{Eq: Homogeneous estimate Lame}
 \| L u \|_{\LL_p (\IR^n ; \IC^n_+)} \simeq \| \nabla^2 u \|_{\LL_p (\IR^n_+ ; \IC^{n^3})} \qquad (u \in \dom(L)).
\end{align}
The estimate
\begin{align*}
 \| L u \|_{\LL_p (\IR^n ; \IC^n_+)} \lesssim \| \nabla^2 u \|_{\LL_p (\IR^n_+ ; \IC^{n^3})} \qquad (u \in \dom(L))
\end{align*}
just follows from the definition of $L.$ For the other direction, let $u \in \dom(L)$. Observe that for any $\alpha > 0$ also the rescaled function $u_{\alpha} (x) := u (\alpha x)$ is in $\dom(L)$. Moreover, the identity $L u_{\alpha} = \alpha^2 (L u)_{\alpha}$ holds, where $(L u)_{\alpha} (x) := (L u) (\alpha x)$. Relying on the continuous embedding $\dom (L) \hookrightarrow \H^2_p (\IR^n_+ ; \IC^n)$ one concludes to the existence of a constant $C > 0$ (independent of $\alpha$) such that
\begin{align*}
 \| \nabla^2 u_{\alpha} \|_{\LL_p (\IR^n_+ ; \IC^{n^3})} \leq C \big( \| u_{\alpha} \|_{\LL_p (\IR^n_+ ; \IC^n)} + \| L u_{\alpha} \|_{\LL_p (\IR^n_+ ; \IC^n)} \big).
\end{align*}
The chain and the substitution rules show that this inequality is equivalent to
\begin{align*}
 \alpha^{2 - \frac{2}{p}} \| \nabla^2 u \|_{\LL_p (\IR^n_+ ; \IC^{n^3})} \leq C \big( \alpha^{- \frac{2}{p}} \| u \|_{\LL_p (\IR^n_+ ; \IC^n)} + \alpha^{2 - \frac{2}{p}} \| L u \|_{\LL_p (\IR^n_+ ; \IC^n)} \big).
\end{align*}
Dividing by $\alpha^{2 - 2 / p}$ reveals that in the limit $\alpha \to \infty$ the following estimate holds
\begin{align*}
 \| \nabla^2 u \|_{\LL_p (\IR^n_+ ; \IC^{n^3})} \lesssim \| L u \|_{\LL_p (\IR^n_+ ; \IC^n)}.
\end{align*}
We conclude that~\eqref{Eq: Homogeneous estimate Lame} is valid and we can thus choose as the space $Y$ in Assumption~\ref{Ass: Homogeneous operator} the homogeneous Bessel potential space $\dot \H^2_p (\IR^n_+ ; \IC^n)$. \par
In order to show that $L$ satisfies Assumption~\ref{Ass: Intersection property}, we see that $\dom(L)$ is a subspace of $\dom(\dot L) \cap \LL_p (\IR^n_+ ; \IC^n)$ by the very definition. For the converse inclusion, notice that
\begin{align*}
 \dom(\dot L) \cap \LL_p (\IR^n_+ ; \IC^n) \subset \dot \H^2_p (\IR^n_+ ; \IC^n) \cap \LL_p (\IR^n_+ ; \IC^n) = \H^2_p (\IR^n_+ ; \IC^n)
\end{align*}
and thus we only need to show that functions $u$ in $\dom(\dot L) \cap \LL_p (\IR^n_+ ; \IC^n)$ satisfy the boundary condition $\IS (u)\cdot\e_n = 0$ on $\partial \IR^n_+$. By definition of $\dom(\dot L)$, there exists a sequence $(u_k)_{k \in \IN} \subset \dom(L)$ such that $u_k \to u$ in $\dot \H^2_p (\IR^n_+ ; \IC^n)$. This implies that $\nabla u|_{\partial \IR^n_+} \in \B^{1 - 1/p}_{p , p} (\partial \Omega)$. Using Lemma~\ref{Lem: Homogeneous trace and extension estimates} one then derives that
\begin{align*}
 \| \IS (u) \cdot \e_n \|_{\dot \B^{1 - \frac{1}{p}}_{p , p} (\partial \IR^n_+ ; \IC^n)} &= \| \IS (u - u_k) \cdot \e_n \|_{\dot \B^{1 - \frac{1}{p}}_{p , p} (\partial \IR^n_+ ; \IC^n)} \\
 &\leq C \| \nabla^2 (u - u_k) \|_{\LL_p (\IR^n_+ ; \IC^{n^3})} \to 0 \quad \text{as} \quad k \to \infty.
\end{align*}
Since $\IS (u) \cdot \e_n \in \LL_p (\partial \IR^n_+ ; \IC^n)$ we thus conclude that $\IS (u) \cdot \e_n = 0$ on $\partial \IR^n_+$. As a consequence, this makes Theorem~\ref{Thm: Full Da Prato - Grisvard} applicable  for the Lam\'e operator. \par
The remaining question is to identify the real interpolation spaces $(X , \dom(\dot L))_{\theta , q}$ for appropriate $\theta \in (0 , 1)$ and $1 \leq q < \infty$. This is done in the following proposition.

\begin{proposition}
For  $1 < p < \infty$, $1 \leq q < \infty$ and $0 < \theta < 1$ with
\begin{align*} 0 < \theta < \frac{n}{2 p} \ \hbox{ if }\ q>1,\andf  0 < \theta \leq \frac{n}{2 p}\ \hbox{ if }\ q=1,\end{align*}
one has
\begin{align*}
 \big( \LL_p (\IR^n_+ ; \IC^n) , \dom(\dot L_p) \big)_{\theta , q} \hookrightarrow  \dot \B^{2 \theta}_{p , q} (\IR^n_+ ; \IC^n).
\end{align*}
 If, furthermore, $0 < \theta < 1 / 2,$ then it holds with equivalent norms
\begin{align*}
\big( \LL_{p , \sigma} (\IR^n_+) , \dom(\dot L_p) \big)_{\theta , q} =  \dot \B^{2 \theta}_{p , q} (\IR^n_+ ; \IC^n).
\end{align*}
\end{proposition}

\begin{proof}
The proof follows exactly the lines of the proof of Proposition~\ref{Prop: Besov domain interpolation}. 
In fact,  the only property  that we still need to establish is the optimal mapping property of the semigroup $(\e^{- t L})_{t \geq 0}$ on the ground space $\H^1_p (\IR^n_+ ; \IC^n)$. This follows from the following discussion.
\par
To prove that the Lam\'e operator generates a bounded analytic semigroup on $\H^1_p (\IR^n_+ ; \IC^n)$ we argue similarly as for the Stokes operator in Proposition~\ref{Prop: Higher regularity resolvent estimate}. Hence, we need  homogeneous estimates of the form
\begin{align}
\label{Eq: First-order resolvent estimate Lame}
 \lvert z \rvert \| \nabla u \|_{\LL_p (\IR^n_+ ; \IC^{n^2})} + \lvert z \rvert^{\frac{1}{2}} \| \nabla^2 u \|_{\LL_p (\IR^n_+ ; \IC^{n^3})} + \| \nabla^3 u \|_{\LL_p (\IR^2_+ ; \IC^{n^4})} \lesssim \| \nabla f \|_{\LL_p (\IR^n_+ ; \IC^{n^2})}.
\end{align}
To establish this estimate, consider the resolvent equation
\begin{align*}
\left\{
\begin{aligned}
 z u - \divergence(\IS (u)) &= f  && \text{in } \IR^n_+ \\
 \IS(u) \cdot \e_n &= 0 && \text{on } \partial \IR^n_+.
 \end{aligned} \right.
\end{align*}
Notice that if one applies the tangential derivative $\nabla_{x^{\prime}}$ to $u,$ one discovers that the function $\nabla_{x^{\prime}} u$ is a solution to the same system but with right-hand side $\nabla_{x^{\prime}} f$. Hence, we immediately derive the validity of the estimate
\begin{align}
\label{Eq: Tangential estimate Lame}
 \lvert z \rvert \| \nabla_{x^{\prime}} u \|_{\LL_p (\IR^n_+)} + \lvert z \rvert^{\frac{1}{2}} \| \nabla \nabla_{x^{\prime}} u \|_{\LL_p(\IR^n_+)} + \| \nabla^2 \nabla_{x^{\prime}} u \|_{\LL_p (\IR^n_+)} \lesssim \| \nabla_{x^{\prime}} f \|_{\LL_p (\IR^n_+)}.
\end{align}
Next, write $u = (u^{\prime} , u^n)$. Then $u^{\prime}$ solves the following Neumann problem for the Laplacian
\begin{align*}
\left\{
\begin{aligned}
 z u^{\prime} - \mu \Delta u^{\prime} &= f^{\prime} + (\mu + \lambda) \nabla_{x^{\prime}} \divergence(u) \in \H^1_p (\IR^n_+ ; \IC^{n - 1}) \\
 - \partial_n u^{\prime} &= \nabla_{x^{\prime}} u^n \quad \text{on } \partial \IR^n_+.
\end{aligned} \right.
\end{align*}
By Proposition~\ref{Prop: Laplacian on half-space} combined with~\eqref{Eq: Tangential estimate Lame}, we find
 \begin{align*}
 \lvert z \rvert \| \nabla u^{\prime} \|_{\LL_p (\IR^n_+)} + \lvert z \rvert^{\frac{1}{2}} \| \nabla^2 u^{\prime} \|_{\LL_p (\IR^n_+)} + \| \nabla^3 u^{\prime} \|_{\LL_p (\IR^n_+)} \lesssim \| \nabla f \|_{\LL_p (\IR^n_+)}.
\end{align*}
Regarding the function $u^n$, we find that it solves the following inhomogeneous Neumann problem:
\begin{align}
\label{Eq: Last component of Lame}
\left\{
\begin{aligned}
 z u^n - (\mu \Delta_{x^{\prime}} + (2 \mu + \lambda) \partial_n^2) u^n &= f^n + (\mu + \lambda) \partial_n \divergence_{x^{\prime}} u^n \in \H^1_p (\IR^n_+) \\
 - \partial_n u^n &= \lambda (2 \mu + \lambda)^{-1} \divergence_{x^{\prime}} u^{\prime} \quad \text{on } \partial \IR^n_+.
\end{aligned} \right.
\end{align}
Clearly, Condition \eqref{eq: ellipticity} implies that
$2 \mu + \lambda > 0$. 
Now, employing the transformation $$v (x) = u (x^{\prime} , ((2 \mu + \lambda) / \mu)^{\frac{1}{2}} x_n),$$ problem~\eqref{Eq: Last component of Lame} can be reduced to the differential operator $\mu \Delta$ and we conclude from Proposition~\ref{Prop: Laplacian on half-space} and~\eqref{Eq: Tangential estimate Lame} that
\begin{align*}
 \lvert z \rvert \| \nabla u^n \|_{\LL_p (\IR^n_+ ; \IC^n)} + \lvert z \rvert^{\frac{1}{2}} \| \nabla^2 u^n \|_{\LL_p (\IR^n_+ ; \IC^{n^2})} + \| \nabla^3 u^n \|_{\LL_p (\IR^n_+ ; \IC^{n^3})} \lesssim \| \nabla f \|_{\LL_p (\IR^n_+ ; \IC^{n^2})}.
\end{align*}
This concludes the proof of~\eqref{Eq: First-order resolvent estimate Lame}.
\end{proof}

The discussion above allows us to employ Theorem~\ref{Thm: Full Da Prato - Grisvard} and to identify the real interpolation spaces for certain values of $p$, $q$ and $\theta$. Summarizing, this results in the following maximal regularity theorem for the Lam\'e system considered in homogeneous Besov spaces.

\begin{theorem}
\label{Thm: Da Prato - Grisvard for Lame}
Let $p\in (1,\infty)$ and $q\in [1 , 2)$ satisfy
\begin{align*}
 \frac{n}{p} + \frac{2}{q} > 2
\end{align*}
and let $s \in (0 , 2 /q - 1)$ be such that 
\begin{align*}
  0 < s < \frac{n}{p} + \frac{2}{q} - 2,\  \mbox{ \ or \ } 0< s \leq \frac{n}{p} + \frac{2}{q} - 2 \mbox{ \ if \ } q = 1.
\end{align*}
Let $f \in \LL_q(\IR_+  ; \dot \B^s_{p , q} (\IR^n_+ ; \IC^n))$ and $u_0 \in \dot \B^{2 + s - 2 / q}_{p , q} (\IR^n_+ ; \IC^n).$ Then, there exists a unique mild solution $u$ to the system
\begin{align*}
 \left\{ \begin{aligned}
  \partial_t u - \divergence(\IS (u)) &= f, && t \in \IR_+ ,\  x \in \IR^n_+ \\
  \IS (u) \cdot \e_n &= 0, && t \in \IR_+ ,\  x \in \partial \IR^n_+ \\
  u|_{t = 0} &= u_0, && x \in \IR^2_+
 \end{aligned} \right.
\end{align*}
with $u \in \cC(\IR_+  ; \dot \B^{2+s-2/q}_{p , q} (\IR^n_+ ; \IC^n))$ such that
\begin{align*}
 \partial_t u \in \LL_q(\IR_+ ; \dot \B^{s}_{p , q} (\IR^n_+;\IC^n)) , \quad \nabla^2 u \in \LL_q(\IR_+ ; \dot \B^{s}_{p , q} (\IR^n_+ ; \IC^{n^3}))
\end{align*}
 and the following inequality holds
 $$\displaylines{\quad
  \|u\|_{\LL_\infty(\IR_+;\dot \B^{2+s-2/q}_{p,q}(\IR^n_+ ; \IC^n))}+
  \|\partial_t u\|_{\LL_q(\IR_+ ; \dot \B^{s}_{p,q}(\IR^n_+ ; \IC^n))} + \| \nabla^2 u\|_{\LL_q(\IR_+ ; \dot \B^{s}_{p,q}(\IR^n_+ ; \IC^{n^{3}}))} \hfill\cr\hfill \leq C\Big( \| f\|_{\LL_q(\IR_+ ; \dot \B^{s}_{p,q}(\IR^n_+ ; \IC^n))} + \|u_0\|_{ \dot \B^{2+s-2/q}_{p,q}(\IR^n_+ ; \IC^n)}\Big)\cdotp}$$
\end{theorem}

\chapter{A free boundary problem for pressureless gases}
\label{Sec: A free boundary problem for pressureless gases}

As a second  application of the linear theory developed in the first part of this article, we  will investigate the free boundary problem for the model of compressible viscous flows of pressureless gases in
a vicinity of the half-space $\IR_+^n$, $n \geq 2$. This problem is described by the following set of equations
\begin{align}\label{eq:pressless}
 \left\{
  \begin{aligned}
   \rho(\partial_t v + v \cdot \nabla v) - \divergence \IS (v) &= 0, \qquad && t \in\IR_+,\  x \in \Omega_t, \\
  \partial_t\rho + \divergence (\rho v) &= 0, \qquad &&  t \in\IR_+ ,\  x \in \Omega_t, \\
   \IS (v) \overline{n} &= 0, \qquad && t \in \IR_+ ,\  x \in \partial \Omega_t, \\
      v \cdot \overline{n} &= - (\partial_t \eta) / \lvert \nabla_x \eta \rvert ,\qquad && t \in\IR_+,\  x \in \partial \Omega_t,\\
\rho|_{t=0}=\rho_0 ,\qquad   v|_{t = 0} & = v_0, \qquad && x \in \Omega, \\
   \Omega_t |_{t = 0} &= \Omega.
   \end{aligned} \right.
\end{align}
Here, the unknowns are the velocity field $v=v(t,x),$ the density  $\rho=\rho(t,x)$ and the time-dependent domain $\Omega_t$.
As in Chapter~\ref{Sec: The Lame operator in the upper half-plane}, $\IS(v)$ denotes the viscous stress tensor given by
\begin{equation}\label{eq:stresslame}
\IS (v):= \mu D(v) + \lambda \divergence v\,\Id\with  D(v) := \nabla v + [\nabla v]^{\top}
\end{equation}
and we  make the  ellipticity assumption 
\begin{equation}\label{eq:ellipticity}
\mu>0\andf 2\mu + n\lambda>0.
\end{equation}
The boundary of $\Omega_t$ is represented by some unknown function $\eta=\eta(t,x)$ through
$$
\partial \Omega_t=\bigl\{x\in\IR^n\,:\, \eta(t,x)=0\bigr\}\cdotp
$$ 
The outward unit normal vector to $\partial \Omega_t,$ 
denoted by  $\overline{n}$, is given by  $\overline{n} = \nabla_x \eta / \lvert \nabla_x \eta \rvert$. 
\medbreak
Our aim is to analyze the situation, where the initial domain $\Omega$ is close to the half-space $\IR^n_+$. By this, we mean that
there exists a function $h:\IR^{n-1} \to \IR$ with $h(0) = 0$ and
$\nabla_{z'}h\in \dot \B^{(n-1)/p}_{p,1}(\IR^{n-1})$ such that  
\begin{equation}\label{o1}
 \Omega=\bigl\{ z \in \IR^n: z_n>h(z') \bigr\}
\andf
 \|\nabla_{z'}h\|_{\dot \B^{(n-1)/p}_{p,1}(\IR^{n-1})} <c \ll 1.
\end{equation}
The chosen regularity for $h$ stems from the need to remain in the critical functional framework. This will eventually allow us to prove the desired $\LL_1$-in-time bounds on the velocity field.

\noindent
\section{The main result}

To solve system~\eqref{eq:pressless} we will perform a Lagrangian change of variables to rewrite the problem equivalently as a system of equations in a \emph{time independent domain}, the initial domain $\Omega$.
More precisely, we define the flow $X$ of $v$ to be the solution of the (integrated) ODE
\begin{equation}\label{eq:X1}
X(t,y)=y+\int_0^tv(\tau,X(\tau,y))\,\d\tau\end{equation}
and introduce the   `Lagrangian' density and velocity:
$$
q(t,y):=\rho(t,X(t,y))\andf u(t,y):=v(t,X(t,y)).
$$ 
Hence,~\eqref{eq:X1} yields
\begin{equation}\label{eq:X2}
X(t,y)=y+\int_0^tu(\tau,y)\,\d\tau.\end{equation}
As  explained in the appendix of~\cite{Danchin14}, the velocity equation of~\eqref{eq:pressless} is then reformulated as 
\begin{equation}\label{eq:presslag}
\left\{\begin{aligned}
q_t +q \divergence_u u &=0 && t \in\IR_+,  x \in \Omega, \\
\rho_0 \partial_t u + \divergence_u\IS_u(u)&=0,\qquad && t \in\IR_+,  x \in \Omega, \\
\IS_u(u)\cdot\overline n_u&=0,\qquad&&t\in\IR_+, x\in\partial\Omega, \\
q|_{t=0}=\rho_0, \qquad u|_{t=0}&=v_0 \qquad&& x\in \Omega, 
\end{aligned}\right.\end{equation}
with $\IS_u(w):= \mu(\nabla_vw+(\nabla_vw)^\top)+\lambda \divergence_vw\,\Id.$

For a  given time-dependent vector-field $u,$ the operators $\divergence_u$ and $\nabla_u$ are defined as follows: setting $A_u:=(DX_u)^{-1}$ with
 \begin{equation}\label{eq:DXv}
 DX_u(t,y)=\Id+\int_0^t Du(\tau,y)\,\d\tau \mbox{ \ and  \ }  J_u(t,y):=\det DX_u(t,y),
\end{equation}
we obtain 
$$\nabla_u:=A_u^\top\nabla_y\andf
\divergence_u:= A_u^\top:\nabla_y= J_u^{-1} \divergence_y(A_u J_u\cdot)=
J_u^{-1} \divergence_y(\adj(DX_u) \cdot),$$
where $\adj(DX_u)$ stands for the adjugate matrix of $DX_u$. 
\medbreak
We may compute the `Lagrangian density' $q$ through
$$J_u(t,y) q(t,y)= \rho_0(y). $$
Then system~\eqref{eq:presslag} describes the transformed equations multiplied by the Jacobian $J_u$.
Thanks to this observation we obtain a simpler form of this system. As in the classical setting,~\eqref{eq:presslag}  takes the form 
\begin{equation}
 q u_t - A_u^{\top} :\nabla_y \IS_u(u)=0.
\end{equation}

We will always work in a regime where $\nabla u$ is small in $\LL_1(\IR_+;\LL_\infty(\Omega))$.  Consequently, 
all the quantities like $A_u-\Id,$  $\adj(DX_u)-\Id$ and $J_u-1$ may be computed by Neumann series expansions and, at first order, we have for instance
$$
\adj(DX_u(t,y))-\Id\approx A_u(t,y) - \Id \approx \int_0^t Du(t^\prime,y) \,\d t^\prime\cdotp
$$
Roughly speaking, our approach will yield a global solution result within the $\LL_1$-setting provided 
$$\IS_u(w)-\IS(w) \approx  \Bigl(\int_0^t Du\,\d t^\prime\Bigr)\cdot \nabla w $$
and the outward unit normal vector $\overline n_u(t,\cdot)|_{\partial \Omega}$
to the moving boundary $X(t,\cdot)|_{\partial\Omega}$ is such that
$$|\overline{n}_u- n_y| \approx \Bigl|\int_0^t Du \, \d t'\Bigr|,$$ where $n_y$ denotes the outward unit normal vector to $\partial \Omega$.

The main result of this section is the following global existence result for system~\eqref{eq:presslag}.

\begin{theorem}\label{thm:lpressl}
Let $\Omega$ be a small perturbation of the half-space $\IR^n_+$ in the sense  of~\eqref{o1}, with $n-1 < p < n.$  
Assume that $v_0\in\dot \B^{n/p-1}_{p,1}(\Omega)$ and that $\rho_0\in\LL_\infty(\Omega)\cap\cM(\dot \B^{n/p-1}_{p,1}(\Omega))$ (where $\cM(X)$ stands for the multiplier space of $X$) 
such that 
\begin{equation}\label{eq:smalldata}
\|\rho_0-1\|_{\LL_\infty(\Omega)\cap\cM(\dot \B^{n/p-1}_{p,1}(\Omega))} + \|v_0\|_{\dot \B^{n/p-1}_{p,1}(\Omega)}\leq c \ll 1.
\end{equation}
Then, system~\eqref{eq:presslag} has a unique, global solution $u$  in the maximal regularity space $\E_p(\Omega)$  defined by 
$$\displaylines{u\in\cC_b(\IR_+;\dot \B^{n/p-1}_{p,1}(\Omega)),\quad
\partial_t u\in\LL_1(\IR_+;\dot \B^{n/p-1}_{p,1}(\Omega)),\quad
\nabla u\in\LL_1(\IR_+;\dot \B^{n/p}_{p,1}(\Omega))\cr\andf
Du|_{\partial\Omega}\in \LL_1\bigl(\IR_+;\dot \B^{\frac{n-1}p}_{p,1}(\partial\Omega)\bigr)\cap 
\dot\B^{\frac{n-1}{2p}}_{1,1}\bigl(\IR_+;\dot\B^{0}_{p,1}(\partial\Omega)\bigr)\cdotp}$$
Furthermore, there exists a constant $C>0$ such that
$$\|u\|_{\E_p(\Omega)}\leq C\|v_0\|_{\dot \B^{n/p-1}_{p,1}(\Omega)}.$$
\end{theorem}

Returning to Eulerian coordinates yields the following result. We note that in the case of a bent half-space $\Omega$ as in~\eqref{o1}, the sets $\Omega_t$ and $\partial \Omega_t$ are given for each $t>0$ by 
$$
\Omega_t=X_v(t,\Omega) \mbox{ and } \partial\Omega_t = X_v(t,\partial\Omega)
$$
and $X_v$ defined by 
$$
X_v(t,x) = x + \int_0^t v(\tau,X_v(\tau,x))d\tau
$$
is  required to satisfy 
$$
\nabla X_v - \Id \in \cC(\IR_+ ; \dot \B^{n/p}_{p , 1} (\Omega)).
$$
A global solution $(v,\rho, \Omega_t)$ to equation~\eqref{eq:pressless}  in the case of a 
bent half-space is then defined as in the definition formulated in the introduction. 

\begin{theorem}\label{Thm:pressl} Let $p\in(n-1,n)$, $\Omega$ as in~\eqref{o1}, 
 $v_0 \in \dot \B^{n/p-1}_{p,1}(\Omega)$ and $\rho_0 \in \LL^\infty(\Omega) \cap  
 \cM(\dot \B^{n/p-1}_{p,1}(\Omega))$, where $\cM(X)$ denotes the multiplier  space of $X$. 
 Then there exists a constant $c>0$ such that if  
$$
\|\rho_0-1\|_{\LL_\infty(\Omega)  \cap \cM(\dot \B^{n/p-1}_{p,1}(\Omega))}  
 +  \|v_0\|_{\dot \B^{n/p-1}_{p,1}(\Omega)}\leq c,
 $$
then system~\eqref{eq:pressless} admits a unique, global  solution  $(v,\rho,\Omega_t)$ and there exists a constant $C>0$ such that 
$$ 
\sup_{t\geq 0} \|\rho-1\|_{\LL_\infty(\Omega_t)}\leq Cc.
$$
\end{theorem}

\begin{remark} As pointed out in~\cite{DM-CPAM},  the interest of assuming that $\rho_0$ belongs to $\cM(\dot \B^{n/p-1}_{p,1}(\Omega))$ lies in the fact that one may consider 
initial densities which are  not continuous along a $\cC^1$ interface. The smallness condition means, however,  that the jump across the interface is small. 
\end{remark}

The proof of the main result will strongly rely on the study of the linearization of~\eqref{eq:presslag} about $\IR^n_+$ that  
 will be performed in the next section. Then, by perturbation methods it will be extended to a rigid
domain $\Omega$.

\section{The boundary value problem for the Lam\'e system in the half-space}

This section is devoted to proving global-in-time a priori estimates 
for the solutions to the following linear initial boundary value problem:
\begin{equation}\label{eq:lame}
\left\{\begin{aligned} 
\partial_t u - \divergence\IS(u) &= f \quad &\hbox{in }\ & \IR_+\times\IR^n_+, \\
\IS(u) \e_n|_{\partial\IR^n_+}&= g \quad &\hbox{on }\ & \IR_+\times\partial\IR^n_+,\\
u|_{t = 0} &= u_0  \quad &\hbox{in }\ & \IR^n_+\end{aligned} \right.
\end{equation}
with $\IS$ defined as in~\eqref{eq:stresslame}.

We aim at proving the following result.

\begin{theorem}\label{Thm:lame} 
Let $1 < p < \infty$ and $0 <s<1/p$ with $s\leq n/p-1$.
Then, there exists a constant $C > 0$ such for all data 
$u_0\in  \dot \B^s_{p , 1} (\IR^{n}_+)$, $f\in \LL_1(\IR_+;\dot\B^s_{p , 1} (\IR^{n}_+))$ and   
\begin{align*}
 g \in \dot\Y^s_p:=\dot \B^{\frac{1}{2}(s+1-\frac1p)}_{1 , 1} (\IR_+ ; \dot \B^0_{p , 1} (\partial\IR^{n}_+)) \cap\dot \B^{\frac{1}{2}(1 - \frac{1}{p})}_{1 , 1} (\IR_+ ; \dot \B^s_{p , 1} (\partial\IR^{n}_+)) \cap \LL_1 (\IR_+ ; \dot \B^{s + 1 - \frac{1}{p}}_{p,1} (\partial\IR^{n}_+)),
\end{align*}
System~\eqref{eq:lame}  has 
a unique solution $u\in\cC_b(\IR_+;\dot\B^s_{p,1}(\IR^n_+))$
with \begin{multline}\label{est:lame} \|u\|_{\E^s_p}:=\| \partial_tu\|_{\LL_1 (\IR_+ ; \dot \B^s_{p , 1} (\IR^n_+))} 
+\|\nabla u \|_{\LL_1 (\IR_+ ; \dot \B^{s+1}_{p , 1} (\IR^n_+))} 
+\|u \|_{\LL_\infty (\IR_+ ; \dot \B^s_{p , 1} (\IR^n_+))} +
\|\nabla u|_{\partial\IR^n_+}\|_{\dot\Y^s_p}\\[5pt]
\leq C\bigl(\|u_0\|_{\dot \B^s_{p , 1} (\IR^{n}_+)}
+\|f \|_{\LL_1 (\IR_+ ; \dot \B^s_{p , 1} (\IR^n_+))} +
\| g \|_{\dot\Y^s_p}\bigr)\cdotp
\end{multline}
Furthermore, in the case $u_0\equiv0$ and $f\equiv0,$
the result holds for all $s\in(-1+1/p,1/p)$ such that
$s\leq n/p-1.$ 
\end{theorem}
\begin{remark}  The condition  $s>0$ for nonzero data $u_0$ and $f$  comes 
from the fact that Theorem~\ref{Thm: Full Da Prato - Grisvard} was applied to the Lam\'e operator on the ground space $X = \LL_p$ leading to positive regularity exponents. 
 Notice that in the situation of a positive regularity index $s$, the space in the middle in the in the definition of $\dot \Y_p^s$ is redundant in the sense that a slight generalization of the interpolation result
of Proposition~\ref{p:interpopopo} gives the middle space as an interpolation of the first and the last space. 

For similar reasons, in the case $s<0$ (that can be considered only if $u_0\equiv0$ and $f\equiv0$ so far),  
 one can keep only the second and last spaces in the definition of $\dot\Y_p^s.$ This exactly fits with 
 the corresponding statement in~\cite{OS}.
\end{remark}
In order to handle the case $u_0\equiv0$ and $f\equiv0$, let us start with the simpler case of the heat 
equation in the half-space with Neumann boundary conditions, namely 
\begin{align*}
 \left\{ \begin{aligned}
 \partial_t u - \Delta u &= 0 && \text{in } \IR_+ \times \IR^n_+ \\
  \partial_{x_n} u &= g && \text{on }  \IR_+\times \partial \IR^n_+\\
  u|_{t=0}&=0&&\text{in } \IR_+^n.
\end{aligned} \right.
\end{align*}
To determine a solution converging  to $0$ at infinity, we employ  the partial Fourier transform in $\IR_t \times \IR^{n - 1}_{x^{\prime}}$.
To this end, it is suitable to first extend  $u$ and $g$ by $0$  to the whole time axis $\IR$. Keeping the same notation for the two functions, we obtain 
\begin{align*}
 \left\{ \begin{aligned}
 \partial_t u - \Delta u &= 0 && \text{in } \IR \times \IR^n_+ \\
  \partial_{x_n} u &= g && \text{on } \IR \times \partial \IR^n_+.
\end{aligned} \right.
\end{align*}
In  Fourier space with respect to $t$ and $x',$  the first equation reads as  
\begin{align*}
\begin{aligned}
 i \tau \cF_{t , x^{\prime}} u + \lvert \xi^{\prime} \rvert^2 \cF_{t , x^{\prime}} u - \partial_{x_n}^2 \cF_{t , x^{\prime}} u = 0\quad\hbox{for }\  \tau \in \IR,\;  \xi^{\prime} \in \IR^{n - 1},\;  x_n \in \IR_+. 
\end{aligned}
\end{align*}
Hence, as we only look for solutions converging  to $0$ for $x_n\to+\infty,$ the function $\cF_{t , x^{\prime}} u$ is of the form
\begin{align*}
 \cF_{t , x^{\prime}} u (\tau,\xi',x_n)= A(\tau,\xi')\, \e^{- r x_n}, 
\end{align*}
where  $r$ is defined to be the square root of $ \ii \tau + \lvert \xi^{\prime} \rvert^2$ with non-negative real part. By elementary trigonometry, it follows that
\begin{equation}\label{eq:r}
\lvert \arg(r) \rvert \leq \frac{\pi}{4}, \qquad
 \Re(r) \leq \lvert r \rvert \leq \frac{\Re(r)}{\cos(\pi / 4)} \andf \lvert r \vert^2 = \lvert r^2 \rvert = \lvert \tau \rvert + \lvert \xi^{\prime} \rvert^2\cdotp
\end{equation}

Taking the derivative in direction $x_n$ and fitting with the boundary data yields
\begin{align*}
 u =- \cF^{-1}_{\tau , \xi^{\prime}} r^{-1} \e^{- r x_n} \cF_{t , x^{\prime}} g.
\end{align*}
To estimate $\partial_t u$ in $\LL_1 (\IR ; \dot \B^s_{p , 1} (\IR^n_+))$, we consider the extension 
\begin{align}\label{Eq: Solution formula}
 E u :=- \cF^{-1}_{\tau , \xi^{\prime}} r^{-1} \e^{- r \lvert x_n \rvert} \cF_{t , x^{\prime}} g \quad \text{for} \quad x_n \in \IR
\end{align}
of $u$ to the whole space. Our aim is  to estimate the time derivative of $Eu$ in $\LL_1 (\IR ; \dot \B^s_{p , 1} (\IR^n))$.

\subsection*{The tensored Littlewood--Paley decomposition}

Let  $\chi : \IR \to \IR$ be a smooth function supported in the ball $B(0 , 4 / 3)$ and having the value $1$ on $B(0 , 3 / 4)$. Set $\psi := \chi (\cdot / 2) - \chi (\cdot)$. 
With this definition we find that 
\begin{align}
\label{Eq: Support of cutoff}
 \supp (\psi) \subset \{ \xi \in \IR : 3 / 4 \leq \xi \leq 8 / 3 \}.
\end{align}
Furthermore, we have
\begin{align*}
 \sum_{k = - \infty}^{\infty} \psi(2^{- k} \lvert \xi^{\prime} \rvert) = 1 \quad \text{for all } \xi^{\prime} \in \IR^{n - 1} \setminus \{ 0 \}
\end{align*}
and
\begin{align*}
 \sum_{k = - \infty}^{\infty} \psi(2^{- k} \xi_n) = 1 \quad \text{for all } \xi_n \in \IR \setminus \{ 0 \}.
\end{align*}
Consequently, for $\xi^{\prime} \in \IR^{n - 1} \setminus \{ 0 \}$ and $\xi_n \in \IR \setminus \{ 0 \},$ we have
\begin{align*}
 1 
 &= \sum_{k = - \infty}^{\infty} \sum_{\ell = - \infty}^k \psi (2^{- k} \lvert \xi^{\prime} \rvert) \psi(2^{- \ell} \xi_n) + \sum_{k = - \infty}^{\infty} \sum_{\ell = k + 1}^{\infty} \psi (2^{- k} \lvert \xi^{\prime} \rvert) \psi(2^{- \ell} \xi_n).
\end{align*}
An application of Tonelli's theorem in the second series and some notational rephrasing in the first series yields
\begin{align*}
 1 &= \sum_{k = - \infty}^{\infty} \psi (2^{- k} \lvert \xi^{\prime} \rvert) \bigg( \sum_{\ell = - \infty}^0 \psi(2^{- \ell} \cdot) \bigg) (2^{- k} \xi_n) + \sum_{\ell = - \infty}^{\infty} \sum_{k = - \infty}^{\ell - 1} \psi (2^{- k} \lvert \xi^{\prime} \rvert) \psi(2^{- \ell} \xi_n) \\
 &= \sum_{k = - \infty}^{\infty} \psi (2^{- k} \lvert \xi^{\prime} \rvert) \bigg( \sum_{\ell = - \infty}^0 \psi(2^{- \ell} \cdot) \bigg) (2^{- k} \xi_n) + \sum_{\ell = - \infty}^{\infty} \psi(2^{- \ell} \xi_n) \bigg(\sum_{k = - \infty}^{0} \psi (2^{- k} \cdot) \bigg) (2^{- (\ell - 1)} \lvert \xi^{\prime} \rvert).
\end{align*}
Now, for $x \in \IR \setminus \{ 0 \}$ we have
\begin{align*}
 \sum_{k = - L}^0 \psi (2^{- k} x) &= \sum_{k = - L}^0 (\chi(2^{- k - 1} x) - \chi (2^{- k} x)) \\
 &= \chi(x / 2) - \chi(2^L x) \to \chi(x / 2) \quad \text{as} \quad L \to \infty.
\end{align*}
Thus, by continuity we  may extend $\sum_{k = - \infty}^{0} \psi (2^{- k} \cdot)$ to be the function $\chi(\cdot / 2)$ on \textit{all of} $\IR$ satisfying  for all $\xi \in \IR^n \setminus \{ 0 \}$ 
the identity
\begin{align}
\label{Eq: Tensored Littlewood-Paley decomposition}
\begin{aligned}
 1 &= \sum_{k = - \infty}^{\infty} \psi (2^{- k} \lvert \xi^{\prime} \rvert) \chi(2^{- (k + 1)} \xi_n) + \sum_{k = - \infty}^{\infty} \psi (2^{- k} \xi_n) \chi(2^{- k} \lvert \xi^{\prime} \rvert) \\
 &= \sum_{k = - \infty}^{\infty} \big\{ \psi (2^{- k} \lvert \xi^{\prime} \rvert) \chi(2^{- (k + 1)} \xi_n) + \psi (2^{- k} \xi_n) \chi(2^{- k} \lvert \xi^{\prime} \rvert) \big\}.
\end{aligned}
\end{align}
Defining the function
\begin{equation}\label{eq:Phi}
 \Phi (\xi) := \psi (\lvert \xi^{\prime} \rvert) \chi(\xi_n / 2) + \psi (\xi_n) \chi (\lvert \xi^{\prime} \rvert).
\end{equation}
the identity~\eqref{Eq: Tensored Littlewood-Paley decomposition} turns into
\begin{align*}
 1 = \sum_{k = - \infty}^{\infty} \Phi(2^{- k} \xi).
\end{align*}
Moreover, notice that $\Phi$ is smooth and that
\begin{align*}
 \supp(\Phi) & \subset \{ (\xi^{\prime} , \xi_n) \in \IR^n : 3 / 4 \leq \lvert \xi^{\prime} \rvert \leq 8 / 3 \text{ and } \lvert \xi_n \rvert \leq 8 / 3 \} \\
 \quad & \cup \{ (\xi^{\prime} , \xi_n) \in \IR^n : \lvert \xi^{\prime} \rvert \leq 4 / 3 \text{ and } 3 / 4 \leq \lvert \xi_n \rvert  \leq 8 / 3 \}.
\end{align*}
Hence,  the support of $\Phi$ has empty intersection
with the ball $B(0,3/4) \subset \IR^n$ and is bounded.
\smallbreak
Next, we associate to  any suitable function $A$ the multiplier operator $A(D_x):=\cF_\xi^{-1} A\cF_x$ and define the Besov space norm on $\IR^n$ of a function $f$  
as (note that the definition of Besov norms is independent of the choice of a Littlewood-Paley decomposition) 
\begin{align}\label{Eq: Besov space on IRn}
 \| f \|_{\dot \B^s_{p , q} (\IR^n)} := \big\| \big(2^{s k} \|\Phi (2^{- k}D_x)f \|_{\LL_p (\IR^n)}\big)_{k \in \IZ} \big\|_{\ell_q (\IZ)}.
\end{align}
Defining
$$\begin{aligned}
&\dot S_k^h:=\cF_{\xi'}^{-1} \chi(2^{-k}|\xi'|) \cF_{x'},
\quad &\dot S_k^v:=\cF_{\xi_n}^{-1} \chi(2^{-k}\xi_n) \cF_{x_n},\\
&\dot\Delta_k^h:=\cF_{\xi'}^{-1} \psi(2^{-k}|\xi'|) \cF_{x'},
\quad &\dot\Delta_k^v:=\cF_{\xi_n}^{-1} \psi(2^{-k}\xi_n) \cF_{x_n},\end{aligned}$$
relation~\eqref{eq:Phi}  translates after rescaling into 
\begin{equation}\label{eq:decompocompo}
\Phi(2^{-k}D_x)=\dot S_{k+1}^v\dot\Delta_k^h+\dot\Delta_k^v\dot S_k^h\quad\hbox{for all }\ k\in\IZ.
\end{equation}
Since we also need to perform the Littlewood-Paley decomposition with respect to the time variable, we further define  
$$
\dot\Delta_k^t:=\cF_{\tau}^{-1} \psi(2^{-k} \tau) \cF_{t}
\andf\dot S_k^t:=\cF_{\tau}^{-1} \chi(2^{-k} \tau) \cF_t. 
$$

\subsection*{Generalizing the symbol}

A glimpse onto~\eqref{Eq: Solution formula} reveals that an extension of the time derivative of $u$ is given by
\begin{align}\label{eq:Eu}
 \partial_t E u = E \partial_t u = -\cF_{\tau , \xi^{\prime}}^{-1} \ii \tau r^{-1} \e^{- r \lvert x_n \rvert} \cF_{t , x^{\prime}} g.
\end{align}
Observe that~\eqref{eq:Eu} may be rewritten as 
\begin{align}
\label{Eq: Fourier multiplier operator}
 E\partial_tu=m(D_{t,x'},x_n) g:= \cF_{\tau , \xi^{\prime}}^{-1} m (\tau , \xi^{\prime} , x_n) \cF_{t , x^{\prime}} g,
\end{align}
where the  symbol $m$ is given by
\begin{equation}\label{eq:m}
 m (\tau , \xi^{\prime} , x_n) =- \frac{\ii \tau}{r} \e^{- r \lvert x_n \rvert}.
\end{equation}
We now  perform estimates of $m(\tau,\xi^\prime,\cdot)$ in $\LL_p(\IR)$ and in $\dot\B^s_{p,1}(\IR).$

Proving an $\LL_p(\IR)$ estimate is straightforward:
taking advantage of~\eqref{eq:r}, we see that we have for some constant $c > 0,$ 
\begin{align*}
 \lvert \e^{- r \lvert x_n \rvert} \rvert \leq \e^{- c (\lvert \tau \rvert^{1 / 2} + \lvert \xi^{\prime} \rvert) \lvert x_n \rvert}.
\end{align*}
As a consequence, there exists  some constant $C > 0$ such that  for all $1 \leq p \leq \infty,$
\begin{align*}
 \| m(\tau,\xi^\prime,\cdot)  \|_{\LL_p(\IR)} \leq {C|\tau|}{(\lvert \tau \rvert^{1 / 2} + \lvert \xi^{\prime} \rvert)^{-1-1 / p}}.
\end{align*}
We claim that  for all $z\in\IC \setminus \{0\}$ 
with $|\arg z|\leq \pi/4$  and  for all $-1+1/p<s < 1 + {1}/{p},$ 
 the function  $x_n \mapsto \e^{- z\lvert x_n \rvert}$  lies in $\dot \B^s_{p , 1} (\IR)$
  and that there exists a constant $C$, depending only on $p,s$, such that 
  \begin{equation}\label{eq:ezxn}
  \| \e^{- z\lvert \cdot \rvert}\|_{\dot\B^s_{p,1}(\IR)}\leq C |z|^{s-1/p}.\end{equation}
Indeed, decomposing $z$ into $z=|z| \e^{\ii\theta},$ the well-known scaling properties
of the homogeneous Besov space $\dot\B^s_{p,1}(\IR)$ guarantee that 
  \begin{equation}\label{eq:ezxn1}\| \e^{- z\lvert \cdot \rvert}\|_{\dot\B^s_{p,1}(\IR)}\approx 
   |z|^{s-1/p} \| h_\theta\|_{\dot\B^s_{p,1}(\IR)}
   \with h_\theta(x_n):= \e^{- \e^{\ii\theta}\lvert x_n \rvert}.\end{equation}
It is obvious that $h_\theta$ is in $\LL_p(\IR)$ with norm 
that can be bounded independently of $\theta$ if $|\theta|\leq \pi/4,$ 
and the same property holds for   $h'_\theta,$ as we just have
$$h'_\theta(x_n) =-\sgn x_n\, \e^{\ii \theta} h_\theta(x_n).$$
Hence, interpolating yields
$$
\|h_\theta\|_{\dot\B^s_{p,1}(\IR)}\leq C_{s,p} \quad\hbox{for all }\ s\in(0,1).
$$
For $s\in(1,1+1/p),$  using the characterization of $\dot\B^{s-1}_{p,1}$
by finite differences allows to show that $h'_\theta$ is 
in  $\dot\B^{s-1}_{p,1}$ uniformly with respect to $\theta.$

Moreover, observe that $h_\theta$ is in $\LL_1(\IR)$ (uniformly in $\theta \in [-\pi/4,\pi/4]$)
and thus in $\dot\B^{-1+1/p}_{p,\infty}(\IR)$ by embedding. 
Putting all these results together and interpolating whenever it is needed, 
we conclude to the validity~\eqref{eq:ezxn} for $|z|=1$ and then to the general case, thanks to~\eqref{eq:ezxn1}.

Reverting to the definition of  our original symbol $m$ in~\eqref{eq:m}, we get, for a constant $C$ independent of $\tau$ and $\xi'$,
\begin{align*}
 \| m(\tau , \xi^{\prime} , \cdot) \|_{\dot \B^s_{p , 1} (\IR)} \leq C \lvert \tau \rvert 
  \big( \lvert \tau \rvert^{\frac{1}{2}} + \lvert \xi^{\prime} \rvert \big)^{s-1- \frac{1}{p}},\qquad
  -1+\frac 1p<s<1+\frac 1p\cdotp
\end{align*}
From explicit computation, we  get a  natural extension of the above estimates for $m$  derivatives with respect to $\tau$ and $\xi^{\prime}$,     
namely,  for all $\tau \in \IR \setminus \{ 0 \}$, $\xi^{\prime} \in \IR^{n - 1}$, $j \in \{0 , 1 , 2\}$, and $\alpha \in \IN_0^n$,
\begin{equation}\label{Eq: Lp norm of multiplier}
\Big\| \partial_{\xi^{\prime}}^{\alpha} \partial_{\tau}^j m (\tau , \xi^{\prime} ,\cdotp) \Big\|_{\LL^p_{x_n} (\IR)}\leq C \lvert \tau \rvert^{1-j }\lvert \xi^{\prime} \rvert^{\lvert \alpha \rvert}   
\big( \lvert \tau \rvert^{1 / 2} + \lvert \xi^{\prime} \rvert \big)^{-1-\frac{1}{p}-2|\alpha|},
\end{equation}
and, if  $-1+1/p<s<1+1/p,$ 
\begin{equation}\label{Eq: Besov norm of multiplier}
\Big\| \partial_{\xi^{\prime}}^{\alpha} \partial_{\tau}^j m (\tau , \xi^{\prime} , \cdotp) \Big\|_{\dot \B^s_{p , 1} (\IR)} \leq C  \lvert \tau \rvert^{1 - j} \lvert \xi^{\prime} \rvert^{\lvert \alpha \rvert } \big( \lvert \tau \rvert^{1 / 2} + \lvert \xi^{\prime} \rvert \big)^{s-1-\frac{1}{p}-2|\alpha|}. 
\end{equation}

\noindent
For any symbol  $m$ satisfying the above two properties,  we have  the following proposition.

\begin{proposition}\label{p:m}
Let $1 < p < \infty$ and $s \in \IR$. There exists a constant $C > 0$ such that for all
\begin{align*}
 g \in E\dot\Y_p^s:=\dot \B^{\frac{1}{2}(s+1 - \frac{1}{p})}_{1 , 1} (\IR ; \dot \B^0_{p , 1} (\IR^{n - 1})) \cap 
 \dot \B^{\frac{1}{2} - \frac{1}{2 p}}_{1 , 1} (\IR ; \dot \B^s_{p , 1} (\IR^{n - 1})) \cap \LL_1 (\IR ; \dot \B^{s + 1 - \frac{1}{p}}_{p,1} (\IR^{n - 1})),
\end{align*}
we have
\begin{align*}
 \| m(D_{t,x'},x_n) g \|_{\LL_1 (\IR ; \dot \B^s_{p , 1} (\IR^n))} \leq C \| g \|_{E\dot\Y^s_p}.
\end{align*}
\end{proposition}

\begin{proof}
Let $v:= m(D_{t,x'},x_n) g.$ We start by applying for fixed $t \in \IR$ the Besov space norm defined by~\eqref{Eq: Besov space on IRn}  to $v(t)$.  Using~\eqref{eq:decompocompo}, this yields
\begin{align*}
 \| v(t) \|_{\dot \B^s_{p , 1} (\IR^n)} &= \sum_{k = - \infty}^{\infty} 2^{s k} \| \Phi(2^{-k}D_x)v(t) \|_{\LL_p (\IR^n)} \\
 &\leq \sum_{k = - \infty}^{\infty} 2^{s k}\|\dot S^v_{k+1}\dot\Delta_k^h\:
 m(D_{t,x'},x_n)g \|_{\LL_p (\IR^n)}  + 
   \sum_{k = - \infty}^{\infty} 2^{s k}\|\dot S^h_{k}\dot\Delta_k^v\:
 m(D_{t,x'},x_n)g \|_{\LL_p (\IR^n)} 
    \\
 &=: M (t) + Y (t).
\end{align*}
Next, we decompose $g$ by using Littlewood-Paley decompositions~\eqref{Eq: Besov space on IRn} in the time  and $x^\prime$ variables $(\ell,j)$. This results in
\begin{eqnarray}\label{eq:M} M(t) &\!\!\!\leq\!\!\!& \sum_{k = - \infty}^{\infty} \sum_{j = - \infty}^{\infty} \sum_{\ell = - \infty}^{\infty} 2^{s k}  \| \dot\Delta^h_j \dot\Delta^t_\ell \dot\Delta_k^h\dot S_{k+1}^v
  m(D_{t,x'},x_n)g \|_{\LL_p (\IR^n)},\\\label{eq:Y}
 Y(t) &\!\!\!\leq\!\!\!& \sum_{k = - \infty}^{\infty} \sum_{j = - \infty}^{\infty} \sum_{\ell = - \infty}^{\infty} 2^{s k} 
 \| \dot\Delta_j^h \dot\Delta_\ell^t\dot S_k^h\dot\Delta_k^v m(D_{t,x'},x_n)g \|_{\LL_p (\IR^n)}.\end{eqnarray}
To  handle  the latter two terms, we will use the facts that  
\begin{equation}\label{eq:LP1}
\dot\Delta^h_j\dot\Delta^h_k=  0 \quad \text{for all} \quad \lvert k - j \rvert \geq 2
\end{equation}
and that
\begin{equation}\label{eq:LP2}
\dot\Delta^h_j\dot S_k^h=0 \quad \text{for all} \quad j \geq k + 1.
\end{equation}
Proving  the proposition reduces to  deriving estimates of the form
\begin{align*}
 \int_{\IR} M(t) \; \d t \leq C \| g \|_{E\dot\Y^s_p} \andf \int_{\IR} Y(t) \; \d t \leq C \| g \|_{E\dot\Y^s_p}.
\end{align*}
Let us first consider $M.$ In light of~\eqref{eq:LP1}, it follows that $M(t)\leq M_1(t)+M_2(t)$ with 
  $$  \begin{aligned}
  M_1(t)&:=  \sum_{k = - \infty}^{\infty} \sum_{|j-k|\leq1} \sum_{\ell = - \infty}^{2k} 2^{s k}  \| \dot\Delta^h_j \dot\Delta^t_\ell\dot\Delta_k^h \dot S_{k+1}^v  m(D_{t,x'},x_n)g \|_{\LL_p (\IR^n)}\\
 \andf
  M_2(t)&:=  \sum_{k = - \infty}^{\infty} \sum_{|j-k|\leq1} \sum_{\ell = 2k+1}^{\infty} 2^{s k} 
   \| \dot\Delta^h_j \dot\Delta^t_\ell\dot\Delta_k^h \dot S_{k+1}^vm(D_{t,x'},x_n)g \|_{\LL_p (\IR^n)}.
   \end{aligned}$$
Write the first  expression in terms of convolutions with respect to $(t,x^\prime)$  as
$$
  \dot\Delta^h_j \dot\Delta^t_\ell\dot\Delta_k^h \dot S_{k+1}^v
  m(D_{t,x'},x_n)g  = \cF^{-1}_{\tau , \xi^{\prime}} \psi (2^{- k} \lvert \xi^{\prime} \rvert) \psi(2^{- \ell} \tau) \cF_{\xi_n}^{-1} \chi (2^{- (k + 1)} \xi_n) \cF_{x_n} m(\tau , \xi^{\prime} , x_n) *_{t , x^{\prime}}  \dot\Delta_j^h g.$$
  For $M_2,$ we further  introduce  the function 
\begin{equation}\label{eq:theta}
\theta(t):=\psi(t/2)+\psi(t)+\psi(2t)
\end{equation}
and observe  that  $\theta\psi=\psi,$ 
so that one can shift $\dot\Delta_\ell^t$ onto $g$ as follows:
$$
\displaylines{\quad \dot\Delta^h_j \dot\Delta^t_\ell \dot\Delta_k^h\dot S_{k+1}^v  m(D_{t,x'},x_n)g \hfill\cr\hfill
= \cF^{-1}_{\tau , \xi^{\prime}} \psi (2^{- k} \lvert \xi^{\prime} \rvert) \theta(2^{-\ell} \tau) \cF_{\xi_n}^{-1} \chi (2^{- (k + 1)} \xi_n) \cF_{x_n} m(\tau , \xi^{\prime} , x_n)  *_{t , x^{\prime}}\dot
\Delta_\ell^t\dot\Delta_j^h g.\quad}
$$
Let $G$ denote a function depending on the variables $t$, $x^{\prime}$ and $x_n$ and let $H$ denote a function depending on the variables $t$ and $x^{\prime}$. Then, Minkowski's inequality 
and Young's inequality for convolutions imply
\begin{align*}
 \| G *_{t , x^{\prime}} H \|_{\LL_p (\IR^n)} \leq \| \| G \|_{\LL^p_{x_n} (\IR)} *_{t , x^{\prime}} \lvert H \rvert \|_{\LL^p_{x^{\prime}} (\IR^{n - 1})} \leq \| G \|_{\LL^1_{x^{\prime}} (\IR^{n - 1} ; \LL^p_{x_n} (\IR))} *_t \| H \|_{\LL^p_{x^{\prime}} (\IR^{n - 1})}.
\end{align*}
Integrating this inequality and employing Young's inequality for convolutions again delivers
\begin{align}
\label{Eq: Multiple Young}
 \| G *_{t , x^{\prime}} H \|_{\LL_1(\IR ; \LL_p (\IR^n))} \leq \| G \|_{\LL^1_{t , x^{\prime}} (\IR \times \IR^{n - 1} ; \LL^p_{x_n} (\IR))} \| H \|_{\LL^1_t (\IR ; \LL^p_{x^{\prime}} (\IR^{n - 1}))}.
\end{align}
Applying~\eqref{Eq: Multiple Young} to the expressions of $M_1$ and $M_2$ leads to
$$\displaylines{
 \int_{\IR} M_1(t) \; \d t\leq \sum_{k = - \infty}^{\infty} \sum_{\lvert j - k \rvert \leq 1} 
  \sum_{\ell = - \infty}^{2 k}2^{s k} \hfill\cr\hfill \cdot\| \cF^{-1}_{\tau , \xi^{\prime}} \psi (2^{- k} \lvert \xi^{\prime} \rvert) \psi(2^{- \ell} \tau) \cF_{\xi_n}^{-1} \chi (2^{- (k + 1)} \xi_n) \cF_{x_n} m(\tau , \xi^{\prime} , x_n) \|_{\LL^1_{t , x^{\prime}} (\IR \times \IR^{n - 1} ; \LL^p_{x_n} (\IR))}
  \| \dot\Delta_j^h g\|_{\LL^1_t (\IR ; \LL^p_{x^{\prime}} (\IR^{n - 1}))}}$$
  $$\displaylines{\andf \int_{\IR} M_2(t) \; \d t  \leq  \sum_{k = - \infty}^{\infty} \sum_{\lvert j - k \rvert \leq 1} 
\sum_{\ell = 2 k + 1}^{\infty}2^{sk}\hfill\cr\hfill\cdot
 \| \cF^{-1}_{\tau, \xi^{\prime}} \psi (2^{- k} \lvert \xi^{\prime} \rvert) \theta(2^{-\ell} \tau) \cF_{\xi_n}^{-1} \chi (2^{- (k + 1)} \xi_n) \cF_{\!x_n} m(\tau , \xi^{\prime} , x_n)\|_{\LL^1_{t , x^{\prime}} (\IR \times \IR^{n - 1} ; \LL^p_{x_n} (\IR))}
 \| \Delta_\ell^t \dot\Delta_j^h g\|_{\LL^1_t (\IR ; \LL^p_{x^{\prime}} (\IR^{n - 1}))}.}$$
At this point we already suspect that the series for $\ell \leq 2 k$ is estimated by $g$ in $\LL_1 (\IR ; \dot \B^{s + 1 - 1 / p}_{p , 1} (\IR^{n - 1}))$ and that the series for 
$\ell \geq 2 k + 1$ is eventually estimated by $g$ in $\dot \B^{1 / 2 - 1 / (2 p)}_{1 , 1} (\IR ; \dot \B^s_{p , 1} (\IR^{n - 1}))$. To confirm our guess, we have to estimate for $\phi\in\{\psi,\theta\},$ 
\begin{align}
\label{Eq: L1 for M}
 \| \cF^{-1}_{\tau , \xi^{\prime}} \psi (2^{- k} \lvert \xi^{\prime} \rvert) \phi(2^{- \ell} \tau) \cF_{\xi_n}^{-1} \chi (2^{- (k + 1)} \xi_n) \cF_{x_n} m(\tau , \xi^{\prime} , x_n) \|_{\LL^1_{t , x^{\prime}} (\IR \times \IR^{n - 1} ; \LL^p_{x_n} (\IR))}.
\end{align}
Given the relation between  $\theta$ and $\psi$, it suffices to prove the following result. 

\begin{proposition}\label{Prop: Multiplier norm}  For $k,\ell  \in \IZ$  and  $(t,x',x_n)\in\IR\times\IR^{n-1}\times\IR$ set
$$
M_{k,\ell}(t,x',x_n):=  \cF^{-1}_{\tau , \xi^{\prime}}\Bigl( \psi (2^{- k} \lvert \xi^{\prime} \rvert) \psi(2^{- \ell} \tau) \cF_{\xi_n}^{-1} \bigl(\chi (2^{- (k + 1)} \xi_n) \cF_{x_n} m(\tau , \xi^{\prime} ,\cdot)
\bigr)(x_n)\Bigr)(t,x')\cdotp
$$
Then, there exists a constant $C > 0$ such that 
\begin{align*}
  \|M_{k,\ell} \|_{\LL^1_{t , x^{\prime}} (\IR \times \IR^{n - 1} ; \LL^p_{x_n} (\IR))} \leq C 2^{\ell} \big( 2^{\frac{\ell}{2}} + 2^k \big)^{-1- \frac{1}{p}}.
\end{align*}
\end{proposition}

\begin{proof} Given $\lambda, \mu > 0$, we use the decomposition
\begin{align*}
    \|M_{k,\ell} \|_{\LL^1_{t , x^{\prime}} (\IR \times \IR^{n - 1} ; \LL^p_{x_n} (\IR))}  
 &=  \int_{\{ \lvert t \rvert < \mu \}} \int_{\{ \lvert x^{\prime} \rvert < \lambda \}}
   \|M_{k,\ell}(t,x^\prime,\cdot)  \|_{\LL^p_{x_n} (\IR)}  \, \d x^{\prime} \; \d t \\
 &+ \int_{\{ \lvert t \rvert \geq \mu \}} \int_{\{ \lvert x^{\prime} \rvert < \lambda \}}
    \|M_{k,\ell}(t,x^\prime,\cdot)  \|_{\LL^p_{x_n} (\IR)}  \, \d x^{\prime} \; \d t  \\
 &+ \int_{\{ \lvert t \rvert < \mu \}} \int_{\{ \lvert x^{\prime} \rvert \geq \lambda \}} 
    \|M_{k,\ell}(t,x^\prime,\cdot)  \|_{\LL^p_{x_n} (\IR)} \,\d x^{\prime} \; \d t \\
 &+ \int_{\{ \lvert t \rvert \geq \mu \}} \int_{\{ \lvert x^{\prime} \rvert \geq \lambda \}}
    \|M_{k,\ell}(t,x^\prime,\cdot)  \|_{\LL^p_{x_n} (\IR)}\, \d x^{\prime} \; \d t \\
 &=: \mathrm{I} + \mathrm{II} + \mathrm{III} + \mathrm{IV}.
\end{align*}
 By definition of the Fourier transform, we have for all  $(t,x',x_n)\in\IR\times\IR^{n-1}\times\IR,$ 
$$
M_{k,\ell}(t,x',x_n)=(2\pi)^{-n}\int_{\IR} \int_{\IR^{n - 1}} 
 \e^{\ii t \tau} \e^{\ii x^{\prime} \cdot \xi^{\prime}} \psi (2^{- k} \lvert \xi^{\prime} \rvert) \psi(2^{- \ell} \tau)  \dot S_{k+1}^v m(\tau , \xi^{\prime} , x_n) \; \d \xi^{\prime} \; \d \tau.
 $$ 
To estimate the  term $\mathrm{I}$, use Minkowski's inequality to pull the integration over $x_n$ into the two innermost integrals. Further, since the modulus is then applied to the exponential functions, which have modulus one, the dependence of the integrands on $t$ and $x^{\prime}$ vanishes. This introduces the measures of the sets $\{ \lvert t \rvert < \mu \}$ and $\{ \lvert x^{\prime} \rvert < \lambda \}$ as factors. Altogether, we find
\begin{align*}
 \mathrm{I} \leq C \mu \lambda^{n - 1} \int_{\IR} \int_{\IR^{n - 1}} \lvert \psi (2^{- k} \lvert \xi^{\prime} \rvert) \psi(2^{- \ell} \tau) \rvert \|  \dot S_{k+1}^v m(\tau , \xi^{\prime} , x_n) \|_{\LL^p_{x_n} (\IR)} \; \d \xi^{\prime} \; \d \tau.
\end{align*}
Since $\dot S_{k+1}^v$ defines an $\LL_p$-bounded operator with a bound independent of $k$, the conditions~\eqref{Eq: Support of cutoff} and~\eqref{Eq: Lp norm of multiplier} deliver
\begin{align*}
 \mathrm{I} \leq C \mu \lambda^{n - 1} 2^{2\ell} 2^{(n - 1) k} \big( 2^{\frac{\ell}{2}} + 2^k \big)^{- 1 - \frac{1}{p}}.
\end{align*}

To estimate the second term, we use the identity
\begin{equation}\label{eq:t}
 \e^{2 \pi \ii t \tau} = - \frac{1}{4 \pi^2 t^2} \frac{\d^2}{\d \tau^2} \e^{2 \pi \ii t \tau}
\end{equation}
and perform an integration by parts in the $\tau$ variable. 
Applying afterwards again Minkowski's inequality and using the fact  that $\dot S_{k+1}^v$ defines an $\LL_p$-bounded operator with a bound independent of $k$ yields
\begin{align*}
 \mathrm{II} \leq C \mu^{-1} \lambda^{n - 1} \int_{\IR} \int_{\IR^{n - 1}} \lvert \psi(2^{- k} \lvert \xi^{\prime} \rvert) \rvert \big\| \partial_{\tau}^2 \{ \psi(2^{- \ell} \tau) 
m (\tau , \xi^{\prime} , x_n) \} \big\|_{\LL^p_{x_n} (\IR)} \; \d \xi^{\prime} \; \d \tau.
\end{align*}
Employing finally  the product rule, followed by~\eqref{Eq: Lp norm of multiplier} and~\eqref{Eq: Support of cutoff} results in
\begin{align*}
 \mathrm{II} \leq C \mu^{-1} \lambda^{n - 1} 2^{(n - 1) k} \big( 2^{\frac{\ell}{2}} + 2^k \big)^{- 1 - \frac{1}{p}}.
\end{align*}

For the third term, we use the identity
\begin{equation}\label{eq:xp}
 \e^{2 \pi \ii x^{\prime} \cdot \xi^{\prime}} = \frac1{(2 \pi)^{2 N} \lvert x^{\prime} \rvert^{2 N}}
 (- \Delta_{\xi^{\prime}})^N  \e^{\ii x^{\prime} \cdot \xi^{\prime}}.
\end{equation}
Take $N := \lceil n / 2 \rceil$. 
Integrating by parts with respect to $\xi^{\prime}$ and using 
 Minkowski's inequality together with the $\LL_p$-boundedness of $\dot S_{k+1}^v$ yields
\begin{align*}
 \mathrm{III} \leq C \mu \lambda^{n - 1 - 2 N} \int_{\IR} \int_{\IR^{n - 1}} \lvert \psi (2^{- \ell} \tau) \rvert \| \Delta_{\xi^{\prime}}^N \{ \psi(2^{- k} \lvert \xi^{\prime} \rvert) m (\tau , \xi^{\prime} , x_n) \} \|_{\LL^p_{x_n} (\IR)} \; \d \xi^{\prime} \; \d \tau.
\end{align*}
By virtue of the product rule combined with~\eqref{Eq: Lp norm of multiplier} and~\eqref{Eq: Support of cutoff} we finally find
\begin{align*}
 \mathrm{III} \leq C \mu \lambda^{n - 1 - 2 N} 2^{2\ell} 2^{(n - 1) k} 2^{- 2 N k} \big( 2^{\frac{\ell}{2}} + 2^k \big)^{-1- \frac{1}{p}}.
\end{align*}

\indent For the  last term, we combine the strategies carried out for the second and third term leading to the estimate
\begin{align*}
 \mathrm{IV} \leq C \mu^{-1} \lambda^{n - 1 - 2 N} \int_{\IR} \int_{\IR^{n - 1}} \big\| (\Delta_{\xi^{\prime}})^N \partial_{\tau}^2 \{ \psi(2^{- \ell} \tau) \psi (2^{- k} \lvert \xi^{\prime} \rvert) m (\tau , \xi^{\prime} , x_n) \} \big\|_{\LL^p_{x_n} (\IR)} \; \d \xi^{\prime} \; \d \tau.
\end{align*}
The Leibniz rule together with~\eqref{Eq: Lp norm of multiplier} implies
\begin{align*}
 &\big\| (\Delta_{\xi^{\prime}})^N \partial_{\tau}^2 \{ \psi(2^{- \ell} \tau) \psi (2^{- k} \lvert \xi^{\prime} \rvert) m (\tau , \xi^{\prime} , x_n) \} \big\|_{\LL^p_{x_n} (\IR)} \\
 &\leq \sum_{j = 0}^2 \sum_{\substack{\alpha , \beta \in \IN_0^n \\ \lvert \alpha \rvert + \lvert \beta \rvert = 2 N}} C_{\alpha , \beta , j} \Big\| \frac{\d^{2 - j}}{\d \tau^{2 - j}} \psi(2^{- \ell} \tau) \partial_{\xi^{\prime}}^{\beta} \psi (2^{- k} \lvert \xi^{\prime} \rvert) \partial_{\xi^{\prime}}^{\alpha} \partial_{\tau}^j m (\tau , \xi^{\prime} , x_n) \Big\|_{\LL^p_{x_n} (\IR)} \\
 &= \sum_{j = 0}^2 \sum_{\substack{\alpha , \beta \in \IN_0^n \\ \lvert \alpha \rvert + \lvert \beta \rvert = 2 N}} C_{\alpha , \beta , j} 2^{- (2 - j) \ell} 2^{- \lvert \beta \rvert k} \Big\lvert \Big(\frac{\d^{2 - j}}{\d \tau^{2 - j}} \psi \Big)(2^{- \ell} \tau) \Big\rvert \lvert (\partial_{\xi^{\prime}}^{\beta} \psi) (2^{- k} \lvert \xi^{\prime} \rvert) \rvert \big\| \partial_{\xi^{\prime}}^{\alpha} \partial_{\tau}^j m (\tau , \xi^{\prime} , x_n) \big\|_{\LL^p_{x_n} (\IR)} \\
 &\lesssim\sum_{j = 0}^2\!\! \sum_{\substack{\alpha , \beta \in \IN_0^n \\ \lvert \alpha \rvert \!+\! \lvert \beta \rvert = 2 N}}\!\!\!\!\! C_{\alpha , \beta , j} 2^{- (2 - j) \ell} 2^{- \lvert \beta \rvert k} \Big\lvert \Big(\frac{\d^{2 - j}}{\d \tau^{2 - j}} \psi \Big)(2^{- \ell} \tau) \Big\rvert \lvert (\partial_{\xi^{\prime}}^{\beta} \psi) (2^{- k} \lvert \xi^{\prime} \rvert) \rvert \lvert \tau \rvert^{1 - j} \lvert \xi^{\prime} \rvert^{\lvert \alpha \rvert} \big( \lvert \tau \rvert^{\frac{1}{2}}\! +\! \lvert \xi^{\prime} \rvert \big)^{- 1 - \frac{1}{p} - 2 \lvert \alpha \rvert}.
\end{align*}
Employing~\eqref{Eq: Support of cutoff} then delivers
\begin{align*}
 \mathrm{IV} \leq C \mu^{-1} \lambda^{n - 1 - 2 N}  2^{(n - 1) k} 2^{- 2 N k} \big( 2^{\frac{\ell}{2}} + 2^k \big)^{-1- \frac{1}{p}}.
\end{align*}

\indent To conclude the estimate of $M_{k , \ell}(t)$, notice first that regarding the variable $\mu$, the sum of the terms $\mathrm{I}$ and $\mathrm{II}$ obeys an estimate of the form
\begin{align*}
 \mathrm{I} + \mathrm{II} \leq C \lambda^{n - 1} 2^{(n - 1) k} \big( 2^{\frac{\ell}{2}} + 2^k \big)^{- 1 - \frac{1}{p}} \big( 2^{2\ell} \mu + \mu^{-1} \big).
\end{align*}
Minimizing the function
\begin{align*}
 (0 , \infty) \ni \mu \mapsto 2^{2\ell} \mu +  \mu^{-1}
\end{align*}
delivers the choice $\mu = 2^{- \ell}$.
Inserting this into the estimates of the terms $\mathrm{I}$, $\mathrm{II}$, $\mathrm{III}$ and $\mathrm{IV}$ yields
\begin{align*}
 \mathrm{I} + \mathrm{II} + \mathrm{III} + \mathrm{IV} \leq C 2^{(n - 1) k} 2^{\ell} \big( 2^{\frac{\ell}{2}} + 2^k \big)^{- 1 - \frac{1}{p}} \lambda^{n - 1} \bigl(1+  
 \lambda^{- 2 N}2^{- 2 N k} \big)\cdotp
\end{align*}
Minimizing the function
\begin{align*}
 (0 , \infty) \ni \lambda \mapsto \lambda^{n - 1} + 2^{- 2 N k} \lambda^{n - 1 - 2 N}
\end{align*}
yields
\begin{align*}
 \lambda := \bigg( \frac{2 N + 1 - n}{n - 1} \bigg)^{\frac{1}{2 N}} 2^{- k},
\end{align*}
Inserting this choice in the estimate of the sum of all four terms finally gives the estimate
\begin{align*}
 \mathrm{I} + \mathrm{II} + \mathrm{III} + \mathrm{IV} \leq C 2^{\ell} \big( 2^{\frac{\ell}{2}} + 2^k \big)^{-1- \frac{1}{p}}. &\qedhere
\end{align*}
\end{proof}

Having at hand Proposition~\ref{Prop: Multiplier norm}, it is easy to  bound the $\LL_1$-norm of $M(t)$. Indeed, by virtue of the inequalities given in ~\eqref{Eq: Multiple Young}, 
we obtain after an application of Proposition~\ref{Prop: Multiplier norm} 
$$\begin{aligned}
 \int_{\IR} M(t) \; \d t &\leq C \sum_{k = - \infty}^{\infty} 2^{s k} \big( \sum_{\ell = - \infty}^{2 k} 2^{\ell} \big( 2^{\frac{\ell}{2}} + 2^k \big)^{-1- \frac{1}{p}}
  \|\dot\Delta_k^h g \|_{\LL^1_t (\IR ; \LL^p_{x^{\prime}} (\IR^{n - 1}))} \\
 &\hspace{3cm} + \sum_{\ell = 2 k + 1}^{\infty} 2^{\ell} \big( 2^{\frac{\ell}{2}} + 2^k \big)^{- 1- \frac{1}{p}} 
 \|\dot\Delta_\ell^t\dot\Delta_k^h g\|_{\LL^1_t (\IR ; \LL^p_{x^{\prime}} (\IR^{n - 1}))} \big) \\
 &\leq C \sum_{k = - \infty}^{\infty} 2^{s k} \big( 2^{k (1 - \frac{1}{p})} \|\dot\Delta_k^h g \|_{\LL^1_t (\IR ; \LL^p_{x^{\prime}} (\IR^{n - 1}))} \\
 &\hspace{3cm} + \sum_{\ell = 2 k + 1}^{\infty} 
 2^{\ell (\frac{1}{2} - \frac{1}{2 p})}
 \|\dot\Delta_\ell^t\dot\Delta_k^h g \|_{\LL^1_t (\IR ; \LL^p_{x^{\prime}} (\IR^{n - 1}))} \big) \cdotp
 \end{aligned}
$$
Hence, 
\begin{equation}\label{eq:MM}
 \int_{\IR} M(t) \; \d t\leq C\bigl(\| g \|_{\LL_1 (\IR ; \dot \B^{s + 1 - \frac{1}{p}}_{p,1} (\IR^{n - 1}))} + \| g \|_{\dot \B^{\frac{1}{2} - \frac{1}{2 p}}_{1 , 1} (\IR ; \dot \B^s_{p , 1} (\IR^{n - 1}))}\bigr)\cdotp
\end{equation}
To complete the proof of Proposition~\ref{p:m}, we need  to bound $\int_{\IR} Y(t) \; \d t$.  Recall that~\eqref{eq:LP2} ensures that 
the series over $j$ in~\eqref{eq:Y}  only runs from $- \infty$ to $k + 1$. Since, moreover,  $\dot S_k^h$ is a bounded operator on $\LL^p_{x^{\prime}} (\IR^{n - 1})$
(with norm independent of $p$ and $k$),   we have  $Y(t)\leq C\bigl( Y_1(t)+Y_2(t)\bigr)$
 with   
 \begin{align*}
 &Y_1(t) := \sum_{k = - \infty}^{\infty} \sum_{j = - \infty}^{k + 1} \sum_{\ell = - \infty}^{2j} 2^{s k}
  \|\dot\Delta_j^h\dot\Delta_\ell^t  \dot\Delta_k^v \, m(D_{t,x'},x_n)g\|_{\LL_p (\IR^n)} \\
  \andf& Y_2(t) := \sum_{k = - \infty}^{\infty} \sum_{j = - \infty}^{k + 1} 
  \sum_{\ell = 2 j + 1}^{\infty} 2^{s k}  \|\dot\Delta_j^h\dot\Delta_\ell^t  \dot\Delta_k^v \, m(D_{t,x'},x_n)g\|_{\LL_p (\IR^n)}.
  \end{align*}
  Now, recalling~\eqref{eq:theta} we write 
  $$\begin{aligned}\dot\Delta_j^h\dot\Delta_\ell^t  \dot\Delta_k^v \, m(D_{t,x'},x_n)g
  &=  \cF^{-1}_{\tau , \xi^{\prime}} \theta (2^{- j} \lvert \xi^{\prime} \rvert) \psi(2^{- \ell} \tau) \cF_{\xi_n}^{-1} \psi (2^{- k} \xi_n) \cF_{x_n} m(\tau , \xi^{\prime} , x_n)  *_{t , x^{\prime}}\dot\Delta_j^h g\\
 & = \cF^{-1}_{\tau , \xi^{\prime}} \theta (2^{- j} \lvert \xi^{\prime} \rvert) \theta(2^{- \ell} \tau) \cF_{\xi_n}^{-1} \psi (2^{- k} \xi_n) \cF_{x_n} m(\tau , \xi^{\prime} , x_n)  *_{t , x^{\prime}} \dot\Delta_j^h\dot\Delta_\ell^t g,
  \end{aligned}$$
and deduce by  virtue  of~\eqref{Eq: Multiple Young} and the definition of $\dot\Delta_k^v$ that 
\begin{align*}
 \int_{\IR} Y_1(t)\, \d t \leq \sum_{k = - \infty}^{\infty}  \sum_{j = - \infty}^{k + 1} \sum_{\ell = - \infty}^{2 j}
 2^{s k}  \| \cF^{-1}_{\tau , \xi^{\prime}} \theta (2^{- j} \lvert \xi^{\prime} \rvert) \psi(2^{- \ell} \tau) 
  \dot\Delta_k^vm(\tau , &\xi^{\prime}, x_n)\|_{\LL^1_{t , x^{\prime}} (\IR \times \IR^{n - 1} ; \LL^p_{x_n} (\IR))} 
\\
&\cdot \| \dot\Delta_j^h g \|_{\LL^1_t (\IR ; \LL^p_{x^{\prime}} (\IR^{n - 1}))} \\
\int_{\IR} Y_2(t)  \,\d t \leq
  \sum_{k = - \infty}^{\infty}  \sum_{j = - \infty}^{k + 1} 
  \sum_{\ell = 2 j + 1}^{\infty} 2^{s k}
   \| \cF^{-1}_{\tau , \xi^{\prime}} \theta (2^{- j} \lvert \xi^{\prime} \rvert) \theta(2^{- \ell} \tau) \dot\Delta_k^vm(\tau,& \xi^{\prime} , x_n) \|_{\LL^1_{t , x^{\prime}} (\IR \times \IR^{n - 1} ; \LL^p_{x_n} (\IR))}\\
  &\cdot \| \dot\Delta_\ell^t\dot\Delta_j^h g \|_{\LL^1_t (\IR ; \LL^p_{x^{\prime}} (\IR^{n - 1}))}.
\end{align*}
By virtue of Fubini's theorem, these estimates might also be written as
\begin{align*}
 \int_{\IR}Y_1(t)\, \d t \leq   \sum_{j = - \infty}^{\infty}  \sum_{\ell = - \infty}^{2 j} \sum_{k = j - 1}^{\infty}2^{s k} \| \cF^{-1}_{\tau , \xi^{\prime}} \theta (2^{- j} \lvert \xi^{\prime} \rvert) \psi(2^{- \ell} \tau) 
  \dot\Delta_k^v m(\tau ,& \xi^{\prime} , x_n)\|_{\LL^1_{t , x^{\prime}} (\IR \times \IR^{n - 1} ; \LL^p_{x_n} (\IR))} \\&\cdot \| \dot\Delta_j^h g \|_{\LL^1_t (\IR ; \LL^p_{x^{\prime}} (\IR^{n - 1}))} \\
\int_{\IR} Y_2(t)  \,\d t \leq
  \sum_{j = - \infty}^{\infty}  \sum_{\ell = 2 j+1}^{\infty} \sum_{k = j - 1}^{\infty}
 2^{s k}
   \|\cF^{-1}_{\tau , \xi^{\prime}} \theta (2^{- j} \lvert \xi^{\prime} \rvert) \theta(2^{- \ell} \tau) 
  \dot\Delta_k^v m(\tau ,& \xi^{\prime} , x_n) \|_{\LL^1_{t , x^{\prime}} (\IR \times \IR^{n - 1} ; \LL^p_{x_n} (\IR))}\\
  &\cdot \| \dot\Delta_\ell^t\dot\Delta_j^h g \|_{\LL^1_t (\IR ; \LL^p_{x^{\prime}} (\IR^{n - 1}))}.
\end{align*}
Thus, a bound on the $\LL_1$-norm of $Y(t)$ would follow, if we are able to prove the following result:

\begin{proposition}\label{Prop: Multiplier norm2} 
For all $j, k,\ell \in \IZ$  and $(t,x',x_n)\in\IR\times\IR^{n-1}\times\IR$ let
$$M_{j,k,\ell}(t,x',x_n):=  \cF^{-1}_{\tau , \xi^{\prime}} \psi (2^{- j} \lvert \xi^{\prime} \rvert) \psi(2^{- \ell} \tau) \dot\Delta_k^v m(\tau , \xi^{\prime} , x_n).$$
Then there  exists a constant $C > 0$ such that for all $j,\ell  \in \IZ$ 
\begin{equation}\label{Eq: Sufficient bound}
  \sum_{k = j - 1}^{\infty} 2^{s k}
    \|M_{j,k,\ell} \|_{\LL^1_{t , x^{\prime}} (\IR \times \IR^{n - 1} ; \LL^p_{x_n} (\IR))}
   \leq  C 2^{\ell}  \big( 2^{\frac{\ell}{2}} + 2^j \big)^{s-1- \frac{1}{p}}.
\end{equation}
\end{proposition}
Indeed, the above estimate  would give us
$$ \int_{\IR} Y_1(t) \; \d t \leq C \sum_{j = - \infty}^{\infty} 
\sum_{\ell = - \infty}^{2 j} 2^{\ell} \big( 2^{\frac{\ell}{2}} + 2^j \big)^{s-1- \frac{1}{p}} \| \dot\Delta_j^h g \|_{\LL^1_t (\IR ; \LL^p_{x^{\prime}} (\IR^{n - 1}))}.$$
Since $2^j \leq 2^{\ell / 2} + 2^j \leq 2^{j + 1}$ for $\ell \leq 2 j$ and $\sum_{\ell = - \infty}^{2 j} 2^{\ell} \sim 2^{2 j}$, one easily deduces that 
\begin{equation}\label{eq:Y1} \int_{\IR} Y_1(t) \; \d t \leq C \| g \|_{\LL^1 (\IR ; \dot \B^{s + 1 - \frac{1}{p}}_{p , 1} (\IR^{n - 1}))}.\end{equation}
Likewise, plugging~\eqref{Eq: Sufficient bound} in the inequality for $Y_2$ gives us
$$ \int_{\IR} Y_2(t) \; \d t \leq C \sum_{j = - \infty}^{\infty}
  \sum_{\ell = 2 j + 1}^{\infty}  2^{\ell}\big( 2^{\frac{\ell}{2}} + 2^j \big)^{s-1- \frac{1}{p}} 
 \| \dot\Delta_\ell^t\dot\Delta_j^h  g \|_{\LL^1_t (\IR ; \LL^p_{x^{\prime}} (\IR^{n - 1}))}.$$
Since $2^{\ell / 2} \leq 2^{\ell / 2} + 2^j$ and $s < 1 + 1 / p$, we  conclude that 
 \begin{equation}\label{eq:Y2} \int_{\IR} Y_2(t) \; \d t \leq C  \| g \|_{\dot \B^{\frac{1}{2}(s+1-\frac1p)}_{1 , 1} (\IR ; \dot \B^0_{p , 1} (\IR^{n - 1}))}.\end{equation}

\begin{proof}[Proof of Proposition~\ref{Prop: Multiplier norm2}]
This is a mere adaptation of the proof of Proposition~\ref{Prop: Multiplier norm}.
 For all $\lambda , \mu > 0$, 
we write 
\begin{align*}
    \|M_{j,k,\ell} \|_{\LL^1_{t , x^{\prime}} (\IR \times \IR^{n - 1} ; \LL^p_{x_n} (\IR))}  
 &=  \int_{\{ \lvert t \rvert < \mu \}} \int_{\{ \lvert x^{\prime} \rvert < \lambda \}}
   \|M_{j,k,\ell}(t,x^\prime,\cdot)  \|_{\LL^p_{x_n} (\IR)}  \, \d x^{\prime} \; \d t \\
 &+ \int_{\{ \lvert t \rvert \geq \mu \}} \int_{\{ \lvert x^{\prime} \rvert < \lambda \}}
    \|M_{j,k,\ell}(t,x^\prime,\cdot)  \|_{\LL^p_{x_n} (\IR)}  \, \d x^{\prime} \; \d t  \\
 &+ \int_{\{ \lvert t \rvert < \mu \}} \int_{\{ \lvert x^{\prime} \rvert \geq \lambda \}} 
    \|M_{j,k,\ell}(t,x^\prime,\cdot)  \|_{\LL^p_{x_n} (\IR)} \,\d x^{\prime} \; \d t \\
 &+ \int_{\{ \lvert t \rvert \geq \mu \}} \int_{\{ \lvert x^{\prime} \rvert \geq \lambda \}}
    \|M_{j,k,\ell}(t,x^\prime,\cdot)  \|_{\LL^p_{x_n} (\IR)}\, \d x^{\prime} \; \d t \\
  &=: \mathrm{I}_k + \mathrm{II}_k + \mathrm{III}_k + \mathrm{IV}_k.
\end{align*}
 By definition of the Fourier transform, 
 we discover that for all  $(t,x',x_n)\in\IR\times\IR^{n-1}\times\IR,$ 
$$M_{j,k,\ell}(t,x',x_n)= (2\pi)^{-n}\int_{\IR} \int_{\IR^{n - 1}} 
 \e^{\ii t \tau} \e^{\ii x^{\prime} \cdot \xi^{\prime}} \psi (2^{- j} \lvert \xi^{\prime} \rvert) \psi(2^{- \ell} \tau) \dot\Delta_k^v m(\tau , \xi^{\prime} , x_n) \; \d \xi^{\prime} \; \d \tau.
 $$

To estimate $\mathrm{I}_k,$ use Minkowski's inequality to pull the integration over $x_n$ into the two innermost integrals. Further, since the modulus is then applied to the exponential functions, which have modulus one, the dependence of the integrands on $t$ and $x^{\prime}$ vanishes. This introduces the measures of the sets $\{ \lvert t \rvert < \mu \}$ and $\{ \lvert x^{\prime} \rvert < \lambda \}$ as factors. Altogether, we find
\begin{align*}
 \mathrm{I}_k \leq C \mu \lambda^{n - 1} \int_{\IR} \int_{\IR^{n - 1}} \lvert \psi (2^{- j} \lvert \xi^{\prime} \rvert) \psi(2^{- \ell} \tau) \rvert \| \dot\Delta_k^v m(\tau , \xi^{\prime} , \cdot) \|_{\LL^p_{x_n} (\IR)} \; \d \xi^{\prime} \; \d \tau.
\end{align*}
Multiplication with $2^{s k}$ and summation over $k \geq j - 1$ delivers by using properties~\eqref{Eq: Support of cutoff} and~\eqref{Eq: Besov norm of multiplier}
\begin{align*}
 \sum_{k = j - 1}^{\infty} 2^{s k} \mathrm{I}_k \leq C \mu \lambda^{n - 1} 2^{2 \ell} 2^{(n - 1) j} \big( 2^{\frac{\ell}{2}} + 2^j \big)^{s-1- \frac{1}{p}}.
\end{align*}

To estimate the second term, use identity~\eqref{eq:t}
and perform an integration by parts in the $\tau$ variable. Apply afterwards again Minkowski's inequality, multiply both sides of the inequality by $2^{s k}$ and sum over $k \geq j - 1$. This yields
\begin{align*}
 \sum_{k = j - 1}^{\infty} 2^{s k} \mathrm{II}_k \leq C \mu^{-1} \lambda^{n - 1} \int_{\IR} \int_{\IR^{n - 1}} \lvert \psi(2^{- j} \lvert \xi^{\prime} \rvert) \rvert \big\| \partial_{\tau}^2 \{ \psi(2^{- \ell} \tau) m (\tau , \xi^{\prime} , x_n) \} \big\|_{\dot \B^s_{p , 1} (\IR)} \; \d \xi^{\prime} \; \d \tau.
\end{align*}
Employing now the product rule, followed by~\eqref{Eq: Besov norm of multiplier} and~\eqref{Eq: Support of cutoff} results in
\begin{align*}
 \sum_{k = j - 1}^{\infty} 2^{s k} \mathrm{II}_k \leq C \mu^{-1} \lambda^{n - 1} 2^{(n - 1) j}  \big( 2^{\frac{\ell}{2}} + 2^j \big)^{s-1- \frac{1}{p}}.
\end{align*}

For the third term, use~\eqref{eq:xp} with $N := \lceil n / 2 \rceil$. Perform several integrations by parts with respect to the Fourier variable $\xi^{\prime}$ and then use Minkowski's inequality, multiply both sides of the inequality by $2^{s k}$ and sum over $k \geq j - 1$ to deduce
\begin{align*}
 \sum_{k = j - 1}^{\infty} 2^{s k} \mathrm{III}_k \leq C \mu \lambda^{n - 1 - 2 N} \int_{\IR} \int_{\IR^{n - 1}} \lvert \psi (2^{- \ell} \tau) \rvert \| \Delta_{\xi^{\prime}}^N \{ \psi(2^{- j} \lvert \xi^{\prime} \rvert) m (\tau , \xi^{\prime} , x_n) \} \|_{\dot \B^s_{p , 1} (\IR)} \; \d \xi^{\prime} \; \d \tau.
\end{align*}
In light of Leibniz' rule combined with~\eqref{Eq: Lp norm of multiplier} and~\eqref{Eq: Support of cutoff} one finally finds
\begin{align*}
 \sum_{k = j - 1}^{\infty} 2^{s k} \mathrm{III}_k \leq C \mu \lambda^{n - 1 - 2 N} 2^{2\ell} 2^{(n - 1) j} 2^{- 2 N j}  \big( 2^{\frac{\ell}{2}} + 2^j \big)^{s-1- \frac{1}{p}}.
\end{align*}

For the last term, we combine the strategies carried out for the second and third terms leading to the estimate
\begin{align*}
 \sum_{k = j - 1}^{\infty} 2^{s k} \mathrm{IV}_k \leq C \mu^{-1} \lambda^{n - 1 - 2 N} \int_{\IR} \int_{\IR^{n - 1}} \big\| (\Delta_{\xi^{\prime}})^N \partial_{\tau}^2 \{ \psi(2^{- \ell} \tau) \psi (2^{- j} \lvert \xi^{\prime} \rvert) m (\tau , \xi^{\prime} , x_n) \} \big\|_{\dot \B^s_{p , 1} (\IR)} \; \d \xi^{\prime} \; \d \tau.
\end{align*}
The product rule implies
\begin{align*}
 &\big\| (\Delta_{\xi^{\prime}})^N \partial_{\tau}^2 \{ \psi(2^{- \ell} \tau) \psi (2^{- k} \lvert \xi^{\prime} \rvert) m (\tau , \xi^{\prime} , x_n) \} \big\|_{\dot \B^s_{p , 1} (\IR)} \\
 &\leq \sum_{j = 0}^2 \sum_{\substack{\alpha , \beta \in \IN_0^n \\ \lvert \alpha \rvert + \lvert \beta \rvert = 2 N}} C_{\alpha , \beta , j} \Big\| \frac{\d^{2 - j}}{\d \tau^{2 - j}} \psi(2^{- \ell} \tau) \partial_{\xi^{\prime}}^{\beta} \psi (2^{- j} \lvert \xi^{\prime} \rvert) \partial_{\xi^{\prime}}^{\alpha} \partial_{\tau}^j m (\tau , \xi^{\prime} , x_n) \Big\|_{\dot \B^s_{p , 1} (\IR)} \\
 &= \sum_{j = 0}^2 \sum_{\substack{\alpha , \beta \in \IN_0^n \\ \lvert \alpha \rvert + \lvert \beta \rvert = 2 N}} C_{\alpha , \beta , j} 2^{- (2 - j) \ell} 2^{- \lvert \beta \rvert j} \Big\lvert \Big(\frac{\d^{2 - j}}{\d \tau^{2 - j}} \psi \Big)(2^{- \ell} \tau) \Big\rvert \lvert (\partial_{\xi^{\prime}}^{\beta} \psi) (2^{- j} \lvert \xi^{\prime} \rvert) \rvert \big\| \partial_{\xi^{\prime}}^{\alpha} \partial_{\tau}^j m (\tau , \xi^{\prime} , x_n) \big\|_{\dot \B^s_{p , 1} (\IR)} \\
 &\!\lesssim \!\sum_{j = 0}^2\!\! \sum_{\substack{\alpha , \beta \in \IN_0^n \\ \lvert \alpha \rvert \!+\! \lvert \beta \rvert = 2 N}}\!\!\!\!\! C_{\alpha , \beta , j} 2^{- (2 - j) \ell} 2^{- \lvert \beta \rvert j} \Big\lvert \Big(\frac{\d^{2 - j}}{\d \tau^{2 - j}} \psi \Big)(2^{- \ell} \tau) \Big\rvert \lvert (\partial_{\xi^{\prime}}^{\beta} \psi) (2^{- j} \lvert \xi^{\prime} \rvert) \rvert \lvert \tau \rvert^{1 - j} \lvert \xi^{\prime} \rvert^{\lvert \alpha \rvert} \big( \lvert \tau \rvert^{\frac{1}{2}} \!+\! \lvert \xi^{\prime} \rvert \big)^{s - 1 - \frac{1}{p} - 2 \lvert \alpha \rvert}.
\end{align*}
Employing~\eqref{Eq: Support of cutoff} then delivers
\begin{align*}
 \sum_{k = j - 1}^{\infty} 2^{s k} \mathrm{IV}_k \leq C \mu^{-1} \lambda^{n - 1 - 2 N} 
 2^{(n - 1) j} 2^{- 2 N j}  \big( 2^{\frac{\ell}{2}} + 2^j \big)^{s-1- \frac{1}{p}}.
\end{align*}

\indent To conclude the estimate take (as for bounding $M$)
  $\mu := 2^{- \ell}$ and insert  this into the estimates of the terms $\mathrm{I}_k$, $\mathrm{II}_k$, $\mathrm{III}_k$ and $\mathrm{IV}_k.$  We end up with 
\begin{align*}
 \sum_{k = j - 1}^{\infty} \big( \mathrm{I}_k + \mathrm{II}_k + \mathrm{III}_k + \mathrm{IV}_k \big) \leq C 2^{(n - 1) j} 2^{\ell}\big( 2^{\frac{\ell}{2}} + 2^j \big)^{s-1- \frac{1}{p}} \lambda^{n - 1} 
 \bigl(1+\lambda^{-2N}2^{- 2jN}\bigr)\cdotp
 \end{align*}
 Then, taking $\lambda=2^{-j},$ one readily gets
\begin{align*}
 \sum_{k = j - 1}^{\infty} \big( \mathrm{I}_k + \mathrm{II}_k + \mathrm{III}_k + \mathrm{IV}_k \big) \leq C 2^{\ell} \big( 2^{\frac{\ell}{2}} + 2^j \big)^{s-1- \frac{1}{p}},
\end{align*}
which completes the proof of Proposition~\ref{Prop: Multiplier norm2}. 
\end{proof}
Then, putting together~\eqref{eq:MM},~\eqref{eq:Y1}
and~\eqref{eq:Y2}  completes the proof of Proposition~\ref{p:m}.
\end{proof}

\section{Maximal regularity for the Lam\'e system in a half-space}
 
Our strategy for proving  Theorem~\ref{Thm:lame} is to decompose the sought solution $u$ into $v+w$ with $v$ satisfying 

\begin{align}\label{eq:lame1}
\left\{\begin{aligned}
\partial_tv-\frac{2\mu}\nu\Delta v & = 0 && \mbox{ in}\quad \IR_+\times\IR_+^n,\\
\IS(v)\cdot \e_n                    & = g && \mbox{ at}\quad\IR_+\times\partial\IR_+^n,\\ 
v|_{t=0}                            & = 0 && \mbox{ in}\quad\IR_+^n
\end{aligned}\right.
\end{align}
and, thus, $w$ is a solution to 
 \begin{align}\label{eq:lame2}\left\{\begin{aligned}
\partial_t w-\divergence\IS(w)&=f+\divergence\IS(v)- \partial_tv && \mbox{ in}\quad \IR_+\times\IR_+^n,\\
\IS(w)\cdot \e_n&=0 && \mbox{ at}\quad\IR_+\times\partial\IR_+^n,\\
w|_{t=0}&=u_0 && \mbox{ in}\quad\IR_+^n.
\end{aligned}\right.\end{align}
Solving~\eqref{eq:lame1} requires two steps:  
first, we remove the boundary  data by  constructing a suitable  extension 
of $g$ over the whole $\IR\times\IR^n_+,$
and use Proposition~\ref{p:m} to  estimate  $\partial_tv,$
then we estimate  $\nabla v$  in $\IR^n_+$ and 
at the boundary. 
Finally, solving~\eqref{eq:lame2} and getting the desired estimates will
rely on results from Chapters~\ref{Sec: The functional setting and basic interpolation results}
and~\ref{Sec: The Lame operator in the upper half-plane}. 
\medbreak
In order to solve  system~\eqref{eq:lame1}, we extend
$g$ by $0$ on negative times, and  consider 
\begin{align*}\left\{\begin{aligned}
\partial_t v-\frac{2\mu}\nu\Delta v&=0 &&\mbox{ in}\quad \IR\times\IR_+^n,\\
\IS(v)\cdot \e_n&=g && \mbox{ at}\quad\IR\times\partial\IR_+^n.
\end{aligned}\right.\end{align*}
For expository, let us first look at the case $n=2.$  Then, denoting 
 $\nu:=\lambda+2\mu,$  the boundary condition translates into
\begin{align}\label{eq:BC}\left\{
\begin{aligned}
\mu(\partial_2v^1+\partial_1v^2) &=g^1,\\
\lambda \partial_1 v^1+\nu\partial_2 v^2 &=g^2.\end{aligned}\right.\end{align}
Defining $\wh v:=\cF_{t,x}v,$ we look at $v$ under the form 
$$\wh v^1(\tau,\xi_1,x_2)= \e^{-rx_2} \wh v^1(\tau,\xi_1,0)\andf
\wh v^2(\tau,\xi_1,x_2)= \e^{-rx_2} \wh v^2(\tau,\xi_1,0)$$
with  $r^2=\ii \tau +2\mu \xi_1^2/\nu$ and $|\arg (r)|\leq\pi/4,$ 
so that~\eqref{eq:BC} implies:
$$
\begin{pmatrix} -\mu r & \ii\mu\xi_1\\
\ii\lambda \xi_1&-\nu r\end{pmatrix}
\begin{pmatrix}  \wh v^1(\tau,\xi_1,0)\\  \wh v^2(\tau,\xi_1,0)\end{pmatrix}
=\begin{pmatrix} \wh g^1(\tau,\xi_1)\\\wh g^2(\tau,\xi_1)\end{pmatrix}\cdotp
$$
 The determinant $\Delta$ of the matrix does not vanish for $(\tau,\xi_1)\not=(0,0)$ since
 $$ \Delta= \mu\nu r^2+\lambda\mu\xi_1^2=\mu\nu(\ii\tau+\xi_1^2).$$
 So,  we eventually find
 $$\begin{pmatrix}  \wh v^1(\tau,\xi_1,0)\\  \wh v^2(\tau,\xi_1,0)\end{pmatrix}=-\frac1{\ii\tau +\xi_1^2}\begin{pmatrix} \frac r\mu& \frac \ii\nu\xi_1\\
 \ii\frac{\lambda}{\mu\nu}\xi_1& \frac r\nu\end{pmatrix} \begin{pmatrix} \wh g^1(\tau,\xi_1)\\\wh g^2(\tau,\xi_1)\end{pmatrix},  $$
 which implies 
  \begin{equation}\label{eq:formula}
  \begin{pmatrix}  \wh{\partial_tv^1}(\tau,\xi_1,x_2)\\  \wh{\partial_tv^2}(\tau,\xi_1,x_2)\end{pmatrix}=-\frac{\ii\tau \e^{-rx_2}}{\ii\tau +\xi_1^2}\begin{pmatrix} \frac r\mu& \frac \ii\nu\xi_1\\
 \ii\frac{\lambda}{\mu\nu}\xi_1& \frac r\nu\end{pmatrix} \begin{pmatrix} \wh g^1(\tau,\xi_1)\\\wh g^2(\tau,\xi_1)\end{pmatrix}\cdotp\end{equation}
 A similar formula may be derived in the $n$-dimensional case. Indeed, the boundary conditions 
 now translate into  
 \begin{align}\label{eq:BCn}\left\{
\begin{aligned}
\mu(\partial_nv'+\nabla'v^n) &=g',\\
\lambda\divergence' v'+\nu\partial_n v^n &=g^n.\end{aligned}\right.\end{align}
 Hence,  we look at $\wh v$ under the form 
$$\wh v(\tau,\xi',x_n)= \e^{-rx_n} \wh v(\tau,\xi',0),$$  the boundary conditions now imply (with $r^2 = \ii \tau + 2 \mu \lvert \xi^{\prime} \rvert^2 / \nu$) that
$$\begin{pmatrix} -\mu r\, \Id_{n-1} & \ii\mu\xi^{\prime}\\
\ii\lambda(\xi')^\top &-\nu r\end{pmatrix}
\begin{pmatrix}  \wh v'(\tau,\xi',0)\\  \wh v^n(\tau,\xi',0)\end{pmatrix}
=\begin{pmatrix} \wh g'(\tau,\xi')\\\wh g^n(\tau,\xi')\end{pmatrix}\cdotp$$
The determinant  of the above matrix is 
  $$\Delta_n= (-1)^n\mu^{n-1}\nu r^{n-2} (\ii\tau+|\xi'|^2),$$
  which allows to obtain a formula similar to~\eqref{eq:formula}. 
  The important conclusion is that all the multipliers appearing  in the right-hand side 
  of~\eqref{eq:formula} or in its $n$-dimensional generalization  
  fulfill~\eqref{Eq: Lp norm of multiplier} and~\eqref{Eq: Besov norm of multiplier}.
   Hence, we may control $\partial_tv$  in   $\LL_1(\IR_+;\dot\B^s_{p,1} (\IR^n_+))$ 
   according to Proposition~\ref{p:m}, and  get
\begin{equation}\label{eq:lame0}
\|\partial_t v\|_{\LL_1(\IR;\dot\B^s_{p,1}(\IR^n_+))} \lesssim \|g\|_{E\dot\Y^s_p}.\end{equation}

In the case $s+1\leq n/p,$ we claim that
\begin{equation}\label{eq:Dvlame}
\| \nabla_x v\|_{\LL_1(\IR;\dot\B^{s+1}_{p,1}(\IR^n_+))} \lesssim \|g\|_{E\dot\Y^s_p}.\end{equation}
To prove our claim, we look  at $v(t,\cdot)$   (for each time) as the solution to  
\begin{align}\label{eq:v}
\left\{\begin{aligned}
2\mu\Delta v&=\nu\partial_t v && \mbox{ in}\quad  \IR_+^n,\\
\IS(v)\cdot \e_n&=g && \mbox{ at}\quad\partial\IR_+^n.
\end{aligned}\right.
\end{align}
We split $v$ into $v=z+\wt z$ where $\wt z$  is a suitable solution of
\begin{align}\label{eq:zt}\left\{\begin{aligned}
2\mu\Delta \wt z&=\nu\partial_t v && \mbox{ in}\quad \IR_+^n,\\
\partial_n\wt z&=0 && \mbox{ at}\quad \partial\IR_+^n
\end{aligned}\right.\end{align}
and 
\begin{align}\label{eq:z}\left\{\begin{aligned}
\Delta z&=0&&\mbox{ in}\quad \IR_+^n,\\
\IS(z)\cdot \e_n&= k:=g-\IS(\wt z)\cdot \e_n&&\mbox{ at}\quad\partial\IR_+^n.
\end{aligned}\right.\end{align}
In order to solve system~\eqref{eq:zt}, it suffices to extend $\partial_tv$ by symmetry 
to $\IR^n$ and, naming $E\partial_tv$ that extension,    to solve 
$$2\mu\Delta \wt z=\nu E\partial_t v\quad\hbox{in }\  \IR^n.$$
Since we are only interested by the construction of the gradient,  one may just  set
\begin{equation}\label{eq:wtilde}
\nabla \wt z:=-\frac{\nu}{2\mu}\nabla(-\Delta)^{-1}E\partial_tv.\end{equation}
Since  $\nabla(-\Delta)^{-1}$ is a multiplier homogeneous of degree $-1$, it  
maps $\dot\B^s_{p,1}(\IR^n)$ to   $\dot\B^{s+1}_{p,1}(\IR^n)$ whenever $s+1\leq n/p$, 
and we get $\nabla\wt z\in\LL_1(\IR_+;\dot\B^{s+1}_{p,1}(\IR^n))$ as well as 
\begin{equation}\label{eq:lame4}
\|\nabla \wt z\|_{\LL_1(\IR_+;\dot\B^{s+1}_{p,1}(\IR^n))}
\lesssim \|E\partial_tv\|_{\LL_1(\IR_+;\dot\B^{s}_{p,1}(\IR^n))}.
\end{equation}
Of course, the restriction of $\nabla\wt z$ to  $\IR^n_+$ satisfies the same estimate
and $\nabla\wt z$ is skewsymmetric; hence the boundary condition of~\eqref{eq:zt} is fulfilled.

Next, we solve~\eqref{eq:z} explicitly by going to the Fourier side  with respect to the horizontal variable  only (the time variable is fixed throughout and omitted). 
Since we want $z$ to  converge  to $0$ at infinity, we have:
\begin{equation}\label{eq:lame5}
\cF_{x'}z(\xi',x_n)  = \e^{-|\xi'|x_n}\, \cF_{x'}z(\xi',0),\quad \xi'\in\IR^{n-1},\: x_n\geq0,
\end{equation}
and the boundary condition of~\eqref{eq:z} tells us that the vector $\wh A$ has to satisfy:
\begin{equation}\label{eq:casn}\begin{pmatrix} -\mu \lvert \xi' \rvert \Id_{n - 1} &\ii \mu\xi^{\prime}\\
\ii\lambda (\xi')^\top&-\nu \lvert \xi' \rvert \end{pmatrix}\begin{pmatrix} \cF_{x'}z^{\prime}(\xi',0)\\\cF_{x'}z^n(\xi',0)\end{pmatrix} = \begin{pmatrix}\cF_{x'} k'(\xi',0)\\ \cF_{x'}k^n(\xi',0)\end{pmatrix}\cdotp \end{equation}
Let us first consider the two-dimensional case. 
Inverting the matrix (which is always possible for $\lambda+\mu\not=0$) and plugging in~\eqref{eq:lame5}, we get
\begin{equation}\label{eq:lame3}\cF_{x_1}z(\xi_1,x_2)=
-\frac{\e^{-|\xi_1| x_2}}{2 \xi_1^2 \mu(\lambda+\mu)} 
\begin{pmatrix} \nu \lvert \xi_1 \rvert&\ii\mu \xi_1\\
\ii\lambda \xi_1&\mu \lvert \xi_1 \rvert\end{pmatrix}\begin{pmatrix} \cF_{x_1}k^1(\xi_1,0)\\  \cF_{x_1} k^2(\xi_1,0)\end{pmatrix}\cdotp\end{equation}
{}Notice that~\eqref{eq:lame3} implies that
$$\begin{aligned}
 \partial_1z (\xi_1 , x_2) &=
-\cF_{\xi_1}^{-1} \frac{\ii \e^{-|\xi_1| x_2}}{2 \mu(\lambda+\mu)} 
\begin{pmatrix} \nu\sgn(\xi_1)&\ii\mu \\
\ii\lambda &\mu\sgn(\xi_1)\end{pmatrix}\begin{pmatrix} \cF_{x_1}k^1(\xi_1,0)\\  \cF_{x_1} k^2(\xi_1,0)\end{pmatrix}\\
\andf \partial_2z (x_1 , x_2) &=
\cF_{\xi_1}^{-1} \frac{\e^{-|\xi_1| x_2}}{2 \mu(\lambda+\mu)} 
\begin{pmatrix} \nu &\ii\mu\sgn(\xi_1)\\
\ii\lambda \sgn(\xi_1)&\mu \end{pmatrix}\begin{pmatrix} \cF_{x_1}k^1(\xi_1,0)\\  \cF_{x_1} k^2(\xi_1,0)\end{pmatrix}\cdotp
\end{aligned}$$
In the $n$-dimensional case, the adjugate matrix of the matrix
$A(\xi)$
appearing in~\eqref{eq:casn} is a homogeneous polynomial 
of degree $n-1$ with respect to $\xi'$ and $|\xi'|,$ 
and $\det A(\xi)=2(-1)^n\mu^{n-1}(\lambda+\mu)|\xi'|^n.$
Hence we get
$$\begin{aligned}
 \nabla' z' (x' , x_n) &=
\cF_{\xi'}^{-1} \frac{\ii\xi' \e^{-|\xi'| x_n}}{2(-1)^n \mu^{n-1}(\lambda+\mu)|\xi'|^n} 
\adj(A(\xi'))\begin{pmatrix} \cF_{x'}k'(\xi',0)\\  
\cF_{x'} k^n(\xi',0)\end{pmatrix}\\
 \andf\partial_nz' (x' , x_n) &=-
\cF_{\xi'}^{-1} \frac{\xi' \e^{-|\xi'| x_n}}{2(-1)^n \mu^{n-1}(\lambda+\mu)|\xi'|^n} 
\adj(A(\xi'))\begin{pmatrix} \cF_{x'}k'(\xi',0)\\  
\cF_{x'} k^n(\xi',0)\end{pmatrix}\cdotp
\end{aligned}$$
Note that in all the above formulae appears a  homogeneous multiplier of order zero  times $\e^{- \lvert \xi' \rvert x_n}.$
Hence,~\cite[Lemma 2]{Danchin_Mucha}  ensures that, provided $0<s+1\leq n/p,$ 
$$\|\nabla  z\|_{\LL_1(\IR_+;\dot\B^{s+1}_{p,1}(\IR^n_+))}\lesssim
\|k\|_{\LL_1(\IR_+;\dot\B^{s+1-1/p}_{p,1}(\partial\IR_+^n))}.$$
Hence, since $k$ is defined as $g - \IS (\wt z) \e_n$, we find by virtue of~\eqref{eq:lame4} and of the trace theorem that
\begin{equation}\label{eq:lame6}
\|\nabla  z\|_{\LL_1(\IR_+;\dot\B^{s+1}_{p,1}(\IR^n_+))}\lesssim
\|g\|_{\LL_1(\IR_+;\dot\B^{s+1-1/p}_{p,1}(\partial\IR_+^n))} + 
\|\partial_t v\|_{\LL_1(\IR_+;\dot\B^{s}_{p,1}(\IR_+^n))}.
\end{equation}
Combining this with~\eqref{eq:lame0} yields~\eqref{eq:Dvlame}. 
\medbreak
Our last task is to bound  $\nabla_{x}v|_{\partial\IR^n_+}$ in $E\dot\Y^s_p$. To this end, we observe that, in dimension 
$n=2,$ we have by construction, 
   \begin{equation}\label{eq:Db}
   \begin{pmatrix}  \cF_{t,x_1}(\partial_1 v^1)(\tau,\xi_1,0)\\  
   \cF_{t,x_1}(\partial_1 v^2)(\tau,\xi_1,0)\end{pmatrix}=-\frac{\ii \xi_1}{\ii\tau +\xi_1^2}\begin{pmatrix} \frac r\mu& \frac \ii\nu\xi_1\\
 \ii\frac{\lambda}{\mu\nu}\xi_1& \frac r\nu\end{pmatrix} \begin{pmatrix} \wh g^1(\tau,\xi_1,0)\\\wh g^2(\tau,\xi_1,0)\end{pmatrix}\cdotp\end{equation}
A similar formula holds in general dimension: 
 $$
   \begin{pmatrix}  \cF_{t,x'}(\nabla_{x'} v^{\prime})(\tau,\xi',0)\\  
   \cF_{t,x'}(\nabla_{x'} v^n)(\tau,\xi',0)\end{pmatrix}=
   \frac{\ii\xi'}{(-1)^n\mu^{n-1}\nu r^{n-2}(\ii\tau+|\xi'|^2)}\adj\begin{pmatrix}
   -\mu r\Id_{n-1}&\ii\mu\xi'\\ \ii\lambda(\xi')^{\top}&-\nu r
   \end{pmatrix}
    \begin{pmatrix} \wh g'(\tau,\xi',0)\\\wh g^n(\tau,\xi',0)\end{pmatrix}\cdotp$$
    We observe that in all cases, $\nabla_{x'}v$ at the boundary is the outcome of a smooth multiplier $m$ 
    satisfying 
    condition~\eqref{eq:mm} below, acting on $g.$
    Hence, to conclude, it suffices to apply
the following result:
\begin{lemma}\label{p:mbis}
Let  $1 < p < \infty$, $s_1,s_2 \in \IR$ and $m$ be a smooth symbol on $(\IR\times\IR^{n-1})\setminus\{(0,0)\}$ such that 
\begin{equation}\label{eq:mm}
|\partial_\tau^\alpha\partial_{\xi'}^\beta m(\tau,\xi')|\leq C_{\alpha,\beta}\, |r|^{-2\alpha-|\beta|}
\quad\hbox{for }\ \alpha=0,1,2\ \hbox{ and all }\  \beta\in\IN^{n-1}_0.\end{equation}
If $s_1<1$ and $s_2<n/p$,  
then the multiplier $m(D_{t,x'})$ is continuous on $\dot\B^{s_1}_{1,1}(\IR;\dot\B^{s_2}_{p,1}(\IR^{n-1})).$
\end{lemma}

\begin{proof}
Let $v:= m(D_{t,x'}) g.$ Then
$$
 \| v\|_{\dot\B^{s_1}_{1,1}(\IR;\dot\B^{s_2}_{p,1}(\IR^{n-1}))}= \sum_{\ell=-\infty}^\infty\sum_{j = - \infty}^{\infty} 2^{s_1\ell} 2^{s_2 j} \|\dot\Delta_\ell^t\dot\Delta_j^hv \|_{\LL_1(\IR;\LL_p (\IR^{n-1}))}.
 $$
Introducing again $\theta:=\psi(\cdot/2)+\psi +\psi(2\cdot)$ and using~\eqref{eq:LP1}, we see that for all $(\ell,k)\in\IZ^2,$ we have
 $$ 
\dot\Delta_\ell^t\dot\Delta_j^hv= M_{j,\ell}*_{t,x'} \dot\Delta_\ell^t\dot\Delta_j^h g
\!\!\with\!\! M_{j,\ell}(t,x):=\int_{\IR\times\IR^{n-1}}\e^{\ii t\tau} \e^{\ii x'\cdot\xi'} \theta(2^{-\ell}\tau)\theta(2^{-j}|\xi'|)
m(\tau,\xi')\frac{\,\d\tau \,\d\xi'}{(2\pi)^n}\cdotp
$$
Hence, the desired conclusion follows from 
\begin{equation}\label{eq:Mjl} 
\sup_{(j,\ell)\in\IZ^2} \|M_{j,\ell}\|_{\LL_1(\IR\times\IR^{n-1})}<\infty.
\end{equation}
To establish~\eqref{eq:Mjl}, we mimic  the  proof of  Proposition~\ref{Prop: Multiplier norm}:
 we split the integral defining the term $\|M_{j,\ell}\|_{\LL_1(\IR\times\IR^{n-1})}$ into
 $$
 \int_{|t|\leq 2^{-\ell}}\int_{|x'|\leq 2^{-j}}|M_{j,\ell}|\,\d t\,\d x'+ \int_{|t|\leq 2^{-\ell}}\int_{|x'|>2^{-j}}\dots
 + \int_{|t|>2^{-\ell}}\int_{|x'|\leq 2^{-j}}\dots+ \int_{|t|>2^{-\ell}}\int_{|x'|>2^{-j}}\dots
 $$
  Denoting by ${\mathrm I},$ $\mathrm{II},$ $\mathrm{III}$ and $\mathrm{IV}$ the corresponding terms, and adapting the strategy used
 in the proof of  Proposition~\ref{Prop: Multiplier norm}, we get for ${\mathrm I}$:
 $$ |{\mathrm I}|\lesssim 2^{-\ell} 2^{-j(n-1)} \int_{\IR}\int_{\IR^{n-1}} |\theta(2^{-\ell}\tau)\theta(2^{-j}|\xi'|)|
 \,\d\tau \,\d\xi'\lesssim 1.$$
 To handle $\mathrm{II},$ we use~\eqref{eq:t} and get, after integrating by parts
 $$|\mathrm{II}|\lesssim 2^{-j(n-1)}\int_{|t|>2^{-\ell}} 
  \int_{\IR}\int_{\IR^{n-1}}\bigl|\partial^2_\tau\bigl(\theta(2^{-\ell}\tau)\theta(2^{-j}|\xi'|)m(\tau,\xi')\bigr) \bigr|
   \,\d\tau \,\d\xi'\frac{\d t}{t^2}\cdotp$$
   Using  the chain rule and~\eqref{eq:mm} with $\alpha=0,1,2$ and $\beta=0,$ we get
    $$|\mathrm{II}|\lesssim 2^{-j(n-1)}2^{\ell} \, 2^\ell 2^{j(n-1)} 2^{- 2 \ell}\lesssim1.$$
  To bound $\mathrm{III},$ we use~\eqref{eq:xp} with $2N>n-1$:
   $$|\mathrm{III}|\lesssim 2^{-\ell} \int_{|x'|>2^{-j}} 
  \int_{\IR}\int_{\IR^{n-1}}\bigl| (\Delta_{\xi'})^N\bigl(\theta(2^{-\ell}\tau)\theta(2^{-j}|\xi'|)m(\tau,\xi')\bigr) \bigr|
   \,\d\tau \,\d\xi'\frac{\d t}{|x'|^{2N}},$$
and, thanks to~\eqref{eq:mm} with $\alpha=0$ and $|\beta|\leq2N,$ we get
  $$|\mathrm{III}|\lesssim  2^{-\ell} 2^{j(2N-n+1)} 2^\ell 2^{j(n-1)}
   2^{-2N j}\lesssim1.$$
 Finally, to bound  $\mathrm{IV},$ we combine the two methods and eventually get
 for $2N>n-1,$ 
  $$ |\mathrm{IV}|\lesssim 2^\ell 2^{j(2N-n+1)} 2^\ell 2^{j(n-1)} 2^{- 2 \ell} 2^{- 2 N j}
  \lesssim1.$$ 
  Putting together the four inequalities completes the proof of~\eqref{eq:Mjl}, and thus of the lemma. 
 \end{proof}
As all the coefficients of the matrix in the right-hand side of~\eqref{eq:Db} (extended to the $n$-dimensional case) fulfill~\eqref{eq:mm}, applying the above proposition to  $\nabla_{x'}v|_{\partial\IR^n_+}$ gives
 $$\|\nabla_{x'}v|_{\partial\IR^n_+}\|_{E\dot\Y^s_p}\lesssim \|g \|_{E\dot Y^s_p}.$$
 Furthermore, on $\partial \IR^n_+$ we have
 $$ \partial_nv'=\mu^{-1}g'-D_{x'}v^n\andf
 \partial_nv^n=\nu^{-1}g^n-\lambda\nu^{-1}\divergence_{x'}v'.$$
 Hence $\partial_nv|_{\partial\IR^n_+}$ also satisfies the desired property, and  
 we eventually have:    
 $$   \|\nabla_{x}v|_{\partial\IR^n_+}\|_{E\dot\Y^s_p}\lesssim \|g\|_{E\dot\Y^s_p}.$$
 The last step is  to set $u=v+w$ with $w$ solving 
 System~\eqref{eq:lame2}. 
Our  previous  computations  ensure that 
the source term of the first equation is in $\LL_1(\IR_+;\dot\B^s_{p,1} (\IR^n_+)).$ 
So the problem reduces to checking that the Lam\'e operator fulfills
the Da Prato -- Grisvard conditions,  which was 
established in Theorem~\ref{Thm: Da Prato - Grisvard for Lame}, 
under the (optimal) condition~\eqref{eq:ellipticity}.
{}From it, provided $0<s\leq n/p$ and $0<s<1,$ 
we get $\partial_tw\in \LL_1(\IR_+;\dot\B^{s}_{p,1}(\IR_+^n))$ and 
$$\|\nabla^2w, \partial_tw\|_{\LL_1(\IR_+;\dot\B^{s}_{p,1}(\IR_+^n))} \lesssim 
\|f , \divergence\IS(v)- \partial_tv\|_{\LL_1(\IR_+;\dot\B^{s}_{p,1}(\IR_+^n))} + \| u_0 \|_{\dot \B^s_{p , 1} (\IR^n_+)}. $$
Hence,  owing to Corollary 
\ref{Cor: Comparison of gradients}, if, in addition,
$s\leq n/p-1$ then we may write:
\begin{align}
\label{Eq: Gradient of w}
\|\nabla w\|_{\LL_1(\IR_+;\dot\B^{s+1}_{p,1}(\IR_+^n))} \lesssim 
\| f , \divergence\IS(v)- \partial_tv\|_{\LL_1(\IR_+;\dot\B^{s}_{p,1}(\IR_+^n))} + \| u_0 \|_{\dot \B^s_{p , 1} (\IR^n_+)}.
\end{align}
Hence, by the trace theorem, as $s+1>1/p,$
\begin{equation}\label{eq:tracew}
\|\nabla w|_{\partial\IR^n_+}\|_{\LL_1(\IR_+;\dot\B^{s+1-\frac1p}_{p,1}(\partial\IR_+^n))} \lesssim 
\|f , \divergence\IS(v)- \partial_tv\|_{\LL_1(\IR_+;\dot\B^{s}_{p,1}(\IR_+^n))} + \| u_0 \|_{\dot \B^s_{p , 1} (\IR^n_+)}. 
\end{equation}
To end the proof of  Theorem~\ref{Thm:lame}, 
there only remains to show that 
$$\nabla_xw|_{\partial\IR_+^n}\in\dot\B^{\frac12(s+1-\frac1p)}_{1,1}(\IR_+;\dot\B^{0}_{p,1}(\partial\IR^n_+))\quad
\andf
\nabla_xw|_{\partial\IR_+^n}\in\dot\B^{\frac12(1-\frac1p)}_{1,1}(\IR_+;\dot\B^{s}_{p,1}(\partial\IR^n_+))
\ \hbox{ if }\ s<0.$$
So far, we know that 
\begin{equation}\label{eq:sofarweknow}
w\in\dot\W^1_1(\IR_+;\dot\B^s_{p,1}(\IR^n_+))\andf 
\nabla w\in\LL_1(\IR_+;\dot\B^{s+1}_{p,1}(\IR_+^n)).\end{equation}
Using  Lemma~\ref{Lem: Extension operators} and Proposition~\ref{Prop: Proper boundedness extension operators}
to extend $w$ on $\IR^n$ then, extending also $w$ to negative times by, e.g., even reflection, 
we get an extended vector-field $Ew$ on $\IR\times\IR^n$  that satisfies
\begin{equation}\label{eq:Ewt}\|(Ew)_t\|_{\LL_1(\IR;\dot\B^s_{p,1}(\IR^n))}\lesssim \|w_t\|_{\LL_1(\IR_+;\dot\B^s_{p,1}(\IR^n_+))}\!\andf\! 
\|\nabla Ew\|_{\LL_1(\IR;\dot\B^{s+1}_{p,1}(\IR^n))}\lesssim \|\nabla w\|_{\LL_1(\IR_+;\dot\B^{s+1}_{p,1}(\IR^n_+))}.\end{equation}
Now, Proposition~\ref{p:interpopopo} guarantees that 
$$\|Ew\|_{\dot \B^{\frac12}_{1,1}(\IR;\dot\B^{s+1}_{p,1}(\IR^n))}\lesssim
\|(Ew)_t\|_{\LL_1(\IR;\dot\B^s_{p,1}(\IR^n))}^{\frac12} 
\|\nabla Ew\|_{\LL_1(\IR;\dot\B^{s+1}_{p,1}(\IR^n))}^{\frac12},$$
whence, back to $w$ and using~\eqref{eq:Ewt}, 
\begin{align}\label{Eq: Interpolation inequality Besov}
\|w\|_{\dot \B^{\frac12}_{1,1}(\IR_+;\dot\B^{s+1}_{p,1}(\IR^n_+))}\lesssim
\|w_t\|_{\LL_1(\IR_+;\dot\B^s_{p,1}(\IR^n_+))}^{\frac12}
\|\nabla w\|_{\LL_1(\IR_+;\dot\B^{s+1}_{p,1}(\IR^n_+))}^{\frac12}.\end{align}
Now, by the trace theorem (observe that $\frac1p< s+1\leq \frac{n}{p}$), 
$$w|_{\partial\IR_+^n} \in \dot\B^{\frac12}_{1,1}(\IR_+;\dot\B^{s+1-\frac1p}_{p,1}(\partial\IR^n_+)),
$$
and thus 
\begin{equation}\label{eq:www1}
\nabla_{x'}(w|_{\partial\IR_+^n}) \in \dot\B^{\frac12}_{1,1}(\IR_+;\dot\B^{s-\frac1p}_{p,1}(\partial\IR^n_+)).
\end{equation}
Notice that~\eqref{Eq: Interpolation inequality Besov} combined with~\eqref{eq:sofarweknow} implies that $w \in \LL_2 (\IR_+;\dot\B^{s+1}_{p,1}(\IR^n_+))$. Thus, for almost every $t > 0$ the function $w(t)$ lies in $\dot\B^{s+1}_{p,1}(\IR^n_+)$. Moreover,~\eqref{Eq: Gradient of w} implies that for almost every $t > 0$ we have $\nabla_{x^{\prime}} w (t) \in \dot \B^{s + 1}_{p , 1} (\IR^n_+)$. Employing Lemma~\ref{Lem: Trace and gradient} below, we find that
\begin{align*}
 (\nabla_{x^{\prime}} w)|_{\partial \IR^n_+} = \nabla_{x^{\prime}} (w|_{\partial \IR^n_+})
\end{align*}
so that we simply write $\nabla_{x'}w|_{\partial\IR_+^n}$ in the following.
\smallbreak
Finally, since we know from~\eqref{eq:tracew} that 
\begin{equation}\label{eq:www2}
(\nabla_{x'}w)|_{\partial\IR_+^n} \in \LL_1(\IR_+;\dot\B^{s+1-\frac1p}_{p,1}(\partial\IR^n_+)),
\end{equation}
 interpolating  between~\eqref{eq:www1} and~\eqref{eq:www2}
 (here we use an obvious generalization of Proposition~\ref{p:interpopopo} and 
 notice that both $(s+1-1/p)/2$ and $(1-1/p)/2$ are in $(0,1/2)$ for $-1+1/p<s<1/p$) yields 
\begin{equation}\label{eq:Dxprime}
\nabla_{x'}w|_{\partial\IR_+^n}\in\dot\B^{\frac12(s+1-\frac1p)}_{1,1}(\IR_+;\dot\B^{0}_{p,1}(\partial\IR^n_+))\andf \nabla_{x'}w|_{\partial\IR_+^n}\in\dot\B^{\frac12(1-\frac1p)}_{1,1}(\IR_+;\dot\B^{s}_{p,1}(\partial\IR^n_+)).\end{equation}
The boundary condition  $\IS(w)\cdot \e_n=0$  ensures that 
$$\partial_{x_n}w'=-\nabla_{x'}w^n\andf \nu\partial_{x_n}w^n=-\lambda{\rm div}_{x'}w'.$$
Hence~\eqref{eq:Dxprime} also holds for $\partial_{x_n}w,$ which  completes the proof of Theorem~\ref{Thm:lame}.\qed

\begin{lemma}
\label{Lem: Trace and gradient}
Let $1 \leq p < \infty$ and $1 / p < \theta , \phi \leq n / p$. Then for all $w \in \dot \B^{\theta}_{p , 1} (\IR^n_+)$ with the property that $\nabla_{x^{\prime}} w \in \dot \B^{\phi}_{p , 1} (\IR^n_+)$ it holds
\begin{align*}
 (\nabla_{x^{\prime}} w)|_{\partial \IR^n_+} = \nabla_{x^{\prime}} (w|_{\partial \IR^n_+}).
\end{align*}
\end{lemma}

\begin{proof}
First of all, we remark that by the extension results proven in Proposition~\ref{Prop: Proper boundedness extension operators} and by the proof of~\cite[Prop.~2.27]{Bahouri_Chemin_Danchin} there exists a sequence $(\varphi_k)_{k \in \IN}$ of Schwartz functions such that
\begin{align*}
 \varphi_k|_{\IR^n_+} \to w \quad \text{in} \quad &\dot \B^{\theta}_{p , 1} (\IR^n_+) \quad \text{as} \quad k \to \infty \\
 \nabla_{x^{\prime}} \varphi_k|_{\IR^n_+} \to \nabla_{x^{\prime}} w \quad \text{in} \quad &\dot \B^{\phi}_{p , 1} (\IR^n_+) \quad \text{as} \quad k \to \infty.
\end{align*}
Those functions obviously satisfy
\begin{align*}
 (\nabla_{x^{\prime}} \varphi_k)|_{\partial \IR^n_+} = \nabla_{x^{\prime}} (\varphi_k|_{\partial \IR^n_+}).
\end{align*}
By the conditions imposed on $\theta$ and $\phi$, we can take the limit $k \to \infty$ and find that $(\nabla_{x^{\prime}} w)|_{\partial \IR^n_+}$ and $\nabla_{x^{\prime}} (w|_{\partial \IR^n_+})$ coincide as elements in $\cS_h^{\prime} (\IR^{n - 1})$.
\end{proof}


\section{The Lam\'e system in a perturbed half-space}

Having Theorem~\ref{Thm:lame} for the half-space at hand, 
we now consider the case where the fluid domain $\Omega$
is  a  small perturbation of $\IR^n_+$ (such a generalization is consistent 
with the study of the full stability issue of the nonlinear system~\eqref{eq:pressless}).
More precisely, we assume that $\Omega$ satisfies~\eqref{o1}.
It follows that there exists  a continuous function
$H:\overline{\IR^n_+} \to \overline\Omega$
such that $H$ is a diffeomorphism  
from $\IR^n_+$ to $\Omega,$
\begin{equation}\label{eq:H}
H(\partial \IR^n_+)=\partial \Omega,\quad \nabla H \in
\dot \B^{n/p}_{p,1}(\IR^n_+)\andf\|\nabla H- \Id\|_{\dot \B^{n/p}_{p,1}(\IR^n_+)} \leq Cc.\end{equation}
To construct $H,$ we set 
 $H(z)=z+Eh(z)$ for $z\in \overline{\IR^n_+}$, where $Eh$ is an extension of $h$
 such that $\nabla Eh \in \dot \B^{n/p}_{p,1}(\IR^n_+)$.
 For $Eh,$ one may 
 use the following harmonic extension:
 $$ Eh(z',z_n):=\cF_{\xi'}^{-1}\Bigl(e^{-|\xi'|z_n}
 \cF_{z'}h\Bigr)(z',z_n),\quad z'\in\IR^{n-1},\: z_n\geq0.$$
 Indeed, we then have
 $$\nabla_{z'}Eh=\cF^{-1}_{\xi'}\Bigl(e^{-|\xi'|z_n}
 \cF_{z'}\nabla_{z'}h\Bigr)\andf
  \partial_{z_n}Eh=i\cF^{-1}_{\xi'}\Bigl(e^{-|\xi'|z_n}
 \frac{\xi'}{|\xi'|}\cF_{z'}\nabla_{z'}h\Bigr) $$
 so that, in light of~\cite[Lemma 2]{Danchin_Mucha}, we have
 $$ \|\nabla Eh\|_{\dot \B^{n/p}_{p,1}(\IR^n_+)}
 \lesssim  \|\nabla_{z'}h\|_{\dot \B^{(n-1)/p}_{p,1}(\IR^{n-1})}\leq Cc.$$
 Then,  we consider the following linear system:
 \begin{align}\label{eq:lameomega}
 \left\{\begin{aligned}
   \rho u_t - \divergence \IS(u) &= f && \mbox{ in} \quad \IR_+ \times \Omega,\\
   \IS(u)\cdot {\bar n} &= g && \mbox{ at} \quad \IR_+ \times \partial \Omega,\\
   u|_{t=0}&= u_0 && \mbox{ in} \quad \Omega,
   \end{aligned}\right.
 \end{align}
 where $\rho$ is a given 
 time independent positive function on $\Omega,$ and $\bar n$ stands for the unit outward normal vector  at $\partial\Omega.$
\medbreak
To formulate the perturbed maximal regularity result, let us introduce the maximal regularity space $\E_p$ and the boundary data space $\dot\Y_p$ pertaining to a domain $\Xi$
of $\IR^n$ as follows:
\begin{align*}
 \dot\Y_p (\partial\Xi) := \dot \B^{\frac{n-1}{2 p}}_{1 , 1} (\IR_+ ; \dot \B^0_{p , 1} (\partial \Xi)) \cap \LL_1 (\IR_+ ; \dot \B^{\frac{n-1}{p}}_{p , 1} (\partial \Xi))\andf
\end{align*}
$$\displaylines{
 \E_p(\Xi)\hfill\cr\hfill  := \{ u \in \cC_b(\IR_+ ; \dot \B^{\frac{n}{p} - 1}_{p , 1} (\Xi)): \nabla u \in \LL_1 (\IR_+ ; \dot \B^{\frac{n}{p}}_{p , 1} (\Xi)) , \, \partial_t u \in \LL_1 (\IR_+ ; \dot \B^{\frac{n}{p} - 1}_{p , 1} (\Xi)) \!\!\andf\!\!\nabla u|_{\partial \Xi} \in \dot\Y_p (\partial\Xi) \}}$$
endowed with the canonical norm. If $\Xi$ is a bent half-space, the properties of $H$ given by~\eqref{eq:H} 
combined with Proposition~\ref{Prop: Change of half} and Lemma~\ref{Lem: Extension of diffeomorphism} 
allow to define the function spaces
$\dot \B^s_{p , 1} (\Xi)$ as well as $\dot \B^s_{p , 1} (\partial \Xi)$ via a pullback of the diffeomorphism $H$ over $\IR_+^n$ and $\IR^{n-1}$, namely
$$
d \in \dot \B^s_{p , 1} (\Xi) \mbox{ \ iff \ }
H^*d = d \circ H \in \dot \B^s_{p , 1} (\IR_+^n)
$$
and
$$
d \in  \dot \B^s_{p , 1} (\partial \Xi) 
\mbox{ \ iff \ } 
H^*d = d \circ H \in \dot \B^s_{p , 1} (\partial \IR^n_+).
$$

In this section we prove the following result.
\begin{theorem}\label{th:lameomega}
Let $n-1 < p < n$ and let $\Omega$ be a small perturbation of $\IR^n_+$ in the sense of~\eqref{o1}. Let  $\rho \in \mathcal{M}(\dot \B^{n/p-1}_{p,1}(\Omega)\cap \LL_\infty(\Omega))$
be  such that
 \begin{equation}\label{eq:smallrho}
  \|\rho-1\|_{\mathcal{M}(\dot \B^{n/p-1}_{p,1}(\Omega)\cap \LL_\infty(\Omega))} \leq c \ll 1.
 \end{equation}
Furthermore, let $f\in \LL_1(\IR_+;\dot \B^{n/p - 1}_{p,1}(\Omega))$, 
$g\in \dot\Y_p (\partial\Omega)$ and $u_0 \in \dot 
\B^{n/p-1}_{p,1}(\Omega)$.
\smallbreak
Then, there exists a unique global solution $u \in \E_p (\Omega)$ to system~\eqref{eq:lameomega}. Furthermore, 
we have
\begin{align*}
 \|u \|_{\E_p (\Omega)} \leq C\big( \|f\|_{\LL_1(\IR_+;\dot \B^{n/p - 1}_{p,1}(\Omega))}+ 
\|g\|_{\dot\Y_p (\partial\Omega)}+ \|u_0\|_{ \dot 
\B^{n/p-1}_{p,1}(\Omega)}\big)\cdotp
\end{align*}
\end{theorem}
\begin{proof}
 At the beginning we define $z=H^{-1}(x)$ 
 and $\bar\nu(z)=\bar n(x),$
 then  set
 $$  w(t,z):=u(t,x),\quad\! \bar\rho(z):=\rho(x),
  \quad\! \bar u_0(z):=u_0(x),
\!\quad   \bar f(t,z):= f(t,x)\!\andf\! 
\bar g(t,z):=g(t,x)$$
so that, denoting  $A (z) := (D H (z))^{-1},$ 
the vector-field $w$ has to satisfy:
\begin{align}\label{eq:lame0-1}\left\{
\begin{aligned}
  \bar \rho w_t - A^{\top} : \nabla_z 
  (\mu(A^{\top} \nabla_z w)+ \mu(A^{\top} \nabla_z w)^{\top} + 
  \lambda A^{\top}:\nabla_z w \Id)&=\bar f && \mbox{ in} \quad \IR_+\times \IR^n_+,\\[5pt]
  (\mu(A^{\top} \nabla_z w)+ \mu(A^{\top} \nabla_z w)^{\top} + \lambda A^{\top}:\nabla_z w \Id)\cdot \bar\nu&=\bar g && \mbox{ at} \quad \IR_+ \times \partial \IR^n_+,\\[5pt]
  w|_{t=0}&=\bar u_0 && \mbox{ in} \quad \IR^n_+.
\end{aligned}\right.
\end{align}
Before solving the above system, two observations are in order. 

First,  the smallness condition in~\eqref{eq:H} 
and the fact that $\dot \B^{n/p}_{p,1}(\IR_+^n)$ is a Banach algebra allow  us to compute  $A-\Id$ by means of the 
following Neumann series expansion:
$$A-\Id =\sum_{k=1}^\infty (-1)^k\bigl(\nabla H-\Id\bigr)^k.$$
Hence we have $A-\Id$ in  $\dot \B^{n/p}_{p,1}(\IR_+^n)$ with 
the estimate
\begin{equation}\label{eq:smallA}
\|A-\Id\|_{\dot \B^{n/p}_{p,1}(\IR_+^n)}\leq Cc.
\end{equation}
Second, in the new coordinates system, the normal vector $\bar\nu$ at the boundary becomes
$$\bar\nu(z',0)=\frac{(D_{z'}h(z'),-1)}{\sqrt{1+|D_{z'}h(z')|^2}}\cdotp$$
Hence, using Taylor series expansion, the fact 
that $\dot\B^{(n-1)/p}_{p,1}(\IR^{n-1})$ is a Banach algebra, and taking advantage of~\eqref{o1}, we get 
$$\|\bar\nu+\e_n\|_{\dot\B^{(n-1)/p}_{p,1}(\IR^{n-1})}
\lesssim\|h^\prime\|_{\dot\B^{(n-1)/p}_{p,1}(\IR^{n-1})}\leq Cc.$$
To solve system~\eqref{eq:lame0-1},  we plan to apply the Banach fixed point theorem to the mapping
\begin{align}
\label{Eq: Linear contraction}
 \Phi : \E_p (\IR^n_+) \to \E_p (\IR^n_+), \quad v \mapsto u,
\end{align}
where $u$ is the unique solution to the following system
\begin{align}\label{eq:lame00}
\left\{\begin{aligned}
u_t -\divergence_z \IS(u)&=\bar f+F(v)   && \mbox{ in} \quad \IR_+\times \IR^n_+,\\
\IS(u)\cdot (-\e_n)|_{z_n=0} &= \bar g +G(v)   && \mbox{ at} \quad \IR_+ \times \partial \IR^n_+,\\
  u|_{t=0}&=\bar u_0 && \mbox{ in} \quad \IR^n_+
\end{aligned} \right.
\end{align}
with
$$
 \begin{array}{l}
F(v)=(1-\overline{\rho})v_t -\divergence_z \IS(v)  + A^{\top} : \nabla_z 
  \left(\mu(A^{\top} \nabla_z v)+ \mu(A^{\top} \nabla_z v)^{\top} +   \lambda A^{\top}:\nabla_z v \Id\right),\\[5pt]
G(v)= -\IS(v)\cdot \e_n - (\mu(A^{\top} \nabla_z v)+ \mu(A^{\top} \nabla_z v)^{\top} + \lambda A^{\top}:\nabla_z v \Id)\cdot \bar\nu.
 \end{array}$$
To show that the mapping $\Phi$ is indeed well-defined, we first have to check that $F(v)$ and $G(v)$ are, for any $v \in \E_p (\IR^n_+)$,  admissible data with regard to Theorem~\ref{Thm:lame}. In order to 
treat $F (v)$ we note that the  product rule formulated in~\cite[Lem.~A.5]{Danchin14} or 
in Proposition~\ref{prod-Bes} yields if $1<p<2n$ 
$$\displaylines{
 \|   A^{\top} : \nabla_z 
  \left(\mu(A^{\top} \nabla_z v)+ \mu(A^{\top} \nabla_z v)^{\top} +   \lambda A^{\top}:\nabla_z v \Id\right)
 -\divergence_z \IS(v) 
 \|_{\LL_1(\IR_+; \dot \B^{n/p-1}_{p,1}(\IR^n_+))} 
 \hfill\cr\hfill\leq C\|\Id-A\|_{\dot \B^{n/p}_{p,1}(\IR^n_+)}
 \|\nabla v\|_{\LL_1(\IR^n_+;\dot \B^{n/p}_{p,1}(\IR^n_+))}.}$$
Hence, 
\begin{multline}\label{eq:F}
 \|F(v)\|_{\LL_1(\IR_+; \dot \B^{n/p-1}_{p,1}(\IR^n_+))} \\\leq 
 \|1-\overline{\rho}\|_{\mathcal{M} (\dot \B^{n/p-1}_{p,1}(\IR^n_+))}\|v_t\|_{ \LL_1(\IR_+; \dot \B^{n/p-1}_{p,1}(\IR^n_+))}+ C\|\Id-A\|_{\dot \B^{n/p}_{p,1}(\IR^n_+)} \|\nabla v\|_{\LL_1(\IR_+;\dot \B^{n/p}_{p,1}(\IR^n_+))}.
\end{multline}
Let us observe that, owing to \eqref{eq:smallrho} and Proposition~\ref{Prop: Change of half}, we have
\begin{equation}\label{eq:smallrhob}
 \|1-\overline{\rho}\|_{\mathcal{M}(\dot\B^{n/p-1}_{p,1}(\IR^n_+))}\leq Cc.
 \end{equation}
 Indeed, we may write by the definition of multiplier spaces:
 $$\begin{aligned}
 \|1-{\rho}\|_{\mathcal{M} (\dot \B^{n/p-1}_{p,1}(\Omega))}&=\sup_{\|a\|_{\dot\B^{n/p-1}_{p,1}(\Omega)}\leq1}
 \|(1-\rho)a\|_{\dot\B^{n/p-1}_{p,1}(\Omega)}\\
 &=\sup_{\|a\|_{\dot\B^{n/p-1}_{p,1}(\Omega)}\leq1}
 \|((1-\bar\rho)\bar a)\circ H^{-1}\|_{\dot\B^{n/p-1}_{p,1}(\Omega)}\\
 &\simeq \sup_{\|\bar a\|_{\dot\B^{n/p-1}_{p,1}(\Omega)}\leq1}
 \|(1-\bar\rho)\bar a\|_{\dot\B^{n/p-1}_{p,1}(\IR^n_+)}\\
 &= \|1-\overline{\rho}\|_{\mathcal{M} (\dot \B^{n/p-1}_{p,1}(\IR^n_+))}.
\end{aligned}
$$
Similarly, since here $A$ is time-independent, we find
arguing as for bounding the boundary terms in the
next section that
\begin{equation}\label{eq:G}
 \|G(v)\|_{\dot\Y_p (\partial\IR^n_+)} \leq C\big( 
\|(\Id-A)|_{z_n=0}\|_{\dot \B^{(n - 1)/p}_{p,1}(\partial \IR^n_+)} 
\|\nabla v|_{z_n=0}\|_{\dot\Y_p (\partial\IR^n_+)}\big)\cdotp
\end{equation}
Now, owing to Theorem~\ref{Thm:lame} (here we need
$n-1<p<n$), we have for some constant $C,$
$$ \|u\|_{\E_p (\IR^n_+)} \leq C\bigl(\|\bar f+F(v)\|_{\LL_1(\IR_+;\dot \B^{n/p}_{p,1}(\IR^n_+))}+ 
\|\overline{g}+G(v)\|_{\dot\Y_p (\partial\IR^n_+)} + \|\overline{u_0}\|_{ \dot 
\B^{n/p-1}_{p,1}(\IR^n_+)}\bigr)\cdotp$$
Hence, putting together with~\eqref{eq:smallrhob},~\eqref{eq:smallA},~\eqref{eq:F} and~\eqref{eq:G}, 
we find that 
\begin{align*}
 \|u\|_{\E_p (\IR^n_+)} \leq C\bigl(\|\bar f\|_{\LL_1(\IR_+;\dot \B^{n/p}_{p,1}(\IR^n_+))}+ 
\|\overline{g}\|_{\dot\Y_p (\partial\IR^n_+)} + \|\overline{u_0}\|_{ \dot 
\B^{n/p-1}_{p,1}(\IR^n_+)} + c \| v \|_{\E_p (\IR^n_+)}\bigr)\cdotp
\end{align*}
In particular, the mapping $\Phi$ is well-defined and a self-map on $\E_p(\IR^n_+).$
\smallbreak
Now, to show that~\eqref{eq:lame00} has a fixed point, it remains to establish that $\Phi$ is a strict contraction. To this end, let $u^0$ be the solution to~\eqref{eq:lame00} with data determined by $v^0$ and let $u^1$ correspond to the data determined by $v^1$ (we keep the same initial 
data $\bar u_0$ and source terms $\bar f$ and $\bar g$ for the two solutions). 
Introduce the abbreviations $\du := u^0-u^1 \andf \dv := v^0-v^1.$ Then,  since $F$ and $G$ are linear, we have
\begin{align}\label{eq:lame01}
 \left\{ \begin{aligned} 
\du_t -\divergence_z \IS(\du)&=F(\dv)   && \mbox{ in} \quad \IR_+\times \IR^n_+,\\
\IS(\du)\cdot (-\e_n)|_{z_n=0} &= G(\dv)   && \mbox{ at} \quad \IR_+ \times \partial \IR^n_+,\\
 \du|_{t=0}&= 0 && \mbox{ in} \quad \IR^n_+.
 \end{aligned} \right.
\end{align}
Then, thanks to \eqref{eq:F} and \eqref{eq:G}, we discover that
\begin{align*}
\|\du\|_{\E_p (\IR^n_+)} \leq Cc \|\dv\|_{\E_p (\IR^n_+)}.
\end{align*}
Thus, taking $c$ such that $Cc<1$ we get that the map defined by~\eqref{Eq: Linear contraction} is a strict contraction and consequently, there exists a unique fixed point to~\eqref{eq:lame00}.
We conclude to the existence of a unique solution $w$ in the set $\E_p (\IR^n_+)$ to system ~\eqref{eq:lame0-1}. 

Since $H:\IR^n_+\to\Omega$ is a diffeomorphism
with $\nabla H\in\dot\B^{n/p}_{p,1}(\IR_+^n),$ we employ 
the product law in~\cite[Lem.~A.5]{Danchin14} together with 
Proposition~\ref{Prop: Change of half} to return to the original system~\eqref{eq:lameomega} and obtain a unique solution in the space $\E_p (\Omega),$ satisfying the desired bound.
\end{proof}

\section{Proof of the global existence result}

The proof is based on the Banach fixed point theorem in the solution space $\E_p$ that has been introduced in the assertion  of  Theorem~\ref{thm:lpressl},
after going to Lagrangian coordinates. In order to solve~\eqref{eq:presslag}, we consider for all time-dependent vector fields $v$ in $\E_p,$ the solution  $u$ to
\begin{align*}
\left\{\begin{aligned} 
\rho_0 \partial_t u - \divergence\IS(u) &= f:= \divergence_v\IS_v(v)-\divergence\IS(v) &&\mbox{ in}\quad  \IR_+\times \Omega, \\
\IS(u) \cdot \bar n|_{\partial\Omega}&= g:=\bigl(\IS(v) \cdot \bar n-\IS_v(v) \cdot \bar n_v\bigr)|_{\partial\Omega}
  &&\mbox{ on}\quad  \IR_+\times\partial\Omega,\\
u|_{t = 0} &= u_0   &&\mbox{ in}\quad \Omega.\end{aligned} \right.
\end{align*}
We claim that  the  map $\Phi: v\mapsto u$ is well-defined on all small enough balls of $\E_p,$ and admits a fixed point. 

\subsubsection*{Step 1. Stability of a small (closed) ball  of $\E_p$ by $\Phi$}

Let $v$ belong to a closed ball $\bar B_{\E_p}(0,R)$ of $\E_p.$ 
We assume that $R>0$ is chosen so small that the following inequality is satisfied (by embedding):
 \begin{equation}\label{eq:smallDv}
\int_{\IR_+} \|Dv\|_{\dot\B^{\frac np}_{p,1}(\Omega)}\,\d\tau\leq c\ll1.
\end{equation} 
By Theorem~\ref{th:lameomega} and remembering  that $\rho_0$ fulfills~\eqref{eq:smalldata}, we get 
\begin{equation}\label{eq:uEp}
\|u\|_{\E_p}\leq C\Bigl(\|u_0\|_{\dot \B^{\frac np-1}_{p,1}(\Omega)} +\|f\|_{\LL_1(\IR_+;\dot \B^{\frac np-1}_{p,1}(\Omega))}
+ \|g\|_{\LL_1(\IR_+;\dot \B^{(n-1)/p}_{p,1}(\partial\Omega))\cap \dot \B^{(n-1)/(2p)}_{1,1}(\IR_+;\dot \B^{0}_{p,1}(\partial\Omega))}\Bigr)\cdotp
\end{equation}
Hence, it only remains  to establish estimates for $f$ in $\LL_1(\IR_+;\dot \B^{n/p-1}_{p,1}(\Omega))$ 
and for $g$ in $\dot\Y_p(\partial\Omega).$
Bounding $f$ mainly stems from the inequality (requiring~\eqref{eq:smallDv}):
\begin{equation}\label{eq:Av}
\|A_v-\Id\|_{\LL_\infty(\IR_+;\dot \B^{n/p}_{p,1}(\Omega))}
+\|\adj(DX_v)-\Id\|_{\LL_{\infty}(\IR_+;\dot \B^{n/p}_{p,1}(\Omega))}\leq 
2\|Dv\|_{\LL_1(\IR_+;\dot \B^{n/p}_{p,1}(\Omega))}.
\end{equation}
The latter inequality  relies on the fact that  $\dot \B^{n/p}_{p,1}(\Omega)$ is an algebra, and on Neumann series expansions like, e.g., 
\begin{equation}\label{eq:Av1} 
A_v(t)-\Id =\sum_{k\geq1} (-1)^k\Bigl(\int_0^t Dv\,\d \tau\Bigr)^k.\end{equation}
Now, we observe that $f$ may be decomposed as
$$f=\mu\divergence\Bigl(\!(\adj(DX_v)A_v^\top-\Id)\nabla v
+Dv(A_v(\adj(DX_v))^\top-\Id)\!\Bigr) +
\lambda\divergence\Bigl(\!(\adj(DX_v)A_v^\top-\Id):\nabla v\!\Bigr)\cdotp$$

Hence, using the fact that the operator $\divergence$ maps $\dot\B^{n/p}_{p,1}(\Omega)$ into  $\dot\B^{n/p-1}_{p,1}(\Omega)$ and the stability of  $\dot \B^{n/p}_{p,1}(\Omega)$ under 
products, we conclude that 
$$
\|f\|_{\LL_1(\IR_+;\dot \B^{n/p-1}_{p,1}(\Omega))}\leq C
 \|Dv\|_{\LL_1(\IR_+;\dot \B^{n/p}_{p,1}(\Omega))}^2.
$$
 Next,  in order to  bound $g$ in $\LL_1(\IR_+;\dot \B^{(n-1)/p}_{p,1}(\partial\Omega)),$ it suffices to consider its expression as the trace on $\partial\Omega$ of
 a function in  $\LL_1(\IR_+;\dot \B^{n/p}_{p,1}(\Omega)).$ 
 More precisely, on $\partial \Omega$, we have 
\begin{align*} g&= \bigl(\IS_v (v) - \IS (v)\bigr) \cdot \e_n + \bigl(\IS(v) - \IS_v (v)\bigr) \cdot (\bar n + \e_n) + \IS_v (v) \cdot (\bar n - \bar n_v) \\
 &= \Bigl(\mu\bigl((A_v^\top-\Id)\cdot\nabla v+Dv\cdot(A_v-\Id)\bigr) 
 +\lambda (A_v^\top-\Id):\nabla v \Bigr) \cdot \e_n \\
 &\qquad + \Bigl(\mu\bigl((\Id - A_v^\top)\cdot\nabla v+Dv\cdot(\Id - A_v)\bigr) 
 +\lambda (\Id - A_v^\top):\nabla v \Bigr) \cdot (\bar n + \e_n) + \IS_v(v) \cdot (\bar n - \bar n_v) \\
 &= g_1 + g_2 + g_3.
\end{align*}
Notice that all terms except,  $\wt n := \bar n + \e_n$ and $\wt n_v := \bar n_v - \bar n,$ admit a canonical extension to all of $\Omega$. Following the proof of Lemma~\ref{Lem: The normal vector} word by word, but by replacing the transformation $X_v$ by $H$, readily yields that for some small constant $c > 0,$ we have
 \begin{align}
 \label{Eq: Normal vector of Omega}
 \| E (\bar n + \e_n) \|_{\dot \B^{n/p}_{p , 1} (\Omega)} \leq c.
 \end{align}
 To extend $\wt n_v$ appropriately to $\Omega$, we need the following lemma.
 
 \begin{lemma}
 \label{Lem: Normal vector at Omega}
 The function $\wt n_v$ has an extension $E \wt n_v$ to all of $\IR_+ \times \Omega$ that satisfies for some constant $C > 0$
 \begin{equation}\label{eq:tildenv}
 \|E\wt n_v\|_{\LL_\infty(\IR_+;\dot\B^{n/p}_{p,1}(\Omega))}\leq C\|Dv\|_{\LL_1(\IR_+;\dot \B^{n/p}_{p,1}(\Omega))}.\end{equation}
 \end{lemma}
 
 \begin{proof}
By definition of  $H$ and of the Lagrangian coordinates transform, the set $\Omega_t$ is given as the image $X_v \circ H(\IR^n_+)$. In this case the tangent space at $\partial \Omega_t$ is spanned by the vectors $\partial_j (X_v \circ H)$ for $j = 1 , \dots , n - 1$. Thus, the outward unit normal vector to $\partial \Omega_t$ is given by
\begin{align*}
 \bar n_v &= - \frac{\partial_1 (X_v \circ H) \times \dots \times \partial_{n - 1} (X_v \circ H)}{\lvert \partial_1 (X_v \circ H) \times \dots \times \partial_{n - 1} (X_v \circ H)\rvert} \\
 &= - \frac{([D X_v - \Id] \partial_1 H + \partial_1 H) \times \dots \times ([D X_v - \Id] \partial_{n - 1} H + \partial_{n - 1} H)}{\lvert ([D X_v - \Id] \partial_1 H + \partial_1 H) \times \dots \times ([D X_v - \Id] \partial_{n - 1} H + \partial_{n - 1} H) \rvert} \\
 &=: - \frac{\partial_1 H \times \dots \times \partial_{n - 1} H + V}{\lvert \partial_1 H \times \dots \times \partial_{n - 1} H + V \rvert}\cdot
\end{align*}
First of all, we see that each of these terms is the restriction of functions that are defined on $\Omega$. This canonically defines an extension of $\bar n_v$ to all of $\Omega$. Second, observe that each summand that constitutes $V$ is of the form
\begin{align*}
 w_1 \times \dots \times w_{n - 1}
\end{align*}
and each $w_j$ is either given by
\begin{align*}
 w_j = \partial_j H = (\partial_j H - \e_j) + \e_j \quad \text{or} \quad w_j = [D X_v - \Id] \partial_j H = [D X_v - \Id] (\partial_j H - \e_j) + [D X_v - \Id] \e_j.
\end{align*}
Moreover, there has to be at least one $j = 1 , \dots , n - 1$ such that
\begin{align*}
 w_j = [D X_v - \Id] (\partial_j H - \e_j) + [D X_v - \Id] \e_j.
\end{align*}
Due to~\eqref{eq:H} and~\eqref{eq:DXv}, $V$ has an estimate of the form
\begin{align*}
 \| V \|_{\LL_{\infty} (\IR_+ ; \dot \B^{n / p}_{p , 1} (\Omega))} \leq C \| D v \|_{\LL_1 (\IR_+ ; \dot \B^{n / p}_{p , 1} (\Omega))}.
\end{align*}
Using Taylor expansion of the function $z \mapsto 1 / (\lvert \partial_1 H \times \dots \times \partial_{n - 1} H \rvert^2 + z)^{1/2}$ around $x = 0$ shows that
\begin{align*}
 E(\bar n_v - \bar n) &= \sum_{k = 0}^{\infty} a_k \frac{(2 (\partial_1 H \times \dots \times \partial_{n - 1} H) \cdot V + \lvert V \rvert^2)^k}{\lvert \partial_1 H \times \dots \times \partial_{n - 1} H \rvert^{k + 1}} V \\
 &= \sum_{k = 1}^{\infty} a_k \frac{(2 (\partial_1 H \times \dots \times \partial_{n - 1} H) \cdot V + \lvert V \rvert^2)^k}{\lvert \partial_1 H \times \dots \times \partial_{n - 1} H \rvert^{k + 1}} \partial_1 H \times \dots \times \partial_{n - 1} H
\end{align*}
for appropriate values of $a_k$. Replacing each $\partial_j H$ by $(\partial_j H - \e_j) + \e_j$, using Taylor expansion of the function $z \mapsto 1 / (1 + z)^{1 / 2}$ together with~\eqref{eq:H} and the fact that $\dot \B^{n / p}_{p , 1}$ is an algebra finally delivers that
\begin{align*}
 \| E (\bar n_v - \bar n) \|_{\LL_{\infty} (\IR_+ ; \dot \B^{n / p}_{p , 1} (\Omega))} \leq C \| D v \|_{\LL_1 (\IR_+ ; \dot \B^{n / p}_{p , 1} (\Omega))}. &\qedhere
\end{align*}
 \end{proof}
 
Making use of Lemma~\ref{Lem: Normal vector at Omega}, of~\eqref{eq:Av} and~\eqref{Eq: Normal vector of Omega}, and using also the stability of   $\dot \B^{n/p}_{p,1}(\Omega)$ under products and, we find that 
$$\|g\|_{\LL_1(\IR_+;\dot \B^{n/p}_{p,1}(\Omega))}
 \lesssim  \|Dv\|_{\LL_1(\IR_+;\dot \B^{n / p}_{p,1}(\Omega))}^2.$$ Hence, in light of the trace theorem, 
\begin{equation}
\|g\|_{\LL_1(\IR_+;\dot \B^{(n-1)/p}_{p,1}(\partial\Omega))}\lesssim 
  \|Dv\|_{\LL_1(\IR_+;\dot \B^{n/p}_{p,1}(\Omega))}^2.\end{equation}
  In order 
   to  bound $g$ in $\dot\B^{(n-1)/(2p)}_{1,1}(\IR_+;\dot\B^{0}_{p,1}(\partial\Omega)),$
   we need to use the fact that  the numerical product of functions is continuous
   from $$\dot \W^1_1(\IR_+;\dot\B^{(n-1)/p}_{p,1}(\partial\Omega))\times
   \dot\B^{(n-1)/(2p)}_{1,1}(\IR_+;\dot\B^{0}_{p,1}(\partial\Omega))\to\dot\B^{(n-1)/(2p)}_{1,1}(\IR_+;\dot\B^{0}_{p,1}(\partial\Omega)).$$
This fact is  deduced from the corresponding product law on $\IR\times\IR^{n - 1}$, which can be found in Proposition~\ref{prod-Bes2}, and by noting that  the classes of functions under consideration may be extended by an even reflection 
 on the whole real line for the time variable.
    
 Hence, we may write 
 \begin{equation}\label{eq:wtf}
 \|(A_v^\top-\Id)\cdot\nabla v\|_{\dot\B^{\frac{n-1}{2p}}_{1,1}(\IR_+;\dot\B^{0}_{p,1}(\partial\Omega))}
 \lesssim \|A_v^\top-\Id\|_{\dot \W^1_1(\IR_+;\dot\B^{\frac{n-1}p}_{p,1}(\partial\Omega))}
 \|\nabla v\|_{\dot\B^{\frac{n-1}{2p}}_{1,1}(\IR_+;\dot\B^{0}_{p,1}(\partial\Omega))}.
 \end{equation}
 Differentiating~\eqref{eq:Av1} yields 
 \begin{equation}\label{eq:dA}\partial_t(A_v(t)-\Id) =\sum_{k\geq1} (-1)^k
 \sum_{j=0}^{k-1}\Bigl(\int_0^t Dv\,\d \tau\Bigr)^{k-1-j}\cdot Dv(t)\cdot \Bigl(\int_0^t Dv\,\d \tau\Bigr)^j.
 \end{equation}
 Hence, thanks to~\eqref{eq:smallDv} and the trace theorem, we have
 $$
 \|\partial_t(A_v(t)-\Id)\|_{\LL_1(\IR_+;\dot\B^{\frac{n-1}p}_{p,1}(\partial\Omega))}
\lesssim  \|\partial_t(A_v(t)-\Id)\|_{\LL_1(\IR_+;\dot\B^{\frac np}_{p,1}(\Omega))}
\lesssim \|Dv\|_{\LL_1(\IR_+;\dot\B^{\frac np}_{p,1}(\Omega))}.
$$ 
 Using also~\eqref{eq:Av}, then reverting to~\eqref{eq:wtf}, we end up with 
 $$ \|(A_v^\top-\Id)\cdot\nabla v\|_{\dot\B^{\frac{n-1}{2p}}_{1,1}(\IR_+;\dot\B^{0}_{p,1}(\partial\Omega))}
 \lesssim \|\nabla v\|_{\LL_1(\IR_+;\dot\B^{\frac np}_{p,1}(\Omega))} \|\nabla v\|_{\dot\B^{\frac{n-1}{2p}}_{1,1}(\IR_+;\dot\B^{0}_{p,1}(\partial\Omega))}.$$
 This readily yields a bound on $g_1$ and similarly one derives a bound on $g_3$ so that one eventually gets
 $$
 \|g_1 , g_3\|_{\dot\B^{\frac{n-1}{2p}}_{1,1}(\IR_+;\dot\B^{0}_{p,1}(\partial\Omega))}
  \lesssim  \|\nabla v\|_{\LL_1(\IR_+;\dot\B^{\frac np}_{p,1}(\Omega))} \|\nabla v\|_{\dot\B^{\frac{n-1}{2p}}_{1,1}(\IR_+;\dot\B^{0}_{p,1}(\partial\Omega))}.$$
The term $g_2$ is estimated by virtue of~\cite[Lem.~A.5]{Danchin14} as
\begin{align*}
 &\|g_2\|_{\dot\B^{\frac{n-1}{2p}}_{1,1}(\IR_+;\dot\B^{0}_{p,1}(\partial\Omega))} \\
 &\lesssim \| \bar n + \e_n \|_{\dot \B^{\frac{n-1}p}_{p , 1} (\partial \Omega)} \| \mu\bigl((\Id - A_v^\top)\cdot\nabla v+Dv\cdot(\Id - A_v)\bigr) 
 +\lambda (\Id - A_v^\top):\nabla v \|_{\dot\B^{\frac{n-1}{2p}}_{1,1}(\IR_+;\dot\B^{0}_{p,1}(\partial\Omega))} \\
 &\leq Cc \|\nabla v\|_{\LL_1(\IR_+;\dot\B^{\frac np}_{p,1}(\Omega))} \|\nabla v\|_{\dot\B^{\frac{n-1}{2p}}_{1,1}(\IR_+;\dot\B^{0}_{p,1}(\partial\Omega))}.
\end{align*}
Summarizing, inserting all the above estimates into~\eqref{eq:uEp}, we get 
$$\|u\|_{\E_p}\leq C\bigl(\|u_0\|_{\dot \B^{n/p-1}_{p,1}(\Omega)}
+\|v\|_{\E_p}^2\bigr).$$
Therefore, taking $R=2Cc,$ then  assuming that  $1 + 2 C R < 2$ yields $u\in \bar B_{\E_p}(0,R)$. 
\medbreak
\subsubsection*{Step 2. Properties of contraction}
Consider $v_1$ and $v_2$ in $\E_p.$ Our aim is
to bound $\|u_2-u_1\|_{\E_p}$ with $u_i:=\Phi(v_i).$
Using the short notation $\divergence_i:=\divergence_{v_i}$ and so on, 
we see that the function 
$\du:=u_2-u_1$ satisfies
\begin{align*}
\left\{\begin{aligned} \rho_0\partial_t\du - \divergence\IS(\du) &= \df:= \divergence_{2}\IS_{2}(v_2)-
\divergence_{1}\IS_{1}(v_1)
 &&\mbox{ in}\quad \IR_+\times\Omega, \\
\IS(\du) \cdot \bar n|_{\partial\Omega}&= \dg:=\bigl((\IS(v_2)-\IS_{2}(v_1)) \cdot \bar n\\
&\hspace{2cm}+\IS_{v_1}(v_1)\cdot\bar n_{1}-\IS_{v_2}(v_2)\cdot\bar n_{2}\bigr)|_{\partial\Omega}
 && \mbox{ on}\quad \IR_+\times\partial\Omega,\\
\du|_{t = 0} &= 0  &&\hbox{ in}\quad \Omega.\end{aligned} \right.
\end{align*}
Thanks to Theorem~\ref{th:lameomega} and to the fact  
that $\rho_0$ fulfills~\eqref{eq:smalldata}, we see that 
\begin{equation}\label{eq:du}\|\du\|_{\E_p}\leq C\Bigl(\|\df\|_{\LL_1(\IR_+;\dot \B^{n/p-1}_{p,1}(\Omega))}
+ \|\dg\|_{\LL_1(\IR_+;\dot \B^{(n-1)/p}_{p,1}(\partial\Omega))\cap \dot \B^{(n-1)/(2p)}_{1,1}(\IR_+;\dot \B^{0}_{p,1}(\partial\Omega))}\Bigr)\cdotp\end{equation}
We compute
$$\displaylines{
\df=\mu\divergence\Bigl(\bigl(\adj(DX_{2})A_2^\top-\adj(DX_1)A_1^\top\bigr)\cdot\nabla v_1
+\bigl(\bigl(\adj(DX_{2})A_2^\top-\adj(DX_1)A_1^\top\bigr)\cdot\nabla v_1\bigr)^\top
\hfill\cr\hfill+\adj(DX_2)A_2^\top\cdot\nabla\dv+\bigl(\adj(DX_2)A_2^\top\cdot\nabla\dv\bigr)^\top\Bigr)
\hfill\cr\hfill
+\lambda\divergence\Bigl(\bigl(\adj(DX_2)A_2^\top-\adj(DX_1)A_1^\top\bigr):\nabla v_1
+\adj(DX_2):\nabla\dv\Bigr)\cdotp}
$$
Since
$$
A_2(t)-A_1(t) = \sum_{k=1}^\infty(-1)^k \sum_{j=0}^{k-1} \Bigl(\int_0^t Dv_2\,\d\tau\Bigr)^j
\Bigl(\int_0^t D\dv\,\d\tau\Bigr)  \Bigl(\int_0^t Dv_1\,\d\tau\Bigr)^{k-1-j},
$$
we get thanks to the stability of $\dot\B^{n/p}_{p,1}(\Omega)$ under products and 
to~\eqref{eq:smallDv},
$$\|A_2-A_1\|_{\LL_{\infty}(\IR_+;\dot \B^{n/p}_{p,1}(\Omega))}
\lesssim
 \|D\dv\|_{\LL_1(\IR_+;\dot\B^{n/p}_{p,1}(\Omega))}.$$
The term $\adj(DX_2)-\adj(DX_1)$ may be bounded similarly (see the appendix of~\cite{Danchin14} for details), and using 
once more the  the stability of $\dot\B^{n/p}_{p,1}(\Omega)$ under products
eventually yields
$$\|\df\|_{\LL_1(\IR_+;\dot \B^{n/p-1}_{p,1}(\Omega))}\lesssim  \bigl(\|Dv_1\|_{\LL_1(\IR_+;\dot\B^{n/p}_{p,1}(\Omega))}+
 \|Dv_2\|_{\LL_1(\IR_+;\dot\B^{n/p}_{p,1}(\Omega))}\bigr)
 \|D\dv\|_{\LL_1(\IR_+;\dot\B^{n/p}_{p,1}(\Omega))}.$$
 In order to bound  the boundary term $\dg,$ we use the decomposition
 \begin{align*}
  \dg &= (\IS_{v_1} (v_1) - \IS_{v_2} (v_1)) \cdot \wt n_2 + (\IS_{v_1} (\dv) - \IS (\dv)) \cdot \e_2 + (\IS (\dv) - \IS_{v_2} (\dv)) \cdot \wt n_2 \\
  &\qquad + (\IS_{v_2} (v_2) - \IS_{v_1} (v_2)) \cdot \e_2 + \IS_{v_1} (v_1) \cdot (\wt n_1 - \wt n_2) \\
  &= \Bigl(\mu\bigl((A_1^\top-A_2^{\top})\cdot\nabla v_1 + Dv_1 \cdot(A_1-A_2)\bigr)
 +\lambda (A_1^\top-A_2^{\top}):\nabla v_1 \Bigr) \cdot \wt n_2 \\
 &\qquad + \Bigl(\mu\bigl((A_1^\top-\Id)\cdot\nabla \dv+D\dv\cdot(A_1-\Id)\bigr)
 +\lambda (A_1^\top-\Id):\nabla \dv \Bigr) \cdot \e_2 \\
 &\qquad + \Bigl(\mu\bigl((\Id-A_2^\top)\cdot\nabla \dv+D\dv\cdot(\Id - A_2)\bigr)
 +\lambda (\Id - A_2^\top):\nabla \dv \Bigr) \cdot \wt n_2 \\
 &\qquad + \Bigl(\mu\bigl((A_2^\top-A_1^{\top})\cdot\nabla v_2 + Dv_2 \cdot(A_2-A_1)\bigr)
 +\lambda (A_2^\top-A_1^{\top}):\nabla v_2 \Bigr) \cdot \e_2 \\
 &\qquad +\IS_{v_1} (v_1) \cdot (\wt n_1 - \wt n_2).
 \end{align*}
 Using  the expression of $\partial_t(A_2-A_1)$  computed from~\eqref{eq:dA}
 (and a variation of it) as well  as the product laws and trace results that 
 have been used in Step 1, we end  up with
 $$\|\dg\|_{\LL_1(\IR_+;\dot \B^{(n-1)/p}_{p,1}(\partial\Omega))\cap
 \dot\B^{(n-1)/(2p)}_{1,1}(\IR_+;\dot\B^{0}_{p,1}(\partial\Omega))}
  \lesssim  (\|v_1\|_{\E_p}+\|v_2\|_{\E_p})
 \|\dv\|_{\E_p}.$$
 Plugging all this in~\eqref{eq:du}, we find  for some constant $C$
 (that may be chosen larger than that of the first step, with no loss of generality):
 $$\|\du\|_{\E_p}\leq C (\|v_1\|_{\E_p}+\|v_2\|_{\E_p})
 \|\dv\|_{\E_p}.$$
  Hence, if $R$ is such that $2CR<1,$ then the map $\Phi$ is contractive, 
 and we deduce the existence of a unique fixed point
 $v$ of $\Phi$ in the ball $\bar B_{\E_p}(0,R).$ 
 \medbreak
\subsubsection*{Step 3. Uniqueness}
Consider two solutions $u_1$ and $u_2$  in $\E_p$ to~\eqref{eq:presslag} 
corresponding to the same data $(\rho_0,v_0).$ 
Then, we have $u_1=\Phi(u_1)$ and $u_2=\Phi(u_2),$ and Step~2 may be carried out
on any interval $[0,T]$ such that 
$$\int_0^T\|\nabla u_1\|_{\dot\B^{n/p}_{p,1}(\Omega)}\,\d t
+\int_0^T\|\nabla u_2\|_{\dot\B^{n/p}_{p,1}(\Omega)}\,\d t\leq c\ll1.$$
Hence, for any $t\in[0,T],$ we have (with obvious notation):
 $$\|\du\|_{\E_p(t)}\leq C (\|u_1\|_{\E_p(t)}+\|u_2\|_{\E_p(t)})
 \|\du\|_{\E_p(t)} \leq  C (2\|u_1\|_{\E_p(t)}+\|\du\|_{\E_p(t)})
 \|\du\|_{\E_p(t)}.$$
 Since the function $t\mapsto \|\du\|_{\E_p(t)}$ is continuous and
 vanishes at $0$ and because one can assume with no loss of generality
 that $u_1$ is the small solution constructed in steps 1-2, 
 we get uniqueness on some nontrivial time interval $[0,T_0].$
 Then, repeating the argument, one gets  uniqueness for all time. 
 The details are left to the reader.

\medbreak
\subsubsection*{Step 4. Back to Eulerian coordinates}

After having constructed a  global and small solution $u\in\E_p$  to~\eqref{eq:presslag} (with corresponding density $q=\rho_0 J_u^{-1}$), 
we now define the `Eulerian'  density and velocity by 
$$\rho(t,x):= q(t,X^{-1}(t,x))\andf v(t,x):= u(t,X^{-1}(t,x))$$
with $X^{-1}(t,\cdot)$ being the inverse diffeomorphism of 
$X(t,\cdot)$ defined by~\eqref{eq:X2}. 

Our construction ensures that $X(t,\cdot)$ is defined  and $\cC^1$ for all $t\geq0.$ 
The reason why $X(t,\cdot)$  is invertible is guaranteed by the fact that 
we also know that $\|\nabla X(t)-\Id\|_{\LL^\infty}$ is small and that 
$X(t,y)$ tends to infinity as $y$ goes to infinity.  Indeed, observe that
$$|X(t,y)-y |\leq t^{1/2} \|u(t,\cdot)\|_{\LL_2(0,t;\LL_\infty(\Omega))}\leq 
Ct^{1/2} \|u\|_{\LL_2(\IR_+;\dot\B^{n/p}_{p,1}(\Omega))}.$$ 
Hence, the relations between operators $\nabla_v$ and $\nabla_y$ are 
justified, and $(\rho,v)$ is thus a solution of~\eqref{eq:pressless} with 
$\Omega_t:= X(t,\Omega)$ and $\partial\Omega_t= X(t,\partial\Omega).$ 
The product law proven in~\cite[Lem.~A.5]{Danchin14} together with Proposition~\ref{Prop: Change of half} and Remark~\ref{Rem: Change of half} allow to transfer the estimates so that $v$ satisfies~\eqref{Eq: Solution estimate pressureless}. See also the explanations in the appendices of~\cite{DM-CPAM} and~\cite{Danchin14}. 
\qed

\chapter{Appendix}

For the reader's convenience, we prove here some continuity results for  product laws  in Besov spaces,  
an interpolation result, 
a solution formula for the Lam\'e system and a higher order resolvent  estimate for the Neumann Laplacian.  

\section{Product estimates}

We here prove two product estimates that have been used 
in our study of free boundary problems.
\begin{proposition}\label{prod-Bes} Let $(s_1,s_2)\in\IR^2$ and $(p,q,r)\in[1,\infty]^3$ satisfy
$$s_2\leq s_1\leq \min\biggl(\frac np,\frac nq\biggr),\quad
s_1+s_2> 0,\quad\frac1p+\frac1q\leq1\andf
\frac1r=\frac1p+\frac1q-\frac{s_1}n\cdotp$$
Then
\begin{equation}\label{eq:prod1} \dot \B^{s_1}_{p , 1}(\IR^n_+) \cdot \dot \B^{s_2}_{q , 1}(\IR^n_+) \subset \dot \B^{s_2}_{r , 1}(\IR^n_+).\end{equation}
If in the limit case $s_1+s_2=0$, 
then
\begin{equation}\label{eq:prod2} \dot \B^{s_1}_{p , 1}(\IR^n_+) \cdot \dot \B^{s_2}_{q , 1}(\IR^n_+) \subset \dot \B^{s_2}_{r , \infty}(\IR^n_+).\end{equation}
\end{proposition}
\begin{proof} Owing to our definition of Besov spaces by restriction, it suffices to prove the statement in $\IR^n.$ 
In this setting, the result is a consequence 
of Bony's decomposition and of the continuity results for the remainder and paraproduct operators. 
More precisely, denote by $R$ and $T,$ respectively, these operators, and  consider
 $u\in \dot \B^{s_1}_{p,1}(\IR^n)$ and $v\in\dot \B^{s_2}_{q,1}(\IR^n).$  Then, Bony's decomposition (first introduced in~\cite{bony}) 
 of the product $uv$ reads  
 \begin{equation}\label{eq:Bony}uv=T_uv+T_vu+R(u,v).\end{equation}
In order to bound the first term, one can use the 
 embedding  $\dot \B^{s_1}_{p,1}(\IR^n)\hookrightarrow \dot \B^{0}_{p^*,1}(\IR^n)$
with $1/{p^*}=1/p-{s_1}/n,$ and the fact that
 $T:\dot \B^{0}_{p^*,1}(\IR^n)\times \dot \B^{s_2}_{q,1}(\IR^n)\to \dot \B^{s_2}_{r,1}(\IR^n)$. 
 Hence, $T_uv$  is in  $\dot \B^{s_2}_{r,1}(\IR^n)$.

For proving that $T_vu$  is in  $\dot \B^{s_2}_{r,1}(\IR^n),$ we combine the embedding
 $ \dot \B^{s_2}_{q,1}(\IR^n)\hookrightarrow \dot \B^{s_2-s_1}_{q^*,1}(\IR^n)$ with 
 $1/q^*=1/q-s_1/n$ and the fact that, since $s_2-s_1\leq0$ and $1/q^*+1/p=1/r,$ 
  we have  $T:  \dot \B^{s_2-s_1}_{q^*,1}(\IR^n)\times  \dot \B^{s_1}_{p,1}(\IR^n)\to  \dot \B^{s_2}_{r,1}(\IR^n).$
 
 Let us finally prove that $R:  \dot \B^{s_1}_{p , 1}(\IR^n)\times \dot \B^{s_2}_{q , 1}(\IR^n) \to  \dot \B^{s_2}_{r , 1}(\IR^n).$
  We start with the observation that owing to the spectral 
 localization properties of $\ddj,$ we have for all $k\in\IZ,$
 $$ \ddk R(u,v)= \sum_{j\geq k-3} \ddk(\ddj u\,\wt\Delta_jv)\with
 \wt\Delta_j :=\dot\Delta_{j-1}+\ddj + \dot\Delta_{j+1}.$$
 Let us define $m$ by $1/m=1/p+1/q.$ 
 According  to Bernstein's inequality 
 we see  that
 $$ \| \ddk(\ddj u\,\wt\Delta_jv)\|_{\LL_r(\IR^n)}\lesssim 2^{nk(\frac1m-\frac1r)} \| \ddk(\ddj u\,\wt\Delta_j v)\|_{\LL_m(\IR^n)}.$$
 Therefore, using H\"older inequality and 
 the definition of $r$ and of $m,$ we get
 $$\sum_{j \geq k-3} 2^{ks_2}\|\ddk(\ddj u\,\wt\Delta_j  v)\|_{\LL_r(\IR^n)}\lesssim \sum_{j\geq k-3}
 2^{(k-j)(s_1+s_2)}\, 2^{js_1}\|\ddj u\|_{\LL_p(\IR^n)}\, 
 2^{js_2}\|\wt\Delta_ju\|_{\LL_q(\IR^n)}.$$ 
 Now, if  $s_1+s_2>0,$ then a standard convolution inequality for series allows to complete 
 the proof.  The borderline case $s_1+s_2=0$ 
 just follows from H\"older inequality for series, 
 but only allows us to control the $\sup$ norm of the left-hand side of the above inequality. 
 Consequently, one ends up with  a Besov space with last index $\infty$ in that case. 
 \end{proof} 

\begin{proposition}\label{prod-Bes2}
Let $m\in\IN,$ $1< p<\infty$ and $0<s<1.$ 
Then, we have the following continuous product law:
$$\dot\W^1_1(\IR;\dot\B^{\frac mp}_{p,1}(\IR^m))\cdotp
\dot\B^s_{1,1}(\IR;\dot\B^0_{p,1}(\IR^m))\subset
\dot\B^s_{1,1}(\IR;\dot\B^0_{p,1}(\IR^m)).$$
\end{proposition}
\begin{proof} It is only a matter of using Bony's decomposition
with respect to the time variable and
the fact that, as a consequence of~\cite[Lem.~A.5]{Danchin14}, we have:
\begin{equation}\label{eq:prod3}
 \dot \B^{\frac mp}_{p , 1}(\IR^m) \cdot \dot \B^{0}_{p , 1}(\IR^m) \subset \dot \B^{0}_{p , 1}(\IR^m).
\end{equation}
Now, we are going to bound the three terms
of~\eqref{eq:Bony} separately where, here, $R$ and $T$
designate the remainder and paraproduct operators 
\emph{with respect to the time variable only}. 
In what follows, we shall use the notation 
$\ddk^t$ and $\dot S_k^t$ of the previous chapter 
to denote the Littlewood-Paley truncation operators
pertaining to the time variable. 

Let us first bound $T_uv.$ We have, by definition of 
$\dot\B^s_{1,1}(\IR;\dot\B^0_{p,1}(\IR^m)),$
$$
\|T_uv\|_{\dot\B^s_{1,1}(\IR;\dot\B^0_{p,1}(\IR^m))}
=\sum_j2^{js} \|\ddj^t(T_uv)\|_{\LL_1(\IR;\dot\B^0_{p,1}(\IR^m))}
\simeq \sum_j 2^{js}
\|\dot S_{j-1}^t u\: \ddj^t v\|_{\LL_1(\IR;\dot\B^0_{p,1}(\IR^m))}.
$$
Hence, using~\eqref{eq:prod3} and H\"older's inequality, 
$$\begin{aligned}
\|T_uv\|_{\dot\B^s_{1,1}(\IR;\dot\B^0_{p,1}(\IR^m))}&\lesssim
\sum_j 2^{js}
\|\dot S_{j-1}^t u\|_{\LL_\infty(\IR;\dot\B^{\frac mp}_{p,1}(\IR^m))}
\|\ddj^t v\|_{\LL_1(\IR;\dot \B^0_{p,1}(\IR^m))}\\
&\lesssim
\sum_j
\|u\|_{\LL_\infty(\IR;\dot\B^{\frac mp}_{p,1}(\IR^m))}\:
2^{js}\|\ddj^t v\|_{\LL_1(\IR;\dot\B^0_{p,1}(\IR^m))}\\
&=\|u\|_{\LL_\infty(\IR;\dot\B^{\frac mp}_{p,1}(\IR^m))}
\|v\|_{\B_{1,1}^s(\IR;\dot\B^0_{p,1}(\IR^m))}.\end{aligned}
$$
For the symmetric term, $T_vu,$ we proceed as follows
(using Bernstein's inequality in the second line): 
$$\begin{aligned}
\|T_vu\|_{\dot\B^s_{1,1}(\IR;\dot\B^0_{p,1}(\IR^m))}
&\simeq \sum_j 2^{js}
\|\dot S_{j-1}^t v\: \ddj^t u\|_{\LL_1(\IR;\dot\B^0_{p,1}(\IR^m))}\\
&\lesssim\sum_j \bigl(2^{j(s-1)}\|\dot S_{j-1}^tv\|_{\LL_\infty(\IR;\dot\B^{0}_{p,1}(\IR^m))}\bigr)
\bigl(\|\ddj^t \partial_tu\|_{\LL_1(\IR;\dot\B^{\frac mp}_{p,1}(\IR^m))}\bigr)\\
&\lesssim \biggl(\sum_j2^{j(s-1)}\|\dot S_{j-1}^tv\|_{\LL_\infty(\IR;\dot\B^{0}_{p,1}(\IR^m))}\biggr)
\biggl(\sup_j \|\ddj^t\partial_t u \|_{\LL_1(\IR;\dot\B^{\frac mp}_{p,1}(\IR^m))}\biggr)\cdotp
\end{aligned}
$$
Note that since $s-1<0,$ Young's  inequality for series,
the fact that $S_{j-1}^tv=\sum_{j'\leq j-2}\dot\Delta^t_{j'}v$
and Bernstein's inequality with respect to $t$
allow to get
$$\begin{aligned}
\sum_j2^{j(s-1)}\|\dot S_{j-1}^tv\|_{\LL_\infty(\IR;\dot\B^{0}_{p,1}(\IR^m))}
&\leq \sum_j\sum_{j'\leq j-2} 
2^{(j-j')(s-1)} 2^{j'(s-1)}
\|\dot\Delta_{j'}^tv\|_{\LL_\infty(\IR;\dot\B^{0}_{p,1}(\IR^m))}\\
&\lesssim \sum_{j''} 
2^{j''(s-1)}
\|\dot\Delta_{j''}^tv\|_{\LL_\infty(\IR;\dot\B^{0}_{p,1}(\IR^m))}\\
&\lesssim \sum_{j''} 
2^{j''s}
\|\dot\Delta_{j''}^tv\|_{\LL_1(\IR;\dot\B^{0}_{p,1}(\IR^m))}.
\end{aligned}$$
Hence, 
$$\|T_vu\|_{\dot\B^s_{1,1}(\IR;\dot\B^0_{p,1}(\IR^m))}\lesssim
\|v\|_{\dot \B_{1,1}^s(\IR;\dot\B^0_{p,1}(\IR^m))}
\|\partial_t u \|_{\LL_1(\IR;\dot\B^{\frac mp}_{p,1}(\IR^m))}.
$$
Finally, for the remainder, we may write 
$$\|R(u,v)\|_{\dot\B^s_{1,1}(\IR;\dot\B^0_{p,1}(\IR^m))}
=\sum_j2^{js}\biggl\|
\sum_{j'\geq j-3} \ddj^t(\dot\Delta_{j'}^tu\:\wt\Delta_{j'}^tv)
\biggr\|_{\LL_1(\IR;\dot\B^0_{p,1}(\IR^m))}. 
$$
Hence, using~\eqref{eq:prod3} and H\"older's inequality, 
$$\|R(u,v)\|_{\dot\B^s_{1,1}(\IR;\dot\B^0_{p,1}(\IR^m))}
\lesssim \sum_j\sum_{j'\geq j-3} 2^{(j-j')s}
\|\dot\Delta_{j'}^tu\|_{\LL_\infty(\IR;\dot\B^{\frac mp}_{p,1}(\IR^m))}
\: 2^{j's}\|\wt\Delta_{j'}^tv\|_{\LL_1(\IR;\dot\B^0_{p,1}(\IR^m))}.$$
Since $s>0,$ one can use Young's inequality for series to conclude 
that 
$$\|R(u,v)\|_{\dot\B^s_{1,1}(\IR;\dot\B^0_{p,1}(\IR^m))}
\lesssim\|u\|_{\LL_\infty(\IR;\dot\B^{\frac mp}_{p,1}(\IR^m))}
\|v\|_{\dot \B_{1,1}^s(\IR;\dot\B^0_{p,1}(\IR^m))}.$$
Then, observing that, in light of the fundamental theorem of calculus, we have
$$\|u\|_{\LL_\infty(\IR;\dot\B^{\frac mp}_{p,1}(\IR^m))}\leq \|\partial_t u \|_{\LL_1(\IR;\dot\B^{\frac mp}_{p,1}(\IR^m))},$$
one can complete  the proof. 
\end{proof}


\section{An interpolation result}

The following statement 
combined with a more classical trace result
has been used in Chapter~\ref{Sec: A free boundary problem for pressureless gases},
to ensure that the boundary values of the velocity 
have the required regularity.
\begin{proposition}\label{p:interpopopo}
Let $\alpha\in(0,1).$ Then, there exists a constant
$C_\alpha$ such that the following a priori
estimate holds true for all $1\leq p\leq\infty$ 
and $s\in\IR$:
$$\|z\|_{\dot\B^\alpha_{1,1}(\IR;\dot\B^{s+2-2\alpha}_{p,1}(\IR^n))}
\leq C_\alpha \|\partial_tz\|_{\LL_1(\IR;\dot \B^s_{p,1}(\IR^n))}^\alpha \|\nabla_xz\|_{\LL_1(\IR;\dot \B^{s+1}_{p,1}(\IR^n))}^{1-\alpha}.$$ 
\end{proposition}
\begin{proof}
We have to bound the following quantity:
$$I:= \|z\|_{\dot\B^\alpha_{1,1}(\IR;\dot\B^{s+2-2\alpha}_{p,1}(\IR^n))}
=\sum_{k\in\IZ} 2^{k(s+2-2\alpha)}I_k\with
I_k:=\sum_{\ell\in\IZ}2^{\alpha\ell}\|\dot\Delta_\ell^t\dot\Delta_k^xz\|_{\LL_1(\IR;\LL_p(\IR^n))}.$$
Let us fix  $k\in\IZ.$ In order 
to bound $I_k,$ we use the
low frequency/high frequency decomposition
$I_k=I_{k,N} + I_k^N$ with 
$$I_{k,N}:= \sum_{\ell\leq N}2^{\alpha\ell}\|\dot\Delta_\ell^t\dot\Delta_k^xz\|_{\LL_1(\IR;\LL_p(\IR^n))}\andf 
 I_k^N:=\sum_{\ell>N} 2^{\alpha\ell}
\|\dot\Delta_\ell^t\dot\Delta_k^xz\|_{\LL_1(\IR;\LL_p(\IR^n))}.$$
Since $\dot\Delta_\ell^t$ maps $\LL_1(\IR;\LL_p(\IR^n))$
to itself with a norm independent of $j$ (and of $p$),
we have
$$I_{k,N}\leq \biggl(\sum_{\ell\leq N} 2^{\alpha\ell}\biggr)
\sup_{\ell\leq N} \|\dot\Delta_\ell^t\dot\Delta_k^xz\|_{\LL_1(\IR;\LL_p(\IR^n))}\leq C_\alpha 2^{\alpha N} 
\|\dot\Delta_k^xz\|_{\LL_1(\IR;\LL_p(\IR^n))}.
$$
Likewise, using the (reverse) Bernstein inequality
with respect to the time variable for getting the last inequality below, we discover that
$$I_k^N\leq \biggl(\sum_{\ell > N} 2^{(\alpha-1)\ell}\biggr)
\sup_{\ell>N} 2^\ell\|\dot\Delta_\ell^t\dot\Delta_k^xz\|_{\LL_1(\IR;\LL_p(\IR^n))}\leq C_\alpha 2^{(\alpha-1) N} 
\|\dot\Delta_k^x\partial_tz\|_{\LL_1(\IR;\LL_p(\IR^n))}.
$$
Hence, choosing the `best' $N\in\IZ,$
that is the integer part of 
$$\log_2\bigl(\|\dot\Delta_k^x\partial_tz\|_{\LL_1(\IR;\LL_p(\IR^n))}/\|\dot\Delta_k^xz\|_{\LL_1(\IR;\LL_p(\IR^n))}\bigr),$$ we conclude that 
$$I_k\leq C_\alpha \|\dot\Delta_k^x\partial_tz\|_{\LL_1(\IR;\LL_p(\IR^n))}^\alpha
\|\dot\Delta_k^x z\|_{\LL_1(\IR;\LL_p(\IR^n))}^{1-\alpha}.$$
Now, plugging this inequality in the definition of $I$ yields
$$
I\leq C_\alpha\sum_{k\in\IZ}
\bigl(2^{ks}\|\dot\Delta_k^x\partial_tz\|_{\LL_1(\IR;\LL_p(\IR^n))}\bigr)^\alpha
\bigl(2^{k(s+2)}\|\dot\Delta_k^x z\|_{\LL_1(\IR;\LL_p(\IR^n))}\bigr)^{1-\alpha}.
$$
At this stage, we apply the reverse Bernstein inequality
with respect to the space variable, followed by H\"older's for series 
inequality so as to get:
$$\begin{aligned}I&\leq C_\alpha 
\sum_{k\in\IZ}
\bigl(2^{ks}\|\dot\Delta_k^x\partial_tz\|_{\LL_1(\IR;\LL_p(\IR^n))}\bigr)^\alpha
\bigl(2^{k(s+1)}\|\dot\Delta_k^x\nabla_x z\|_{\LL_1(\IR;\LL_p(\IR^n))}\bigr)^{1-\alpha}\\
&\leq C_\alpha\biggl(\int_\IR 
\sum_{k\in\IZ}
2^{ks}\|\dot\Delta_k^x\partial_tz\|_{\LL_p(\IR^n)} \, \d t
\biggr)^\alpha
\biggl(\int_\IR\sum_{k\in\IZ} 2^{k(s+1)}\|\dot\Delta_k^x\nabla_x z\|_{\LL_1(\IR;\LL_p(\IR^n))} \, \d t\biggr)^{1-\alpha}.
\end{aligned}$$
This is exactly what we wanted to prove.
\end{proof}

\section{Regularity properties of diffeomorphisms}

In this section we discuss regularity properties of changes of variables in homogeneous Besov spaces. The following proposition builds the foundation of this discussion and is proven in~\cite[Prop.~A.1]{Danchin14}.

\begin{proposition}
\label{Prop: Change of variables whole space}
Let $X$ be a globally bi-Lipschitz diffeomorphism of $\IR^n$, $1 \leq p < \infty$ and $-\min(n/p,n/p')< s \leq n / p$. If 
\begin{align*}
D X - \Id \in \dot \B^{n / p}_{p , 1} (\IR^n ; \IR^{n \times n}) \quad \text{and} \quad D X^{-1} - \Id \in \dot \B^{n / p}_{p , 1} (\IR^n ; \IR^{n \times n}),
\end{align*}
then $u \mapsto u \circ X$ is a self-map over $\dot \B^s_{p , 1} (\IR^n).$ Furthermore, there exists $C > 0$ such that for all $u \in \dot \B^s_{p , 1} (\IR^n)$ it holds
\begin{align*}
 C^{-1} \| u \|_{\dot \B^s_{p , 1} (\IR^n)} \leq \| u \circ X \|_{\dot \B^s_{p , 1} (\IR^n)} \leq C \| u \|_{\dot \B^s_{p , 1} (\IR^n)}.
\end{align*}
\end{proposition}

\begin{remark}
An analysis of the proof of the proposition in~\cite[Prop.~A.1]{Danchin14} reveals that $C$ can be chosen to be uniform in the class of all bi-Lipschitz diffeomorphisms that satisfy for some $M > 0$
\begin{align*}
 \| D X - \Id \|_{\dot \B^{n / p}_{p , 1} (\IR^n ; \IR^{n \times n})} \leq M \quad \text{and} \quad \| D X^{-1} - \Id \|_{\dot \B^{n / p}_{p , 1} (\IR^n ; \IR^{n \times n})} \leq M.
\end{align*}
Note that the  continuous embedding from $\dot \B^{n/p}_{p , 1} (\IR^n ; \IR^{n \times n})$ into $\LL_{\infty} (\IR^n ; \IR^{n \times n})$ implies that for some constant $C > 0$ depending only on $n$ and $p$ one has that
\begin{align*}
 \| D X \|_{\LL_{\infty} (\IR^n ; \IR^{n \times n})} \leq 1 + C M \quad \text{and} \quad \| D X^{-1} \|_{\LL_{\infty} (\IR^n ; \IR^{n \times n})} \leq 1 + C M.
\end{align*}
In particular, this implies that the bi-Lipschitz constant of $X$ is bounded by $1 + C M$.
\end{remark}

The goal of this section is to transfer this result to an appropriate class of diffeomorphisms that map from the half-space $\IR^n_+$ onto a perturbation $\Omega$ of the half-space. As a preparation, we study diffeomorphisms of a certain structure, that allows to extend them to diffeomorphisms of $\IR^n$.

\begin{lemma} 
\label{Lem: Extension of diffeomorphism}
Let $1 < p < \infty$. There exists a constant $c > 0$ depending only on $n$ and $p$ such that any bi-Lipschitz diffeomorphism $X : \IR^n_+ \to \Omega$ onto some domain $\Omega \subset \IR^n$ of the form
\begin{align*}
 X(x) = x + g(x)
\end{align*}
satisfying
\begin{align*}
 \| D g \|_{\dot \B^{n / p}_{p , 1} (\IR^n_+ ; \IR^{n \times n})} \leq c
\end{align*}
can be extended to a bi-Lipschitz diffeomorphism $\wt X : \IR^n \to \IR^n$ such that
\begin{align*}
 \| D \wt X - \Id \|_{\dot \B^{n / p}_{p , 1} (\IR^n ; \IR^{n \times n})} \leq C \| D g \|_{\dot \B^{n / p}_{p , 1} (\IR^n_+ ; \IR^{n \times n})}
\end{align*}
and
\begin{align*}
 \| D \wt X^{-1} - \Id \|_{\dot \B^{n / p}_{p , 1} (\IR^n ; \IR^{n \times n})} \leq C \| D g \|_{\dot \B^{n / p}_{p , 1} (\IR^n_+ ; \IR^{n \times n})}.
\end{align*}
\end{lemma}

\begin{proof}
Let $E$ denote an extension operator with parameter $m = n + 1$ that was constructed in Lemma~\ref{Lem: Extension operators}. Notice that $E$ is constructed in such a way, that it maps $\cC^m$-functions to $\cC^m$-functions. \par
Let $\ell = 1 , \dots , n$ and let $g_{\ell}$ denote the $\ell$-th component of $g$. Define 
\begin{align*}
 \wt X_{\ell} (x) := x_{\ell} + E g_{\ell} (x) = \begin{cases} x + g_{\ell} (x), & x \in \IR^n_+ \\
 x + \sum_{j = 0}^m \alpha_j g_{\ell} \big( x^{\prime} , - \frac{x_n}{j + 1} \big), & x \in \IR^n \setminus \IR^n_+.\end{cases}
\end{align*}
Recall that the numbers $\alpha_j$ are chosen such that for each $\iota = 0 , \dots , m$ it holds
\begin{align*}
 \sum_{j = 0}^m \Big( - \frac{1}{j + 1} \Big)^{\iota} \alpha_j = 1.
\end{align*}
{}From its definition, it is clear  that $E$ commutes with all the tangential derivatives $\partial_k,$   $k = 1 , \dots , n - 1$. Moreover,
\begin{align*}
 \partial_n E g_{\ell} (x) = \partial_n g_{\ell} (x) \chi_{\IR^n_+} - \sum_{j = 0}^m \frac{\alpha_j}{j + 1} (\partial_n g_{\ell}) \Big( x^{\prime} , - \frac{x_n}{j + 1} \Big) \chi_{\IR^n \setminus \IR^n_+} =: E^{\prime} \partial_n g_{\ell} (x).
\end{align*}
By the condition imposed on the numbers $\alpha_j$, the operator $E^{\prime}$ still maps $\cC^{m - 1}$-functions on $\IR^n_+$ to $\cC^{m - 1}$-functions on $\IR^n$. Thus, one may argue as in the proof of Proposition~\ref{Prop: Proper boundedness extension operators} to conclude that $E^{\prime}$ acts as a bounded operator from $\dot \B^{n / p}_{p , 1} (\IR^n_+)$ to $\dot \B^{n / p}_{p , 1} (\IR^n)$. Thus, there exists a constant $C_1 > 0$ depending only on the boundedness constants of $E$ as well as $E^{\prime}$ such that
\begin{align*}
 \| D E g_{\ell} \|_{\dot \B^{n / p}_{p , 1} (\IR^n ; \IR^n)} \leq C_1 \| D g_{\ell} \|_{\dot \B^{n / p}_{p , 1} (\IR^n_+ ; \IR^n)}.
\end{align*}
With $C_2 > 0$ denoting the embedding constant of $\dot \B^{n / p}_{p , 1} (\IR^n ; \IR^n)$ into $\LL_{\infty} (\IR^n ; \IR^n)$ we conclude that if $c \leq C_1 C_2 / (2 n)$ then the function $E g_{\ell}$ satisfies
\begin{align*}
 \| D E g_{\ell} \|_{\LL_{\infty} (\IR^n ; \IR^n)} \leq \frac{1}{2n}\cdotp
\end{align*}
So, altogether one finds that  $E g$ defines a strict contraction on $\IR^n$ (with Lipschitz constant less than $1/2$), which already ensures that $\wt X$ is injective. Indeed, $\wt X (x) = \wt X (y)$ implies that
\begin{align*}
 x - y = E g (y) - E g (x) \quad \text{so that} \quad \lvert x - y \rvert = \lvert E g (y) - E g (x) \rvert \leq \frac{\lvert x - y \rvert}{2}\cdotp
\end{align*}
Moreover, the strict contractivity implies that $\wt X$ maps $\IR^n$ onto $\IR^n$. To see the ontoness, let $y \in \IR^n$ and notice that $\wt X$ maps to $y$ if and only if the function
\begin{align*}
 h : \IR^n \to \IR^n, \quad x \mapsto y - E g(x)
\end{align*}
has a fixed point. Since $E g$ is a strict contraction this fixed point exists by Banach's fixed point theorem. Finally, the inverse function theorem implies that $\wt X$ is a $\cC^1$-diffeomorphism and a Neumann series argument shows that
\begin{align*}
 D \wt X^{-1} - \Id = \sum_{k = 1}^{\infty} (-1)^k (D E g)^k
\end{align*}
and thus that
\begin{align*}
 \| D \wt X^{-1} - \Id \|_{\dot \B^{n / p}_{p , 1} (\IR^n ; \IR^{n \times n})} \leq C \| D g \|_{\dot \B^{n / p}_{p , 1} (\IR^n_+ ; \IR^{n \times n})},
\end{align*}
since $\dot \B^{n / p}_{p , 1} (\IR^n ; \IR^{n \times n})$ is an algebra. In particular, $\wt X$ is a bi-Lipschitz diffeomorphism.
\end{proof}

The following proposition is the counterpart of Proposition~\ref{Prop: Change of variables whole space} for bi-Lipschitz diffeomorphisms from $\IR^n_+$ onto some open set $\Omega$ of an appropriate structure.

\begin{proposition}\label{Prop: Change of half}
Let $1< p < \infty$,  $-\min(n/p,n/p')< s \leq n / p,$ and let $c > 0$ denote the constant from Lemma~\ref{Lem: Extension of diffeomorphism}. If $X : \IR^n_+ \to \Omega$ denotes a bi-Lipschitz diffeomorphism onto some domain $\Omega \subset \IR^n$ of the form
\begin{align*}
 X(x) = x + g(x)
\end{align*}
satisfying
\begin{align}
\label{Eq: Uniform bi-Lipschitz}
 \| D g \|_{\dot \B^{n / p}_{p , 1} (\IR^n_+ ; \IR^{n \times n})} \leq c,
\end{align}
then there exists a constant $C > 0$, depending only on $n$, $p$, $s$ and $c$ such that for all $v \in \dot \B^s_{p , 1} (\Omega)$ one has
\begin{align*}
 C^{-1} \| v \|_{\dot \B^s_{p , 1} (\Omega)} \leq \| v \circ X \|_{\dot \B^s_{p , 1} (\IR^n_+)} \leq C \| v \|_{\dot \B^s_{p , 1} (\Omega)}.
\end{align*}
\end{proposition}

\begin{proof}
Let $v \in \dot \B^s_{p , 1} (\Omega)$ and let $u := v \circ X$. Let $\wt X$ denote the extended bi-Lipschitz diffeomorphism of $X$ on $\IR^n$ that is constructed in Lemma~\ref{Lem: Extension of diffeomorphism} and let $V \in \dot \B^s_{p , 1} (\IR^n)$ denote an arbitrary extension of $v$ to $\IR^n$. Then $U := V \circ \wt X$ is an extension of $u$ to $\IR^n$. Thus,
\begin{align*}
 \| u \|_{\dot \B^s_{p , 1} (\IR^n_+)} \leq \| U \|_{\dot \B^s_{p , 1} (\IR^n)} \leq C \| V \|_{\dot \B^s_{p , 1} (\IR^n)},
\end{align*}
where in the second inequality we used Proposition~\ref{Prop: Change of variables whole space}. As the extension $V$ of $v$ was arbitrary, we can take the infimum and obtain
\begin{align*}
 \| u \|_{\dot \B^s_{p , 1} (\IR^n_+)} \leq C \| v \|_{\dot \B^s_{p , 1} (\Omega)}.
\end{align*}
The same argument but by taking $v := u \circ X^{-1}$ for some $u \in \dot \B^s_{p , 1} (\IR^n_+)$ yields that
\begin{align*}
 \| v \|_{\dot \B^s_{p , 1} (\Omega)} \leq C \| u \|_{\dot \B^s_{p , 1} (\IR^n_+)}. &\qedhere
\end{align*}
\end{proof}

\begin{remark}
\label{Rem: Change of half}
If $X : \IR^n_+ \to \Omega_t$, $t \geq 0$, is a family of bi-Lipschitz diffeomorphisms with $X(x) = x + g_t$ and $g_t$ satisfying~\eqref{Eq: Uniform bi-Lipschitz} uniformly with respect to $t$, then
\begin{align*}
 C^{-1} \| v \|_{\dot \B^s_{p , 1} (\Omega_t)} \leq \| v \circ X_t \|_{\dot \B^s_{p , 1} (\IR^n_+)} \leq C \| v \|_{\dot \B^s_{p , 1} (\Omega_t)}
\end{align*}
holds with a constant $C > 0$ \textit{independent of $t$}.
\end{remark}

\section{Explicit solution formula for the Lam\'e system in the half-plane}

We start from~\eqref{eq:lame} and 
take the Fourier transform with respect to the tangential and time directions: 
we name 
$$
 v:=\cF_{t,x_1}u.$$
Then, denoting $\nu=\lambda+2\mu,$  the first line of~\eqref{eq:lame} with homogeneous right-hand side
in the case $n=2$ reads
\begin{equation}\label{eq:blame}
 \begin{array}{l}
  \ii\tau v^1+\nu\xi_1^2 v^1 - \mu \partial_2^2 v^1 -(\mu+\lambda) \ii \xi_1 \partial_2 v^2=0,\\[5pt]
  \ii\tau v^2 + \mu \xi_1^2 v^2 - \nu \partial_2^2 v^2 -(\mu+\lambda) \ii \xi_1 \partial_2 v^1=0.
 \end{array}
\end{equation}
We  look at the system in the following matrix form:
\begin{align*}
 \frac{\d}{\d x_2} \begin{pmatrix}
   v^1\\
   v^2\\
   \partial_2 v^1 \\
   \partial_2 v^2
  \end{pmatrix}
  = \begin{pmatrix}
 0 & 0 & 1 & 0 \\
 0 & 0 & 0 & 1 \\
 r^2_1 & 0 & 0 & -\ii\xi_1 a \\
 0 & r^2_2 & - \ii \xi_1 b & 0 
\end{pmatrix}
\begin{pmatrix}
 v^1\\
   v^2\\
   \partial_2 v^1 \\
   \partial_2 v^2
\end{pmatrix}
\end{align*}
where 
\begin{equation}
 r^2_1=\frac{1}{\mu} (\ii\tau + \nu\xi_1^2), \qquad r^2_2=\frac{1}{\nu}
 (\ii\tau + \mu \xi_1^2), \qquad a=\frac{\mu+\lambda}{\mu}, \qquad b=\frac{\mu+\lambda}\nu\cdotp
\end{equation}
Except in the particular case  $\lambda+\mu=0$
(that will not be treated here since we assume that 
$\lambda+\mu>0$) or $\tau=0,$ the above matrix has four distinct 
 eigenvalues $\pm\lambda_1$ and $\pm\lambda_2,$ given by 
\begin{equation}
 \lambda^2_1= \frac{\ii}{\nu} \tau + \xi_1^2\andf
 \lambda^2_2=\frac{\ii}{\mu} \tau +\xi_1^2.
\end{equation}
We agree that $\lambda_1$ and $\lambda_2$ have non-negative real parts, 
and we only consider solutions that vanish at $x_2\to +\infty.$ 
Hence, those solutions  are  valued in the sum  of eigenspaces for 
 $-\lambda_1$ and $- \lambda_2.$
 
 A highly informative observation is that the corresponding eigenvectors 
 are determined by the 
system for the vorticity and divergence, and are thus  quite fast to compute.
To proceed, let us first determine  an eigenvector  $L_1=(l_1,l_2,l_3,l_4)$
for $-\lambda_1.$  Then we have to solve
\begin{equation}
 l_1 \lambda_1 + l_3=0, \qquad l_2 \lambda_1 + l_4=0, \qquad l_1r^2_1 + l_3 \lambda_1 - \ii \xi_1 a l_4=0.
\end{equation}
Noting that
$$
r^2_1-\lambda_1^2=\frac{\ii\tau}{\mu}+\frac{\nu}{\mu}\xi_1^2 - \frac{\ii\tau}{\nu} - \xi_1^2= \frac{\mu+\lambda}{\mu\nu} \ii\tau+\frac{\mu+\lambda}{\mu}\xi_1^2=a\lambda^2_1,
$$
we find that one can take 
\begin{equation}
 L_1=(\ii\xi_1,-\lambda_1,-\ii\xi_1 \lambda_1,\lambda^2_1).
\end{equation}
For the second eigenvector, we use 
\begin{equation}
 r^2_2-\lambda_2^2= -b \lambda_2^2,
\end{equation}
and find that 
\begin{equation}
 L_2=(\lambda_2,\ii\xi_1,-\lambda^2_2,-\ii\xi_1 \lambda_2).
\end{equation}

So the general form of the solutions we are looking for reads
\begin{align}
 \begin{pmatrix}
   v^1\\
   v^2\\
   \partial_2 v^1 \\
   \partial_2 v^2
  \end{pmatrix} = 
\begin{pmatrix}
 \ii\xi_1 \\
 -\lambda_1 \\
 -\ii\xi_1 \lambda_1 \\
 \lambda^2_1
\end{pmatrix}
 \phi^1 \e^{-\lambda_1 x_2} +
\begin{pmatrix}
 \lambda_2 \\
 \ii\xi_1 \\
 -\lambda^2_2 \\
 -\ii\xi_1 \lambda_2
\end{pmatrix}
\phi^2 \e^{-\lambda_2 x_2}.
\end{align}
Note  that the first  part is curl-free, while  the second one is divergence-free. 
\medbreak
Plugging that ansatz in~\eqref{eq:blame} the boundary equation in~\eqref{eq:lame} yields
\begin{align*}
 \cA \begin{pmatrix} \phi^1  \\ \phi^2 \end{pmatrix} := \begin{pmatrix} - 2 \ii \xi_1 \lambda_1 \mu & - \mu (\xi_1^2 + \lambda_2^2) \\ \nu\lambda_1^2 - \lambda \xi_1^2 & - 2 \mu \ii \xi_1 \lambda_2 \end{pmatrix} \begin{pmatrix} \phi^1 \\ \phi^2 \end{pmatrix} = \begin{pmatrix} \cF g_1 \\ \cF g_2 \end{pmatrix}\cdotp
\end{align*}
Thus, the solution $u = (u^1 , u^2)$ to the Lam\'e system is formally given by
\begin{align*}
 \begin{pmatrix} u^1 \\ u^2 \end{pmatrix} = \cF^{-1}_{\tau , \xi_1} \bigg[ \frac{\cB}{\det(\cA)} \bigg] \begin{pmatrix} \cF g_1 \\ \cF g_2 \end{pmatrix},
\end{align*}
with $\cB$ being given by
\begin{align*}
 \cB = \begin{pmatrix} 2 \mu \xi_1^2 \lambda_2 \e^{- \lambda_1 x_2} + \lambda_2 (\lambda \xi_1^2 - \nu \lambda_1^2) \e^{- \lambda_2 x_2} & \ii \mu \xi_1 (\xi_1^2 + \lambda_2^2) \e^{- \lambda_1 x_2} - 2 \mu \ii \xi_1 \lambda_1 \lambda_2 \e^{- \lambda_2 x_2} \\ \ii \xi_1 (\lambda \xi_1^2 - \nu \lambda_1^2) \e^{- \lambda_2 x_2} + 2 \mu \ii \xi_1 \lambda_1 \lambda_2 \e^{- \lambda_1 x_2} & 2 \mu \xi_1^2 \lambda_1 \e^{- \lambda_2 x_2} - \mu \lambda_1 (\xi_1^2 + \lambda_2^2) \e^{- \lambda_1 x_2} \end{pmatrix}\cdotp
\end{align*}
To investigate the determinant of the matrix $\cA$ one calculates 
\begin{align*}
 \det (\cA) = \mu (\xi_1^2 + \lambda_2^2) \big( \nu\lambda_1^2 - \lambda \xi_1^2 \big) - 4 \mu^2 \xi_1^2 \lambda_1 \lambda_2.
\end{align*}
Observe that
\begin{align*}
 \nu\lambda_1^2 - \lambda \xi_1^2 = \ii \tau + 2 \mu \xi_1^2= \mu(\xi_1^2 + \lambda_2^2).
\end{align*}
Hence
$$ \det (\cA) = (\ii \tau + 2 \mu \xi_1^2)^2 - 4 \mu^2 \xi_1^2 \lambda_1 \lambda_2.
 $$
 We have to keep in mind that, by definition and elementary trigonometry, both
 $\arg(\lambda_1)$ and $\arg(\lambda_2)$ are in $[-\pi/4,\pi/4].$

\begin{proposition} The determinant of $\cA$ vanishes if and only if $\tau=0,$ and 
there exist two positive real numbers $c<C$ 
 such that, for all $(\tau,\xi_1)\in\IR^2,$ we have
\begin{equation}\label{eq:detA}
c\,|\tau|\,\sqrt{\tau^2+\mu^2\xi_1^4}\leq |\det (\cA)| \leq 
C (\tau^2+\mu^2\xi_1^4).\end{equation}
\end{proposition}

\begin{proof}
Ruling out the obvious case $\tau=0,$ let us 
set $r=(\tau^2+\mu^2\xi_1^4)^{1/4},$ $\zeta_0:=\tau/r^2$
and $\zeta:=\xi_1/{r}.$ Then, we have
$$
 \det (\cA) = r^4\Bigl((\ii \zeta_0 + 2 \mu \zeta^2)^2 - 4 \mu^2 \zeta^2 \mu_1 \mu_2\Bigr)
 \quad\hbox{with}\quad \ \mu_1^2:= \ii\nu^{-1}\zeta_0+\zeta^2
 \quad\hbox{and}\quad
 \mu_2^2:=\ii\mu^{-1}\zeta_0+\zeta^2.$$
 It is obvious that the function  $(\zeta_0,\zeta)\mapsto (\ii \zeta_0 + 2 \mu \zeta^2)^2
  - 4 \mu^2 \zeta^2 \mu_1 \mu_2$ is  bounded (by an absolute constant)  on the 
  set  $\bigl\{\zeta_0^2+\mu^2\zeta^4=1\bigr\},$  and 
  we thus get the right inequality of~\eqref{eq:detA}. 
  Next,  let us   prove  that  $$ \det (\cA)  =0\ \Longrightarrow \ \tau=0.$$
Indeed, if $ \det (\cA) $ vanishes then 
 $$
 (\ii \zeta_0 + 2 \mu \zeta^2)^2=4 \mu^2 \zeta^2 \mu_1 \mu_2.
 $$
 Let us square both sides and compute (using the value of $\mu_1^2\mu_2^2$):
 $$
 \zeta_0^4-8\ii\mu\zeta_0^3\zeta^2-24\mu^2\zeta_0^2\zeta^4+32\ii\mu^3\zeta_0\zeta^6
 +16\mu^4\zeta^8=16\mu^4\zeta^8+16\ii\mu^4(\mu^{-1}+\nu^{-1})\zeta_0\zeta^6
-16\mu^3\nu^{-1}\zeta_0^2\zeta^4.$$
 Identifying real and imaginary parts, we eventually  find that either $\zeta_0=0,$ or
$$\zeta_0\not=0\andf 2\bigl(1-\mu/\nu)\mu^2\zeta^4=\zeta_0^2=8(3-2\mu/\nu)\mu^2\zeta^4.$$
This may occur only  if $\mu/\nu=11/7,$ but as $\lambda+\mu>0$ and $\mu > 0$, this cannot happen.  
\medbreak
From this and the compactness of the set $\bigl\{\zeta_0^2+\mu^2\zeta^4=1,\ 
|\zeta_0|\geq \mu |\zeta|^2 / 10 \bigr\},$
 we readily deduce that, whenever (say) $|\tau|\geq \mu|\xi_1|^2/10,$ 
 the left inequality of~\eqref{eq:detA} is satisfied. 
 \medbreak
 To handle the case  $|\tau|\leq \mu|\xi_1|^2/10$ and $(\tau,\xi_1)\not=(0,0),$  
 one can notice that 
 $$ \det (\cA) = \frac{(\ii \tau + 2 \mu \xi_1^2)^4 - 16 \mu^4 \xi_1^4 \lambda_1^2 \lambda_2^2}
 {(\ii \tau + 2 \mu \xi_1^2)^2 +  4 \mu^2 \xi_1^2 \lambda_1 \lambda_2}\cdotp $$
 As   $|\tau|\leq \mu|\xi_1|^2/10,$ the denominator may be bounded  by $C\mu^2\xi_1^4.$
 Moreover, we have 
 $$ (\ii \tau + 2 \mu \xi_1^2)^4 - 16 \mu^4 \xi_1^4 \lambda_1^2 \lambda_2^2
 =\biggl(\tau^3-8\ii\mu\tau^2\xi_1^2+8\mu^2\Bigl(\frac{2\mu}\nu-3\Bigr)\tau\xi_1^4
 -16\ii\mu^3\Bigl(3+\frac\mu\nu\Bigr)\xi_1^6\biggr)\tau.$$
   Hence there exists $c>0$ such that 
   $$   \bigl| (\ii \tau + 2 \mu \xi_1^2)^4 - 16 \mu^4 \xi_1^4 \lambda_1^2 \lambda_2^2\bigr|\geq 
   c\mu^3 |\tau|\,|\xi_1|^6.$$
   This completes  the proof of the left inequality of~\eqref{eq:detA}.
\end{proof}

\section{Resolvent estimates for the Neumann Laplacian}

In this section we provide the proof of the higher regularity estimates for the resolvent problem of the Neumann Laplacian that are needed in Section~\ref{Sec: The functional setting and basic interpolation results}. We start with the well-known whole space result.


\begin{lemma}
\label{Lem: Laplacian on the whole space}
Let $n \geq 1$, $1 < p < \infty$, and $\theta \in (0 , \pi)$. For all $\lambda \in \Sec_{\theta}$ and $f \in \LL_p (\IR^n),$ there exists a unique solution $u \in \H^2_p (\IR^n)$ to $\lambda u - \Delta u = f$. Furthermore, there exists a constant $C > 0$ that is independent of $\lambda$, $u$, and $f$ such that
\begin{align*}
 \lvert \lambda \rvert \| u \|_{\LL_p (\IR^n)} + \lvert \lambda \rvert^{\frac{1}{2}} \| \nabla u \|_{\LL_p (\IR^n ; \IC^n)} + \| \nabla^2 u \|_{\LL_p (\IR^n ; \IC^{n^2})} \leq C \| f \|_{\LL_p (\IR^n)}.
\end{align*}
\end{lemma}

If $f \in \H^k_p (\IR^n)$ with $k \in \IN$, then again Marcinkiewicz's theorem implies the following corollary.  


\begin{corollary}
\label{Cor: Laplacian on whole space}
Let $n \geq 1$, $1 < p < \infty$, $k \in \IN_0$, and $\theta \in (0 , \pi)$. For all $\lambda \in \Sec_{\theta}$ and  $f \in \H^k_p (\IR^n),$ there exists a unique solution $u \in \H^{k + 2}_p (\IR^n)$ to $\lambda u - \Delta u = f$. Furthermore, there exists a constant $C > 0$ that is independent of $\lambda$, $u$, and $f$ such that for all $\alpha \in \IN_0^n$ with $\lvert \alpha \rvert \leq k,$ we have 
\begin{align*}
 \lvert \lambda \rvert \| \partial^{\alpha} u \|_{\LL_p (\IR^n)} + \lvert \lambda \rvert^{\frac{1}{2}} \| \nabla \partial^{\alpha} u \|_{\LL_p (\IR^n ; \IC^n)} + \| \nabla^2 \partial^{\alpha} u \|_{\LL_p (\IR^n ; \IC^{n^2})} \leq C \| \partial^{\alpha} f \|_{\LL_p (\IR^n)}.
\end{align*}
\end{corollary}
Let us now consider  the Neumann Laplacian on the half-space supplemented with nonzero  Neumann data, namely
\begin{align}\label{Eq: Poisson problem half-space}
 \left\{
\begin{aligned}
 \lambda v - \Delta v &= f && \text{in } \IR^n_+ \\
 \partial_n v &= G|_{\partial \IR^n_+} && \text{on } \partial \IR^n_+.
\end{aligned}
\right.
\end{align}
It is well known that Lemma~\ref{Lem: Laplacian on the whole space} may be adapted to~\eqref{Eq: Poisson problem half-space}.
Our aim is to prove a higher regularity result in the spirit of Corollary 
\ref{Cor: Laplacian on whole space}:
\begin{proposition}
\label{Prop: Laplacian on half-space}
Let $n \geq 1$, $1 < p < \infty$, and $\theta \in (0 , \pi)$. For all $\lambda \in \Sec_{\theta},$  $f \in \H^1_p (\IR^n_+)$ and $G\in \H^2_p (\IR^n_+),$ there exists a unique solution $v \in \H^3_p (\IR^n)$ to~\eqref{Eq: Poisson problem half-space}.
Furthermore, there exists a constant $C > 0$ that is independent of $\lambda$, $v$, $f$ and $G,$ such that
$$\displaylines{\quad
 \lvert \lambda \rvert \| \nabla v \|_{\LL_p (\IR^n_+ ; \IC^n)} + \lvert \lambda \rvert^{\frac{1}{2}} \| \nabla^2 v \|_{\LL_p (\IR^n_+ ; \IC^{n^2})} + \| \nabla^3 v \|_{\LL_p (\IR^n_+ ; \IC^{n^3})} \hfill\cr\hfill \leq C \Big( \| \nabla f \|_{\LL_p (\IR^n_+ ; \IC^n)} + \lvert \lambda \rvert \| G \|_{\LL_p (\IR^n_+)} + \lvert \lambda \rvert^{1 / 2} \| \nabla G \|_{\LL_p(\IR^n_+ ; \IC^n)} + \| \nabla^2 G \|_{\LL_p (\IR^n_+ ; \IC^{n^2})} \Big)\cdotp\quad}$$
\end{proposition}

\begin{proof}
We reduce the problem to the case  $f = 0$, by extending $f$ to a function $Ef \in \H^1_p (\IR^n)$ by an even reflection. Note that the reflection operator satisfies the homogeneous estimate
\begin{align*}
 \| \nabla E f \|_{\LL_p (\IR^n ; \IC^n)} \leq C \| \nabla f \|_{\LL_p (\IR^n_+ ; \IC^n)} \qquad (f \in \H^1_p (\IR^n)).
\end{align*}
Now, let $U \in \H^3_p (\IR^n)$ be the solution to $\lambda U - \Delta U = E f$ provided by Corollary~\ref{Cor: Laplacian on whole space}. Notice that $\nabla U \in \H^2_p (\IR^n)$ and define $H := \partial_n U - G$. Then, denoting by  $h$  the restriction of $H$ to $\partial \IR^n_+,$
 the solution $v$ to~\eqref{Eq: Poisson problem half-space} is given by $v = U - u$, where $u$ is the solution to
\begin{align}\label{Eq: Neumann problem half-space}
 \left\{
\begin{aligned}
 \lambda u - \Delta u &= 0 && \text{in } \IR^n_+ \\
 \partial_n u &= h && \text{on } \partial \IR^n_+.
\end{aligned}
\right.
\end{align}
 Thus, the investigation of problem~\eqref{Eq: Poisson problem half-space} is reduced to~\eqref{Eq: Neumann problem half-space}, with data $h$ being the restriction of a function $H \in \H^2_p (\IR^n_+)$.  Applying the Fourier transform with respect to the first $(n - 1)$ variables shows that the solution $u$ to~\eqref{Eq: Neumann problem half-space} is given by
\begin{align*}
 u(x) = - \cF_{\xi^{\prime}}^{-1} \bigg[ \frac{\e^{- B x_n}}{B} [\cF_{y^{\prime}} h(y^{\prime})] (\xi^{\prime}) \bigg] (x^{\prime})
  \with B^2 := \lambda + |\xi'|^2\quad\hbox{such that}\ \Re B>0.
\end{align*}
We find that
$$
 \partial_n u = \cF_{\xi^{\prime}}^{-1} \big[ \e^{- B x_n} [\cF_{y^{\prime}} h]\big]
\andf\partial_j u = - \cF_{\xi^{\prime}}^{-1} \bigg[ \frac{\e^{- B x_n}}{B} [\cF_{y^{\prime}} \partial_{y_j} h] \bigg] 
\ \hbox{ for }\ j \in \{ 1 , \dots , n - 1 \}.$$
 On the right-hand sides of these formulas, we use the fundamental theorem of calculus together with the fact that $h$ is the restriction to $\partial \IR^n_+$ of the function $H$ to write
\begin{align}
\label{Eq: Representation of solution of Laplacian}
\begin{aligned}
u(x) &= \int_0^{\infty} \frac{\d}{\d y_n} \cF_{\xi^{\prime}}^{-1} \bigg[\frac{\e^{- B (x_n + y_n)}}{B} [\cF_{y^{\prime}} H(y^{\prime} , y_n)] (\xi^{\prime}) \bigg] (x^{\prime}) \; \d y_n \\
 &= - \int_0^{\infty} \cF_{\xi^{\prime}}^{-1} \Big[ \e^{- B (x_n + y_n)} [\cF_{y^{\prime}} H(y^{\prime} , y_n)] (\xi^{\prime}) \Big] (x^{\prime}) \; \d y_n \\
 &\hspace{3cm}+ \int_0^{\infty} \cF_{\xi^{\prime}}^{-1} \bigg[ \frac{\e^{- B (x_n + y_n)}}{B} [\cF_{y^{\prime}} \partial_{y_n} H(y^{\prime} , y_n)] (\xi^{\prime}) \bigg] (x^{\prime}) \; \d y_n.
\end{aligned}
\end{align}
This implies that  for $j \in \{ 1 , \dots , n - 1 \},$ 
$$\displaylines{ \partial_j u(x) = - \int_0^{\infty} \cF_{\xi^{\prime}}^{-1} \Big[ \e^{- B (x_n + y_n)} [\cF_{y^{\prime}} \partial_{y_j} H(y^{\prime} , y_n)] (\xi^{\prime}) \Big] (x^{\prime}) \; \d y_n \hfill\cr\hfill
+ \int_0^{\infty} \cF_{\xi^{\prime}}^{-1} \bigg[ \frac{\e^{- B (x_n + y_n)}}{B} [\cF_{y^{\prime}} \partial_{y_j} \partial_{y_n} H(y^{\prime} , y_n)] (\xi^{\prime}) \bigg] (x^{\prime}) \; \d y_n.}$$
Using that $B^2= \lambda+|\xi'|^2,$ we get
\begin{equation}\label{Eq: Representation of tangential derivative}
 \partial_j u(x) =: \lambda \cdot \mathrm{I}_{\lambda} (H) + \mathrm{II}_{\lambda} (\Delta_{y^{\prime}} H) + \mathrm{III}_{\lambda} (\partial_{y_j} \partial_{y_n} H)
\end{equation}
with
$$\begin{aligned} \mathrm{I}_{\lambda} (H) &:= - \int_0^{\infty} \cF_{\xi^{\prime}}^{-1} \bigg[ \ii \xi_j \frac{\e^{- B (x_n + y_n)}}{B^2} [\cF_{y^{\prime}} H(y^{\prime} , y_n)] (\xi^{\prime}) \bigg] (x^{\prime}) \; \d y_n, \\
 \mathrm{II}_{\lambda} (\Delta_{y^{\prime}} H)&:= \int_0^{\infty} \cF_{\xi^{\prime}}^{-1} \bigg[ \ii \xi_j \frac{\e^{- B (x_n + y_n)}}{B^2} [\cF_{y^{\prime}} \Delta_{y^{\prime}} H(y^{\prime} , y_n)] (\xi^{\prime}) \bigg] (x^{\prime}) \; \d y_n \\
 \andf  \mathrm{III}_{\lambda} (\partial_{y_j} \partial_{y_n} H)&:=\int_0^{\infty} \cF_{\xi^{\prime}}^{-1} \bigg[ \frac{\e^{- B (x_n + y_n)}}{B} [\cF_{y^{\prime}} \partial_{y_j} \partial_{y_n} H(y^{\prime} , y_n)] (\xi^{\prime}) \bigg] (x^{\prime}) \; \d y_n.
\end{aligned}
$$
Analogously, we have
$$ \begin{aligned}\partial_n u(x) &:=  \lambda \cdot \mathrm{IV}_{\lambda} (H) + \mathrm{V}_{\lambda} (\Delta_{y^{\prime}} H) + \mathrm{VI}_{\lambda} (\partial_{y_n} H)\\
\with  \mathrm{IV}_{\lambda} (H)&:= \int_0^{\infty} \cF_{\xi^{\prime}}^{-1} \bigg[ \frac{\e^{- B (x_n + y_n)}}{B} [\cF_{y^{\prime}} H(y^{\prime} , y_n)] (\xi^{\prime}) \bigg] (x^{\prime}) \; \d y_n, \\
 \mathrm{V}_{\lambda} (\Delta_{y^{\prime}} H) &:= - \int_0^{\infty} \cF_{\xi^{\prime}}^{-1} \bigg[ \frac{\e^{- B (x_n + y_n)}}{B} [\cF_{y^{\prime}} \Delta_{y^{\prime}} H(y^{\prime} , y_n)] (\xi^{\prime}) \bigg] (x^{\prime}) \; \d y_n \\
  \andf\mathrm{VI}_{\lambda} (\partial_{y_n} H)&:=- \int_0^{\infty} \cF_{\xi^{\prime}}^{-1} \bigg[ \e^{- B (x_n + y_n)} [\cF_{y^{\prime}} \partial_{y_n} H(y^{\prime} , y_n)] (\xi^{\prime}) \bigg] (x^{\prime}) \; \d y_n.\end{aligned}$$
To proceed we have to prove $\LL_p$  estimates for  $\mathrm{I}_{\lambda}$, $\mathrm{II}_{\lambda}$, $\mathrm{III}_{\lambda}$, $\mathrm{IV}_{\lambda}$, $\mathrm{V}_{\lambda}$, and $\mathrm{VI}_{\lambda}$. \par
\smallbreak
To do so, introduce the following classes of multipliers. Let $\theta \in (\pi / 2 , \pi)$. 
The  imaginary part of  elements $\lambda \in \Sec_{\theta}$ is denoted by $\tau$. For $s \in \IR$ we say that $m : \Sec_{\theta} \times \IR^{n - 1} \setminus \{ 0 \} \to \IC$ is a multiplier of order $s$ and of type $1$ if $m$ is smooth with respect to $\tau=\Im\lambda$ and $\xi^{\prime}$ (with $\lambda\in\Sec_{\theta}$ and $\xi^{\prime}\in\IR^{n-1}$), and if there exist constants $C_1 , C_2 > 0$ such that for all $\lambda \in \Sec_{\theta}$, $\xi^{\prime} \in \IR^{n - 1} \setminus \{ 0 \}$ and $\alpha \in \IN_0^{n - 1},$ it holds
\begin{align*}
 \big\lvert \partial_{\xi^{\prime}}^{\alpha} m (\lambda , \xi^{\prime}) \big\rvert \leq C_1 (\lvert \lambda \rvert^{1 / 2} + |\xi'|)^{s - \lvert \alpha \rvert} \quad \text{and} \quad \big\lvert \partial_{\xi^{\prime}}^{\alpha} \tau \partial_{\tau} m (\lambda , \xi^{\prime}) \big\rvert \leq C_2 (\lvert \lambda \rvert^{1 / 2} + |\xi'|)^{s - \lvert \alpha \rvert}.
\end{align*}
In this case, we write $m \in {\bf M}_{s , 1 , \theta},$ see~\cite[Sec.~5]{Shibata_Shimizu} for more information. It is proven in~\cite[Lem.~5.2]{Shibata_Shimizu} that $B^s \in {\bf M}_{s , 1 , \theta}$ for all $s \in \IR$. Beside, a simple calculation shows that for $r \geq 0,$
\begin{align*}
 (\lambda , \xi^{\prime}) \mapsto \xi_j \in {\bf M}_{1 , 1 , \theta} \quad \text{and} \quad (\lambda , \xi^{\prime}) \mapsto \lvert \lambda \rvert^r \in {\bf M}_{2r , 1 , \theta} \quad \text{and} \quad (\lambda , \xi^{\prime}) \mapsto 1 \in {\bf M}_{0 , 1 , \theta}.
\end{align*}
 The multiplication rule~\cite[Lem.~5.1]{Shibata_Shimizu} shows that $\xi_j B^s \in {\bf M}_{s + 1 , 1 , \theta}$ for $s \in \IR$ and $\lvert \lambda \rvert^r B^s \in {\bf M}_{2r + s, 1 , \theta}$. This property permits to apply~\cite[Lem.~5.6]{Shibata_Shimizu} to conclude that there exists a constant $C > 0$ independent of $\lambda$ and $H$ such that
 $$\displaylines{
 \lvert \lambda \rvert \| \mathrm{I}_{\lambda} (H) \|_{\LL_p (\IR^n_+)} + \lvert \lambda \rvert^{1 / 2} \| \nabla \mathrm{I}_{\lambda} (H) \|_{\LL_p (\IR^n_+ ; \IC^n)} + \| \nabla^2 \mathrm{I}_{\lambda} (H) \|_{\LL_p (\IR^n_+ ; \IC^{n^2})} \leq C \| H \|_{\LL_p (\IR^n_+)}, \cr
 \lvert \lambda \rvert \| \mathrm{II}_{\lambda} (\Delta_{y^{\prime}} H) \|_{\LL_p (\IR^n_+)} +  \lvert \lambda \rvert^{1 / 2} \|\nabla \mathrm{II}_{\lambda} (\Delta_{y^{\prime}} H) \|_{\LL_p (\IR^n_+ ; \IC^n)} + \| \nabla^2 \mathrm{II}_{\lambda} (\Delta_{y^{\prime}} H) \|_{\LL_p (\IR^n_+ ; \IC^{n^2})} \cr
\hfill \leq C \| \Delta_{y^{\prime}} H \|_{\LL_p (\IR^n_+)}\cr
 \hbox{and }\ 
 \lvert \lambda \rvert \| \mathrm{III}_{\lambda} (\partial_{y_j} \partial_{y_n} H) \|_{\LL_p (\IR^n_+)} + \lvert \lambda \rvert^{1 / 2} \| \nabla \mathrm{III}_{\lambda} (\partial_{y_j} \partial_{y_n} H) \|_{\LL_p (\IR^n_+ ; \IC^n)} + \| \nabla^2 \mathrm{III}_{\lambda} (\partial_{y_j} \partial_{y_n} H) \|_{\LL_p (\IR^n_+ ; \IC^{n^2})} \cr
\hfill \leq C \| \partial_{y_j} \partial_{y_n} H \|_{\LL_p (\IR^n_+)}.}$$
 Combining this with~\eqref{Eq: Representation of tangential derivative} shows that for all $j \in \{ 1 , \dots , n - 1 \}$
\begin{multline}
\label{Eq: Tangential estimate}
 \lvert \lambda \rvert \| \partial_j u \|_{\LL_p (\IR^n_+)} + \lvert \lambda \rvert^{1 / 2} \| \nabla \partial_j u \|_{\LL_p (\IR^n_+ ; \IC^n)}+ \| \nabla^2 \partial_j u \|_{\LL_p (\IR^n_+ ; \IC^{n^2})} \\
 \leq C \Big( \lvert \lambda \rvert \| H \|_{\LL_p (\IR^n_+)} + \| \nabla^2 H \|_{\LL_p (\IR^n_+ ; \IC^{n^2})} \Big)\cdotp
\end{multline}
To complete the proof, only bounds on $\lvert \lambda \rvert \partial_n u$, $\lvert \lambda \rvert^{1 / 2} \partial_n \partial_n u$, and $\partial_n \partial_n \partial_n u$ have to be established. Now,  since
\begin{align*}
 \partial_n \partial_n \partial_n u = \lambda \partial_n u - \Delta_{y^{\prime}} \partial_n u,
\end{align*}
the bound on $\partial_n \partial_n \partial_n u$ follows from that on $\lvert \lambda \rvert \partial_n u$. This latter bound stems  from~\cite[Lem.~5.4]{Shibata_Shimizu} since $1 , \lvert \lambda \rvert B^{-1} \in {\bf M}_{0 , 1 , \theta}$. Indeed, this lemma provides the estimates
\begin{align*}
 \lvert \lambda \rvert \| \mathrm{IV}_{\lambda} (H) \|_{\LL_p (\IR^n_+)} &\leq C \| H \|_{\LL_p (\IR^n_+)}\\
  \lvert \lambda \rvert \| \mathrm{V}_{\lambda} (\Delta_{y^{\prime}} H) \|_{\LL_p (\IR^n_+)} &\leq C \| \Delta_{y^{\prime}} H \|_{\LL_p (\IR^n_+)} \\
 \lvert \lambda \rvert^{1 / 2} \| \mathrm{VI}_{\lambda} (\partial_{y_n} H) \|_{\LL_p (\IR^n_+)} &\leq C \| \partial_{y_n} H \|_{\LL_p (\IR^n_+)}.
\end{align*}
Altogether, this yields
\begin{align*}
 \lvert \lambda \rvert \| \partial_n u \|_{\LL_p (\IR^n_+)} \leq C \Big( \lvert \lambda \rvert \| H \|_{\LL_p (\IR^n_+)} + \lvert \lambda \rvert^{1 / 2} \| \nabla H \|_{\LL_p (\IR^n_+ ; \IC^n)} + \| \nabla^2 H \|_{\LL_p (\IR^n_+ ; \IC^{n^2})} \Big)\cdotp
\end{align*}
The only term that remains is $\lvert \lambda \rvert^{1 / 2} \partial_n \partial_n u$. This can be written as
\begin{align*}
 \lvert \lambda \rvert^{1 / 2} \partial_n \partial_n u = \lambda \lvert \lambda \rvert^{1 / 2} u - \lvert \lambda \rvert^{1 / 2} \Delta_{y^{\prime}} u.
\end{align*}
Since the term $\lvert \lambda \rvert^{1 / 2} \Delta_{y^{\prime}} u$ can already be controlled by~\eqref{Eq: Tangential estimate} it remains to estimate $\lambda \lvert \lambda \rvert^{1 / 2} u$. By virtue of~\eqref{Eq: Representation of solution of Laplacian} the same arguments that were used before result in the estimate
\begin{align*}
 \lvert \lambda \rvert^{1 / 2} \| u \|_{\LL_p (\IR^n_+)} \leq C \Big( \| H \|_{\LL_p (\IR^n_+)} + \lvert \lambda \rvert^{- 1 / 2} \| \nabla H \|_{\LL_p (\IR^n_+ ; \IC^n)} \Big)\cdotp
\end{align*}
Altogether, we obtain
\begin{align*}
 \lvert \lambda \rvert^{1 / 2} \| \partial_n \partial_n u \|_{\LL_p (\IR^n_+)} \leq C \Big( \lvert \lambda \rvert \| H \|_{\LL_p (\IR^n_+)} + \lvert \lambda \rvert^{1 / 2} \| \nabla H \|_{\LL_p (\IR^n_+ ; \IC^n)} + \| \nabla^2 H \|_{\LL_p (\IR^n_+ ; \IC^{n^2})} \Big)\cdotp
\end{align*}
This concludes the proof.
\end{proof}

\end{document}